\documentclass[a4paper, 11pt]{article}
\usepackage{comment} 
\usepackage{lipsum} 
\usepackage{fullpage} 

\usepackage[pdftex,plainpages=false,hypertexnames=false,pdfpagelabels]{hyperref}
\newcommand{\arxiv}[1]{\href{http://arxiv.org/abs/#1}{\tt arXiv:\nolinkurl{#1}}}
\newcommand{\arXiv}[1]{\href{http://arxiv.org/abs/#1}{\tt arXiv:\nolinkurl{#1}}}
\newcommand{\doi}[1]{\href{http://dx.doi.org/#1}{{\tt DOI:#1}}}

\newcommand{\mathscinet}[1]{\href{http://www.ams.org/mathscinet-getitem?mr=#1}{\tt
#1}}
\newcommand{\googlebooks}[1]{(preview at
\href{http://books.google.com/books?id=#1}{google books})}

\usepackage{xcolor}
\definecolor{medium-blue}{rgb}{0,0,.8}
\hypersetup{
   colorlinks, linkcolor={purple},
   citecolor={medium-blue}, urlcolor={medium-blue}
}

\usepackage{url}
\usepackage{abstract}
\usepackage{indentfirst}
\usepackage{ulem}
\usepackage{cite}
\usepackage{setspace}
\usepackage{latexsym}
\usepackage{amssymb,amsthm,amsmath}
\usepackage{graphicx}
\usepackage{paralist}
\usepackage{enumitem}
\usepackage[toc,page,title,titletoc,header]{appendix}
\usepackage{listings}
\usepackage[mathscr]{euscript}
\usepackage{bibtopic}

\usepackage{xypic}
\usepackage{tabu}

\usepackage{xcolor}
\definecolor{dark-red}{rgb}{0.7,0.25,0.25}
\definecolor{dark-blue}{rgb}{0.15,0.15,0.55}
\definecolor{medium-blue}{rgb}{0,0,.8}
\definecolor{shaded-blue}{RGB}{98,140,255}
\definecolor{DarkGreen}{RGB}{0,150,0}
\definecolor{rho}{named}{red}
\definecolor{Salmon}{RGB}{255, 144, 144}
\definecolor{pink}{RGB}{255,192,203}
\definecolor{gray1}{RGB}{225,225,225}
\definecolor{gray2}{RGB}{192,192,192}
\definecolor{gray3}{RGB}{128,128,128}
\definecolor{green1}{RGB}{75,200,75}

\DeclareMathOperator{\op}{op}
\usepackage{longtable}
\usepackage{fullpage}

\usepackage{blindtext}
\usepackage{lipsum}

\usepackage{tikz}
   \usetikzlibrary{arrows,backgrounds}
\usepackage[all]{xy}
\usetikzlibrary{shapes}
\usetikzlibrary{cd}
\usetikzlibrary{backgrounds}
\usetikzlibrary{decorations,decorations.pathreplacing,decorations.markings}
\usetikzlibrary{fit,calc,through}
\usetikzlibrary{external}
\usetikzlibrary{arrows}
\tikzset{vertex/.style = {shape=circle,draw,fill=black,inner sep=0pt,minimum size=5pt}}
\tikzset{edge/.style = {->,> = latex', bend right}}
\tikzset{
 super thick/.style={line width=3pt}
}
\tikzset{
    quadruple/.style args={[#1] in [#2] in [#3] in [#4]}{
        #1,preaction={preaction={preaction={draw,#4},draw,#3}, draw,#2}
    }
}
\tikzstyle{shaded}=[fill=gray!25!white]
\tikzstyle{shadedpink}=[left color= white, right color = Salmon]
\tikzstyle{unshaded}=[fill=white]
\tikzstyle{empty box}=[circle, draw, thick, fill=white, opaque, inner sep=2mm]
\tikzstyle{annular}=[scale=.7, inner sep=1mm, baseline]
\tikzstyle{rectangular}=[scale=.75, inner sep=1mm, baseline=-.1cm]
\tikzstyle{mid>}=[decoration={markings, mark=at position 0.5 with {\arrow{>}}}, postaction={decorate}]
\tikzstyle{mid<}=[decoration={markings, mark=at position 0.5 with {\arrow{<}}}, postaction={decorate}]
\tikzstyle{over}=[double, draw=white, super thick, double=]
\tikzstyle{relativecommutantshading}=[fill=blue!20!white]
\tikzstyle{ctwoshading}=[fill=red!15!white]
\tikzstyle{Bshading}=[fill=blue!15!white]

\newcommand{\roundNbox}[6]{
 \draw[rounded corners=5pt,  thick, #1] ($#2+(-#3,-#3)+(-#4,0)$) rectangle ($#2+(#3,#3)+(#5,0)$);
 \coordinate (ZZa) at ($#2+(-#4,0)$);
 \coordinate (ZZb) at ($#2+(#5,0)$);
 \node at ($1/2*(ZZa)+1/2*(ZZb)$) {#6};
}

\newcommand{\nbox}[6]{
 \draw[thick, #1] ($#2+(-#3,-#3)+(-#4,0)$) rectangle ($#2+(#3,#3)+(#5,0)$);
 \coordinate (ZZa) at ($#2+(-#4,0)$);
 \coordinate (ZZb) at ($#2+(#5,0)$);
 \node at ($1/2*(ZZa)+1/2*(ZZb)$) {#6};
}
\newcommand{\dbox}[6]{
\filldraw[white] ($#2+(-#3,-#3)+(-#4,0)$) rectangle ($#2+(#3,#3)+(#5,0)$);
 \draw[thick, #1] ($#2+(-#3,-#3)+(-#4,0)$) -- ($#2+(#3,-#3)+(#5,0)$);
 \draw[thick, #1] ($#2+(#3,-#3)+(#5,0)$) -- ($#2+(#3,#3)+(#5,0)$);
 \draw[thick, #1] ($#2+(-#3,#3)+(-#4,0)$) -- ($#2+(#3,#3)+(#5,0)$);
 \coordinate (ZZa) at ($#2+(-#4,0)$);
 \coordinate (ZZb) at ($#2+(#5,0)$);
 \node at ($1/2*(ZZa)+1/2*(ZZb)$) {#6};
}

\newcommand{\sm}{{\mathbin{-}}}
\newcommand{\+}{{\mathbin{+}}}
\newcommand{\End}{{\operatorname{End}}}
\newcommand{\id}{{\operatorname{id}}}
\newcommand{\alt}{{\operatorname{alt}}}
\newcommand{\ev}{{\operatorname{ev}}}
\newcommand{\coev}{{\operatorname{coev}}}
\newcommand{\Tr}{{\operatorname{Tr}}}
\newcommand{\tr}{{\operatorname{tr}}}
\newcommand{\Hom}{{\operatorname{Hom}}}

\theoremstyle{plain}
\newtheorem{theorem}{Theorem}[subsection]
\newtheorem*{thm*}{Theorem}
\newtheorem{corollary}[theorem]{Corollary}
\newtheorem{lemma}[theorem]{Lemma}
\newtheorem{proposition}[theorem]{Proposition}
\newtheorem*{claim*}{Claim}
\newtheorem{thmx}{Theorem}

\theoremstyle{definition}
\newtheorem{definition}[theorem]{Definition}

\newtheorem{example}[theorem]{Example}
\newtheorem{notation}[theorem]{Notation}
\newtheorem{remark}[theorem]{Remark}

\newtheorem{construction}[theorem]{Construction}

\def\circle#1{\raisebox{.9pt}{\textcircled{\raisebox{-.9pt}{#1}}}}

\title{Standard $\lambda$-lattices, rigid ${\rm C}^*$ tensor categories, and (bi)modules}
\author{Quan Chen}
\date{September 18, 2020}

\begin{document}
\maketitle
\begin{abstract}
In this article, we construct a 2-shaded rigid ${\rm C}^*$ multitensor category with canonical unitary dual functor directly from a standard $\lambda$-lattice. 
We use the notions of traceless Markov towers and lattices to define the notion of module and bimodule over standard $\lambda$-lattice(s), and we explicitly construct the associated module category and bimodule category over the corresponding 2-shaded rigid ${\rm C}^*$ multitensor category.

As an example, we compute the modules and bimodules for Temperley-Lieb-Jones standard $\lambda$-lattices in terms of traceless Markov towers and lattices. 
Translating into the unitary 2-category of bigraded Hilbert spaces, we recover DeCommer-Yamshita's classification of $\mathcal{TLJ}$ modules in terms of 
edge weighted graphs, and a classification of $\mathcal{TLJ}$ bimodules in terms of biunitary connections on square-partite weighted graphs.

As an application, we show that every (infinite depth) subfactor planar algebra embeds into the bipartite graph planar algebra of its principal graph.
\end{abstract}

\section*{Introduction} Since Jones’ landmark article \cite{Jo83}, the modern theory of subfactors has developed deep connections to numerous branches of mathematics, including representation theory, category theory, knot theory, topological quantum field theory, statistical mechanics, conformal field theory, and free probability. The standard invariant of a type $\text{II}_1$ subfactor was first defined as a standard $\lambda$-lattice \cite{Po95}.  Since, it has been reinterpreted as a planar algebra \cite{Jo99} and a Q-system \cite{Lo89}, or unitary Frobenius algebra object, in a rigid ${\rm C}^*$ tensor category \cite{Mu03}.


The following theorem  is a well known folklore result. 
It is for instance mentioned in this form in \cite[Remark 2.1]{AV15}.
A similar result with planar algebras in place of tensor categories was announced in \cite{Li14}.
The folklore proof of this result makes use of Popa's subfactor reconstruction theorem \cite[Thm.~3.1]{Po95}.\footnote{Similarly, for a given standard $\lambda$-lattice, Jones proved in \cite[Thm.~4.2.1]{Jo99} that one can construct a subfactor planar algebra by passing through Popa's subfactor reconstruction theorem \cite[Thm.~3.1]{Po95}. }
One primary motivation of this paper is to give a direct argument without making a detour via subfactors.

\begin{thm*}[Folklore]
There is a bijective correspondence between equivalence classes of the following:
\[\left\{\, \parbox{3.2cm}{\rm Standard $\lambda$-lattices $A=(A_{i,j})_{0\le i\le j}$} \,\right\}\ \, \cong\ \, \left\{\, \parbox{9cm}{\rm Pairs $(\mathcal{A}, X)$ with $\mathcal{A}$ a 2-shaded rigid ${\rm C}^*$ multitensor category with a generator $X$, i.e., $1_\mathcal{A}=1^+\oplus 1^-$, $1^+,1^-$ are simple and $X=1^+\otimes X\otimes 1^-$ 
} \,\right\}\]
Equivalence on the left hand side is unital $*$-isomorphism of standard $\lambda$-lattices; 
equivalence on the right hand side is unitary equivalence 
between their Cauchy completions 
which maps generator to generator.
\end{thm*}

Given $(\mathcal{A},X)$, it is well known that one can obtain a standard $\lambda$-lattice $A$ by
$$
A_{i,j}:=\begin{cases}\id_{X^{\alt\otimes 2k}}\otimes\End\left(X^{\alt\otimes (j-2k)}\right)& i=2k \\ \id_{X^{\alt\otimes (2k+1)}}\otimes\End\left(\overline{X}^{\alt\otimes (j-2k-1)}\right)& i=2k+1\end{cases}
$$
where $\overline{X}$ is a dual of $X$ and 
$$
X^{\alt\otimes n}
:=
\underbrace{X\otimes \overline{X}\otimes X\otimes\cdots}_{\text{$n$ tensorands}}
$$
and similarly for $\overline{X}^{\alt\otimes n}$. 
The inclusion $A_{i,j}\subset A_{i,j+1}$ sends $x$ to $x\otimes \id$, the inclusion $A_{i+1,j}\subset A_{i,j}$ sends $x$ to $x$. 
The Jones projections are defined using the canonical balanced evaluation and coevaluation for $X$.

Going the other way directly is harder.
Using \cite[Def.~3.1]{CHPS18},
we construct a skeletal (when $d>1$) $\rm W^*$ category explicitly from $A$ whose objects are $[n,\pm]$ for $n\geq 0$ and whose hom spaces can be identified with the algebras $A_{i,j}$. 
We endow it with a tensor structure using the 2-shift map in the standard $\lambda$-lattice, which is a trace-preserving $*$-isomorphism $S_{i,j}:A_{i,j}\to A_{i+2,j+2}$ \cite[Cor.~2.8]{Bi97}.
We call this skeletal category a \textbf{planar tensor category},
and we provide a string diagram calculus to perform computations. 
The Cauchy completion of this planar tensor category is the target 2-shaded rigid ${\rm C}^*$ multitensor category.

Given a standard $\lambda$-lattice $A$, an $A$-module is a Markov tower as a standard $A-$module. 
In more detail, let $A=(A_{i,j})_{0\le i\le j<\infty}$ be a standard $\lambda$-lattice with Jones projection $\{e_i\}_{i\ge 1}$ and compatible conditional expectations.
An $A-$module is a \textbf{Markov tower} of finite dimensional von Neumann algebras $(M_n)_{n\geq 0}$ 
such that $A_{0,n}\subset M_n$
together with conditional expectations $E_i:M_i\to M_{i-1}$ implemented by the Jones projections,
which satisfy the appropriate commuting square conditions.
$$
\begin{matrix}
M_0 & \subset & M_1 & \subset & M_2 & \subset & \cdots & \subset & M_n & \subset &\cdots\\
\cup & & \cup & & \cup & & & & \cup\\
A_{0,0} & \subset & A_{0,1} & \subset & A_{0,2} & \subset & \cdots & \subset & A_{0,n} & \subset & \cdots\\
& & \cup & & \cup & & & & \cup\\
& & A_{1,1} & \subset & A_{1,2} & \subset & \cdots & \subset & A_{1,n} & \subset & \cdots 
\end{matrix}
$$
We refer the reader to Definition \ref{def traceless MT} below for the complete definition.

We warn the reader that our definition is slightly different from the original one from \cite[Def.~3.1]{CHPS18}; our tower of algebras $(M_n)_{n\geq 0}$ does not necessarily have a Markov trace.
An important difference in our construction is that we do \textit{not} use the trace, but rather the commuting square of conditional expectations. 
In \S\ref{MT to planar mod cat}, by using this technique, we are able to discuss arbitrary modules over a standard $\lambda$-lattice instead of merely pivotal modules.


We call an $A-$module \textbf{standard} if $[M_i,A_{k,l}]=0$ for $i\le k\le l$.
By similar techniques used in our new proof of the Folklore Theorem 
above, we obtain the following theorem.

\begin{thmx}
\label{thmx:ModuleEquivalence}
There is a bijective correspondence between equivalence classes of the following:
\[\left\{\, \parbox{5.1cm}{\rm 
Traceless Markov towers $M=(M_i)_{i\ge 0}$ with $\dim(M_0)=1$ 
as standard right modules over a standard $\lambda$-lattice $A$} \,\right\}\ \, \cong\ \, \left\{\, \parbox{7cm}{\rm Pairs $(\mathcal{M}, Z)$ with $\mathcal{M}$ an indecomposable semisimple right $\mathcal{A}-$module $\rm C^*$ category together with a choice of simple object $Z= Z\lhd 1^+_\mathcal{A}$} \,\right\}\]
Equivalence on the left hand side is $*$-isomorphism of traceless Markov towers as standard $A-$modules; 
equivalence on the right hand side is unitary $\mathcal{A}-$module equivalence on Cauchy completions which maps the simple base object to simple base object.

Tracial Markov towers as standard $A-$modules correspond to pivotal $\mathcal{A}-$module categories.
\end{thmx}

In \S\ref{Cpt 3}, we discuss bimodules. 
Given two standard $\lambda$-lattices $A$ and $B$, 
we define an $A-B$ bimodule as a \textbf{standard Markov lattice}, which consists of a doubly indexed sequence $M=(M_{i,j})_{i,j\ge 0}$ of finite dimensional von Neumann algebras with two sequences of Jones projections $(e_i)_{i\geq 1}$ and $(f_j)_{j\geq 1}$ where the following conditions hold.
\begin{enumerate}[label=(\alph*),leftmargin=*]

\item $M_{i,j}\subset M_{i,j+1}$ and $M_{i,j}\subset M_{i+1,j}$ are unital inclusions.
\item $M_{-,j}=(M_{i,j},E^{M,l}_{i,j},e_{i+1})_{i\ge 0}$ are Markov towers with the same modulus $d_0$ and $e_i\in M_{i+1,j}$ for all $i$; $M_{i,-}=(M_{i,j},E^{M,r}_{i,j},f_{j+1})_{j\ge 0}$ are Markov towers with the same modulus $d_1$ and $f_j\in M_{i,j+1}$ for all $j$. We call $M$ of modulus $(d_0,d_1)$.
$$
\begin{matrix}
M_{i+1,j} & {\subset} & M_{i+1,j+1}\\
\cup & & \cup \\
M_{i,j} & {\subset} & M_{i,j+1} 
\end{matrix}
$$
\item The commuting square condition:
$$\xymatrix@+0.5pc{
M_{i+1,j} \ar[d]_{E^{M,l}_{i+1,j}} &M_{i+1,j+1}  \ar[l]_{E^{M,r}_{i+1,j+1}}\ar[d]^{E^{M,l}_{i+1,j+1}}\\
M_{i,j} & M_{i,j+1}\ar[l]^{E^{M,r}_{i,j+1}}
}$$
is a commuting square, i.e., $E^{M,r}_{i,j+1}\circ E^{M,l}_{i,j}=E^{M,l}_{i,j+1}\circ E^{M,r}_{i+1,j+1}$.
\end{enumerate}
\vspace{.2cm}

We require $A^{\op}_{i,0}\subset M_{i,0}$ and $B_{0,j}\subset M_{0,j}$ for all $i,j$ with conditional expectations satisfying the appropriate commuting square conditions.  Here, we take the \textit{opposite} $\lambda$-lattice $A^{\op}$ of $A$,
where $A_{i,j}^{\op}$ is the opposite algebra of $A_{i,j}$, so the indices for $A$ and $B$ are transposed.
$$
\begin{matrix}
\textcolor{red}{\cup} & & \textcolor{red}{\cup} & & \cup & & \cup & & \cup & & \cup & & \\
\textcolor{red}{A_{3,1}}  & \textcolor{red}{\subset} & \textcolor{red}{A_{3,0}} & \subset & M_{3,0} & \subset & M_{3,1} & \subset & M_{3,2} & \subset & M_{3,3} & \subset \\
\textcolor{red}{\cup} & & \textcolor{red}{\cup} & & \cup & & \cup & & \cup & & \cup & & \\
\textcolor{red}{A_{2,1}} & \textcolor{red}{\subset} & \textcolor{red}{A_{2,0}} & \subset & M_{2,0} & \subset & M_{2,1} & \subset & M_{2,2} & \subset & M_{2,3} & \subset \\
\textcolor{red}{\cup} & & \textcolor{red}{\cup} & & \cup & & \cup & & \cup & & \cup & & \\
\textcolor{red}{A_{1,1}} & \textcolor{red}{\subset} & \textcolor{red}{A_{1,0}} & \subset & M_{1,0} & \subset & M_{1,1} & \subset & M_{1,2} & \subset & M_{1,3} & \subset \\
 & & \textcolor{red}{\cup} & & \cup & & \cup & & \cup & & \cup & & \\
& & \textcolor{red}{A_{0,0}} & \subset & M_{0,0} & \subset & M_{0,1} & \subset & M_{0,2} & \subset & M_{0,3} & \subset \\
 & & & & \cup & & \cup & & \cup & & \cup & & \\
& & & & \textcolor{blue}{B_{0,0}} & \textcolor{blue}{\subset} & \textcolor{blue}{B_{0,1}} & \textcolor{blue}{\subset} & \textcolor{blue}{B_{0,2}} & \textcolor{blue}{\subset} & \textcolor{blue}{B_{0,3}} & \textcolor{blue}{\subset}\\
 & & & & & & \textcolor{blue}{\cup} & & \textcolor{blue}{\cup} & & \textcolor{blue}{\cup} & & \\
& & & & & & \textcolor{blue}{B_{1,1}} & \textcolor{blue}{\subset} & \textcolor{blue}{B_{1,2}} & \textcolor{blue}{\subset} & \textcolor{blue}{B_{1,3}} & \textcolor{blue}{\subset}\\
\end{matrix}$$
We call an $A-B$ bimodule \textbf{standard} if $[M_{i,j},A_{p,q}]=0$ for $i\le q\le p$; $[M_{i,j},B_{k,l}]=0$, for $j\le k\le l$.
Similar to the proof
of the Folklore Theorem 
and Theorem 
\ref{thmx:ModuleEquivalence} above, we obtain the following theorem.

\begin{thmx}
\label{thmx:BimoduleEquivalence}
There is a bijective correspondence between equivalence classes of the following:
\[\left\{\, \parbox{5.7cm}{\rm 
Traceless Markov lattices $M=(M_{i,j})_{i,j\ge 0}$ 
with $\dim(M_{0,0})=1$
as standard $A-B$ bimodules over standard $\lambda$-lattices $A,B$} \,\right\}\ \, \cong\ \, \left\{\, \parbox{6.5cm}{\rm Pairs $(\mathcal{M}, Z)$ with $\mathcal{M}$ an indecomposable semisimple ${\rm C}^*$  $\mathcal{A}-\mathcal{B}$ bimodule category together with a choice of simple object $Z= 1^+_\mathcal{A}\rhd Z\lhd 1^+_\mathcal{B}$} \,\right\}\]
Equivalence on the left hand side is $*$-isomorphism on the traceless Markov lattice as a standard $A-B$ bimodule;  equivalence on the right hand side is unitary $\mathcal{A}-\mathcal{B}$ bimodule equivalence between their Cauchy completions which maps the simple base object to simple base object.

Tracial Markov lattices as standard $A-B$ bimodules correspond to pivotal $\mathcal{A}-\mathcal{B}$ bimodule categories.
\end{thmx}

\paragraph{Examples}
As a natural corollary from Theorem \ref{thmx:ModuleEquivalence}, a Markov tower corresponds to a Temperley-Lieb-Jones($\mathcal{TLJ}$) module category. 
This result generalizes the pivotal module case from \cite[Thm.~A.]{CHPS18}.
To translate our classification into that of \cite{DY15} which uses fair and balanced graphs, 
we obtain an elegant graphical version of a Markov tower using a $\rm W^*$ 2-subcategory $\mathcal{C}(\Lambda,\omega)$ of bigraded Hilbert spaces \textsf{BigHilb} which is built from a fair and balanced graph $(\Lambda,\omega)$.
Our approach is inspired by Ocneanu's path algebras \cite{Oc88}\cite{EK98}\cite[\S5.4]{JS97}.
The following diagram shows how these notions are related to each other in \S\ref{Cpt 4}:
\[
\begin{tikzpicture}
\draw[dashed] (-2.8,-.7) rectangle (2.8,.3);
\node at (0,0) {indecomposable semisimple ${\rm C}^*$};
\node at (0,-.45) {$\mathcal{TLJ}(d)-$module category $\mathcal{M}$};

\draw[dashed] (7.8,-.7) rectangle (12.2,.3);
\node at (10,0) {2-subcategory $\mathcal{C}(K,\ev_K)$};
\node at (10,-.45) {of \textsf{BigHilb}};

\draw[dashed] (-1.8,1.3) rectangle (1.8,2.3);
\node at (0,2) {Markov tower};
\node at (0,1.55) {$M$ with modulus $d$};

\draw[dashed] (8,1.3) rectangle (12,2.3);
\node at (10,2) {balanced $d$-fair};
\node at (10,1.55) {bipartite graph $(\Lambda,\omega)$};

\begin{scope}[>=latex]
\draw[<->] (3.3,-.2) -- (7.5,-.2);

\draw[->] (2.3,1.8) -- (7.5,1.8);

\draw[<->] (0,.4) -- (0,1.2);

\draw[<->] (10,.4) -- (10,1.2);

\draw[->] (7.5,.4) -- (2.3,1.2);
\end{scope}

\node at (5.4,-.6) {\S 4.7};

\node at (5,2.1) {\S 4.6};

\node at (-.4,.8) {\S 2};

\node at (10.5,.8) {\S 4.3};

\node at (5.4,1) {\S 4.4};
\end{tikzpicture}
\]

As an application, in the unitary pivotal/tracial setting, we obtain 
the embedding theorem for (infinite depth) subfactor planar algebras (cf.~\cite{MW10}).

\begin{thmx}
Every (infinite depth) subfactor planar algebra embeds in
any bipartite graph planar algebra of its fusion graph with respect to a module category.
In particular, it embeds in the bipartite graph planar algebra of its (dual) principal graph.
\end{thmx}

By Theorem \ref{thmx:BimoduleEquivalence} above, a Markov lattice corresponds to a $\mathcal{TLJ}-\mathcal{TLJ}$ bimodule category. 
By work-in-progress of Penneys-Peters-Snyder,
pivotal $\mathcal{TLJ}-\mathcal{TLJ}$ bimodule categories correspond to Ocneanu's biunitary connections on associative square-partite graphs with vertex weightings.
For the non-pivotal case, the weighting on the square-partite graph is the edge-weighting and we obtain the non-pivotal analog of a biunitary connection.
To translate between these classifications,
we use the well-known fact that a commuting square of finite dimensional von Neumann algebras gives a biunitary connection \cite{JS97}. 
We then introduce a graphical version of a Markov lattice using a $\rm W^*$ 2-subcategory $\mathcal{C}(\Phi)$ of \textsf{BigHilb} obtained from a biunitary connection $\Phi$. 
It turns out that the biunitary connection $\Phi$ corresponds to the bimodule associator of the bimodule category. 
The following diagram shows how these notions are related to each other in \S\ref{Cpt 5}:
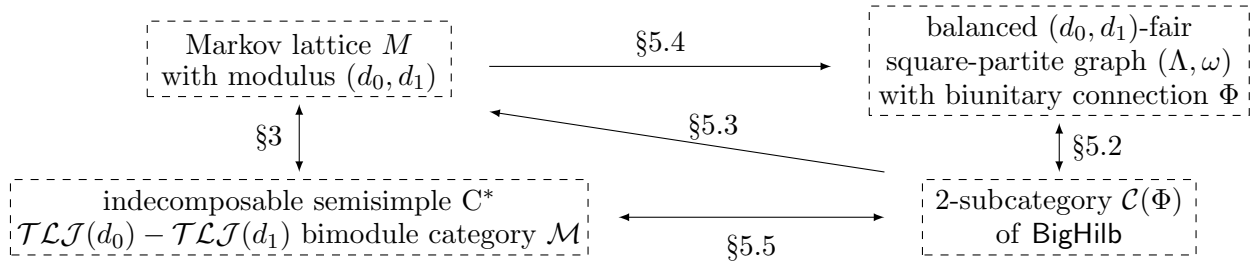
\[\begin{tikzpicture}
\draw[dashed] (-3.8,-.7) rectangle (3.8,.3);
\node at (0,0) {indecomposable semisimple ${\rm C}^*$};
\node at (0,-.45) {$\mathcal{TLJ}(d_0)-\mathcal{TLJ}(d_1)$ bimodule category $\mathcal{M}$};

\draw[dashed] (8.2,-.7) rectangle (11.8,.3);
\node at (10,0) {2-subcategory $\mathcal{C}(\Phi)$};
\node at (10,-.45) {of \textsf{BigHilb}};

\draw[dashed] (-2,1.4) rectangle (2,2.4);
\node at (0,2.1) {Markov lattice $M$};
\node at (0,1.65) {with modulus $(d_0,d_1)$};

\draw[dashed] (7.5,1.15) rectangle (12.5,2.6);
\node at (10,2.3) {balanced $(d_0,d_1)$-fair};
\node at (10,1.85) {square-partite graph $(\Lambda,\omega)$};
\node at (10,1.4) {with biunitary connection $\Phi$};

\begin{scope}[>=latex]
\draw[<->] (4.2,-.2) -- (7.7,-.2);

\draw[->] (2.5,1.8) -- (7,1.8);

\draw[<->] (0,.4) -- (0,1.3);

\draw[<->] (10,.4) -- (10,1.05);

\draw[->] (7.7,.4) -- (2.5,1.2);
\end{scope}

\node at (5.95,-.6) {\S 5.5};

\node at (4.75,2.1) {\S 5.4};

\node at (-.4,.85) {\S 3};

\node at (10.5,.725) {\S 5.2};

\node at (5.45,1) {\S 5.3};
\end{tikzpicture}\]

\paragraph{Acknowledgements} I thank David Penneys and Corey Jones for providing this project and some necessary techniques, including the idea of the 2-shift map in a standard $\lambda$-lattice, unitary dual functors and biunitary connections for the pivotal case. I thank Peter Huston for clarifying Lemma 1.5.4, and also providing good suggestion in the writing. I want to thank Jamie Vicary for clarifying the graphical calculus for 2-categories during the  Summer Research Program on Quantum Symmetries at Ohio State University, 2019. The author is supported by the Mathematics Department in Ohio State University as Graduate Teaching Associate and David Penneys' NSF DMS grant 1654159. 

\vspace{1cm}
\tableofcontents

\vspace{1cm}
\section{Standard $\lambda$-lattices and tensor category}
\subsection{Traceless Markov tower and its properties}
\begin{definition}
Let $A\subset B$ be a unital inclusion of finite von Neumann algebras. A \textbf{conditional expectation} $E:M\to N$ is a positive linear map satisfying the following conditions:
\begin{compactenum}[(a)]
\item $E(x)=x$ for all $x\in A$,
\item $E(axb)=aE(x)b$ for all $a,b\in A$, $x\in B$.
\end{compactenum}
\end{definition}

\begin{definition}
Let $C$ be a unital ${\rm C}^*$-algebra. We call a linear functional $\tr:C\to \mathbb{C}$ a \textbf{trace} if it satisfies the following conditions:
\begin{compactenum}[(a)]
\item(tracial) $\tr(xy)=\tr(yx)$, for all $x,y\in C$.
\item(positive) $\tr(x^*x)\ge 0$, for all $x\in C$.
\item(faithful) $\tr(x^*x)=0$ if and only if $x=0$.
\end{compactenum}
In addition, we call $\tr$ \textbf{unital} if $\tr(1)=1$. 
\end{definition}

\begin{definition}\label{def traceless MT}
A \textbf{traceless Markov tower} $M = (M_n, E_n, e_{n+1})_{n\ge 0}$ consists of a sequence $(M_n)_{n\geq 0}$ of finite dimensional von Neumann algebras, such that $M_n$ is unitally included in $M_{n+1}$. 
For each $n$, there is a faithful normal conditional expectation $E_n:M_n\to M_{n-1}$ together with a sequence of \textbf{Jones projections} $e_n \in M_{n+1}$ for all $n\ge 1$, such that:
\begin{compactenum}[({M}1)]
\item The projections $(e_n)$ satisfy the Temperley-Lieb-Jones relations:
\begin{compactenum}[({TLJ}1)]
\item $e_n^2=e_n=e_n^*$ for all $n$.
\item $[e_i,e_j]=0$ for $|i-j|>1$.
\item There is a fixed constant called the \textbf{modulus} $d>0$ such that $e_ne_{n\pm 1}e_n=d^{-2}e_n$ for all $n$.
\end{compactenum}
\item For all $x\in M_n$, $e_nxe_n=E_n(x)e_n$.
\item $E_{n+1}(e_n)=d^{-2}\cdot 1$ for all $n\ge 1$.
\item (pull down) $M_{n+1}e_n=M_ne_n$ for all $n\ge 1$.
\end{compactenum}
\end{definition}

In the following, all Markov towers are  \textbf{traceless} unless stated otherwise.

\begin{proposition}\label{Markov tower properties} Some properties of a traceless Markov tower include: 
\begin{compactenum}[\rm(1)]
\item $[x,e_k]=0$, for $x\in M_n$, $k\ge n+1$.
\item The map $M_n\ni x\mapsto xe_n\in M_{n+1}$ is injective.
\item For $x\in M_{n+1}$, $d^2 E_{n+1}(xe_n)$ is the unique element $y\in M_n$ such that $xe_n=ye_n$.
\item Property \rm{(3)} is equivalent to \rm{(M3)}.
\item If $x\in M_n$ and $[x,e_n]=0$, then $x\in M_{n-1}$. Together with (1), we have $M_{n-1}=M_n\cap \{e_n\}'$.
\item $e_nM_{n+1}e_n=M_{n-1}e_n$.
\end{compactenum}
\end{proposition}
\begin{proof} 
\item[(1)] For $x\in M_n$ and $k\ge n+1$, $E_k(x)=x,E_k(x^*)=x^*$, then 
$$xe_k = E_k(x)e_k=e_kxe_k=(e_kx^*e_k)^*= (E_k(x^*)e_k)^*=(x^*e_k)^* =e_kx.$$

\item[(2)] If $x\in M_n$ and $xe_n=0$, then by (M3),
$$0=E_{n+1}(xe_n)=xE_{n+1}(e_n)=d^{-2}x.$$
Thus, $x\mapsto xe_n$ is injective.

\item[(3)] By (M4) and (2), the existence and uniqueness hold. Then by (M3),
$$E_{n+1}(xe_n)=E_{n+1}(ye_n)=yE_{n+1}(e_n)=d^{-2}y,$$
so $y=d^2E_{n+1}(xe_n)$.

\item[(4)] First, let's prove that in this setting, $M_n\in x\mapsto xe_n\in M_{n+1}$ is injective. If $xe_n=0$, then 
$$0=d^2E_{n+1}(xe_n)=d^2 x E_{n+1}(e_n).$$
Note that $E_{n+1}$ is faithful and $E_{n+1}(e_n)\ne 0$, so $x=0$. 

Let $x=e_n$, then we have $d^2E_{n+1}(e_n)e_n=e_n$. Since $d^2E_{n+1}(e_n)$ and $1\in M_n$, we have $d^2E_{n+1}(e_n)=1$ by \rm{(2)}.

\item[(5)] Since $xe_n=e_nx$,
$$E_n(x)e_n=e_nxe_n=xe_ne_n=xe_n.$$
Then by (2), $E_n(x)=x$, which implies $x\in M_{n-1}$.

\item[(6)] By (M2) and (M4). 
\end{proof}

We will explore more properties of traceless Markov towers in \S\ref{Cpt 4}.

\begin{remark}
If there is a faithful normal trace on $\bigcup_{n=0}^\infty M_n$ and $E_n$ is the canonical faithful normal trace-preserving conditional expectation for  $n=1,2,\cdots$, then $M$ is called a \textbf{tracial Markov tower}. Thus, tracial Markov towers defined in \cite{CHPS18} are traceless Markov towers. 
\end{remark}

\begin{example}[Markov tower without a trace]\label{example traceless MT}

Let $d>0$ such that $d^2>4$. There is a unique $\lambda\in (0,\frac{1}{2})$ such that $d^{-2}=\lambda(1-\lambda)$. Then $d\lambda+d(1-\lambda)=d$ and $\frac{1}{d\lambda}+\frac{1}{d(1-\lambda)}=d$. Let $e_{ij}$ denote the matrix units of $M_2(\mathbb{C})$, $i,j=1,2$, and $1=e_{11}+e_{22}\in M_2(\mathbb{C})$.

Define $E_\lambda:M_2(\mathbb{C})\to \mathbb{C}$ by $E_\lambda(e_{11})=\lambda$, $E_\lambda(e_{22})=1-\lambda$ and $E_\lambda(e_{12})=E_\lambda(e_{21})=0$. It is clear that $E_\lambda$ is a normal faithful conditional expectation and not tracial. 

Define $e_\lambda\in M_2(\mathbb{C})\otimes M_2(\mathbb{C})$ by
$$e_\lambda =(1-\lambda)e_{11}\otimes e_{11}+\lambda e_{22}\otimes e_{22}+\sqrt{\lambda(1-\lambda)}(e_{12}\otimes e_{12}+e_{21}\otimes e_{21}), $$
and one can check that:
\begin{compactenum}[(a)]
\item $e_\lambda$ is a projection.
\item $E_\lambda(e_\lambda)=d^{-2}(e_{11}+e_{22})=d^{-2}\cdot 1$.
\item $(e_\lambda\otimes 1)(1\otimes e_{1-\lambda})(e_\lambda\otimes 1)=d^{-2}(e_\lambda\otimes 1)$ and $(e_{1-\lambda}\otimes 1)(1\otimes e_\lambda)(e_{1-\lambda}\otimes 1)=d^{-2}(e_{1-\lambda}\otimes 1)$.
\end{compactenum}
\vspace{.5cm}

Define $\id: M_2(\mathbb{C})\to M_2(\mathbb{C})$ to be the identity map.  Let $M_n:=M_2(\mathbb{C})^{\otimes n}$. The inclusion $M_n\subset M_{n+1}$ maps $x$ to $x\otimes \id$. Jones projection $e_{2n+1}=1^{\otimes 2n}\otimes e_{1-\lambda}\in M_{2n+2}$ and $e_{2n+2}=1^{\otimes 2n+1}\otimes e_\lambda\in M_{2n+3}$, $n=0,1,2,\cdots$ The conditional expectation is defined as follows:
$$E_{2n+1}=\id^{\otimes 2n+1}\otimes E_\lambda\qquad E_{2n+2}=\id^{\otimes 2n+2}\otimes E_{1-\lambda}.$$
Now we build a Markov tower with modulus $d$ and without a trace:
\[\xymatrix@+1.5pc{
1 & \ar[l]_{E_\lambda} M_2(\mathbb{C}) & \ar[l]_{\id\otimes E_{1-\lambda}} M_2(\mathbb{C})^{\otimes 2} & \ar[l]_{\id^{\otimes 2}\otimes E_{\lambda}} M_2(\mathbb{C})^{\otimes 3} & \ar[l]_{\id^{\otimes 3}\otimes E_{1-\lambda}} M_2(\mathbb{C})^{\otimes 4} & \ar[l] \cdots
}\]
\end{example}

\subsection{Standard $\lambda$-lattice and its properties}
\begin{definition}[{\cite{Po95}}]\label{standard lambda lattice}
Let $A=(A_{i,j})_{0\le i\le j<\infty}$ be a system of finite dimensional ${\rm C}^*$ algebras with $A_{i,i}=\mathbb{C}$ with unital inclusions $A_{i,j}\subset A_{k,l}$, for $k\le i,j\le l$. 

$$\begin{matrix}
A_{0,0} & \subset & A_{0,1} & \subset & A_{0,2} & \subset & A_{0,3} & \subset & A_{0,4}  & \subset & \cdots\\
& & \cup & & \cup & & \cup & & \cup\\
& & A_{1,1} & \subset & A_{1,2} & \subset & A_{1,3} & \subset & A_{1,4}  & \subset & \cdots\\
& & & & \cup & & \cup & & \cup \\
& & & & A_{2,2} & \subset & A_{2,3} & \subset & A_{2,4}  & \subset & \cdots\\
& & & & & & \cup & & \cup \\
& & & & & & A_{3,3} & \subset & A_{3,4}  & \subset & \cdots \\
& & & & & & & & \cup\\
& & & & & & & & A_{4,4}  & \subset & \cdots\\
& & & & & & & & & & \ddots
\end{matrix}$$

Let $E^r_{i,j}:A_{i,j}\to A_{i,j-1}$ be the (horizontal) faithful normal conditional expectation, $j=1,2\cdots, i=0,\cdots,j-1$ and $E^l_{i,j}:A_{i,j}\to A_{i+1,j}$ be the (vertical) faithful normal conditional expectation $i=0,1,\cdots, j=i+1,i+2,\cdots$. We also require that

\begin{compactenum}[(a)]
\item (commuting square condition)
$$\xymatrix@+0.5pc{
A_{i,j} \ar[d]_{E^l_{i,j}} &A_{i,j+1}  \ar[l]_{E^r_{i,j+1}}\ar[d]^{E^l_{i,j+1}}\\
A_{i+1,j} & A_{i+1,j+1}\ar[l]^{E^r_{i+1,j+1}}
}$$
is a commuting square, i.e., $E^l_{i,j}\circ E^r_{i,j+1}=E^r_{i+1,j+1}\circ E^l_{i,j+1}$.\\

\item (existence of Jones $\lambda$-projections)\\
There exists a sequence of Jones projections $\{e_i\}_{i\ge 1}$ in $\bigcup_n A_{0,n}$ such that 
\begin{compactenum}[({b}1)]
\item $e_j\in A_{i-1,k}$, for $1\le i\le j+1\le k$.
\item The projections satisfy the Temperley-Lieb-Jones relations:
\begin{compactenum}[({TLJ}1)]
\item $e_i^2=e_i=e_i^*$ for all $i$.
\item $e_ie_j=e_je_i$ for $|i-j|>1$.
\item There is a fixed constant $d>0$ called the modulus such that $e_ie_{i\pm 1}e_i=d^{-2}e_i$ for all $i$.
\end{compactenum}
\item $e_jxe_j=E^r_{i,j}(x)e_j$, for $x\in A_{i,j},i+1\le j$.
\item $e_ixe_i=E^l_{i,j}(x)e_i$, for $x\in A_{i,j},  i+1\le j$.
\end{compactenum}
\vspace{10pt}
\item (Markov conditions)
\begin{compactenum}[({c}1)]
\item $\dim A_{i,j}=\dim A_{i,j+1}e_{j}=\dim A_{i+1,j+1}$, for $i\le j$.
\item $E^r_{i,j+1}(e_j)=E^l_{j-1,k}(e_j)=d^{-2} 1$, for $j\ge i+1,k\ge j+1$.
\end{compactenum}
\end{compactenum}
Then $A=(A_{i,j})_{0\le i\le j<\infty}$ is called a \textbf{ $\lambda$-lattice} of commuting squares. If there is a faithful normal trace $\tr$ on $\bigcup_{n=0}^\infty A_{0,n}$ and $E^r_{i,j},E^l_{i,j}$ are the canonical faithful normal trace-preserving conditional expectation, then $A$ is called a \textbf{tracial $\lambda$-lattice}.
\end{definition}

\begin{definition}[{\cite{Po95}}]
A $\lambda$-lattice $(A_{i,j})_{0\le i\le j}$ is called a \textbf{standard $\lambda$-lattice} if $[A_{i,j},A_{k,l}]=0$ for $i\le j\le k\le l$. This condition is called the \textbf{standard condition}.
\end{definition}

\begin{remark}
In the definition of (standard) $\lambda$-lattice, we may not require a trace and the conditional expectations are trace-preserving. In fact, the reader can construct an example of (standard) $\lambda$-lattice without a trace from Example \ref{example traceless MT} easily. We will not further discuss the traceless standard $\lambda$-lattices, though the following statements do NOT require the trace at all!
\end{remark}

\begin{remark} \label{lambda lattice properties}
Each row $A_i=(A_{i,j})_{j\ge i}$ is a Markov tower, $i=0,1,2,\cdots$; each column $A_j=(A_{i,j})_{i=j}^0$ is a Markov tower, $j=1,2,\cdots$. From Proposition \ref{Markov tower properties}, we have 
\begin{compactenum}[(1)]
\item If $x\in A_{i,j}$, $[x,e_k]=0$ for $k\ge j+1$; $[x,e_l]=0$ for $1\le l\le i-1$. 
\item The map $A_{i,j}\ni x\mapsto xe_j\in A_{i,j+1}$ is injective; the map $A_{i,j}\ni x\mapsto xe_i\in A_{i-1,j}$ is injective.
\item The Markov condition is equivalent to the pull-down condition:
\begin{compactenum}[({c}1{)'}]
\item $d^2 E^r_{i,j+1}(x e_j)e_j=x e_j$, for $x\in A_{i,j+1}$, $j\ge i\ge 0$.
\item $d^2 E^l_{i-1,j}(x e_i)e_i=x e_i$, for $x\in A_{i-1,j}$, $j\ge i\ge 1$.
\end{compactenum}
\end{compactenum}
\end{remark}

The following property was proved in \cite[Prop.~1.4]{Po95} by using the trace, here we provide another proof without it.

\begin{proposition}
Let 
$$\begin{matrix}
A_{0,0} & \subset & A_{0,1} & \subset & A_{0,2} & \subset & A_{0,3}   & \subset & \cdots\\
& & \cup & & \cup & & \cup \\
& & A_{1,1} & \subset & A_{1,2} & \subset & A_{1,3} &  \subset & \cdots
\end{matrix}$$
be a $\lambda$-sequence of commuting squares, and define $A_{i,j}:=A_{i-1,j}\cap\{e_{i-1}\}'= A_{1,j}\cap \{e_1,\cdots,e_{i-1}\}'$, $2\le i\le j$. Then $(A_{i,j})_{0\le i\le j<\infty}$ is a $\lambda$-lattice of commuting squares.
\end{proposition}
\begin{proof}
We construct $A_{i,j}$ and conditional expectation $E^l_{i-1,j}:A_{i-1,j}\to A_{i,j}$ by induction on $i$, and show that Jones projections $\{e_{i+1},\cdots,e_{j-1}\}\subset A_{i,j}$ for $i+2\le j$. Suppose $A_{i-1,j}$ is constructed (or given) and $\{e_i,\cdots,e_{j-1}\}\subset A_{i-1,j}$, We define $A_{i,j}:= A_{i-1,j}\cap \{e_{i-1}\}'$. Then clearly, $\{e_{i+1},\cdots,e_{j-1}\}\subset A_{i,j}$.  

According to Proposition \ref{Markov tower properties}(5) and (6), for each $x\in A_{i-1,j}\subset A_{i-2,j}$, there exists a $y\in A_{i,j}$ such that 
$$ye_{i-1}=e_{i-1}xe_{i-1}.$$
By Proposition \ref{Markov tower properties}(2), $A_{i-1,j}\ni y\mapsto ye_{i-1}\in A_{i-2,j}$ is injective, so $y$ is unique for each given $x$. This technique is often used in this chapter. We define $E^l_{i-1,j}(x):=y$. Now we show that $E^l_{i-1,j}$ is a faithful normal conditional expectation:
\begin{compactenum}[(a)]
\item It is clear that $E^l_{i-1,j}$ is linear,  and $E^l_{i-1,j}(1)=1$. The ultraweak continuity/normality follows from the definition.
\item $E^l_{i-1,j}(x^*)=E^l_{i-1,j}(x)^*$: 
$$E^l_{i-1,j}(x)^*e_{i-1}=(e_{i-1}E^l_{i-1,j}(x))^*=(e_{i-1}xe_{i-1})^*=e_{i-1}x^*e_{i-1}=E^l_{i-1,j}(x^*)e_{i-1}.$$
\item $E^l_{i-1,j}(axb)=aE^l_{i-1,j}(x)b$ for $a,b\in A_{i,j}$: Note that $[a,e_{i-1}]=[b,e_{i-1}]=0$, then $$E^l_{i-1,j}(axb)e_{i-1}=e_{i-1}axbe_{i-1}=ae_{i-1}xe_{i-1}b=aE^l_{i-1,j}(x)e_{i-1}b=aE^l_{i-1,j}(x)be_{i-1}.$$
\item $E^l_{i-1,j}(x^*x)\ge E^l_{i-1,j}(x)^*E^l_{i-1,j}(x)$, which follows that $E^l_{i-1,j}$ is positive: $$E^l_{i-1,j}(x)^*E^l_{i-1,j}(x)e_{i-1}=E^l_{i-1,j}(x)^*e_{i-1}xe_{i-1}=e_{i-1}x^*e_{i-1}xe_{i-1}\le e_{i-1}x^*xe_{i-1}=E^l_{i-1,j}(x^*x)e_{i-1},$$
so $E^l_{i-1,j}(x^*x)\ge E^l_{i-1,j}(x)^*E^l_{i-1,j}(x)$ by applying the inductive hypothesis that $E^l_{i-2,j}$ is a positive conditional expectation and $E^l_{i-2,j}(e_{i-1})=d^{-2}\cdot 1$.
\item $E^l_{i-1,j}(x^*x)=0$ if and only if $x=0$, i.e., $E^l_{i-1,j}$ is faithful:
$$0=E^l_{i-1,j}(x^*x)e_{i-1}=e_{i-1}x^*xe_{i-1}=(xe_{i-1})^*(xe_{i-1}),$$
which follows that $xe_{i-1}=0$. Note that $A_{i-1,j}\ni x\mapsto xe_{i-1}\in A_{i-2,j}$ is an injection, so $x=0$.
\end{compactenum}

Then define $E^r_{i,j+1}:A_{i,j+1}\to A_{i,j}$ as the restriction of $E^r_{i-1,j+1}$ on $A_{i,j+1}$, which is also a conditional expectation. 

Now we prove the commuting square condition $E^l_{i-1,j}\circ E^r_{i-1,j+1}=E^r_{i,j+1}\circ E^l_{i-1,j+1}$: for $x\in A_{i-1,j+1}$,
\begin{align*}
    E^l_{i-1,j}(E^r_{i-1,j+1}(x))e_{i-1} & = e_{i-1}E^r_{i-1,j+1}(x)e_{i-1}\\
    E^r_{i,j+1}(E^l_{i-1,j+1}(x))e_{i-1} & = E^r_{i-1,j+1}(E^l_{i-1,j+1}(x))e_{i-1} = E^r_{i-1,j+1}(E^l_{i-1,j+1}(x)e_{i-1})  \\
    & = E^r_{i-1,j+1}(e_{i-1}xe_{i-1})= e_{i-1}E^r_{i-1,j+1}(x)e_{i-1}.
\end{align*}

Finally, we prove the Markov condition:
\begin{compactenum}[(a)]
\item $\dim A_{i,j} = \dim A_{i-1,j}\cap \{e_{i-1}\}' = \dim A_{i-1,j}\cap \{e_{j-1}\}' =\dim A_{i-1,j-1}$.
\item $E^r_{i,j+1}(e_j)=E^r_{i-1,j+1}(e_j)=d^{-2}1$.
\item $E^l_{i-1,j}(e_i)e_{i-1}=e_{i-1}e_ie_{i-1}=d^{-2}e_{i-1}$, so $E^l_{i-1,j}(e_i)=d^{-2}\cdot 1$.
\end{compactenum} 
\end{proof}
\begin{corollary}
Let $(A_{i,j})_{i\le j,i=0,1}$ be a $\lambda$-sequence of commuting squares. If $A_{i,j}:=\{e_1,\cdots,e_{i-1}\}'\cap A_{i,j}$, for all $2\le i\le j$, then $(A_{i,j})_{0\le i\le j}$ is a standard $\lambda$-lattice if and only if $(A_{i,j})_{i\le j,i=0,1}$ satisfies
\begin{align*}
    [A_{0,1},A_{1,j}]&=0,\ \forall 1\le j\\
    [A_{0,i},A_{i,j}]&=0,\ \forall 2\le i\le j.
\end{align*}
\end{corollary}

Now we define the opposite standard $\lambda$-lattice, which will be used in Definition \ref{def ML std bimod}.

\begin{definition}\label{Def:opposite Aop}
$A^{\text{op}}=(A_{i,j})_{0\le j\le i}$ is the \textbf{opposite} of $\lambda$-lattice $A$ if $A^{\text{op}}_{j,i}=A_{i,j}$ as opposite algebras, $E^{\text{op},l}_{j,i}=E^r_{i,j}$, $E^{\text{op},r}_{j,i}=E^l_{i,j}$ for $i\le j$.
\end{definition}

\begin{example}\label{TLJ lambda lattice}
The Temperley-Lieb-Jones algebra TLJ$(d)$ forms a standard $\lambda$-lattice with the modulus $d$ by letting $A_{i,i}=A_{i,i+1}=\mathbb{C}$ and $A_{i,j}=\langle e_{i+1},\cdots,e_{j-1}\rangle$ for $j-i\ge 2$, which is called a Temperley-Lieb-Jones standard $\lambda$-lattice.
\end{example}

\begin{example}[{\cite{Po95}}]
If $A_0\subset A_1$ is a unital inclusion of type $\text{II}_1$ subfactors with finite index and $A_0\subset A_1\subset A_2\subset A_3\subset\cdots$ is the Jones tower, then $A_{i,j}:=A_i'\cap A_j$ forms a standard $\lambda$-lattice, which is called the standard invariant of $A_0\subset A_1$. 
\end{example}

\subsection{The 2-shift map}
In this section, we discuss an important type of $*$-isomorphism in a standard $\lambda$-lattice, so-called the 2-shift map \cite{Bi97}. Here we provide the definition by using the conditional expectations and Jones projections instead of trace and Pimsner-Popa basis.

For $i,k\ge 0$, define the following element of $A_{l,i+2k}$, $l+1\le i+2k$: $$e^i_k := d^{k(k-1)}(e_{k+i}e_{k+i-1}\cdots e_{i+1})(e_{k+i+1}e_{k+i}\cdots e_{n-k+2})\cdots(e_{2k+i-1}e_{2k+i-2}\cdots e_{k+i}).$$

For $i,j,k\ge 0$, define the following element of $A_{l,i+j+2k}$, $l+1\le i+j+2k$, 
$$e^i_{j,k}=d^{jk}e^i_ke^{i+1}_k\cdots e^{i+j}_k.$$

Clearly, $e_n=e^{n-1}_1=e^{n-1}_{0,1}$, $e^i_k=e^i_{0,k}$, $(e^i_k)^2=(e^i_k)^*=e^i_k$ and $e^i_{j,k}(e^i_{j,k})^*=e^i_{0,k},\ (e^i_{j,k})^*e^i_{j,k}=e^{i+j}_{0,k}$.

\begin{definition} [Multi-step condition expectation]\label{Def:multistep conditional expectation}
Define the $k$-step horizontal conditional expectation as $E^{r,k}_{i,j}=E^r_{i,j+1-k}\circ E^r_{i,j+2-k}\circ\cdots\circ E^r_{i,j}:A_{i,j}\to A_{i,j-k}$ for $k\le j-i$ and we have $E^{r,1}_{i,j}=E^r_{i,j}$; the $k$-step vertical conditional expectation as
$E^{l,k}_{i,j}=E^l_{i+k-1,j}\circ E^r_{i+k-2,j}\circ\cdots\circ E^l_{i,j}:A_{i,j}\to A_{i+k}$ for $k\le j-i$ and we have $E^{l,1}_{i,j}=E^l_{i,j}$. 

In particular, the trace is made by the composition of conditional expectations, i.e., $E^{l,i-j+k}_{i,j-k}\circ E^{r,k}_{i,j}=\tr=E^{r,j-i-t}_{i+t,j}\circ E^{l,t}_{i,j}$, for $0\le k\le j-i$, $0\le t\le j-i$
\end{definition}

\begin{definition}[$2$-shift map]
Define the \textbf{$2$-shift} map $S_{i,j}:A_{i,j}\to A_{i+2,j+2}$, $i\le j$ by $$S_{i,j}(x) := d^{2j-2i+2} E^l_{i,j+2}(e_{i+1}e_{i+2}\cdots e_jxe_{j+1}e_j\cdots e_{i+1}).$$
\end{definition}

\begin{proposition} The followings are the properties of the 2-shift map. \label{2-shift map prop}
\begin{compactenum}[\rm(1)]
\item $S_{i,j}$ is well defined, i.e., $S_{i,j}(x)\in A_{i+2,j+2}$ for $x\in A_{i,j}$.
\item $S_{i,j}$ is a unital $*$-isomorphism.
\item $($commuting parallelogram$)$ $S_{i,j-1}\circ E^r_{i,j}=E^r_{i+2,j+2}\circ S_{i,j}$ and $S_{i+1,j}\circ E^l_{i,j}=E^l_{i+2,j+2}\circ S_{i,j}$.
\item $S_{i,j+1}(x)=S_{i,j}(x)$ for $x\in A_{i,j}$ and $S_{i-1,j}(x)=S_{i,j}(x)$ for $x\in A_{i,j}$. 
\item $($shift$)$ $e_{i+1}e_{i+2}\cdots e_{j+1}x=S_{i,j}(x)e_{i+1}e_{i+2}\cdots e_{j+1}$ for $x\in A_{i,j}$. Taking adjoints, $xe_{j+1}e_j\cdots e_{i+1}=e_{j+1}e_j\cdots e_{i+1}S_{i,j}(x)$. In other word, $e^i_{j-i,1}x=S_{i,j}(x)e^i_{j-i,1}$.
\item $S_{i,j}$ is trace-preserving.
\item $S_{i,j}(e_k)=e_{k+2}$, where $i+1\le k\le j-1$.
\end{compactenum}
\end{proposition}
\begin{proof} 
\item[(1)] Note that $S_{i,j}(x)\in A_{i+1,j+2}$, we shall show that $E^l_{i+1,j+2}(S_{i,j}(x))=S_{i,j}(x)$. Since $E^l_{i+1,j+2}(S_{i,j}(x))-S_{i,j}(x)\in A_{i+1,j+2}$ and the map $A_{i+1,j+2}\ni y\mapsto ye_{i+1}\in A_{i,j+2}$ is injective, we shall show that $E^l_{i+1,j+2}(S_{i,j}(x))e_{i+1}=S_{i,j}(x)e_{i+1}$.
\begin{align*}
    E^l_{i+1,j+2}(S_{i,j}(x))e_{i+1} & = e_{i+1}S_{i,j}(x)e_{i+1} \\
    & = d^{2j-2i+2} e_{i+1}E^l_{i,j+2}(e_{i+1}e_{i+2}\cdots e_jxe_{j+1}e_j\cdots e_{i+1})e_{i+1} \\
    & = d^{2j-2i}e_{i+1}(e_{i+1}e_{i+2}\cdots e_jxe_{j+1}e_j\cdots e_{i+1}) \tag{\text{pull down}}\\
    & = d^{2j-2i} e_{i+1}e_{i+2}\cdots e_jxe_{j+1}e_j\cdots e_{i+1}\\
    & = d^{2j-2i+2} E^l_{i,j+2}(e_{i+1}e_{i+2}\cdots e_jxe_{j+1}e_j\cdots e_{i+1})e_{i+1} \tag{\text{pull down}}\\
    & = S_{i,j}(x)e_{i+1}.
\end{align*}
\item[(2)] For $x\in A_{i,j}$, we have $[x,e_{j+1}]=0$. First, we show that $S_{i,j}$ is a homomorphism, i.e., $S_{i,j}(xy)=S_{i,j}(x)S_{i,j}(y)$ for $x,y\in A_{i,j}$. Note that the map $A_{i+2,j+2}\subset A_{i+1,j+2}\ni y\mapsto ye_{i+1}\in A_{i,j+2}$ is injective, we shall prove that $S_{i,j}(xy)e_{i+1}=S_{i,j}(x)S_{i,j}(y)e_{i+1}$.
\begin{align*}
    S_{i,j}(x)S_{i,j}(y)e_{i+1} & = d^{2j-2i+2} S_{i,j}(x)E^l_{i,j+2}(e_{i+1}e_{i+2}\cdots e_jye_{j+1}e_j\cdots e_{i+1})e_{i+1}\\
    & = d^{2j-2i} S_{i,j}(x)e_{i+1}e_{i+2}\cdots e_jye_{j+1}e_j\cdots e_{i+1}\tag{\text{pull down}}\\
    & = d^{2j-2i}\cdot d^{2j-2i}(e_{i+1}e_{i+2}\cdots e_jxe_{j+1}e_j\cdots e_{i+1})(e_{i+1}e_{i+2}\cdots e_jye_{j+1}e_j\cdots e_{i+1})\tag{\text{pull down}}\\
    & = d^{2j-2i+2} e_{i+1}e_{i+2}\cdots e_jxe_{j+1}e_jye_{j+1}e_j\cdots e_{i+1} \tag{$e_{k}e_{k\pm 1}e_{k}=d^{-2}e_{k}$}\\
    & = d^{2j-2i+2} e_{i+1}e_{i+2}\cdots e_jxe_{j+1}e_je_{j+1}ye_j\cdots e_{i+1} \tag{$[y,e_{j+1}]=0$}\\
    & = d^{2j-2i} e_{i+1}e_{i+2}\cdots e_jxe_{j+1}ye_j\cdots e_{i+1}\\
    & = d^{2j-2i} e_{i+1}e_{i+2}\cdots e_jxye_{j+1}e_j\cdots e_{i+1}\\
    & = d^{2j-2i+2} E^l_{i,j+2}(e_{i+1}e_{i+2}\cdots e_jxye_{j+1}e_j\cdots e_{i+1})e_{i+1} \tag{\text{pull down}}\\
    & = S_{i,j}(xy)e_{i+1}.
\end{align*}

Next, $S_{i,j}$ is a $*$-homomorphism. Note that $E^l_{i,j+2}$ is a $*$-homomorphism, we have
\begin{align*}
    S_{i,j}(x^*) & = d^{2j-2i+2} E^l_{i,j+2}(e_{i+1}e_{i+2}\cdots e_jx^*e_{j+1}e_j\cdots e_{i+1}) \\
    & = d^{2j-2i+2} E^l_{i,j+2}((e_{i+1}e_{i+2}\cdots e_jxe_{j+1}e_j\cdots e_{i+1})^*)\\
    & = d^{2j-2i+2} E^{l,*}_{i,j+2}(e_{i+1}e_{i+2}\cdots e_jxe_{j+1}e_j\cdots e_{i+1})\\
    & = S^*_{i,j}(x).
\end{align*}

When $x=1$, 
\begin{align*}
    e_{i+1}e_{i+2}\cdots e_je_{j+1}e_j\cdots e_{i+1} &= d^{-2}e_{i+1}e_{i+2}\cdots e_{j-1}e_{j}e_{j-1}\cdots e_{i+1} \\
    & = \cdots =d^{2(i-j+2)}e_{i+1}e_{i+2}e_{i+1} =d^{2(i-j)}e_{i+1}.
\end{align*}
Thus, $S_{i,j}(1)=d^{2}E^l_{i,j+2}(e_{i+1})=1$, i.e., $S_{i,j}$ is unital.\\

In order to prove that $S_{i,j}$ is an isomorphism, we shall show $S_{i,j}$ is injective and surjective.

If $S_{i,j}(x)=0$, then 
\begin{align*}
   0 &= S_{i,j}(x)e_{i+1}=d^{2j-2i} e_{i+1}e_{i+2}\cdots e_jxe_{j+1}e_j\cdots e_{i+1}\\
   &= d^{2j-2i}(e_{i+1}e_{i+2}\cdots e_j)xe_{j+1}(e_{i+1}e_{i+2}\cdots e_j)^*,
\end{align*}
which follows that $xe_{j+1}=0$. Since the map $A_{i,j}\ni y\mapsto ye_{j+1}\in A_{i,j+1}$ is injective, we have $x=0$.

Note that $\dim A_{i,j}=\dim A_{i+1,j+1}=\dim A_{i+2,j+2}<\infty$, so the injectivity implies the surjectivity. Thus, $S_{i,j}$ is a unital $*$-isomorphism.
\item[(3)] For $x\in A_{i,j}$, $E^r_{i,j}(x)\in A_{i,j-1}$ and $[E^r_{i,j}(x),e_j]=0$,
\begin{align*}
    S_{i,j-1}\circ E^r_{i,j}(x) & = d^{2j-2i} E^l_{i,j+1}(e_{i+1}e_{i+2}\cdots e_jE^r_{i,j}(x)e_{j+1}e_j\cdots e_{i+1})\\
    & = d^{2j-2i} E^l_{i,j+1}(e_{i+1}e_{i+2}\cdots E^r_{i,j}(x)e_je_{j+1}e_j\cdots e_{i+1})\\
    & = d^{2j-2i+2} E^l_{i,j+1}(e_{i+1}e_{i+2}\cdots E^r_{i,j}(x)e_j\cdots e_{i+1})\\
    & = d^{2j-2i+2} E^l_{i,j+1}(e_{i+1}e_{i+2}\cdots e_jxe_j\cdots e_{i+1}),\\
    E^r_{i+2,j+2}\circ S_{i,j}(x) & = E^r_{i+2,j+2}\circ E^l_{i+1,j+2}\circ S_{i,j}(x)\\
    & = E^l_{i+1,j+1}\circ E^r_{i+1,j+2}\circ S_{i,j}(x) \tag{commuting square} \\
    & = d^{2j-2i}E^l_{i+1,j+1}\circ E^r_{i+1,j+2}\circ E^l_{i,j+2}(e_{i+1}e_{i+2}\cdots e_jxe_{j+1}e_j\cdots e_{i+1})\\
    & = d^{2j-2i} E^l_{i+1,j+1}\circ E^l_{i,j+1}\circ E^r_{i,j+2}(e_{i+1}e_{i+2}\cdots e_jxe_{j+1}e_j\cdots e_{i+1})\\
    & = d^{2j-2i}E^l_{i+1,j+1}\circ E^l_{i,j+1}(e_{i+1}e_{i+2}\cdots e_jxE^r_{i,j+2}(e_{j+1})e_j\cdots e_{i+1})\\
    & = d^{2j-2i+2}E^l_{i+1,j+1}\circ E^l_{i,j+1}(e_{i+1}e_{i+2}\cdots e_jxe_j\cdots e_{i+1}).\\
    & = E^l_{i+1,j+1}(S_{i,j-1}\circ E^r_{i,j}(x))\tag{since $S_{i,j-1}\circ E^r_{i,j}(x)\in A_{i+2,j+1}$}\\
    & = S_{i,j-1}\circ E^r_{i,j}(x).
\end{align*}
Thus, $S_{i,j-1}\circ E^r_{i,j}=E^r_{i+2,j+2}\circ S_{i,j}$.\\

Note that $\{e_{i+1},\cdots, e_{j-1}\}\subset A_{i,j}$, we have
\begin{align*}
    E^l_{i,j+2}(e_kxe_n)=e_k E^l_{i,j+2}(x)e_n\qquad\text{for all }k,n\in \{i+1,\cdots,j-1\}.\tag{$*$}
\end{align*}

In order to prove that $S_{i+1,j}\circ E^l_{i,j}=E^l_{i+2}\circ S_{i,j}$, by Remark \ref{lambda lattice properties} (2), we shall show that $S_{i+1,j}\circ E^l_{i,j}(x)e_{i+2}=E^l_{i+2,j+2}\circ S_{i,j}(x)e_{i+2}$ for all $x\in A_{i,j}$.
\begin{align*}
    S_{i+1,j}\circ E^l_{i,j}(x)e_{i+2} & = d^{2j-2i} E^l_{i+1,j+2}(e_{i+2}\cdots e_j E^l_{i,j}(x)e_{j+1}\cdots e_{i+2})e_{i+2}\\
    & = d^{2j-2i-2}e_{i+2}\cdots e_j E^l_{i,j}(x)e_{j+1}\cdots e_{i+2}, \tag{pull down}\\
    E^l_{i+2,j+2}\circ S_{i,j}(x)e_{i+2} & = d^{2j-2i+2} E^l_{i+2,j+2}(E^l_{i,j+2}(e_{i+1}\cdots e_j x e_{j+1}\cdots e_{i+1}))e_{i+2}\\
    & = d^{2j-2i+2} e_{i+2}E^l_{i,j+2}(e_{i+1}e_{i+2}\cdots e_j x e_{j+1}\cdots e_{i+2}e_{i+1})e_{i+2} \tag{by ($*$)}\\
    & =d^{2j-2i+2} E^l_{i,j+2}(e_{i+2}e_{i+1}e_{i+2}\cdots e_j x e_{j+1}\cdots e_{i+2}e_{i+1}e_{i+2})\\
    & =d^{2j-2i-2} E^l_{i,j+2}(e_{i+2}\cdots e_j x e_{j+1}\cdots e_{i+2})\\
    & = d^{2j-2i-2} e_{i+2}e_{i+1}\cdots e_j E^l_{i,j+2}(x) e_{j+1}\cdots e_{i+1}e_{i+2} \tag{by ($*$)}\\
    & = d^{2j-2i-2} e_{i+2}e_{i+1}\cdots e_j E^l_{i,j}(x) e_{j+1}\cdots e_{i+1}e_{i+2} \tag{commuting square}\\
    & = S_{i+1,j}\circ E^l_{i,j}(x)e_{i+2}.
\end{align*}
Thus, $S_{i+1,j}\circ E^l_{i,j}=E^l_{i+2}\circ S_{i,j}$. 

\item[(4)] This is a particular case of (3) by the property of conditional expectation.

\item[(5)] For $x\in A_{i,j}$, $[x,e_{j+1}]=0$, 
\begin{align*}
    S_{i,j}(x)e_{i+1}e_{i+2}\cdots e_{j+1} & = d^{2j-2i+2} E^l_{i,j+2}(e_{i+1}e_{i+2}\cdots e_jxe_{j+1}e_j\cdots e_{i+2}e_{i+1})e_{i+1}e_{i+2}\cdots e_{j+1} \\
    & = d^{2j-2i}(e_{i+1}e_{i+2}\cdots e_jxe_{j+1}e_j\cdots e_{i+2}e_{i+1})e_{i+2}\cdots e_{j+1} \tag{pull down}\\
    & = d^{2j-2i-2}(e_{i+1}e_{i+2}\cdots e_jx)\cdot e_{j+1}\cdots e_{i+2}\cdots e_{j+1}\\
    & = \cdots\\
    & = e_{i+1}e_{i+2}\cdots e_jxe_{j+1}\\
    & = e_{i+1}e_{i+2}\cdots e_je_{j+1}x.
\end{align*}

\item[(6)] By (3) and Definition \ref{Def:multistep conditional expectation}.

\item[(7)] Note that the map $A_{i+2,j+2}\subset A_{i+1,j+2}\ni y\mapsto ye_{i+1}\in A_{i,j+2}$ is injective, we shall prove that $S_{i,j}(e_k)e_{i+1}=e_{k+2}e_{i+1}$. For $i+1\le k\le j-1$,
\begin{align*}
    S_{i,j}(e_k)e_{i+1} & = d^{2j-2i+2} E^l_{i,j+2}(e_{i+1}e_{i+2}\cdots e_je_ke_{j+1}e_j\cdots e_{i+1})e_{i+1}\\
    & = d^{2j-2i} e_{i+1}e_{i+2}\cdots e_je_ke_{j+1}e_j\cdots e_{i+1} \tag{\text{pull down}}\\
    & = d^{2j-2i} e_{i+1}\cdots e_{k-1}e_ke_{k+1}e_k e_{k+2}\cdots e_je_{j+1}e_j\cdots e_{k+2}e_{k+1}\cdots e_{i+1} \tag{$[e_i,e_j]=0$ for $|i-j|\ge 2$}\\
    & = d^{2j-2i} e_{i+1}\cdots e_{k-1}(e_ke_{k+1}e_k) (e_{k+2}\cdots e_je_{j+1}e_j\cdots e_{k+2})e_{k+1}\cdots e_{i+1}\\
    & = d^{2k-2i} e_{i+1}\cdots e_{k-1}e_k e_{k+2}e_{k+1}e_k\cdots e_{i+1} \tag{$e_te_{t\pm1}e_t=d^{-2}e_t$}\\
    & = d^{2k-2i} e_{k+2} e_{i+1}\cdots e_{k-1}e_k e_{k+1}e_k\cdots e_{i+1}\\
    & = e_{k+2}e_{i+1} \tag{$e_te_{t\pm1}e_t=d^{-2}e_t$}
\end{align*}
\end{proof}

\begin{definition} [$2n$-shift map]
Define $S^{(n)}_{i,j}:A_{i,j}\to A_{i+2n,j+2n}$ by $$S^{(n)}_{i,j}=S_{i+2(n-1),j+2(n-1)}\circ S^{(n-1)}_{i,j}=S_{i+2(n-1),j+2(n-1)}\circ S_{i+2(n-2),j+2(n-2)}\circ \cdots \circ S_{i,j}$$
to be the \textbf{$2n$-shift map}. 
\end{definition}

\begin{proposition}\label{2n-shift map prop} The followings are the properties of the $2n$-shift map.
\begin{compactenum}[\rm (1)]
\item $S^{(n)}_{i,j}$ is a unital $*$-isomorphism.
\item $($commuting parallelogram$)$ $S^{(n)}_{i,j-1}\circ E^{r,k}_{i,j}=E^{r,k}_{i+2n,j+2n}\circ S^{(n)}_{i,j}$ and $S^{(n)}_{i+1,j}\circ E^{l,k}_{i,j}=E^{l,k}_{i+2n,j+2n}\circ S^{(n)}_{i,j}$.
\item $S^{(n)}_{i,j+k}(x)=S^{(n)}_{i,j}(x)$ for $x\in A_{i,j}$ and $S^{(n)}_{i-k,j}(x)=S^{(n)}_{i,j}(x)$ for $x\in A_{i,j}$.
\item $($shift$)$  For $x\in A_{i,j}$, $e^i_{j-i,n}x=S^{(n)}_{i,j}(x)e^i_{j-i,n}$. By taking adjoint, $xe^{i,*}_{j-i,n}=e^{i,*}_{j-i,n}S^{(n)}_{i,j}(x)$.
\item $S^{(n)}_{i,j}$ is trace-preserving.
\end{compactenum}
\end{proposition}
\begin{proof} (1),(2),(3),(5) follow from Proposition \ref{2-shift map prop}. 

\noindent(4) First, we show that $e^{i+2(n-1)}_{j-i,1}e^{i+2(n-2)}_{j-i,1}\cdots e^{i}_{j-i,1}x=S^n_{i,j}(x)e^{i+2(n-1)}_{j-i,1}e^{i+2(n-2)}_{j-i,1}\cdots e^{i}_{j-i,1}$ for $x\in A_{i,j}$.

\begin{align*}
    S^n_{i,j}(x)e^{i+2(n-1)}_{j-i,1}e^{i+2(n-2)}_{j-i,1}\cdots e^{i}_{j-i,1} & = S_{i+2(n-1),j+2(n-1)}(S^{(n-1)}_{i,j}(x))e^{i+2(n-1)}_{j-i,1}e^{i+2(n-2)}_{j-i,1}\cdots e^{i}_{j-i,1}\\
    & = e^{i+2(n-1)}_{j-i,1}S^{(n-1)}_{i,j}(x)e^{i+2(n-2)}_{j-i,1}\cdots e^{i}_{j-i,1}\\
    & = \cdots\\
    & =  e^{i+2(n-1)}_{j-i,1}e^{i+2(n-2)}_{j-i,1}\cdots e^{i}_{j-i,1}x.
\end{align*}

Second, $e^i_{j-i,n}=a^i_{j-i,n}e^{i+2(n-1)}_{j-i,1}e^{i+2(n-2)}_{j-i,1}\cdots e^{i}_{j-i,1}b^i_{j-i,n}$ with $a^i_{j-i,n}\in A_{i,i+2n}$ and $b^i_{j-i,n}\in A_{j,j+2n}$, which will be showed below in Lemma \ref{lemma 1} and \ref{lemma 2}. Then by the standard condition, since $x\in A_{i,j}$ and $S^{(n)}(x)\in A_{i+2n,j+2n}$, we have $[S^{(n)}_{i,j}(x),a^i_{j-i,n}]=0$ and $[x,b^i_{j-i,n}]=0$, which follows that 
\begin{align*}
   S^{(n)}_{i,j}(x)e^i_{j-i,n} & = S^{(n)}_{i,j}(x)a^i_{j-i,n}e^{i+2(n-1)}_{j-i,1}e^{i+2(n-2)}_{j-i,1}\cdots e^{i}_{j-i,1} b^i_{j-i,n}\\
   & = a^i_{j-i,n}S^{(n)}_{i,j}(x)e^{i+2(n-1)}_{j-i,1}e^{i+2(n-2)}_{j-i,1}\cdots e^{i}_{j-i,1} b^i_{j-i,n}\\
   & = a^i_{j-i,n}e^{i+2(n-1)}_{j-i,1}e^{i+2(n-2)}_{j-i,1}\cdots e^{i}_{j-i,1}x b^i_{j-i,n}\\
   & = a^i_{j-i,n}e^{i+2(n-1)}_{j-i,1}e^{i+2(n-2)}_{j-i,1}\cdots e^{i}_{j-i,1} b^i_{j-i,n}x\\
   & = e^i_{j-i,n}x.
\end{align*}
\end{proof}

\subsection{String diagram explanation}\label{TLJ diagram explanation}
In this section, we use Temperley-Lieb-Jones (TLJ) string diagram to explain the elements in $A_{i,j}$, horizontal (right) and vertical (left) conditional expectations, the Jones projections, $2n$-shift maps and their properties. 

In the following sections, we will use these diagrams to do the algebraic computation and readers may interpret these diagrams directly into the algebraic computations by looking at the dictionary here. \\

\begin{compactenum}[({$\lambda$}1)]
\item Element $x\in A_{i,j}$. $A_{i,j}$ is a (rectangular) box space with $j$ shaded/unshaded strands where the left $i$ strands are straight strands and together with a $j-i$ box space. We set the left part of left most strand to be always unshaded; the shading on the left part of the $j-i$ box space depends on the parity of $i$: 
\[

\]

\textbf{Remark}: The reader shall understand the meaning of rectangular box and round box of an element. And the shading type of an element is the shading on the left of the round box. 

\item Horizontal inclusion $x\in A_{i,j}\subset A_{i,j+1}$. The inclusion $A_{i,j}\subset A_{i,j+1}$ means adding one straight strand on the right and regarding the $j-i$ box space in $A_{i,j}$ as a part of the $j-i+1$ box space in $A_{i,j+1}$ together with the straight strand, which does not change the shading type of the box space:
\[
%
\]

\item Vertical inclusion $x\in A_{i,j}\subset A_{i-1,j}$. The inclusion $A_{i,j}\subset A_{i-1,j}$ means regarding the right most straight strand together with the original $j-i$ box space in $A_{i,j}$ as a part of the $j-i+1$ box space in $A_{i-1,j}$, which changes the shading type of the box space: 
\[
%
\]

\textbf{Remark}: See the string diagram calculation of Jones projections in the Temperley-Lieb-Jones algebra.\\

\item Horizontal (right) conditional expectation $E^r_{i,j}:A_{i,j}\to A_{i,j-1}$, $x\in A_{i,j}$: 
\[ E^r_{i,j}(x)=
d^{-1}
%
\]

\item Vertical (left) conditional expectation $E^l_{i,j}:A_{i,j}\to A_{i+1,j}$, $x\in A_{i,j}$. The vertical (left) conditional expectation is the left conditional expectation acting on the left of the box space and then adding one straight strand on the left of the box space, which changes the shading type of box space:
\[ \text{If } 2\mid i:E^l_{i,j}(x)= d^{-1}
%
\]

\item Commuting square of conditional expectation: For $x\in A_{i,j}$, $E^l_{i,j}\circ E^r_{i,j+1}(x)=E^r_{i+1,j+1}\circ E^l_{i,j+1}(x)$:
\[E^l_{i,j}\circ E^r_{i,j+1}(x)=
%
\]

\item Conditional expectation property $E^r_{i,j}(axb)=aE^r_{i,j}(x)b$, for $x\in A_{i,j},\ a,b\in A_{i,j-1}$; $E^l_{i,j}(axb)=aE^l_{i,j}(x)b$, for $x\in A_{i,j},\ a,b\in A_{i+1,j}$.
\[
d^{-1}\
%
\]
\item Standard condition: For $x\in A_{i,j}$, $y\in A_{k,l}$ with $k\ge j$, then we regard $x,y$ as elements in $A_{i,l}$, $xy=yx$.
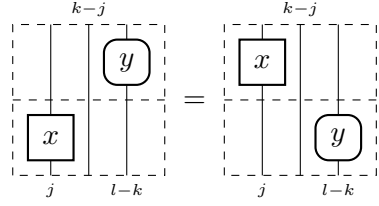
\[
%
\]
\item Pull down condition

$d^2 E^r_{i,j+1}(x e_j)e_j=x e_j$, for $x\in A_{i,j+1}$, $j\ge i\ge 0$;

$d^2 E^l_{i-1,j}(x e_i)e_i=x e_i$, for $x\in A_{i-1,j}$, $j\ge i\ge 1$: 
\[
%
\]

\item Commuting parallelogram:

For $x\in A_{i,j}$, $S^{(n)}_{i,j-1}\circ E^{r,k}_{i,j}(x)=E^{r,k}_{i+2n,j+2n}\circ S^{(n)}_{i,j}(x)$;

For $x\in A_{i,j}$, $S^{(n)}_{i+1,j}\circ E^{l,k}_{i,j}(x)=E^{l,k}_{i+2n,j+2n}\circ S^{(n)}_{i,j}(x)$. 
\[
S^{(n)}_{i,j-1}\circ E^{r,k}_{i,j}(x)=
%
=E^{l,k}_{i+2n,j+2n}\circ S^{(n)}_{i,j}(x)
\]

\item Shift property: For $x\in A_{i,j}$, $e^i_{j,k}x=S^k_{i,j}(x)e^i_{j,k}$.
\[
%
\]
\end{compactenum}

\subsection{Some useful lemmas}
In this section, we are going to show some important lemmas. One can interpret the string diagram computation into algebraic computation by the above dictionary.
\begin{lemma}\label{lemma 1}
\[
%
\]
\end{lemma}
\begin{proof}
By the above lemma. 
\end{proof}

\vspace{.5cm}
These two lemmas are used in the proof of Proposition \ref{2n-shift map prop}(4). 

\begin{lemma}\label{lemma 3}
\[
%
\]

\noindent List of the formulas used in above equalities:

\noindent \circle1: top uses ($\lambda$8) and bottom uses Jones projection property; \hfill \circle2: uses ($\lambda$9);

\noindent \circle3: middle uses ($\lambda$8) and bottom uses Jones projection property; \hfill \circle4: uses ($\lambda$9);

\noindent \circle5: uses ($\lambda$7); \hfill \circle6: uses ($\lambda$9);

\noindent \circle7: uses ($\lambda$8). 
\end{proof}

\subsection{From standard $\lambda$-lattice to pivotal planar tensor category} \label{std lattice to planar cat}
\subsubsection{Planar tensor category}\label{planar tensor category}
\begin{definition}
A planar tensor category $\mathcal{A}_0$ has the following properties.
\begin{compactenum}[\rm (a)]
\item $\mathcal{A}_0$ is a 2-shaded category with objects $[n,+],[n,-]$, $n\in\mathbb{Z}_{\ge 0}$, where $1^+:=[0,+],1^-:=[0,-]$ are simple and the tensor unit $1_{\mathcal{A}_0}=1^+\oplus 1^-$, which means $\mathcal{A}_0$ is 2-shaded. 
\item $\mathcal{A}_0$ is a strict tensor category. The tensor product of objects are
$$\begin{tabu} to 0.8\textwidth { | X[c] | X[c] | X[c] | X[c] | X[c] | }
 \hline
  $[m,?]\otimes [n,?]$ & $[2i,+]$ & $[2i+1,+]$ & $[2i,-]$ & $[2i+1,-]$ \\
 \hline
 $[n,+]$  & $[2i+n,+]$  & 0 & 0 & $[2i+1+n,-]$  \\
\hline
$[n,-]$  & 0 & $[2i+1+n,+]$ & $[2i+n,-]$ & 0  \\
\hline
\end{tabu}$$
There is an involution $\overline{(\cdot)}$ such that $\overline{[2i,\pm]}=[2i,\pm]$, $\overline{[2i+1,+]}=[2i+1,-]$ and $\overline{\overline{(\cdot)}}=\id$.
\item Only $\mathcal{A}_0([m,+]\to [m\pm 2i,+])$ and $\mathcal{A}_0([m,-]\to[m\pm 2i,-])$ are non-empty, $m,i\in\mathbb{Z}_{\ge 0}$, and $\mathcal{A}_0([m,+]\to[m,+])$ and $\mathcal{A}_0([m,-]\to[m,-])$ are finite-dimensional ${\rm C}^*$-algebras. The tensor product of morphisms should match the shading types. 

\item $\mathcal{A}_0$ is a dagger category. There is a dagger structure $\dagger$ such that $[n,+]^\dagger=[n,+],[n,-]^\dagger=[n,-]$ and a anti-linear map $\mathcal{A}_0([m,?]\to [n,?])\to \mathcal{A}_0([n,?]\to [m,?])$ with $\dagger^2=\id$ such that morphism $(x\circ y)^\dagger= y^\dagger\circ x^\dagger$ and $(x\otimes y)^\dagger = x^\dagger\otimes y^\dagger$. In fact, $\mathcal{A}_0$ is a $\rm C^*$ category, see \cite[\S3.4]{CHPS18}.

\item $\mathcal{A}_0$ is rigid. For $X\in \mathcal{A}_0$, there exist
\begin{compactenum}[(1)]
\item $\ev_{X}:\overline{X}\otimes X\to 1^?$, where $?=+$ if $X$ is unshaded on the right, i.e., $X=1^+\otimes X$, $?=-$ if $X$ is shaded on the right, i.e., $X=1^-\otimes X$;
\item $\coev_{X}:1^?\to X\otimes\overline{X}$, where $?=+$ if $X$ is unshaded on the left, $?=-$ if $X$ is shaded on the left.
\end{compactenum}
such that 
\begin{compactenum}[$\bullet$]
\item $(\id_X\otimes \ev_X)\circ (\coev_X\otimes \id_X)=\id_X$
\item $(\ev_X\otimes\id_{\overline{X}})\circ (\id_{\overline{X}}\otimes \coev_X)=\id_{\overline{X}}$.
\item $\ev_{\overline{X}}:=(\coev_X)^\dagger$ and $\coev_{\overline{X}}=(\coev_X)^\dagger$.
\end{compactenum}
In other word, $\overline{(\cdot)}$ is a unitary dual functor, which will be discussed in \S\ref{unitary dual functor}.
\end{compactenum} 
\end{definition}

\begin{definition}\label{pivotal planar tensor category}
We call a planar tensor category $\mathcal{A}_0$ \textbf{pivotal}, if the left trace $\Tr_L$ and right trace $\Tr_R$ defined as follows are faithful normal tracial and equal. For $X=[2k+1,+]$ and $f\in \mathcal{A}_0(X\to X)$, since $\overline{[2k+1,+]}=[2k+1,-]$, we define
\begin{align*}
    \ev_{X}\circ (\id_{\overline{X}}\otimes f)\circ \ev_{X}^\dagger &=: \Tr_L(f)\id_{1^+}\\
    \coev_{X}^\dagger\circ (f\otimes \id_{\overline{X}})\circ \coev_{X} &=: \Tr_R(f)\id_{1^-}
\end{align*}
We require that $\Tr_R(f)=\Tr_L(f)$. Similar for other three cases $[2k,+],[2k,-]$ and $[2k+1,-]$.

And there exists a $d>0$ such that $\ev_{\overline{[n,?]}}\circ \coev_{[n,?]}=d^{2n}\cdot 1^?$, $?=+,-$.
\end{definition}

\begin{remark}
The traces $\Tr_L,\Tr_R$ are defined in the sense of Definition \ref{Def:spherical trace}.
\end{remark}

\begin{definition}\label{TLJ(d) category}
The 2-shaded Temperley-Lieb-Jones multitensor category $\mathcal{TLJ}(d)$ is a planar tensor category with the endomorphism spaces being 2-shaded Temperley-Lieb-Jones algebras with modulus $d$, namely, $\End([n,+])$ is a 2-shaded Temperley-Lieb-Jones algebra with $n$ points on one side and unshaded on the left; $\End([n,-])$ is a 2-shaded Temperley-Lieb algebra with $n$ points on one side and shaded on the left.
\end{definition}

\begin{remark} \label{main idea remark endo,domain,range}
The morphisms in $\mathcal{A}_0$ are determined by its representation in endomorphism and its domain and range.

There is a canonical isomorphism $\phi:\mathcal{A}_0([m,+],[m+ 2i,+])\to \mathcal{A}_0([m+i,?]\to [m+i,?])$ by Frobenius reciprocity, where $?=+$ if $i$ is even and $?=-$ if $i$ is odd.

\[\phi:
\begin{tikzpicture}[baseline=-.1cm]
\draw (0,-.7) -- (0,0);
\draw (-.1,0) -- (-.1,.7);
\draw (.1,0) -- (.1,.7);

\nbox{unshaded}{(0,0)}{.3}{0}{0}{$x$};
\node at (0,-.9) {\tiny{$m$}};
\node at (-.2,.9) {\tiny{$m\+ i$}};
\node at (.2,.5) {\tiny{$i$}};
\end{tikzpicture}
\mapsto
\begin{tikzpicture}[baseline=-.1cm]
\draw (0,-.7) -- (0,0);
\draw (-.1,0) -- (-.1,.7);
\draw (.1,0) -- (.1,.5);
\draw (.1,.5) arc (180:0:.15);
\draw (.4,-.7) -- (.4,.5);

\nbox{unshaded}{(0,0)}{.3}{0}{0}{$x$};
\node at (0,-.9) {\tiny{$m$}};
\node at (-.1,.9) {\tiny{$m\+ i$}};
\node at (.4,-.9) {\tiny{$i$}};
\end{tikzpicture}
\qquad\qquad \phi^{-1}:\ 
\begin{tikzpicture}[baseline=-.1cm]
\draw (0,0) -- (0,.7);
\draw (-.1,-.7) -- (-.1,0);
\draw (.1,-.7) -- (.1,0);
\nbox{unshaded}{(0,0)}{.3}{0}{0}{$x$};
\node at (0,.9) {\tiny{$m\+i$}};
\node at (-.1,-.9) {\tiny{$m$}};
\node at (.1,-.9) {\tiny{$i$}};
\end{tikzpicture}
\mapsto
\begin{tikzpicture}[baseline=-.1cm]
\draw (0,0) -- (0,.7);
\draw (-.1,-.7) -- (-.1,0);
\draw (.1,-.5) -- (.1,0);
\draw (.1,-.5) arc (-180:0:.15);
\draw (.4,-.5) -- (.4,.7);
\nbox{unshaded}{(0,0)}{.3}{0}{0}{$x$};
\node at (0,.9) {\tiny{$m\+i$}};
\node at (-.1,-.9) {\tiny{$m$}};
\node at (.4,.9) {\tiny{$i$}};
\end{tikzpicture}
\]

For morphism $x\in \mathcal{A}([m,?]\to [n,?])$, we can write a triple $(\phi(x);[m,?],[n,?])$ to represent $x$, where $\phi(x)\in \End([\frac{m+n}{2},?])$, which is called the \textbf{endomorphism representation part} of $x$. In the following context, we simply write $x$ instead of $\phi(x)$ in the triple $(x;[m,?],[n,?])$. 
\end{remark}

\subsubsection{From standard $\lambda$-lattice to pivotal planar tensor category}
\label{triple represent morphism}
We regard the elements in algebra $A_{i,j}$ as endomorphisms in the category
and the idea in Remark \ref{main idea remark endo,domain,range} gives us the way to represent the morphism by using its corresponding endomorphism, source and target,
then we can construct a pivotal planar tensor category from a given standard $\lambda$-lattice.

\begin{definition}\label{Def:std lambda lattice to planar tensor category}
Let $A=(A_{i,j})_{0\le i\le j}$ be a standard $\lambda$-lattice. We define a planar tensor category $\mathcal{A}_0$ from $A$ as follows.
\begin{compactenum}[(a)]
\item The objects of $\mathcal{A}_0$ are the symbols $[n,+],[n,-]$ for $n\in\mathbb{Z}_{\ge 0}$.
\item Given $n\ge 0$, define $\mathcal{A}_0([n,+]\to [n,+]):= A_{0,n}$ and $\mathcal{A}_0([n,-]\to [n,-]):= A_{1,n+1}$. Define $1:=[0,+]\oplus [0,-]$.
\item The identity morphism in $\mathcal{A}_0([n,+]\to [n,+])$ is $1_{A_{0,n}}$ and in $\mathcal{A}_0([n,-]\to [n,-])$ is $1_{A_{1,n+1}}$.
\item For $(x;[n,+],[n+2k,+])$ (or $(x;[n+2k,+],[n,+])$), we define the dagger structure as $(x;[n,+],[n+2k,+])^\dagger := (x^*;[n+2k,+],[n,+])$, where $x,x^*\in A_{0,n+k}$; for $(x;[n,-],[n+2k,-])$ (or $(x;[n+2k,-],[n,-])$), we define $(x;[n,-],[n+2k,-])^\dagger := (x^*;[n+2k,-],[n,-])$, where $x,x^*\in A_{1,n+k+1}$.
\item We define composition in six cases.
\begin{compactenum}[{(C}1{)}]
\item $(y;[n+2i,+],[n+2i+2j,+])\circ (x;[n,+],[n+2i,+])= (d^i E^{r,i}_{0,n+2i+j}(yxe^n_{j,i});[n,+],[n+2i+2j,+])$, where $x\in A_{0,n+i},y\in A_{0,n+2i+j}$ and $d^i E^{r,i}_{0,n+2i+j}(yxe^n_{j,i})\in A_{0,n+i+j}$. 
\item $(y;[n+2i+2j,+],[n+2i,+])\circ (x;[n,+],[n+2i+2j,+])= (d^i E^{r,i+j}_{0,n+2i+j}(yxe^{n,*}_{j,i});[n,+],[n+2i,+])$, where $x\in A_{0,n+i+j},y\in A_{0,n+2i+j}$ and $d^i E^{r,i+j}_{0,n+2i+j}(yxe^{n,*}_{j,i})\in A_{0,n+i}$. 
\item $(y;[n,+],[n+2i+2j,+])\circ (x;[n+2i,+],[n,+])= (d^i ye^{n,*}_{j,i}x;[n+2i,+],[n+2i+2j,+])$, where $x\in A_{0,n+i},y\in A_{0,n+i+j}$ and $d^i ye^{n,*}_{j,i}x\in A_{0,n+2i+j}$. 
\item $(y;[n+2i,-],[n+2i+2j,-])\circ (x;[n,-],[n+2i,-])= (d^i E^{r,i}_{1,n+2i+j+1}(yxe^{n+1}_{j,i});[n,+],[n+2i+2j,+])$, where $x\in A_{1,n+i+1},y\in A_{1,n+2i+j+1}$ and $d^i E^{r,i}_{1,n+2i+j+1}(yxe^{n+1}_{j,i})\in A_{1,n+i+j+1}$. 
\item $(y;[n+2i+2j,-],[n+2i,-])\circ (x;[n,-],[n+2i+2j,-])= (d^i E^{r,i+j}_{1,n+2i+j+1}(yxe^{n+1,*}_{j,i});[n,-],[n+2i,-])$, where $x\in A_{1,n+i+j+1},y\in A_{1,n+2i+j+1}$ and $d^i E^{r,i+j}_{1,n+2i+j+1}(yxe^{n+1,*}_{j,i})\in A_{1,n+i+1}$. 
\item $(y;[n,-],[n+2i+2j,-])\circ (x;[n+2i,-],[n,-])= (d^i ye^{n+1,*}_{j,i}x;[n+2i,-],[n+2i+2j,-])$, where $x\in A_{1,n+i+1},y\in A_{1,n+i+j+1}$ and $d^i ye^{n+1,*}_{j,i}x\in A_{1,n+2i+j+1}$. 
\end{compactenum}

If $x\in \mathcal{A}_0([n+2i,-]\to [n,-])$ and $y\in \mathcal{A}_0([n,-]\to [n+2i+2j,-])$, we define
$$y\circ x:= d^i ye^{n+1,*}_{j,i}x\in A_{1,n+2i+j+1}=\mathcal{A}_0([n+2i,-]\to [n+2i+2j,-]).$$
We define the composition $x^\dagger\circ y^\dagger:= (y\circ x)^\dagger$, which defines composition
$$\mathcal{A}_0([n+2i+2j,-]\to [n,-])\otimes \mathcal{A}_0([n,-]\to [n+2i,-])\to \mathcal{A}_0([n+2i+2j,-]\to [n+2i,-]).$$

\end{compactenum}
\end{definition}
It has been proved in \cite[\S3.4]{CHPS18} that the composition and dagger structure are well defined as Markov tower, and $\mathcal{A}_0$ is a ${\rm C}^*$ category.\\

Before we define the tensor product of morphisms, we use the string diagrams to explain the composition. The box space in the following diagram is always the endomorphism representation of corresponding morphism.

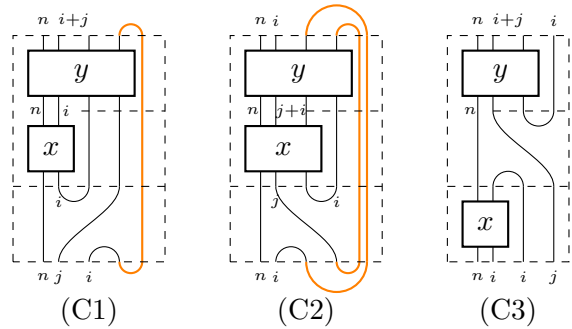
\begin{center}
\begin{tabular}{ c c c c c }
\begin{tikzpicture}[baseline=0.5cm]
\draw (0.4,-1) -- (0.4,2);
\draw (0.6,0) -- (0.6,2);
\draw (1,0) -- (1,2);  
\draw (1.4,0) -- (1.4,2);

\draw (1,-1) arc (180:0:.2cm);
\draw (.6,0) arc (-180:0:.2cm);
\draw (.6,-1)  .. controls ++(90:.35cm) and ++(270:.35cm) .. (1.4,0);

\draw[orange,thick] (1.7,-1) -- (1.7,2);
\draw[orange,thick] (1.4,2) arc (180:0:.15cm);
\draw[orange,thick] (1.4,-1) arc (-180:0:.15cm);

\nbox{unshaded}{(.9,1.5)}{.3}{0.4}{0.4}{$y$};
\nbox{unshaded}{(.5,.5)}{.3}{0}{0}{$x$};
  
\draw[dashed] (0.9,1) -- (2,1);
\draw[dashed] (0,0) -- (2,0);
\draw[dashed] (0,-1) rectangle (2,2);
\node at (0.4,2.2) {\tiny{$n$}};
\node at (0.8,2.2) {\tiny{$i\+j$}};
\node at (.7,1) {\tiny{$i$}};
\node at (.4,-1.2) {\tiny{$n$}};
\node at (.3,1) {\tiny{$n$}};
\node at (.6,-.2) {\tiny{$i$}};
\node at (1,-1.2) {\tiny{$i$}};
\node at (.6,-1.2) {\tiny{$j$}};
\end{tikzpicture} 
& &
\begin{tikzpicture}[baseline=0.5cm]
\draw (0.4,-1) -- (0.4,2);
\draw (0.6,0) -- (0.6,2);
\draw (1,0) -- (1,2);  
\draw (1.4,0) -- (1.4,2);

\draw (.6,-1) arc (180:0:.2cm);
\draw (1,0) arc (-180:0:.2cm);
\draw (1.4,-1)  .. controls ++(90:.35cm) and ++(270:.35cm) .. (0.6,0);

\draw[orange,thick] (1.7,-1) -- (1.7,2);
\draw[orange,thick] (1.4,2) arc (180:0:.15cm);
\draw[orange,thick] (1.4,-1) arc (-180:0:.15cm);
 
\draw[orange,thick] (1.8,-1) -- (1.8,2);
\draw[orange,thick] (1,2) arc (180:0:.4cm);
\draw[orange,thick] (1,-1) arc (-180:0:.4cm);

\nbox{unshaded}{(.9,1.5)}{.3}{0.4}{0.4}{$y$};
\nbox{unshaded}{(.7,.5)}{.3}{0.2}{0.2}{$x$};
  
\draw[dashed] (1,1) -- (2,1);
\draw[dashed] (0,0) -- (2,0);
\draw[dashed] (0,-1) rectangle (2,2);
\node at (0.4,2.2) {\tiny{$n$}};
\node at (0.6,2.2) {\tiny{$i$}};
\node at (.8,1) {\tiny{$j\+ i$}};
\node at (.4,-1.2) {\tiny{$n$}};
\node at (.3,1) {\tiny{$n$}};
\node at (.6,-.2) {\tiny{$j$}};
\node at (1.4,-.2) {\tiny{$i$}};
\node at (.6,-1.2) {\tiny{$i$}};
\end{tikzpicture}
& &
\begin{tikzpicture}[baseline=0.5cm]
\draw (0.4,-1) -- (0.4,2);
\draw (0.6,0) -- (0.6,-1);
\draw (0.6,1) -- (0.6,2);
\draw (1,1) -- (1,2);  
\draw (1,0) -- (1,-1);  
\draw (1.4,0) -- (1.4,-1);
\draw (1.4,2) -- (1.4,1);

\draw (.6,0) arc (180:0:.2cm);
\draw (1,1) arc (-180:0:.2cm);
\draw (1.4,0) .. controls ++(90:.35cm) and ++(270:.35cm) .. (.6,1);

\nbox{unshaded}{(.7,1.5)}{.3}{0.2}{0.2}{$y$};
\nbox{unshaded}{(.5,-.5)}{.3}{0}{0}{$x$};
  
\draw[dashed] (0.6,1) -- (1.6,1);
\draw[dashed] (0,0) -- (1.6,0);
\draw[dashed] (0,-1) rectangle (1.6,2);
 
\node at (0.4,2.2) {\tiny{$n$}};
\node at (0.8,2.2) {\tiny{$i\+j$}};
\node at (1.4,-1.2) {\tiny{$j$}};
\node at (.4,-1.2) {\tiny{$n$}};
\node at (.3,1) {\tiny{$n$}};
\node at (1.4,2.2) {\tiny{$i$}};
\node at (1,-1.2) {\tiny{$i$}};
\node at (.6,-1.2) {\tiny{$i$}};
\end{tikzpicture}\\
(C1) & & (C2) & & (C3)
\end{tabular}
\end{center}
The string diagram of case (C4) comes from the string diagram of case (C1) by adding a straight strand on the leftmost of the diagram and change the shading. In the same way, we obtain (C5) from (C2) and (C6) from (C3).\\

Now we define the tensor product of morphisms.

\begin{definition} \label{x otimes 1, 1 otimes y} $x\otimes 1$ and $1\otimes y$, $x,y\in \Hom(\mathcal{A}_0)$: 

First, we define $x\otimes 1$ as 
\begin{center}
\small
\begin{tabular}{ |c|c| } 
 \hline
 $x$ & $x\otimes 1_j$  \\
 \hline
  $(x;[m,+], [m+2i,+]),\ i\le j$ & $(xe^m_{j-i,i};[m+j,+],[m+2i+j,+])$  \\
\hline
  $(x;[m,+], [m+2i,+]),\ i>j$ & $(xe^{m,*}_{i-j,j};[m+j,+],[m+2i+j,+])$  \\
\hline
  $(x;[m,-], [m+2i,-]),\ i\le j$ & $(xe^{m+1}_{j-i,i};[m+j,-],[m+2i+j,-])$  \\
\hline
  $(x;[m,-], [m+2i,-]),\ i> j$  & $(xe^{m+1,*}_{i-j,j};[m+j,-],[m+2i+j,-])$  \\
\hline
\end{tabular}
\end{center}

Because of the shading, we define $1\otimes y$ as:
\begin{center}
\small
\begin{tabular}{ |c|c|c| } 
 \hline
 $y$ & $1_{2i}\otimes y$ & $1_{2i+1}\otimes y$ \\ 
 \hline
 $(y;[n,+], [n\pm 2j,+])$ & $(S^{(i)}_{0,n\pm j}(y);[n+2i,+],[n+2i\pm 2j,+])$ & 0\\
 \hline
 $(y;[n,-], [n\pm 2j,-])$ & 0 & $(S^{(i)}_{1,n+1\pm j}(y);[n+2i,-],[n+2i\pm 2j,-])$ \\
 \hline
\end{tabular}

\normalsize
\begin{tabular}{ c c c }
\begin{tikzpicture}[baseline=1cm]
\draw (0.4,0) -- (0.4,2);
\draw (0.6,1) -- (0.6,2);
\draw (1,1) -- (1,2);
\draw (1.4,1) -- (1.4,2);
\nbox{unshaded}{(.5,1.5)}{.3}{0.0}{0.0}{$x$};

\draw (.6,1) arc (-180:0:.2cm);
\draw (1,0) arc (180:0:.2cm);
\draw (.6,0)  .. controls ++(90:.35cm) and ++(270:.35cm) .. (1.4,1);
 
\draw[dashed] (0,1) -- (1.6,1);
\draw[dashed] (0,0) rectangle (1.6,2);
\node at (0.4,-.2) {\tiny{$n$}};
\node at (1.4,2.2) {\tiny{$j\sm i$}};
\node at (1,-.2) {\tiny{$i$}};
\node at (.5,2.2) {\tiny{$n\+i$}};
\node at (1,2.2) {\tiny{$i$}};
\end{tikzpicture}
& &
\begin{tikzpicture}[baseline=1cm]
\draw (0.4,0) -- (0.4,2);
\draw (0.6,1) -- (0.6,1.5);
\draw (1,1) -- (1,2);
\draw (1.4,1) -- (1.4,2);
\nbox{unshaded}{(.7,1.5)}{.3}{0.2}{0.2}{$x$};

\draw (1,1) arc (-180:0:.2cm);
\draw (.6,0) arc (180:0:.2cm);
\draw (1.4,0)  .. controls ++(90:.35cm) and ++(270:.35cm) .. (.6,1);
 
\draw[dashed] (0,1) -- (1.6,1);
\draw[dashed] (0,0) rectangle (1.6,2);
\node at (0.4,-.2) {\tiny{$n$}};
\node at (0.6,-.2) {\tiny{$j$}};
\node at (1.4,-.2) {\tiny{$i\sm j$}};
\node at (.4,2.2) {\tiny{$n$}};
\node at (1,2.2) {\tiny{$i$}};
\node at (1.4,2.2) {\tiny{$j$}};
\end{tikzpicture}\\
$i\le j$ & & $i>j$ 
\end{tabular}
\qquad
\begin{tabular}{ c c c }
\begin{tikzpicture}
\draw[cyan,thick] (-.5,-.7) -- (-.5,.7);
 \draw (0,-.7) -- (0,.7);
 \roundNbox{unshaded}{(0,0)}{.3}{0.0}{0.0}{$y$};
 \node at (0,-.9) {\tiny{$n$}};
 \node at (-.5,-.9) {\tiny{$2i$}};
\end{tikzpicture}
& &
\begin{tikzpicture}
\fill[shaded] (-.5,-.7) rectangle (0,.7);
\draw[cyan,thick] (-.8,-.7) -- (-.8,.7);
\draw (-.5,-.7) -- (-.5,.7);
 \draw (0,-.7) -- (0,.7);
 \roundNbox{unshaded}{(0,0)}{.3}{0.0}{0.0}{$y$};
 \node at (0,-.9) {\tiny{$n\sm 1$}};
 \node at (-.8,-.9) {\tiny{$2i$}};
 \node at (-.5,-.9) {\tiny{$1$}};
\end{tikzpicture}
\end{tabular}
\end{center}
\end{definition}

\begin{proposition}\label{tensor product def prop}
For $x,y\in \Hom(\mathcal{A}_0)$, $(x\otimes 1)\circ (1\otimes y)=(1\otimes y)\circ (x\otimes 1)$. 
\end{proposition}
\begin{proof} Here, we check the case
$(x;[m,+],[m+2i,+])$ and $(y;[n,+],[n+2j,+])$, where $2\mid m$ (or $(y;[n,-],[n+2j,-])$ if $2\nmid m$) and $n+j\ge i$. We shall prove that 
\begin{align*}
     & ((x;[m,+],[m+2i,+])\otimes (1;[n+2j,+],[n+2j,+]))\circ ((1;[m,+],[m,+])\otimes (y;[n,+],[n+2j,+])) \\
    =& ((1;[m+2i,+],[m+2i,+])\otimes (y;[n,+],[n+2j,+]))\circ ((x;[m,+],[m+2i,+])\otimes (1;[n,+];[n,+]))
\end{align*}

First, they both in $\mathcal{A}_0([m+n,+]\to [m+n+2i+2j,+])$. 

The right hand side:

$((1;[m+2i,+],[m+2i,+])\otimes (y;[n,+],[n+2j,+]))\circ ((x;[m,+],[m+2i,+])\otimes (1;[n,+];[n,+]))$: 
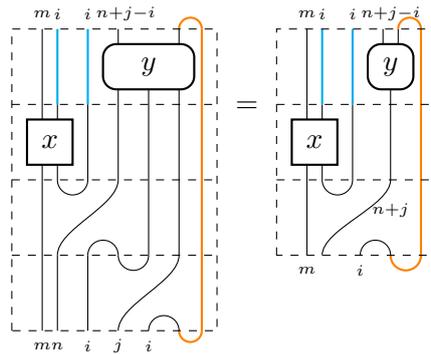
\[

\]

The left hand side:

$((x;[m,+],[m+2i,+])\otimes (1;[n+2j,+],[n+2j,+]))\circ ((1;[m,+],[m,+])\otimes (y;[n,+],[n+2j,+]))$:

\item[(1)] If $i\le j$,
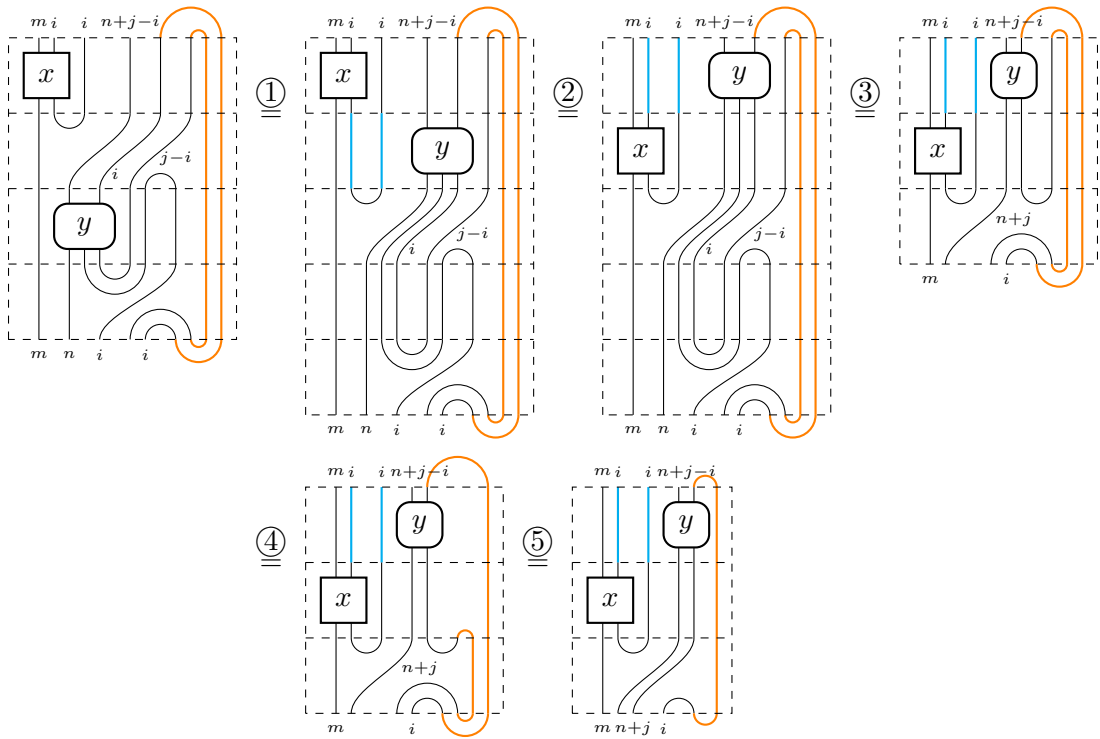
\begin{align*}
%
\end{align*}

\noindent List the formulas used in above equalities: 

\noindent \circle1: uses ($\lambda$16); \hfill \circle2: uses ($\lambda$11) and ($\lambda$15);

\noindent \circle3: Jones projection property; \hfill  \circle4: uses ($\lambda$10);

\noindent \circle5: uses Lemma \ref{lemma 3}.

\item[(2)] If $i>j$,
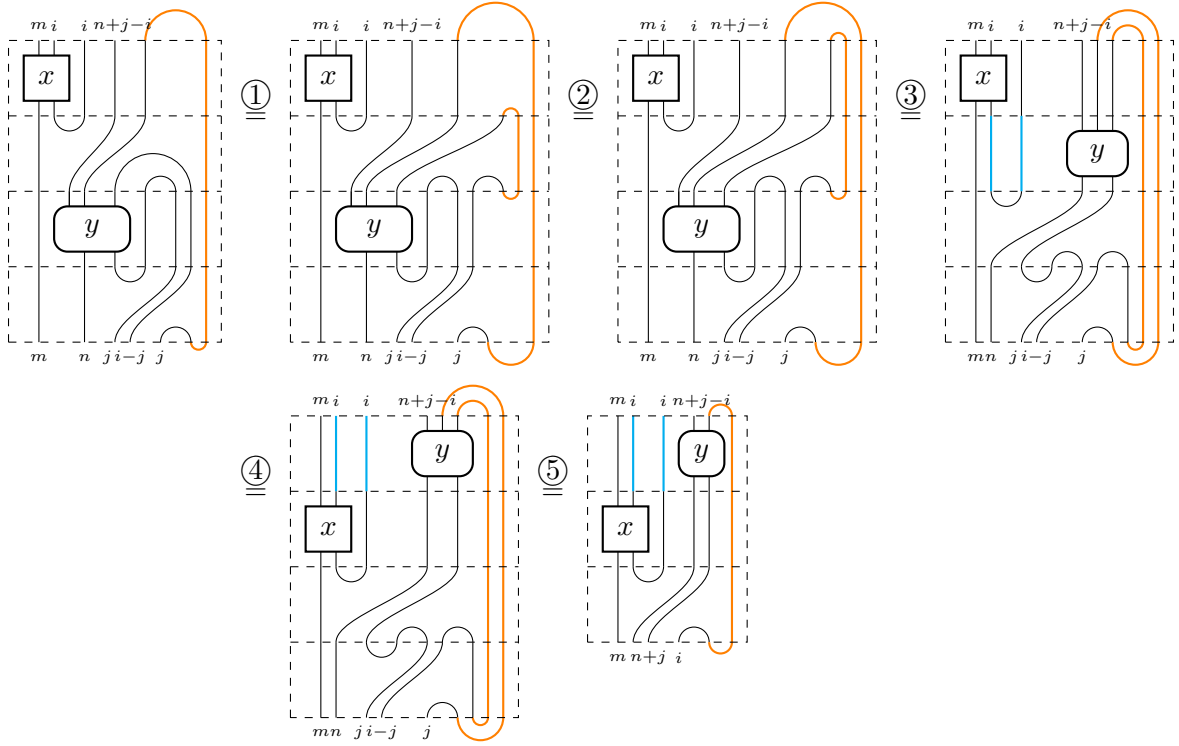
\begin{align*}
%
\end{align*}

\noindent List of the formulas used in above equalities: 

\noindent \circle1: uses Lemma \ref{lemma 3}; \hfill \circle2: uses ($\lambda$10);

\noindent \circle3: uses ($\lambda$16) and Lemma \ref{lemma 2}; \hfill \circle4: uses ($\lambda$11) and ($\lambda$15);

\noindent \circle5: Jones projection property.\\

Therefore, $(x\otimes 1)\circ (1\otimes y)=(1\otimes y)\circ (x\otimes 1)$ in this case. The remaining cases are left to the reader. 
\end{proof}

\begin{definition}[tensor product of morphisms]
Define $x\otimes y :=(x\otimes 1)\circ (1\otimes y)$.\label{definion x otimes y}
\end{definition}

We need to prove that the tensor product defined above is functorial and associative.

\begin{proposition}\label{strictness prop}
Tensor product is associative and strict, i.e., 
for $x,y,z\in \Hom(\mathcal{A}_0)$, $(x\otimes y)\otimes z= x\otimes (y\otimes z)$.
\end{proposition}
\begin{proof}Here,we check the case $(x;[m,+],[m+2i,+])$, $(y;[n,+],[n+2j,+])$ and $(z;[l,-],[l+2k,-])$, where $2\mid m,2\nmid n$ and $n+j\ge i$, $l+k\ge i+j$. Then $(x\otimes y)\otimes z$, $x\otimes (y\otimes z)\in \mathcal{A}_0([m+n+l,+]\to [m+n+l+2i+2j+2k,+])$. 

By Proposition \ref{tensor product def prop}, the endomorphism representation parts of $x\otimes y$ and $y\otimes z$ are defined in this way:  
\[
\begin{tikzpicture}[baseline=1.9cm]
\draw (0.4,0) -- (0.4,3);
\draw[cyan,thick] (.6,2) -- (.6,3);
\draw[cyan,thick] (1,2) -- (1,3);
\draw (.6,1) -- (.6,2);
\draw (1,1) -- (1,2);
\draw (1.4,2.5) -- (1.4,3);
\draw (1.6,2.5) -- (1.6,3);
\draw (1.5,1) -- (1.5,2.5);
\roundNbox{unshaded}{(1.5,2.5)}{.3}{0}{0}{$y$};

\draw (.6,1) arc (-180:0:.2cm);
\draw (.6,0)  .. controls ++(90:.35cm) and ++(270:.35cm) .. (1.5,1);
\draw (1.1,0) arc (180:0:.2cm);
\nbox{unshaded}{(0.5,1.5)}{.3}{0.0}{0.0}{$x$};

\draw[orange,thick] (1.9,0) -- (1.9,3);
\draw[orange,thick] (1.6,3) arc (180:0:.15cm);
\draw[orange,thick] (1.5,0) arc (-180:0:.2cm);

\draw[dashed] (0,2) -- (2.1,2);
\draw[dashed] (0,1) -- (2.1,1);
\draw[dashed] (0,0) rectangle (2.1,3);
 
\node at (.4,3.2) {\tiny{$m$}};
\node at (.6,3.2) {\tiny{$i$}};
\node at (1,3.2) {\tiny{$i$}};
\node at (1.5,3.2) {\tiny{$n\+ j\sm i$}};
 
\node at (.4,-.2) {\tiny{$m$}};
\node at (1.5,0.6) {\tiny{$n\+ j$}};
\node at (1.1,-.2) {\tiny{$i$}};
\end{tikzpicture}
\qquad\qquad
\begin{tikzpicture}[baseline=1.9cm]
\draw (0.4,0) -- (0.4,3);
\draw[cyan,thick] (.6,2) -- (.6,3);
\draw[cyan,thick] (1,2) -- (1,3);
\draw (.6,1) -- (.6,2);
\draw (1,1) -- (1,2);
\draw (1.2,1) -- (1.2,3);
\draw (.6,0)  .. controls ++(90:.35cm) and ++(270:.35cm) .. (1.2,1);
\draw (1.6,2.5) -- (1.6,3);
\draw (1.8,2.5) -- (1.8,3);
\draw (1.7,1) -- (1.7,2.5);
\roundNbox{unshaded}{(1.7,2.5)}{.3}{0}{0}{$z$};
\draw (.6,1) arc (-180:0:.2cm);
\draw (.8,0)  .. controls ++(90:.35cm) and ++(270:.35cm) .. (1.7,1);
\draw (1.3,0) arc (180:0:.2cm);
\nbox{unshaded}{(0.5,1.5)}{.3}{0.0}{0.0}{$y$};
\draw[orange,thick] (2.1,0) -- (2.1,3);
\draw[orange,thick] (1.8,3) arc (180:0:.15cm);
\draw[orange,thick] (1.7,0) arc (-180:0:.2cm);
\draw[dashed] (0,2) -- (2.3,2);
\draw[dashed] (0,1) -- (2.3,1);
\draw[dashed] (0,0) rectangle (2.3,3);
\node at (.4,3.2) {\tiny{$n$}};
\node at (.6,3.2) {\tiny{$j$}};
\node at (1,3.2) {\tiny{$j$}};
\node at (1.7,3.2) {\tiny{$l\+ k\sm j$}};
\node at (.4,-.2) {\tiny{$n$}};
\node at (.6,-.2) {\tiny{$1$}};
\node at (1.7,0.6) {\tiny{$l\+ k$}};
\node at (1.3,-.2) {\tiny{$j$}};
\end{tikzpicture}
\]

Then $(x\otimes y)\otimes z$:

\[\hspace*{-1.2cm}
\begin{tikzpicture}[baseline=2.9cm]
\draw (.4,0) -- (.4,5);
\draw (.6,2) -- (.6,3);
\draw (1,2) -- (1,3);
\draw[cyan,thick] (.6,3) -- (.6,5);
\draw[cyan,thick] (1,3) -- (1,5);
\draw[cyan,thick] (1.4,4) -- (1.4,5);
\draw (1.4,2) -- (1.4,4);
\draw (1.6,2) -- (1.6,4);
\roundNbox{unshaded}{(1.5,3.5)}{.3}{0}{0}{$y$};
\draw (.6,2) arc (-180:0:.2cm);
\draw (.6,1)  .. controls ++(90:.35cm) and ++(270:.35cm) .. (1.4,2);
\draw (.8,1)  .. controls ++(90:.35cm) and ++(270:.35cm) .. (1.6,2);
\draw (1.1,1) arc (180:0:.2cm);
\nbox{unshaded}{(0.5,2.5)}{.3}{0.0}{0.0}{$x$};
\draw[orange,thick] (1.9,1) -- (1.9,4);
\draw[orange,thick] (1.6,4) arc (180:0:.15cm);
\draw[orange,thick] (1.5,1) arc (-180:0:.2cm);
\draw (.6,0) -- (.6,1);
\draw (1.1,1) arc (-180:0: .5cm and .3cm);
\draw (.8,1) arc (-180:0: .75cm and .4cm);
\draw (2.1,1) -- (2.1,4);
\draw (2.3,1) -- (2.3,4);
\draw[cyan,thick] (2.1,4) -- (2.1,5);
\draw[cyan,thick] (2.3,4) -- (2.3,5);
\draw (2.5,1) -- (2.5,5);
\draw (1.8,0)  .. controls ++(90:.35cm) and ++(270:.35cm) .. (2.5,1);
\draw (3,1) -- (3,4.5);
\draw (2.9,4.5) -- (2.9,5);
\draw (3.1,4.5) -- (3.1,5);
\roundNbox{unshaded}{(3,4.5)}{.3}{0}{0}{$z$};
\draw (2.1,0)  .. controls ++(90:.35cm) and ++(270:.35cm) .. (3,1);
\draw (2.6,0) arc (180:0:.2cm);
\draw[orange,thick] (3.4,0) -- (3.4,5);
\draw[orange,thick] (3.1,5) arc (180:0:.15cm);
\draw[orange,thick] (3,0) arc (-180:0:.2cm);
\draw[dashed] (0,4) -- (3.6,4);
\draw[dashed] (0,3) -- (3.6,3);
\draw[dashed] (0,2) -- (3.6,2);
\draw[dashed] (0,1) -- (3.6,1);
\draw[dashed] (0,0) rectangle (3.6,5);
\node at (.4,-.2) {\tiny{$m$}};
\node at (.6,-.2) {\tiny{$n$}};
\node at (1.8,-.2) {\tiny{$1$}};
\node at (2.1,-.2) {\tiny{$l\+k$}};
\node at (2.6,-.2) {\tiny{$i\+j$}};
\node at (.4,5.2) {\tiny{$m$}};
\node at (.6,5.2) {\tiny{$i$}};
\node at (1,5.2) {\tiny{$i$}};
\node at (1.4,5.2) {\tiny{$n$}};
\node at (2.1,5.2) {\tiny{$j$}};
\node at (2.3,5.2) {\tiny{$j$}};
\node at (3.2,5.2) {\tiny{$l\+ k\sm i\sm j$}};
\end{tikzpicture}
\stackrel{\circle1}{=}d^{-i-j}\ 
\begin{tikzpicture}[baseline=2.9cm]
\draw (.4,-1) -- (.4,5);
\draw (.6,2) -- (.6,3);
\draw (1,2) -- (1,3);
\draw[cyan,thick] (.6,3) -- (.6,5);
\draw[cyan,thick] (1,3) -- (1,5);
\draw[cyan,thick] (1.4,4) -- (1.4,5);
\draw (1.4,2) -- (1.4,4);
\draw (1.6,2) -- (1.6,4);
\roundNbox{unshaded}{(1.5,3.5)}{.3}{0}{0}{$y$};
\draw (.6,2) arc (-180:0:.2cm);
\draw (.6,1)  .. controls ++(90:.35cm) and ++(270:.35cm) .. (1.4,2);
\draw (.8,1)  .. controls ++(90:.35cm) and ++(270:.35cm) .. (1.6,2);
\draw (1.1,1) arc (180:0:.2cm);
\nbox{unshaded}{(0.5,2.5)}{.3}{0.0}{0.0}{$x$};
\draw[orange,thick] (1.9,1) -- (1.9,4);
\draw[orange,thick] (1.6,4) arc (180:0:.15cm);
\draw[orange,thick] (1.5,1) arc (-180:0:.2cm);
\draw (.6,-1) -- (.6,1);
\draw (1.1,1) arc (-180:0: .5cm and .3cm);
\draw (.8,1) arc (-180:0: .75cm and .4cm);
\draw (2.1,1) -- (2.1,4);
\draw (2.3,1) -- (2.3,4);
\draw[cyan,thick] (2.1,4) -- (2.1,5);
\draw[cyan,thick] (2.3,4) -- (2.3,5);
\draw (2.5,0) -- (2.5,5);
\draw (1.8,-1)  .. controls ++(90:.35cm) and ++(270:.35cm) .. (2.5,0);
\draw (3,0) -- (3,4.5);
\draw (2.9,4.5) -- (2.9,5);
\draw (3.1,4.5) -- (3.1,5);
\roundNbox{unshaded}{(3,4.5)}{.3}{0}{0}{$z$};
\draw (2.1,-1)  .. controls ++(90:.35cm) and ++(270:.35cm) .. (3,0);
\draw (2.6,-1) arc (180:0:.2cm);
\draw (1,0) arc (180:-180:.2cm);
\draw (1.6,0) arc (180:-180:.2cm);
\draw[orange,thick] (3.4,-1) -- (3.4,5);
\draw[orange,thick] (3.1,5) arc (180:0:.15cm);
\draw[orange,thick] (3,-1) arc (-180:0:.2cm);
\draw[dashed] (0,4) -- (3.6,4);
\draw[dashed] (0,3) -- (3.6,3);
\draw[dashed] (0,2) -- (3.6,2);
\draw[dashed] (0,1) -- (3.6,1);
\draw[dashed] (0,0) -- (3.6,0);
\draw[dashed] (0,-1) rectangle (3.6,5);
\node at (.4,-1.2) {\tiny{$m$}};
\node at (.6,-1.2) {\tiny{$n$}};
\node at (1.8,-1.2) {\tiny{$1$}};
\node at (2.1,-1.2) {\tiny{$l\+k$}};
\node at (2.6,-1.2) {\tiny{$i\+j$}};
\node at (.4,5.2) {\tiny{$m$}};
\node at (.6,5.2) {\tiny{$i$}};
\node at (1,5.2) {\tiny{$i$}};
\node at (1.4,5.2) {\tiny{$n$}};
\node at (2.1,5.2) {\tiny{$j$}};
\node at (2.3,5.2) {\tiny{$j$}};
\node at (3.2,5.2) {\tiny{$l\+ k\sm i\sm j$}};
\node at (1,.3) {\tiny{$i$}};
\node at (1.6,.3) {\tiny{$j$}};
\end{tikzpicture}
\stackrel{\circle2}{=}d^{-i-j}\
\begin{tikzpicture}[baseline=2.9cm]
\draw (.4,0) -- (.4,5);
\draw (.6,2) -- (.6,3);
\draw (1,2) -- (1,3);
\draw[cyan,thick] (.6,3) -- (.6,5);
\draw[cyan,thick] (1,3) -- (1,5);
\draw (1.4,2) -- (1.4,4);
\draw (1.6,2) -- (1.6,4);
\draw[cyan,thick] (1.4,4) -- (1.4,5);
\draw[cyan,thick] (1.6,4) -- (1.6,5);
\roundNbox{unshaded}{(1.5,3.5)}{.3}{0}{0}{$y$};
\draw (.6,2) arc (-180:0:.2cm);
\draw (.6,1)  .. controls ++(90:.35cm) and ++(270:.35cm) .. (1.4,2);
\draw (1.6,2) arc (-180:0:.2cm);
\nbox{unshaded}{(0.5,2.5)}{.3}{0.0}{0.0}{$x$};
\draw (.6,0) -- (.6,1);
\draw (2,2) -- (2,4);
\draw[cyan,thick] (2,4) -- (2,5);
\draw (2.2,1) -- (2.2,5);
\draw (1.5,0)  .. controls ++(90:.35cm) and ++(270:.35cm) .. (2.2,1);
\draw (2.7,1) -- (2.7,4.5);
\draw (2.6,4.5) -- (2.6,5);
\draw (2.8,4.5) -- (2.8,5);
\roundNbox{unshaded}{(2.7,4.5)}{.3}{0}{0}{$z$};
\draw (1.8,0)  .. controls ++(90:.35cm) and ++(270:.35cm) .. (2.7,1);
\draw (2.3,0) arc (180:0:.2cm);
\draw (1,1) arc (180:-180:.2cm);
\draw (1.6,1) arc (180:-180:.2cm);
\draw[orange,thick] (3.1,0) -- (3.1,5);
\draw[orange,thick] (2.8,5) arc (180:0:.15cm);
\draw[orange,thick] (2.7,0) arc (-180:0:.2cm);
\draw[dashed] (0,4) -- (3.3,4);
\draw[dashed] (0,3) -- (3.3,3);
\draw[dashed] (0,2) -- (3.3,2);
\draw[dashed] (0,1) -- (3.3,1);
\draw[dashed] (0,0) rectangle (3.3,5);
\node at (.4,-.2) {\tiny{$m$}};
\node at (.6,-.2) {\tiny{$n$}};
\node at (1.5,-.2) {\tiny{$1$}};
\node at (1.8,-.2) {\tiny{$l\+k$}};
\node at (2.3,-.2) {\tiny{$i\+j$}};
\node at (.4,5.2) {\tiny{$m$}};
\node at (.6,5.2) {\tiny{$i$}};
\node at (1,5.2) {\tiny{$i$}};
\node at (1.4,5.2) {\tiny{$n$}};
\node at (1.6,5.2) {\tiny{$j$}};
\node at (2,5.2) {\tiny{$j$}};
\node at (2.9,5.2) {\tiny{$l\+k\sm i\sm j$}};
\node at (1,1.3) {\tiny{$i$}};
\node at (1.6,1.3) {\tiny{$j$}};
\end{tikzpicture}
\stackrel{\circle3}{=}\
\begin{tikzpicture}[baseline=2.9cm]
\draw (.4,1) -- (.4,5);
\draw (.6,2) -- (.6,3);
\draw (1,2) -- (1,3);
\draw[cyan,thick] (.6,3) -- (.6,5);
\draw[cyan,thick] (1,3) -- (1,5);
\draw (1.4,2) -- (1.4,4);
\draw (1.6,2) -- (1.6,4);
\draw[cyan,thick] (1.4,4) -- (1.4,5);
\draw[cyan,thick] (1.6,4) -- (1.6,5);
\roundNbox{unshaded}{(1.5,3.5)}{.3}{0}{0}{$y$};
\draw (.6,2) arc (-180:0:.2cm);
\draw (.6,1)  .. controls ++(90:.35cm) and ++(270:.35cm) .. (1.4,2);
\draw (1.6,2) arc (-180:0:.2cm);
\nbox{unshaded}{(0.5,2.5)}{.3}{0.0}{0.0}{$x$};
\draw (2,2) -- (2,4);
\draw[cyan,thick] (2,4) -- (2,5);
\draw (2.2,2) -- (2.2,5);
\draw (1.5,1)  .. controls ++(90:.35cm) and ++(270:.35cm) .. (2.2,2);
\draw (2.7,2) -- (2.7,4.5);
\draw (2.6,4.5) -- (2.6,5);
\draw (2.8,4.5) -- (2.8,5);
\roundNbox{unshaded}{(2.7,4.5)}{.3}{0}{0}{$z$};
\draw (1.8,1)  .. controls ++(90:.35cm) and ++(270:.35cm) .. (2.7,2);
\draw (2.3,1) arc (180:0:.2cm);
\draw[orange,thick] (3.1,1) -- (3.1,5);
\draw[orange,thick] (2.8,5) arc (180:0:.15cm);
\draw[orange,thick] (2.7,1) arc (-180:0:.2cm);
\draw[dashed] (0,4) -- (3.3,4);
\draw[dashed] (0,3) -- (3.3,3);
\draw[dashed] (0,2) -- (3.3,2);
\draw[dashed] (0,1) rectangle (3.3,5);
\node at (.4,.8) {\tiny{$m$}};
\node at (.6,.8) {\tiny{$n$}};
\node at (1.5,.8) {\tiny{$1$}};
\node at (1.8,.8) {\tiny{$l\+k$}};
\node at (2.3,.8) {\tiny{$i\+j$}};
\node at (.4,5.2) {\tiny{$m$}};
\node at (.6,5.2) {\tiny{$i$}};
\node at (1,5.2) {\tiny{$i$}};
\node at (1.4,5.2) {\tiny{$n$}};
\node at (1.6,5.2) {\tiny{$j$}};
\node at (2,5.2) {\tiny{$j$}};
\node at (2.9,5.2) {\tiny{$l\+k\sm i\sm j$}};
\end{tikzpicture}
\]

\noindent List of the formulas used in above equalities:

\noindent \circle1: Jones projection property;\hfill \circle2: uses Lemma \ref{lemma 4};

\noindent \circle3: Jones projection property.\\

And $x\otimes (y\otimes z)$:
\[
\begin{tikzpicture}[baseline=2.9cm]
\draw (.4,0) -- (.4,5);
\draw (.6,1) -- (.6,3);
\draw (1,1) -- (1,3);
\draw[cyan,thick] (.6,3) -- (.6,5);
\draw[cyan,thick] (1,3) -- (1,5);
\nbox{unshaded}{(0.5,1.5)}{.3}{0.0}{0.0}{$x$};
\draw (.6,1) arc (-180:0:.2cm);
\draw (1.4,1) -- (1.4,4);
\draw (1.6,3) -- (1.6,4);
\draw (2,3) -- (2,5);
\draw[cyan,thick] (1.4,4) -- (1.4,5);
\draw[cyan,thick] (1.6,4) -- (1.6,5);
\draw[cyan,thick] (2,4) -- (2,5);
\draw (2.5,4.5) -- (2.5,5);
\draw (2.7,4.5) -- (2.7,5);
\draw (2.9,4.5) -- (2.9,5);
\draw (2.7,3) -- (2.7,4.5);
\draw (2.2,3) -- (2.2,5);
\roundNbox{unshaded}{(1.5,3.5)}{.3}{0}{0}{$y$};
\roundNbox{unshaded}{(2.7,4.5)}{.3}{0}{0}{$z$};
\draw (1.6,3) arc (-180:0:.2cm);
\draw (1.6,2)  .. controls ++(90:.35cm) and ++(270:.35cm) .. (2.2,3);
\draw (1.8,2)  .. controls ++(90:.35cm) and ++(270:.35cm) .. (2.7,3);
\draw (2.3,2) arc (180:0:.2cm);
\draw (1.6,1) -- (1.6,2);
\draw (1.8,1) -- (1.8,2);
\draw (2.3,1) -- (2.3,2);
\draw (.6,0)  .. controls ++(90:.35cm) and ++(270:.35cm) .. (1.4,1);
\draw (.8,0)  .. controls ++(90:.35cm) and ++(270:.35cm) .. (1.6,1);
\draw (1,0)  .. controls ++(90:.35cm) and ++(270:.35cm) .. (1.8,1);
\draw (1.5,0)  .. controls ++(90:.35cm) and ++(270:.35cm) .. (2.3,1);
\draw (2.1,0) arc (180:0:.2cm);
\draw[orange,thick] (3.1,2) -- (3.1,5);
\draw[orange,thick] (2.9,5) arc (180:0:.1cm);
\draw[orange,thick] (2.7,2) arc (-180:0:.2cm);
\draw[orange,thick] (3.3,0) -- (3.3,5);
\draw[orange,thick] (2.7,5) arc (180:0:.3cm);
\draw[orange,thick] (2.5,0) arc (-180:0:.4cm);
\draw[dashed] (0,4) -- (3.5,4);
\draw[dashed] (0,3) -- (3.5,3);
\draw[dashed] (0,2) -- (3.5,2);
\draw[dashed] (0,1) -- (3.5,1);
\draw[dashed] (0,0) rectangle (3.5,5);
\node at (.4,-.2) {\tiny{$m$}};
\node at (.6,.3) {\tiny{$n$}};
\node at (1,-.2) {\tiny{$l\+k$}};
\node at (1.5,-.2) {\tiny{$j$}};
\node at (2.1,-.2) {\tiny{$i$}};
\node at (1.6,2.3) {\tiny{$1$}};
\node at (.4,5.2) {\tiny{$m$}};
\node at (.6,5.2) {\tiny{$i$}};
\node at (1,5.2) {\tiny{$i$}};
\node at (1.4,5.2) {\tiny{$n$}};
\node at (1.6,5.2) {\tiny{$j$}};
\node at (2,5.2) {\tiny{$j$}};
\node at (2.7,5.2) {\tiny{$l\+k\sm i\sm j$}};
\end{tikzpicture}
\stackrel{\circle1}{=}\
\begin{tikzpicture}[baseline=3.9cm]
\draw (.4,0) -- (.4,6);
\draw (.6,1) -- (.6,4);
\draw (1,1) -- (1,4);
\draw[cyan,thick] (.6,4) -- (.6,6);
\draw[cyan,thick] (1,4) -- (1,6);
\nbox{unshaded}{(0.5,3.5)}{.3}{0.0}{0.0}{$x$};
\draw (.6,1) arc (-180:0:.2cm);
\draw (1.4,1) -- (1.4,5);
\draw (1.6,3) -- (1.6,5);
\draw (2,3) -- (2,5);
\draw[cyan,thick] (1.4,5) -- (1.4,6);
\draw[cyan,thick] (1.6,5) -- (1.6,6);
\draw[cyan,thick] (2,5) -- (2,6);
\draw (2.2,3) -- (2.2,6);
\draw (2.5,5.5) -- (2.5,6);
\draw (2.7,4.5) -- (2.7,6);
\draw (2.9,5.5) -- (2.9,6);
\draw (2.7,3) -- (2.7,4.5);
\roundNbox{unshaded}{(1.5,4.5)}{.3}{0}{0}{$y$};
\roundNbox{unshaded}{(2.7,5.5)}{.3}{0}{0}{$z$};
\draw (1.6,3) arc (-180:0:.2cm);
\draw (1.6,2)  .. controls ++(90:.35cm) and ++(270:.35cm) .. (2.2,3);
\draw (1.8,2)  .. controls ++(90:.35cm) and ++(270:.35cm) .. (2.7,3);
\draw (2.3,2) arc (180:0:.2cm);
\draw (1.6,1) -- (1.6,2);
\draw (1.8,1) -- (1.8,2);
\draw (2.3,1) -- (2.3,2);
\draw (.6,0)  .. controls ++(90:.35cm) and ++(270:.35cm) .. (1.4,1);
\draw (.8,0)  .. controls ++(90:.35cm) and ++(270:.35cm) .. (1.6,1);
\draw (1,0)  .. controls ++(90:.35cm) and ++(270:.35cm) .. (1.8,1);
\draw (1.5,0)  .. controls ++(90:.35cm) and ++(270:.35cm) .. (2.3,1);
\draw (2.1,0) arc (180:0:.2cm);
\draw[orange,thick] (3.1,2) -- (3.1,6);
\draw[orange,thick] (2.9,6) arc (180:0:.1cm);
\draw[orange,thick] (2.7,2) arc (-180:0:.2cm);
\draw[orange,thick] (3.3,0) -- (3.3,6);
\draw[orange,thick] (2.7,6) arc (180:0:.3cm);
\draw[orange,thick] (2.5,0) arc (-180:0:.4cm);
\draw[dashed] (0,5) -- (3.5,5);
\draw[dashed] (0,4) -- (3.5,4);
\draw[dashed] (0,3) -- (3.5,3);
\draw[dashed] (0,2) -- (3.5,2);
\draw[dashed] (0,1) -- (3.5,1);
\draw[dashed] (0,0) rectangle (3.5,6);
\node at (1.6,2.3) {\tiny{$1$}};
\node at (.4,-.2) {\tiny{$m$}};
\node at (.6,.3) {\tiny{$n$}};
\node at (1,-.2) {\tiny{$l\+k$}};
\node at (1.5,-.2) {\tiny{$j$}};
\node at (2.1,-.2) {\tiny{$i$}};
\node at (.4,6.2) {\tiny{$m$}};
\node at (.6,6.2) {\tiny{$i$}};
\node at (1,6.2) {\tiny{$i$}};
\node at (1.4,6.2) {\tiny{$n$}};
\node at (1.6,6.2) {\tiny{$j$}};
\node at (2,6.2) {\tiny{$j$}};
\node at (2.7,6.2) {\tiny{$l\+k\sm i\sm j$}};
\end{tikzpicture}
\stackrel{\circle2}{=}\
\begin{tikzpicture}[baseline=3.9cm]
\draw (.4,0) -- (.4,6);
\draw (.6,1) -- (.6,4);
\draw (1,1) -- (1,4);
\draw[cyan,thick] (.6,4) -- (.6,6);
\draw[cyan,thick] (1,4) -- (1,6);
\nbox{unshaded}{(0.5,3.5)}{.3}{0.0}{0.0}{$x$};
\draw (.6,1) arc (-180:0:.2cm);
\draw (1.4,1) -- (1.4,5);
\draw (1.6,3) -- (1.6,5);
\draw (2,3) -- (2,5);
\draw[cyan,thick] (1.4,5) -- (1.4,6);
\draw[cyan,thick] (1.6,5) -- (1.6,6);
\draw[cyan,thick] (2,5) -- (2,6);
\draw (2.2,3) -- (2.2,6);
\draw (2.5,5.5) -- (2.5,6);
\draw (2.7,4.5) -- (2.7,6);
\draw (2.9,5.5) -- (2.9,6);
\draw (2.7,3) -- (2.7,4.5);
\roundNbox{unshaded}{(1.5,4.5)}{.3}{0}{0}{$y$};
\roundNbox{unshaded}{(2.7,5.5)}{.3}{0}{0}{$z$};
\draw (1.6,3) arc (-180:0:.2cm);
\draw (1.6,2)  .. controls ++(90:.35cm) and ++(270:.35cm) .. (2.2,3);
\draw (1.8,2)  .. controls ++(90:.35cm) and ++(270:.35cm) .. (2.7,3);
\draw (2.3,2) arc (180:0:.2cm);
\draw (1.6,1) -- (1.6,2);
\draw (1.8,1) -- (1.8,2);
\draw (2.3,1) -- (2.3,2);
\draw (.6,0)  .. controls ++(90:.35cm) and ++(270:.35cm) .. (1.4,1);
\draw (.8,0)  .. controls ++(90:.35cm) and ++(270:.35cm) .. (1.6,1);
\draw (1,0)  .. controls ++(90:.35cm) and ++(270:.35cm) .. (1.8,1);
\draw (1.5,0)  .. controls ++(90:.35cm) and ++(270:.35cm) .. (2.3,1);
\draw (2.1,0) arc (180:0:.2cm);
\draw (2.7,0) -- (2.7,2);
\draw[orange,thick] (3.1,0) -- (3.1,6);
\draw[orange,thick] (2.9,6) arc (180:0:.1cm);
\draw[orange,thick] (2.7,0) arc (-180:0:.2cm);
\draw[orange,thick] (3.3,0) -- (3.3,6);
\draw[orange,thick] (2.7,6) arc (180:0:.3cm);
\draw[orange,thick] (2.5,0) arc (-180:0:.4cm);
\draw[dashed] (0,5) -- (3.5,5);
\draw[dashed] (0,4) -- (3.5,4);
\draw[dashed] (0,3) -- (3.5,3);
\draw[dashed] (0,2) -- (3.5,2);
\draw[dashed] (0,1) -- (3.5,1);
\draw[dashed] (0,0) rectangle (3.5,6);
\node at (.4,-.2) {\tiny{$m$}};
\node at (.6,.3) {\tiny{$n$}};
\node at (1,-.2) {\tiny{$l\+k$}};
\node at (1.5,-.2) {\tiny{$j$}};
\node at (2.1,-.2) {\tiny{$i$}};
\node at (1.6,2.3) {\tiny{$1$}};
\node at (.4,6.2) {\tiny{$m$}};
\node at (.6,6.2) {\tiny{$i$}};
\node at (1,6.2) {\tiny{$i$}};
\node at (1.4,6.2) {\tiny{$n$}};
\node at (1.6,6.2) {\tiny{$j$}};
\node at (2,6.2) {\tiny{$j$}};
\node at (2.7,6.2) {\tiny{$l\+k\sm i\sm j$}};
\end{tikzpicture}
\stackrel{\circle3}{=}\
\begin{tikzpicture}[baseline=2.9cm]
\draw (.4,1) -- (.4,5);
\draw (.6,2) -- (.6,3);
\draw (1,2) -- (1,3);
\draw[cyan,thick] (.6,3) -- (.6,5);
\draw[cyan,thick] (1,3) -- (1,5);
\draw (1.4,2) -- (1.4,4);
\draw (1.6,2) -- (1.6,4);
\draw[cyan,thick] (1.4,4) -- (1.4,5);
\draw[cyan,thick] (1.6,4) -- (1.6,5);
\roundNbox{unshaded}{(1.5,3.5)}{.3}{0}{0}{$y$};
\draw (.6,2) arc (-180:0:.2cm);
\draw (.6,1)  .. controls ++(90:.35cm) and ++(270:.35cm) .. (1.4,2);
\draw (1.6,2) arc (-180:0:.2cm);
\nbox{unshaded}{(0.5,2.5)}{.3}{0.0}{0.0}{$x$};
\draw (2,2) -- (2,4);
\draw[cyan,thick] (2,4) -- (2,5);
\draw (2.2,2) -- (2.2,5);
\draw (1.5,1)  .. controls ++(90:.35cm) and ++(270:.35cm) .. (2.2,2);
\draw (2.7,2) -- (2.7,4.5);
\draw (2.6,4.5) -- (2.6,5);
\draw (2.8,4.5) -- (2.8,5);
\roundNbox{unshaded}{(2.7,4.5)}{.3}{0}{0}{$z$};
\draw (1.8,1)  .. controls ++(90:.35cm) and ++(270:.35cm) .. (2.7,2);
\draw (2.3,1) arc (180:0:.2cm);
\draw[orange,thick] (3.1,1) -- (3.1,5);
\draw[orange,thick] (2.8,5) arc (180:0:.15cm);
\draw[orange,thick] (2.7,1) arc (-180:0:.2cm);
\draw[dashed] (0,4) -- (3.3,4);
\draw[dashed] (0,3) -- (3.3,3);
\draw[dashed] (0,2) -- (3.3,2);
\draw[dashed] (0,1) rectangle (3.3,5);
\node at (.4,.8) {\tiny{$m$}};
\node at (.6,.8) {\tiny{$n$}};
\node at (1.5,.8) {\tiny{$1$}};
\node at (1.8,.8) {\tiny{$l\+k$}};
\node at (2.3,.8) {\tiny{$i\+j$}};
\node at (.4,5.2) {\tiny{$m$}};
\node at (.6,5.2) {\tiny{$i$}};
\node at (1,5.2) {\tiny{$i$}};
\node at (1.4,5.2) {\tiny{$n$}};
\node at (1.6,5.2) {\tiny{$j$}};
\node at (2,5.2) {\tiny{$j$}};
\node at (2.9,5.2) {\tiny{$l\+k\sm i\sm j$}};
\end{tikzpicture}
\]

\noindent List of the formulas used in above equalities: 

\noindent \circle1: uses ($\lambda$11);\hfill \circle2: uses ($\lambda$10); 

\noindent \circle3: Jones projection property.\\

Therefore, $(x\otimes y)\otimes z=x\otimes (y\otimes z)$ in this case. Readers can check the rest of the cases by using the string diagram dictionary and the lemmas. \end{proof}

\begin{proposition}\label{circ tensor 1}
For $x,y\in \Hom(\mathcal{A}_0)$, $(x\circ y)\otimes 1=(x\otimes 1)\circ (y\otimes 1)$ and $1\otimes (x\circ y)=(1\otimes x)\circ (1\otimes y)$.
\end{proposition}
\begin{proof} By our construction, $1\otimes (x\circ y)=(1\otimes x)\circ (1\otimes y)$ only uses the fact that the shift map is a $*$-homomorphism.

As for $(x\circ y)\otimes 1=(x\otimes 1)\circ (y\otimes 1)$, we check the case
$(x;[m,+],[m+2i,+])$ and $(y;[m+2i],[m+2i+2j,+])$, where $n\ge i+j$. Then $(x\circ y)\otimes 1_n$, $(x\otimes 1_n)\circ (y\otimes 1_n)\in \mathcal{A}_0([m+n,+]\to [m+n+2i+2j,+])$. Next, let us compare their endomorphism representation parts.

$(x\circ y)\otimes 1_n$:
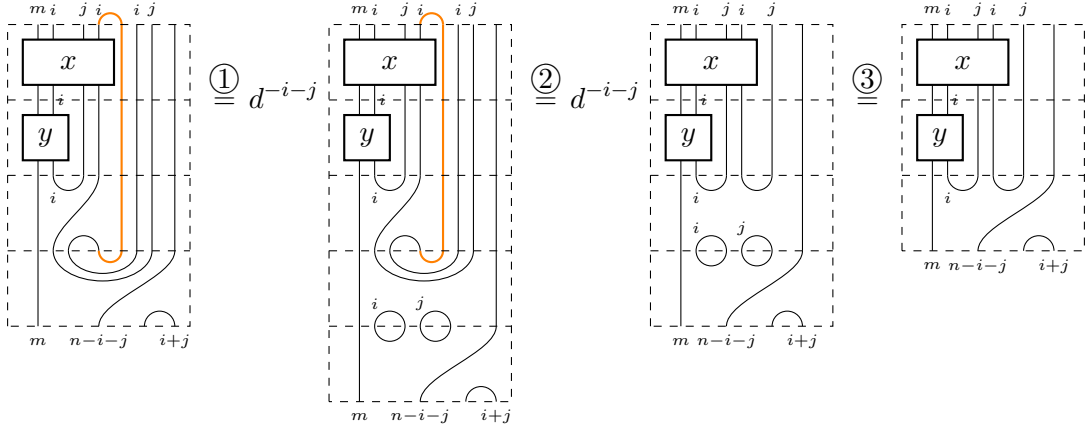
\begin{align*}
\begin{tikzpicture}[baseline=2.9cm]
\draw (.4,0) -- (.4,4);
\draw (.6,2) -- (.6,4);
\draw (1,2) -- (1,4);
\draw (1.2,2) -- (1.2,4);
\draw (1.7,1) -- (1.7,4);
\draw (1.9,1) -- (1.9,4);
\draw (2.2,1) -- (2.2,4);
\nbox{unshaded}{(.5,2.5)}{.3}{0}{0}{$y$};
\nbox{unshaded}{(0.8,3.5)}{.3}{0.3}{0.3}{$x$};
\draw (.6,2) arc (-180:0:.2cm);
\draw (.6,1)  .. controls ++(90:.35cm) and ++(270:.35cm) .. (1.2,2);
\draw (.8,1) arc (180:0:.2cm);
\draw (.8,1) arc (-180:0: .45cm and .3cm);
\draw (.6,1) arc (-180:0: .65cm and .4cm);
\draw (1.8,0) arc (180:0:.2cm);
\draw (1.2,0)  .. controls ++(90:.35cm) and ++(270:.35cm) .. (2.2,1);
\draw[orange,thick] (1.5,1) -- (1.5,4);
\draw[orange,thick] (1.2,4) arc (180:0:.15cm);
\draw[orange,thick] (1.2,1) arc (-180:0:.15cm);
\draw[dashed] (0,3) -- (2.4,3);
\draw[dashed] (0,2) -- (2.4,2);
\draw[dashed] (0,1) -- (2.4,1);
\draw[dashed] (0,0) rectangle (2.4,4);
\node at (.4,-.2) {\tiny{$m$}};
\node at (1.2,-.2) {\tiny{$n\sm i\sm j$}};
\node at (2.2,-.2) {\tiny{$i\+ j$}};
\node at (.6,1.7) {\tiny{$i$}};
\node at (.7,3) {\tiny{$i$}};
\node at (.4,4.2) {\tiny{$m$}};
\node at (.6,4.2) {\tiny{$i$}};
\node at (1,4.2) {\tiny{$j$}};
\node at (1.2,4.2) {\tiny{$i$}};
\node at (1.7,4.2) {\tiny{$i$}};
\node at (1.9,4.2) {\tiny{$j$}};
\end{tikzpicture}
\stackrel{\circle1}{=}d^{-i-j}\
\begin{tikzpicture}[baseline=2.9cm]
\draw (.4,-1) -- (.4,4);
\draw (.6,2) -- (.6,4);
\draw (1,2) -- (1,4);
\draw (1.2,2) -- (1.2,4);
\draw (1.7,1) -- (1.7,4);
\draw (1.9,1) -- (1.9,4);
\draw (2.2,0) -- (2.2,4);
\nbox{unshaded}{(.5,2.5)}{.3}{0}{0}{$y$};
\nbox{unshaded}{(0.8,3.5)}{.3}{0.3}{0.3}{$x$};
\draw (.6,2) arc (-180:0:.2cm);
\draw (.6,1)  .. controls ++(90:.35cm) and ++(270:.35cm) .. (1.2,2);
\draw (.8,1) arc (180:0:.2cm);
\draw (.8,1) arc (-180:0: .45cm and .3cm);
\draw (.6,1) arc (-180:0: .65cm and .4cm);
\draw (.6,0) arc (180:-180:.2cm);
\draw (1.2,0) arc (180:-180:.2cm);
\draw (1.8,-1) arc (180:0:.2cm);
\draw (1.2,-1)  .. controls ++(90:.35cm) and ++(270:.35cm) .. (2.2,0);
\draw[orange,thick] (1.5,1) -- (1.5,4);
\draw[orange,thick] (1.2,4) arc (180:0:.15cm);
\draw[orange,thick] (1.2,1) arc (-180:0:.15cm);
\draw[dashed] (0,3) -- (2.4,3);
\draw[dashed] (0,2) -- (2.4,2);
\draw[dashed] (0,1) -- (2.4,1);
\draw[dashed] (0,0) -- (2.4,0);
\draw[dashed] (0,-1) rectangle (2.4,4);
\node at (.6,1.7) {\tiny{$i$}};
\node at (.7,3) {\tiny{$i$}};
\node at (.6,.3) {\tiny{$i$}};
\node at (1.2,.3) {\tiny{$j$}};
\node at (.4,4.2) {\tiny{$m$}};
\node at (.6,4.2) {\tiny{$i$}};
\node at (1,4.2) {\tiny{$j$}};
\node at (1.2,4.2) {\tiny{$i$}};
\node at (1.7,4.2) {\tiny{$i$}};
\node at (1.9,4.2) {\tiny{$j$}};
\node at (.4,-1.2) {\tiny{$m$}};
\node at (1.2,-1.2) {\tiny{$n\sm i\sm j$}};
\node at (2.2,-1.2) {\tiny{$i\+ j$}};
\end{tikzpicture}
\stackrel{\circle2}{=}d^{-i-j}\
\begin{tikzpicture}[baseline=2.9cm]
\draw (.4,0) -- (.4,4);
\draw (.6,2) -- (.6,4);
\draw (1,2) -- (1,4);
\draw (1.2,2) -- (1.2,4);
\draw (1.6,2) -- (1.6,4);
\draw (2,1) -- (2,4);
\nbox{unshaded}{(.5,2.5)}{.3}{0}{0}{$y$};
\nbox{unshaded}{(0.8,3.5)}{.3}{0.3}{0.3}{$x$};
\draw (.6,2) arc (-180:0:.2cm);
\draw (1.2,2) arc (-180:0:.2cm);
\draw (.6,1) arc (-180:180:.2cm);
\draw (1.2,1) arc (-180:180:.2cm);
\draw (1,0)  .. controls ++(90:.35cm) and ++(270:.35cm) .. (2,1);
\draw (1.6,0) arc (180:0:.2cm);
\draw[dashed] (0,3) -- (2.4,3);
\draw[dashed] (0,2) -- (2.4,2);
\draw[dashed] (0,1) -- (2.4,1);
\draw[dashed] (0,0) rectangle (2.4,4);
\node at (.4,-.2) {\tiny{$m$}};
\node at (1,-.2) {\tiny{$n\sm i\sm j$}};
\node at (2,-.2) {\tiny{$i\+ j$}};
\node at (.4,4.2) {\tiny{$m$}};
\node at (.6,4.2) {\tiny{$i$}};
\node at (1,4.2) {\tiny{$j$}};
\node at (1.2,4.2) {\tiny{$i$}};
\node at (1.6,4.2) {\tiny{$j$}};
\node at (.6,1.7) {\tiny{$i$}};
\node at (.7,3) {\tiny{$i$}};
\node at (.6,1.3) {\tiny{$i$}};
\node at (1.2,1.3) {\tiny{$j$}};
\end{tikzpicture}
\ \stackrel{\circle3}{=}\
\begin{tikzpicture}[baseline=2.9cm]
\draw (.4,1) -- (.4,4);
\draw (.6,2) -- (.6,4);
\draw (1,2) -- (1,4);
\draw (1.2,2) -- (1.2,4);
\draw (1.6,2) -- (1.6,4);
\draw (2,2) -- (2,4);
\nbox{unshaded}{(.5,2.5)}{.3}{0}{0}{$y$};
\nbox{unshaded}{(0.8,3.5)}{.3}{0.3}{0.3}{$x$};
\draw (.6,2) arc (-180:0:.2cm);
\draw (1.2,2) arc (-180:0:.2cm);
\draw (1,1)  .. controls ++(90:.35cm) and ++(270:.35cm) .. (2,2);
\draw (1.6,1) arc (180:0:.2cm);
\draw[dashed] (0,3) -- (2.4,3);
\draw[dashed] (0,2) -- (2.4,2);
\draw[dashed] (0,1) rectangle (2.4,4);
\node at (.4,.8) {\tiny{$m$}};
\node at (1,.8) {\tiny{$n\sm i\sm j$}};
\node at (2,.8) {\tiny{$i\+ j$}};
\node at (.4,4.2) {\tiny{$m$}};
\node at (.6,4.2) {\tiny{$i$}};
\node at (1,4.2) {\tiny{$j$}};
\node at (1.2,4.2) {\tiny{$i$}};
\node at (1.6,4.2) {\tiny{$j$}};
\node at (.6,1.7) {\tiny{$i$}};
\node at (.7,3) {\tiny{$i$}};
\end{tikzpicture}
\end{align*}

\noindent List of the formulas used in above equalities:

\noindent \circle1: Jones projection property;\hfill \circle2: uses Lemma \ref{lemma 4};

\noindent \circle3: Jones projection property.\\

$(x\otimes 1_n)\circ (y\otimes 1_n)$:
\begin{align*}
\begin{tikzpicture}[baseline=3.9cm]
\draw (.4,-1) -- (.4,5);
\draw (.6,1) -- (.6,5);
\draw (1,1) -- (1,5);
\draw (1.2,4) -- (1.2,5);
\draw (1.6,4) -- (1.6,5);
\draw (2.2,4) -- (2.2,5);
\draw (1.8,4) -- (1.8,5);
\draw (1.2,4) arc (-180:0:.2cm);
\draw (1.2,3)  .. controls ++(90:.35cm) and ++(270:.35cm) .. (1.8,4);
\draw (1.5,3)  .. controls ++(90:.35cm) and ++(270:.35cm) .. (2.2,4);
\draw (1.8,3) arc (180:0:.2cm);
\nbox{unshaded}{(.5,1.5)}{.3}{0}{0}{$y$};
\nbox{unshaded}{(0.8,4.5)}{.3}{0.3}{0.3}{$x$};
\draw (.6,1) arc (-180:0:.2cm);
\draw (.6,0)  .. controls ++(90:.35cm) and ++(270:.35cm) .. (1.2,1);
\draw (.8,0)  .. controls ++(90:.35cm) and ++(270:.35cm) .. (1.5,1);
\draw (1,0) arc (180:0:.2cm);
\draw (2.2,0) -- (2.2,3);
\draw (1.8,0) -- (1.8,2);
\draw (1.5,1) -- (1.5,2);
\draw (1.2,1) -- (1.2,3);
\draw (1.8,-1) arc (180:0:.2cm);
\draw (1.4,-1)  .. controls ++(90:.35cm) and ++(270:.35cm) .. (2.2,0);
\draw (1.4,0) arc (-180:0:.2cm);
\draw (.6,-1) -- (.6,0);
\draw (.8,-1) -- (.8,0);
\draw (1,-1) -- (1,0);
\draw[orange,thick] (2.5,-1) -- (2.5,5);
\draw[orange,thick] (2.2,5) arc (180:0:.15cm);
\draw[orange,thick] (2.2,-1) arc (-180:0:.15cm);
\draw[dashed] (0,4) -- (2.7,4);
\draw[dashed] (0,3) -- (2.7,3);
\draw[dashed] (0,2) -- (2.7,2);
\draw[dashed] (0,1) -- (2.7,1);
\draw[dashed] (0,0) -- (2.7,0);
\draw[dashed] (0,-1) rectangle (2.7,5);
\node at (.4,-1.2) {\tiny{$m$}};
\node at (.6,-1.2) {\tiny{$\uparrow$}};
\node at (.6,-1.4) {\tiny{$n\sm i\sm j$}};
\node at (.8,-1.2) {\tiny{$j$}};
\node at (1,-1.2) {\tiny{$i$}};
\node at (1.4,-1.2) {\tiny{$j$}};
\node at (1.8,-1.2) {\tiny{$i$}};
\node at (.4,5.2) {\tiny{$m$}};
\node at (.6,5.2) {\tiny{$i$}};
\node at (1,5.2) {\tiny{$i$}};
\node at (1.2,5.2) {\tiny{$j$}};
\node at (1.6,5.2) {\tiny{$j$}};
\node at (1.8,5.2) {\tiny{$\downarrow$}};
\node at (1.8,5.4) {\tiny{$n\sm i\sm j$}};
\node at (.6,.7) {\tiny{$i$}};
\node at (.7,2.2) {\tiny{$i$}};
\end{tikzpicture}
\ \stackrel{\circle1}{=}\
\begin{tikzpicture}[baseline=3.9cm]
\draw (.4,-1) -- (.4,5);
\draw (.6,1) -- (.6,5);
\draw (1,1) -- (1,5);
\draw (1.2,3) -- (1.2,5);
\draw (1.6,3) -- (1.6,5);
\draw (2.2,3) -- (2.2,5);
\draw (1.8,3) -- (1.8,5);
\draw (1.2,3) arc (-180:0:.2cm);
\draw (1.2,2)  .. controls ++(90:.35cm) and ++(270:.35cm) .. (1.8,3);
\draw (1.5,2)  .. controls ++(90:.35cm) and ++(270:.35cm) .. (2.2,3);
\draw (1.8,2) arc (180:0:.2cm);
\nbox{unshaded}{(.5,3.5)}{.3}{0}{0}{$y$};
\nbox{unshaded}{(0.8,4.5)}{.3}{0.3}{0.3}{$x$};
\draw (.6,1) arc (-180:0:.2cm);
\draw (.6,0)  .. controls ++(90:.35cm) and ++(270:.35cm) .. (1.2,1);
\draw (.8,0)  .. controls ++(90:.35cm) and ++(270:.35cm) .. (1.5,1);
\draw (1,0) arc (180:0:.2cm);
\draw (2.2,0) -- (2.2,2);
\draw (1.8,0) -- (1.8,1);
\draw (1.5,1) -- (1.5,1);
\draw (1.2,1) -- (1.2,2);
\draw (1.8,-1) arc (180:0:.2cm);
\draw (1.4,-1)  .. controls ++(90:.35cm) and ++(270:.35cm) .. (2.2,0);
\draw (1.4,0) arc (-180:0:.2cm);
\draw (.6,-1) -- (.6,0);
\draw (.8,-1) -- (.8,0);
\draw (1,-1) -- (1,0);
\draw[orange,thick] (2.5,-1) -- (2.5,5);
\draw[orange,thick] (2.2,5) arc (180:0:.15cm);
\draw[orange,thick] (2.2,-1) arc (-180:0:.15cm);
\draw[dashed] (0,4) -- (2.7,4);
\draw[dashed] (0,3) -- (2.7,3);
\draw[dashed] (0,2) -- (2.7,2);
\draw[dashed] (0,1) -- (2.7,1);
\draw[dashed] (0,0) -- (2.7,0);
\draw[dashed] (0,-1) rectangle (2.7,5);
\node at (.4,-1.2) {\tiny{$m$}};
\node at (.6,-1.2) {\tiny{$\uparrow$}};
\node at (.6,-1.4) {\tiny{$n\sm i\sm j$}};
\node at (.8,-1.2) {\tiny{$j$}};
\node at (1,-1.2) {\tiny{$i$}};
\node at (1.4,-1.2) {\tiny{$j$}};
\node at (1.8,-1.2) {\tiny{$i$}};
\node at (.4,5.2) {\tiny{$m$}};
\node at (.6,5.2) {\tiny{$i$}};
\node at (1,5.2) {\tiny{$i$}};
\node at (1.2,5.2) {\tiny{$j$}};
\node at (1.6,5.2) {\tiny{$j$}};
\node at (1.8,5.2) {\tiny{$\downarrow$}};
\node at (1.8,5.4) {\tiny{$n\sm i\sm j$}};
\node at (.6,.7) {\tiny{$i$}};
\node at (.7,4) {\tiny{$i$}};
\end{tikzpicture}
\ \stackrel{\circle2}{=}\
\begin{tikzpicture}[baseline=3.9cm]
\draw (.4,-1) -- (.4,5);
\draw (.6,1) -- (.6,5);
\draw (1,1) -- (1,5);
\draw (1.2,3) -- (1.2,5);
\draw (1.6,3) -- (1.6,5);
\draw (1.8,3) -- (1.8,5);
\draw (1.2,3) arc (-180:0:.2cm);
\draw (1.2,2)  .. controls ++(90:.35cm) and ++(270:.35cm) .. (1.8,3);
\draw (1.5,2)  .. controls ++(90:.35cm) and ++(270:.35cm) .. (2.2,3);
\draw (1.8,2) arc (180:0:.2cm);
\nbox{unshaded}{(.5,3.5)}{.3}{0}{0}{$y$};
\nbox{unshaded}{(0.8,4.5)}{.3}{0.3}{0.3}{$x$};
\draw (.6,1) arc (-180:0:.2cm);
\draw (.6,0)  .. controls ++(90:.35cm) and ++(270:.35cm) .. (1.2,1);
\draw (.8,0)  .. controls ++(90:.35cm) and ++(270:.35cm) .. (1.5,1);
\draw (1,0) arc (180:0:.2cm);
\draw (2.2,0) -- (2.2,2);
\draw (1.8,0) -- (1.8,1);
\draw (1.5,1) -- (1.5,1);
\draw (1.2,1) -- (1.2,2);
\draw (1.8,-1) arc (180:0:.2cm);
\draw (1.4,-1)  .. controls ++(90:.35cm) and ++(270:.35cm) .. (2.2,0);
\draw (1.4,0) arc (-180:0:.2cm);
\draw (.6,-1) -- (.6,0);
\draw (.8,-1) -- (.8,0);
\draw (1,-1) -- (1,0);
\draw[orange,thick] (2.5,-1) -- (2.5,3);
\draw[orange,thick] (2.2,3) arc (180:0:.15cm);
\draw[orange,thick] (2.2,-1) arc (-180:0:.15cm);
\draw[dashed] (0,4) -- (2.7,4);
\draw[dashed] (0,3) -- (2.7,3);
\draw[dashed] (0,2) -- (2.7,2);
\draw[dashed] (0,1) -- (2.7,1);
\draw[dashed] (0,0) -- (2.7,0);
\draw[dashed] (0,-1) rectangle (2.7,5);
\node at (.4,-1.2) {\tiny{$m$}};
\node at (.6,-1.2) {\tiny{$\uparrow$}};
\node at (.6,-1.4) {\tiny{$n\sm i\sm j$}};
\node at (.8,-1.2) {\tiny{$j$}};
\node at (1,-1.2) {\tiny{$i$}};
\node at (1.4,-1.2) {\tiny{$j$}};
\node at (1.8,-1.2) {\tiny{$i$}};
\node at (.4,5.2) {\tiny{$m$}};
\node at (.6,5.2) {\tiny{$i$}};
\node at (1,5.2) {\tiny{$i$}};
\node at (1.2,5.2) {\tiny{$j$}};
\node at (1.6,5.2) {\tiny{$j$}};
\node at (1.8,5.2) {\tiny{$\downarrow$}};
\node at (1.8,5.4) {\tiny{$n\sm i\sm j$}};
\node at (.6,.7) {\tiny{$i$}};
\node at (.7,4) {\tiny{$i$}};
\end{tikzpicture}
\ \stackrel{\circle3}{=}\
\begin{tikzpicture}[baseline=2.9cm]
\draw (.4,1) -- (.4,4);
\draw (.6,2) -- (.6,4);
\draw (1,2) -- (1,4);
\draw (1.2,2) -- (1.2,4);
\draw (1.6,2) -- (1.6,4);
\draw (2,2) -- (2,4);
\nbox{unshaded}{(.5,2.5)}{.3}{0}{0}{$y$};
\nbox{unshaded}{(0.8,3.5)}{.3}{0.3}{0.3}{$x$};
\draw (.6,2) arc (-180:0:.2cm);
\draw (1.2,2) arc (-180:0:.2cm);
\draw (1,1)  .. controls ++(90:.35cm) and ++(270:.35cm) .. (2,2);
\draw (1.6,1) arc (180:0:.2cm);
\draw[dashed] (0,3) -- (2.4,3);
\draw[dashed] (0,2) -- (2.4,2);
\draw[dashed] (0,1) rectangle (2.4,4);
\node at (.4,.8) {\tiny{$m$}};
\node at (1,.8) {\tiny{$n\sm i\sm j$}};
\node at (2,.8) {\tiny{$i\+ j$}};
\node at (.4,4.2) {\tiny{$m$}};
\node at (.6,4.2) {\tiny{$i$}};
\node at (1,4.2) {\tiny{$j$}};
\node at (1.2,4.2) {\tiny{$i$}};
\node at (1.6,4.2) {\tiny{$j$}};
\node at (.6,1.7) {\tiny{$i$}};
\node at (.7,3) {\tiny{$i$}};
\end{tikzpicture}
\end{align*}

where only the straight strands are allowed in the blank.  \\

\noindent List of the formulas used in above equalities: 

\noindent \circle1: uses ($\lambda$11);\hfill \circle2: uses ($\lambda$10); 

\noindent \circle3: uses Lemma \ref{lemma 3} and Jones projection property.\\

Therefore, $(x\circ y)\otimes 1=(x\otimes 1)\circ (y\otimes 1)$ in this case. 
Readers can check the rest of the cases by using the string diagram dictionary and the lemmas. 
\end{proof}

\begin{proposition} 
The tensor product is functorial.
For $x,y,z,w\in \Hom(\mathcal{A}_0)$, $(x\circ y)\otimes (z\circ w)=(x\otimes z)\circ (y\otimes w)$.
\end{proposition}
\begin{proof}
Based on Proposition \ref{tensor product def prop} and Proposition \ref{circ tensor 1}, we have
\begin{align*}
    (x\circ y)\otimes (z\circ w) & = ((x\circ y)\otimes 1)\circ (1\otimes (z\circ w)) \\
    & = ((x\otimes 1)\circ (y\otimes 1))\circ ((1\otimes z)\circ (1\otimes w))\\
    & = (x\otimes 1)\circ ((y\otimes 1)\circ (1\otimes z))\circ (1\otimes w)\\
    & = (x\otimes 1)\circ ((1\otimes z)\circ (y\otimes 1))\circ (1\otimes w)\\
    & = ((x\otimes 1)\circ (1\otimes z))\circ ((y\otimes 1)\circ (1\otimes w))\\
    & = (x\otimes z)\circ (y\otimes w).
\end{align*}
\end{proof}

Therefore, the tensor product in Definition \ref{definion x otimes y} is well-defined.\\

Next, we show that $\mathcal{A}_0$ has a pivotal structure.

\begin{definition}[ev and coev]
Note that $[n,\pm]\otimes \overline{[n,\pm]}=[2n;\pm]$; 
$\overline{[n,+]}\otimes [n,+]=[2n,+]$ if $2\mid n$ and $[2n,-]$ if $2\nmid n$; $\overline{[n,-]}\otimes [n,-]=[2n,-]$ if $2\mid n$ and $[2n,+]$ if $2\nmid n$.

Define 
\begin{align*}
    \coev_{[n,+]}:1^+\to [2n,+]=[n,+]\otimes \overline{[n,+]} \quad &\text{ as }\quad \coev_{[n,+]}=(d^n;[0,+],[2n,+])\\
    \ev_{[n,+]}:\overline{[n,+]}\otimes [n,+]=[2n,?]\to 1^? \quad &\text{ as }\quad \ev_{[n,+]}=(d^n;[2n,?],[0,?]),\ ?=+,\text{ if }2\mid n\\
    \coev_{[n,-]}:1^-\to [2n,-]=[n,-]\otimes \overline{[n,-]} \quad &\text{ as }\quad \coev_{[n,-]}=(d^n;[0,-],[2n,-])\\
    \ev_{[n,-]}:\overline{[n,-]}\otimes [n,-]=[2n,?]\to 1^? \quad &\text{ as }\quad \ev_{[n,-]}=(d^n;[2n,?],[0,?]),\ ?=-,\text{ if }2\mid n
\end{align*}
\end{definition}

\begin{proposition}
$\mathcal{A}_0$ is rigid.
\end{proposition}
\begin{proof} 
First we prove that 
$$(\id_{[n,+]}\otimes \ev_{[n,+]})\circ (\coev_{[n,+]}\otimes \id_{[n,+]})=\id_{[n,+]}.$$

Note that $\id_{[n,+]}\otimes \ev_{[n,+]} = (S^{(n)}(d^n);[2n+n,+],[0+n,+])=(d^n;[3n,+],[n,+])$ and $\coev_{[n,+]}\otimes \id_{[n,+]} = (d^n e^0_{(n-n),n};[0+n,+],[2n+n,+])= (d^n e^0_{0,n};[n,+],[3n,+])$.

Then by the composition case (C2), where $i=0,j=n$,
\begin{align*}
    (\id_{[n,+]}\otimes \ev_{[n,+]})\circ (\coev_{[n,+]}\otimes \id_{[n,+]}) &= (d^n;[3n,+],[n,+])\circ (d^ne^0_{0,n};[n,+],[3n,+])\\
    & = (d^0E^{r,0+n}_{0,n+2n}(d^{2n}e^0_{0,n}e^{n,*}_{j,i});[n,+],[n+2i,+])\\
    & = (d^{2n}E^{r,n}_{0,3n}(e^0_{0,n});[n,+],[n,+]) \\
    & = (1;[n,+],[n,+]) = \id_{[n,+]}.
\end{align*}

The other three cases are left to the reader. Therefore, $\mathcal{A}_0$ is rigid. 
\end{proof}

\begin{proposition}
$\mathcal{A}_0$ has a pivotal structure. \label{A_0 is pivotal from A}
\end{proposition}
\begin{proof} 
First, we prove that the right trace $\Tr_R$ is a normal faithful trace. 
Let $X=[n,+]$. 
Given $(f;[n,+],[n,+])$, $f\otimes \id_{\overline{[n,+]}}=(f;[2n,+],[2n,+])$, then 
\begin{align*}
    \Tr_R(f) & =\coev^\dagger_{[n,+]}\circ (f\otimes \id_{\overline{{[n,+]}}})\circ \coev_{[n,+]}\\
    & = (d^n;[2n,+],[0,+])\circ (f;[2n,+],[2n,+])\circ (d^n;[0,+],[2n,+])\\
    & = (d^n;[2n,+],[0,+])\circ (d^nE^{r,n}_{0,2n}(f\cdot d^ne^0_{0,n});[0,+],[2n,+])\\
    & = (d^n;[2n,+],[0,+])\circ (f;[0,+],[2n,+])\\
    & = (d^0E^{r,n}_{0,n}(fe^{0,*}_{n,0});[0,+];[0,+])\\
    & = (\tr(f);[0,+],[0,+]).
\end{align*}
The third equality uses (C1), where $n=0,i=n,j=0$; the forth equality uses ($\lambda10$); 
the fifth equality uses (C2), where $n=i=0,j=n$.

The case $X=[n,-]$ is left to the reader.

Next, we prove that the left trace $\Tr_L$ is a normal faithful trace. 
Let $X=[2n,+]$. Given $(f;[2n,+],[2n,+])$, $\id_{\overline{[2n,+]}}\otimes f=(S^{(n)}_{0,2n}(f);[4n,+],[4n,+])$, then 
\begin{align*}
    \Tr_L(f) & = \ev_{[2n,+]}\circ (\id_{\overline{[2n,+]}}\otimes f)\circ \ev_{[2n,+]}^\dagger \\
    & = (d^{2n};[4n,+],[0,+])\circ (S^{(n)}_{0,2n}(f);[4n,+],[4n,+])\circ (d^{2n};[0,+],[4n,+])\\
    & = (d^{2n};[4n,+],[0,+])\circ (d^{2n}E^{r,2n}_{0,4n}(S^{(n)}_{0,2n}(f)\cdot d^{2n}e^0_{0,2n});[0,+],[4n,+])\\
    & = (d^{4n}\cdot d^0 E^{r,2n}_{0,2n}(E^{r,2n}_{0,4n}(S^{(n)}_{0,2n}(f)e^0_{0,2n}) e^{0,*}_{0,2n});[0,+],[0,+])\\
    & = (\tr(f);[0,+],[0,+]).
\end{align*}

The last equality: since $e^{0,*}_{0,2n}=1$ and $E^{r,2n}_{0,2n}\circ E^{r,2n}_{0,4n}=\tr=E^{r,2n}_{2n,4n}\circ E^{l,2n}_{0,4n}$, $S^{(n)}_{0,2n}(f)\in A_{2n,4n}$ and $S^{(n)}_{0,2n}$ is trace-preserving,
then 
\begin{align*}
    d^{4n}\cdot d^0 E^{r,2n}_{0,2n}(E^{r,2n}_{0,4n}(S^{(n)}_{0,2n}(f)e^0_{0,2n}) e^{0,*}_{0,2n}) &=d^{4n}\tr(S^{(n)}_{0,2n}(f)e^0_{0,2n})\\
    & = d^{4n} E^{r,2n}_{2n,4n} \circ E^{l,2n}_{0,4n}(S^{(n)}_{0,2n}(f)e^0_{0,2n}) \tag{by ($\lambda$10)}\\
    & = E^{r,2n}_{2n,4n}(S^{(n)}_{0,2n}(f)) \tag{by Prop \ref{2n-shift map prop}(2)}\\
    & = E^{r,2n}_{0,2n}(f) =\tr(f).
\end{align*}

The cases $X=[2n+1,+],[n,-]$ are left to the reader.

Therefore, $\Tr_R=\Tr_L$ are the trace, so $\mathcal{A}_0$ has a pivotal structure.  

Moreover, by the composition case (C2), where $i=n=0, j=n$,
\begin{align*}
    \ev_{[n,+]}\circ \coev_{[n,+]} & = (d^n;[2n,+],[0,+])\circ (d^n;[0,+],[2n,+]) \\
    & = (d^0E^{r,n}_{0,2n}(d^{2n}e^{0,*}_{n,0});[0,+],[0,+]) \\
    & = (d^{2n};[0,+],[0,+]) = d^{2n}\cdot 1^+.
\end{align*}

Similarly, $\ev_{\overline{[n,-]}}\circ \coev_{[n,-]}=d^{2n} \cdot 1^-$. 
\end{proof}

Combining above propositions, $\mathcal{A}_0$ constructed from a standard $\lambda$-lattice is a pivotal planar tensor category.

\subsection{From 2-shaded rigid ${\rm C}^*$ multitensor category to standard $\lambda$-lattice}

In this section, we show the relation between the 2-shaded rigid $\rm C^*$ multitensor category and planar tensor category, and give the construction from the category to standard $\lambda$-lattice.

\subsubsection{Rigid ${\rm C}^*$ multitensor category}\label{unitary dual functor}
In this subsection, we are going to review the unitary dual functors in a rigid ${\rm C}^*$ (multi)tensor category $\mathcal{C}$\cite{Pe18}.

\begin{definition}
Recall that every object $c\in \mathcal{C}$ is \textbf{dualizable}, i.e., there is an object $\overline{c}\in\mathcal{C}$ together with morphisms $\ev_c\in \mathcal{C}(\overline{c}\otimes c\to 1_{\mathcal{C}})$ and $\coev_c\in\mathcal{C}(1_{\mathcal{C}}\to c\otimes \overline{c})$ satisfying the zig-zag condition:
\begin{align*}
    (\id_c\otimes \ev_c)\circ (\coev_c\otimes \id_c)=\id_c\\
    (\ev_c\otimes \id_{\overline{c}})\circ (\id_{\overline{c}}\otimes \coev_c)=\id_{\overline{c}}.
\end{align*}
We also require that every object $c\in \mathcal{C}$ admits a predual object $\underline{c}$ such that $\overline{(\underline{c})}\cong c$.
\end{definition}

\begin{definition}\label{Def:bar action}
A choice of dual for every object in $\mathcal{C}$ assembles into a
\textbf{dual functor} $\overline{(\cdot)}:\mathcal{C}\to \mathcal{C}^{\text{mop}}$, which is a tensor functor with a canonical tensorator $\nu_{a,b}$. 
To be precise, for a morphism $f\in \mathcal{C}(a\to b)$, define
\begin{align*}
    \overline{f}:=(\ev_b\otimes \id_{\overline{a}})\circ (\id_{\overline{b}}\otimes f\otimes \id_{\overline{a}})\circ (\id_{\overline{b}}\otimes \coev_a):\overline{b}\to \overline{a}.
\end{align*}
\[\overline{f}:=\ 
\begin{tikzpicture}[baseline=-.1cm]
\draw (0,-.5) -- (0,.5);
\draw (0,.5) arc (0:180:.3);
\draw (0,-.5) arc (-180:0:.3);
\draw (-.6,-.8) -- (-.6,.5);
\draw (.6,-.5) -- (.6,.8);

\nbox{unshaded}{(0,0)}{.3}{0.0}{0.0}{$f$};

\node at (-.6,-1) {$\overline{b}$};
\node at (.6,1) {$\overline{a}$};
\node at (.2,.5) {$b$};
\node at (-.2,-.5) {$a$};
\end{tikzpicture}
\]
The tensorator $\nu_{a,b}:\overline{a}\otimes \overline{b}\to \overline{b\otimes a}$ is defined as
$$\nu_{a,b}:=(\ev_a\otimes \id_{\overline{b\otimes a}})\circ(\id_{\overline{a}}\otimes\ev_b\otimes \id_a\otimes \id_{\overline{b\otimes a}})\circ (\id_{\overline{a}\otimes \overline{b}}\otimes \coev_{b\otimes a}).$$
Note that $\nu$ is completely determined by $\ev$ and $\coev$.
\end{definition}

\begin{proposition}\label{equiv btw dual functors}
Any two dual functors $\overline{(\cdot)}_1$ and $\overline{(\cdot)}_2$ are equivalent up to a unique natural isomorphism. Define $\zeta:\overline{(\cdot)}_2\to \overline{(\cdot)}_1$ as follows: for $c\in\mathcal{C}$,
$$\zeta_c:=(\ev_c^2\otimes \id_{\overline{c}_1})\circ (\id_{\overline{c}_2}\otimes \coev_c^1).$$
\[\zeta_c = \begin{tikzpicture}[baseline=-.1cm]
\draw[red] (0,0) arc (0:180:.5);
\draw (0,0) arc (-180:0:.5);
\draw[red] (-1,0) -- (-1,-.5);
\draw (1,0) -- (1,.5);

\node at (.5,-.7) {$\coev^1_c$};
\node at (-.5,.8) {$\ev^2_c$};
\node at (-.2,0) {$c$};
\node at (-1,-.7) {$\overline{c}_2$};
\node at (1,.7) {$\overline{c}_1$};
\end{tikzpicture}\]
Then we have $\zeta(\overline{f}_2)=\zeta_a\circ \overline{f}_2\circ \zeta_b^{-1}=\overline{\zeta(f)}_1$ for all $f\in \mathcal{C}(a\to b)$. 
\end{proposition}

\begin{definition}\cite{EGNO15} A \textbf{pivotal structure} on a rigid monoidal category $\mathcal{C}$ is a pair $(\overline{(\cdot)},\varphi)$, where $\overline{(\cdot)}$ is a dual functor and $\varphi:\id\Rightarrow \overline{\overline{(\cdot)}}$ is a monoidal natural isomorphism. To be precise, for all $a,b\in \mathcal{C}$, the following diagram commutes:
$$\xymatrix{
a\otimes b \ar[rr]^{\varphi_a\otimes \varphi_b}\ar[d]_{\varphi_{a\otimes b}} && \overline{\overline{a\otimes b}} \ar[d]^{\overline{\nu_{\overline{b},\overline{a}}}}\\
\overline{\overline{a}}\otimes \overline{\overline{b}} \ar[rr]_{\nu_{\overline{a},\overline{b}}} && \overline{\overline{b}\otimes \overline{a}} 
}$$

\end{definition}

\begin{definition}[Pivotal trace] Let $1_{\mathcal{C}}=\bigoplus_{i=1}^r 1_i$ be a decomposition into simples. For $c\in \mathcal{C}$ and $f\in\mathcal{C}(c\to c)$, define the left/right \textbf{pivotal traces} $\tr_L^\varphi$ and $\tr_R^\varphi:\mathcal{C}(c\to c)\to \mathcal{C}(1_{\mathcal{C}}\to 1_{\mathcal{C}})\cong M_r(\mathbb{C})$ by
\begin{align*}
    \tr^\varphi_L(f) &:= \ev_c\circ (\id_{\overline{c}}\otimes f)\circ (\id_{\overline{c}}\otimes \varphi_c^{-1})\circ \coev_{\overline{c}}\\
    \tr^\varphi_R(f) &:= \ev_{\overline{c}}\circ (\varphi_c\otimes \id_{\overline{c}})\circ (f\otimes \id_{\overline{c}})\circ \coev_{c}.
\end{align*}
\[\tr^\varphi_L(f)=\begin{tikzpicture}[baseline=-.1cm]
\draw (0,-1) -- (0,1);
\draw (-.6,-1) -- (-.6,1);
\draw (0,1) arc (0:180:.3);
\draw (0,-1) arc (0:-180:.3);

\nbox{unshaded}{(0,.5)}{.3}{0.0}{0.0}{$f$};
\nbox{unshaded}{(0,-.5)}{.3}{0.0}{0.0}{\tiny{$\varphi_c^{-1}$}};

\node at (.2,0) {$c$};
\node at (.2,-1) {$\overline{\overline{c}}$};
\node at (.2,1) {$c$};
\node at (-.8,0) {$\overline{c}$};
\end{tikzpicture} \qquad \qquad
\tr^\varphi_R(f)=\begin{tikzpicture}[baseline=-.1cm]
\draw (0,-1) -- (0,1);
\draw (.6,-1) -- (.6,1);
\draw (0,1) arc (180:0:.3);
\draw (0,-1) arc (-180:0:.3);

\nbox{unshaded}{(0,-.5)}{.3}{0.0}{0.0}{$f$};
\nbox{unshaded}{(0,.5)}{.3}{0.0}{0.0}{$\varphi_c$};

\node at (-.2,0) {$c$};
\node at (-.2,1.05) {$\overline{\overline{c}}$};
\node at (-.2,-1) {$c$};
\node at (.8,0) {$\overline{c}$};
\end{tikzpicture}
\]
The traces are tracial and non-degenerate.
\end{definition}

\begin{definition}\label{Def:spherical trace}
Let $p_i\in\mathcal{C}(1_{\mathcal{C}}\to 1_{\mathcal{C}})$ be the projection onto $1_i$, $i=1,2,\cdots,r$. We define the $M_r(\mathbb{C})$-valued traces $\Tr^\varphi_L$ and $\Tr^\varphi_R$ by the formulas:
\begin{align*}
    (\Tr^\varphi_L(f))_{i,j}\id_{1_j} &:= \tr^\varphi_L(p_i\otimes f\otimes p_j)\\
    (\Tr^\varphi_R(f))_{i,j}\id_{1_i} &:= \tr^\varphi_R(p_i\otimes f\otimes p_j).
\end{align*}
Note that $\Tr^\varphi_L$ and $\Tr^\varphi_R$ are tracial, and $\Tr^\varphi_L(\overline{f})=\Tr^\varphi_R(f)^T$ for all $f\in\mathcal{C}(c\to c)$.

We call the pivotal structure $(\overline{(\cdot)},\varphi)$ \textbf{spherical},
if $\Tr^\varphi_L(f)=\Tr^\varphi_R(f)$, for all $c\in \mathcal{C},\ f\in\mathcal{C}(c\to c)$.
\end{definition}

\begin{definition}\label{def dimension}
For each $c\in C$, define $\text{Dim}^\varphi_L,\text{Dim}^\varphi_R\in M_r(\mathbb{C})$ by
$$\text{Dim}^\varphi_L(c):=\Tr^\varphi_L(\id_c)\qquad\qquad \text{Dim}^\varphi_R(c):=\Tr^\varphi_R(\id_c).$$

If $c$ is simple, then $\text{Dim}^\varphi_L(c),\text{Dim}^\varphi_R(c)$ have only one non-zero entry, which we denote $\dim^\varphi_L(c),\dim^\varphi_R(c)$ respectively.

If the pivotal structure $(\overline{(\cdot)},\varphi)$ is spherical, $\text{Dim}^\varphi_L(c)=\text{Dim}^\varphi_R(c):=\text{Dim}(c)$ for all object $c$. 
\end{definition}

\begin{definition}
A \textbf{dagger structure} on a $\mathbb{C}$-linear category is a collection of anti-linear maps $\dagger:\mathcal{C}(c\to d)\to \mathcal{C}(d\to c)$ for all $c,d\in\mathcal{C}$ such that $(f\circ g)^\dagger = g^\dagger \circ f^\dagger$ and $(f^\dagger)^\dagger = f$. A morphism $f:\mathcal{C}(a\to b)$ is called unitary if $f^\dagger =f^{-1}$.

A dagger (multi)tensor category is a (multi)tensor category equipped with a dagger structure so that $(f\otimes g)^\dagger=f^\dagger\otimes g^\dagger$ for all morphisms $f,g$, and all associator and unitors are unitary.
\end{definition}

\begin{definition}
A functor between dagger categories $F:\mathcal{C}\to\mathcal{D}$ is called a \textbf{dagger functor} if $F(f^\dagger)=F(f)^\dagger$ for all $f\in\Hom(\mathcal{C})$.
\end{definition}

\begin{definition}[Rigid ${\rm C}^*$ (multi)tensor category]
A ${\rm C}^*$ category is a dagger category which is Cauchy complete and each endomorphism algebra is a ${\rm C}^*$-algebra, where the dagger structure is compatible with the $*$-structure.

A ${\rm C}^*$ (multi)tensor category is a dagger (multi)tensor category whose underlying dagger category is ${\rm C}^*$.

A rigid ${\rm C}^*$ (multi)tensor category is a ${\rm C}^*$ (multi)tensor category equipped with a dual functor. It is known that a rigid ${\rm C}^*$ multitensor category is Cauchy complete if and only if it is semisimple \cite{LR96}.
\end{definition}

\begin{proposition}[Unitary dual functor] Fix a dual functor $\overline{(\cdot)}$ on a rigid ${\rm C}^*$ (multi)tensor category $\mathcal{C}$, the followings are equivalent:
\begin{compactenum}[(1)]
\item $\overline{(\cdot)}$ is a \textbf{unitary dual functor}, i.e., for all $a,b\in\mathcal{C},\ f\in \mathcal{C}(a\to b)$, the tensorator $\nu_{a,b}$ is unitary and $\overline{f}^\dagger= \overline{f^\dagger}$.
\item Defining $\varphi_c:=(\coev_c^\dagger\otimes \id_{\overline{\overline{c}}})\circ (\id_c\otimes \coev_{\overline{c}})$ is a pivotal structure $\varphi:\id\Rightarrow \overline{\overline{(\cdot)}}$.
\end{compactenum}
\end{proposition}
\begin{proof}
\cite{Se11}, see also  \cite[Prop.~3.9]{Pe18}.
\end{proof}

\begin{definition}
Two unitary dual functors are called \textbf{unitary equivalent}, if the canonical natural transformation $\zeta$ from Proposition \ref{equiv btw dual functors} is unitary, i.e., $\zeta_c$ is unitary for all $c\in\mathcal{C}$.
\end{definition}

\begin{proposition}
For a unitary dual functor $\overline{(\cdot)}$, the left/right pivotal traces have alternate formulas:
\begin{align*}
    \tr^\varphi_L(f) &= \ev_c\circ (\id_{\overline{c}}\otimes f)\circ \ev_c^\dagger\\
    \tr^\varphi_R(f) &= \coev_c^\dagger\circ (f\otimes \id_{\overline{c}})\circ \coev_c.
\end{align*}
\end{proposition}

\begin{theorem}[{\cite{BDH14} \cite[Prop.~3.24]{Pe18}}]

For a rigid ${\rm C}^*$ (multi)tensor category $\mathcal{C}$, there exists a unique unitary dual functor whose induced pivotal structure is spherical up to unitary equivalence. In other words, the pivotal structure can be trivial, so that $\ev_{\overline{c}}=\coev_c^\dagger$ and $\coev_{\overline{c}}=\ev_c^\dagger$ for all $c\in\mathcal{C}$. 
\end{theorem}

\subsubsection{2-shaded rigid ${\rm C}^*$ multitensor category with a choice of generator and planar tensor category}\label{cat A to A_0}
Let $\mathcal{A}$ be a 2-shaded rigid ${\rm C}^*$ multitensor category together with $1=1^+\oplus 1^-$, where $1^+,1^-$ are simple, and a generator $X=1^+\otimes X\otimes 1^-$. Here, the generating means for any simple object $P$, it is a direct summand of $X^{\alt\otimes n}$ or $\overline{X}^{\alt\otimes n}$ (defined below) for some $n\in\mathbb{Z}_{\ge 0}$.

Let $\overline{(\cdot)}$ be a unitary dual functor that induced a spherical pivotal structure $\varphi$. Note that only $(+,-)$ entry of $\text{Dim}(X)$ is non-zero and we denote this number as $d_X$ to be the modulus of category $\mathcal{C}$.

\begin{construction}
We construct a planar tensor category $\mathcal{A}_0$ from $(\mathcal{A},X)$. By MacLane's coherence theorem, $\mathcal{A}$ is unitary equivalent to a strict tensor category with the above properties and the dual functor is strict, WLOG, we also denote it as $\mathcal{A}$. Construct the pivotal planar tensor category $\mathcal{A}_0$ as follows:
\begin{compactenum}[(a)]
\item Objects: Define $[0,+]:=1^+$, $[0,-]:=1^-$, and
$$[n,+]:=[n-1,+]\otimes X^?=\underbrace{(\cdots(X\otimes \overline{X})\otimes X)\otimes \cdots)\otimes X^?}_{n \text{ tensorands}}=:X^{\alt\otimes n},$$ 
where $X^?=\overline{X}$ if $n$ is even and $X$ if $n$ is odd, and 
$$[n,-]:=[n-1,-]\otimes X^?=\underbrace{(\cdots(\overline{X}\otimes X)\otimes \overline{X})\otimes \cdots)\otimes X^?}_{n \text{ tensorands}}=:\overline{X}^{\alt\otimes n},$$
where $X^?=X$ if $n$ is even and $\overline{X}$ if $n$ is odd, for $n\in\mathbb{Z}_{\ge 0}$. 
\item Morphisms: $\mathcal{A}_0$ is the full subcategory of $\mathcal{A}$ with above objects.
\item Duality: The dual functor is unitary as a dual functor on the subcategory, which also induces a spherical pivotal structure on the subcategory. 
\end{compactenum}
\end{construction}

Given $\mathcal{A}_0$ to be a pivotal planar tensor category, then its Cauchy completion $\widehat{\mathcal{A}_0}$ is a Cauchy completed 2-shaded rigid ${\rm C}^*$ multitensor category with a generator $[1,+]$ and a canonical unitary dual functor $\overline{(\cdot)}_1$.

\begin{proposition}
Suppose $\mathcal{A}_0$ is a pivotal planar tensor category constructed from $(\mathcal{A},X)$, then there is a unitary equivalence between $(\widehat{\mathcal{A}_0},[1,+])$ and the Cauchy completion of $(\mathcal{A},X)$ with respect to their unitary dual functors.
\end{proposition}

\begin{remark}
Suppose $\mathcal{A},\mathcal{B}$ are two 2-shaded rigid $\rm C^*$ multitensor categories with generator $X$ and $Y$ respectively and $\mathcal{A}_0,\mathcal{B}_0$ are corresponding pivotal planar tensor categories. Then $\mathcal{A}_0$ and $\mathcal{B}_0$ are unitary equivalent if and only if the Cauchy completions of $\mathcal{A}$ and $\mathcal{B}$ are unitary equivalent which maps generator to generator.
\end{remark}

\begin{remark}
The planar tensor category $\mathcal{A}_0$ is not Cauchy complete, i.e., additive complete and idempotent complete.
In fact, as for skeletalness, strictness and Cauchy complete, most tensor categories can require at most two of them. Vec$(G)$ is an exception. 
\end{remark}

\subsubsection{From planar tensor category to standard $\lambda$-lattice}\label{A_0 to std A}
\begin{construction}
Let $\mathcal{A}_0$ be a pivotal planar tensor category with modulus $d$. Define  $A_{0,j}=\End([j,+])$, $A_{1,j}=\id_{[1,+]}\otimes \End([j-1,-])$, $j\in\mathbb{Z}_{\ge 0}$, so that $A_{0,0}=A_{1,1}=\mathbb{C}$. In general, for $i\le j$, define
$$A_{i,j}=\begin{cases}\id_{[i,+]}\otimes \End([j-i,+])& 2\mid i\\
\id_{[i,+]}\otimes\End([j-i,-])& 2\nmid i.
\end{cases}$$
Then we check $A=(A_{i,j})_{i,j\ge 0}$ to be a standard $\lambda$-lattice.
\begin{compactenum}[(a)]
\item The vertical inclusion $A_{i+1,j}\subset A_{i,j}$ is clear. The right inclusion: the right inclusion send $x\in A_{i,j}$ to $x\otimes \id_{[1,?]}\in A_{i,j+1}$, where $?=+$ if $2\mid j$ and $?=-$ if $2\nmid j$.
\item Horizontal conditional expectation: Define $E^r_{i,j}:A_{i,j}\to A_{i,j-1}$ by 
\begin{align*}
    E^r_{i,2k}(x)&=d^{-1}(\id_{[2k-1,+]}\otimes \ev_{\overline{[1,+]}})\circ (x\otimes [1,+])\circ (\id_{[2k-1,+]}\otimes \coev_{[1,+]})\\
    E^r_{i,2k+1}(x)&=d^{-1}(\id_{[2k,+]}\otimes \ev_{\overline{[1,-]}})\circ (x\otimes [1,-])\circ (\id_{[2k,+]}\otimes \coev_{[1,-]}).
\end{align*}
\item Vertical conditional expectation: Define $E^l_{i,j}:A_{i,j}\to A_{i+1,j}$ by
\begin{align*}
    E^l_{2k,j} &= d^{-1}(\id_{[2k+2,+]}\otimes \ev_{\overline{[1,+]}})\circ (\id_{[2,+]}\otimes x)\circ (\id_{[2k+2,+]}\otimes \coev_{[1,+]})\\
    E^l_{2k+1,j} &= d^{-1}(\id_{[2k+3,+]}\otimes \ev_{\overline{[1,-]}})\circ (\id_{[2,+]}\otimes x)\circ (\id_{[2k+3,+]}\otimes \coev_{[1,-]}).
\end{align*}
\item Jones projection: the $n$-th Jones projection is defined as 
\begin{align*}
    e_{2k+1} &= d^{-1}\cdot\id_{[2k,+]}\otimes (\coev_{[1,+]}\circ \ev_{\overline{[1,+]}})\in A_{i,2k+2}\\
    e_{2k+2} &= d^{-1}\cdot\id_{[2k+1,+]}\otimes (\coev_{[1,-]}\circ \ev_{\overline{[1,-]}})\in A_{i,2k+3}.
\end{align*}
\end{compactenum}
The check that $A=(A_{i,j})_{j\ge i\ge 0}$ satisfies Definition \ref{standard lambda lattice}(a), (b), (c) and standard condition is left to the reader. In particular, $e_ne_{n\pm 1}e_n=d^{-2}e_n$, $E^r_{i,j+1}(e_j)=E^l_{j-1,k}(e_j)=d^{-2} 1$.

Note that the dual functor is unitary and we divide the loop parameter, the composition of these conditional expectations is actually a unital trace on $A$.
\end{construction}

\begin{remark}
The idea of drawing the string diagram explanation in \S\ref{TLJ diagram explanation} comes from here.
\end{remark}

In this section, the class of unitary equivalent pairs $(\mathcal{A},X)$ with $\mathcal{A}$ a 2-shaded rigid ${\rm C}^*$ multitensor category and $X$ a generator induces the class of isomorphic pivotal planar tensor categories; in \S\ref{std lattice to planar cat}, the class of isomorphic pivotal planar tensor categories is one to one corresponding to the class of isomorphic standard $\lambda$-lattices.

Combining above discussion, we can deduce the equivalence between standard $\lambda$-lattice $A$ and pair 2-shaded rigid ${\rm C}^*$ multitensor category with a generator $(\mathcal{A},X)$. 

\begin{theorem} There is a bijective correspondence between equivalence classes of the following:
\[\left\{\, \parbox{3.2cm}{\rm Standard $\lambda$-lattices $A=(A_{i,j})_{0\le i\le j}$} \,\right\}\ \, \cong\ \, \left\{\, \parbox{9cm}{\rm Pairs $(\mathcal{A}, X)$ with $\mathcal{A}$ a 2-shaded rigid ${\rm C}^*$ multitensor category with a generator $X$, i.e., $1_\mathcal{A}=1^+\oplus 1^-$, $1^+,1^-$ are simple and $X=1^+\otimes X\otimes 1^-$ 
} \,\right\}\]
Equivalence on the left hand side is unital $*$-isomorphism of standard $\lambda$-lattices; 
equivalence on the right hand side is unitary equivalence 
between their Cauchy completions 
which maps generator to generator.
\end{theorem}

\section{Markov towers as standard right module over standard $\lambda$-lattice and module categories}\label{Cpt 2}

Now we move to the module case. One motivation that regards a Markov tower as a right module over a standard $\lambda$-lattice is to answer the question in \cite[Rmk.~3.34]{CHPS18}. 

\subsection{Markov tower as a standard right module over standard $\lambda$-lattice}
\begin{definition}\label{Def:Markov tower as standard module}
$$\begin{matrix}
M_0 & \subset & M_1 & \subset & M_2 & \subset & \cdots & \subset & M_n & \subset &\cdots\\
\cup & & \cup & & \cup & & & & \cup\\
A_{0,0} & \subset & A_{0,1} & \subset & A_{0,2} & \subset & \cdots & \subset & A_{0,n} & \subset & \cdots\\
& & \cup & & \cup & & & & \cup\\
& & A_{1,1} & \subset & A_{1,2} & \subset & \cdots & \subset & A_{1,n} & \subset & \cdots 
\end{matrix}$$
Let $A=(A_{i,j})_{0\le i\le j<\infty}$ be a standard $\lambda$-lattice with Jones projection $\{e_i\}_{i\ge 1}$ and compatible conditional expectations. Let $M=(M_n, e_n)_{n\ge 0}$ be a Markov tower with conditional expectation $E_i:M_i\to M_{i-1}$, $i\ge 1$. ($M$ and $A$ share the same Jones projections) We call a Markov tower $M$ a \textbf{standard right $A-$module}, if it satisfies the following three conditions.
\begin{compactenum}[(a)]
\item $A_{0,i}\subset M_i$ is a unital inclusion, $i=0,1,2,\cdots$.
\item $E_i|_{A_{0,i}}=E^r_{0,i}$, $i=1,2,\cdots$.
\item (standard condition) $[M_i,A_{k,l}]=0$ for $i\le k\le l$.
\end{compactenum}
\end{definition}

In the rest of this Chapter, we only consider the Markov tower with $\dim (M_0)=1$ unless stated.

\subsection{String diagram explanation}\label{TLJ diagram explanation, module}
We now introduce the diagrammatic explanation of the element, conditional expectation, Jones projection and their relations in a Markov tower with the same spirit in \S\ref{TLJ diagram explanation}.
\begin{compactenum}[({MT}1)]
\item Element $x\in M_n$:
\[
\begin{tikzpicture}[baseline=-.1cm]
\draw (.5,-.7) -- (.5,.7);
\draw[thick,blue] (.2,-.7) -- (.2,.7);
\dbox{unshaded}{(.4,0)}{.3}{.1}{0}{$x$};
\node at (.5,-.9) {\tiny{$n$}};
\end{tikzpicture}
:=
\begin{tikzpicture}[baseline=-.1cm]
\filldraw[gray1] (.5,-.7) rectangle (.8,.7);
\filldraw[gray1] (1.1,-.7) rectangle (1.4,.7);
\draw (1.4,-.7) -- (1.4,.7);
\draw (1.1,-.7) -- (1.1,.7);
\draw (.8,-.7) -- (.8,.7);
\draw (.5,-.7) -- (.5,.7);
\draw[thick,blue] (.2,-.7) -- (.2,.7);
\dbox{unshaded}{(.4,0)}{.3}{.1}{.9}{$x$};
\node at (.8,-.9) {\tiny{$n$}};
\end{tikzpicture}
\]

\item Vertical inclusion $x\in A_{0,n}\subset M_n$:
\[
\begin{tikzpicture}[baseline=-.1cm]
\draw (.7,-.7) -- (.7,.7);
\dbox{unshaded,dashed}{(.4,0)}{.35}{.05}{.3}{};
\draw[thick,blue] (.2,-.7) -- (.2,.7);
\nbox{unshaded}{(.7,0)}{.3}{0}{0}{$x$};

\node at (.7,-.9) {\tiny{$n$}};
\end{tikzpicture}
\]
\item Horizontal inclusion $x\in M_n\subset M_{n+1}$:
\[
\begin{tikzpicture}[baseline=-.1cm]
\draw (.5,-.7) -- (.5,.7);
\draw[thick,blue] (.2,-.7) -- (.2,.7);
\dbox{unshaded,dashed}{(.4,0)}{.35}{.05}{.3}{};
\draw (.8,-.7) -- (.8,.7);
\dbox{unshaded}{(.3,0)}{.3}{0}{0.05}{$x$};
\node at (.5,-.9) {\tiny{$n$}};
\node at (.8,-.9) {\tiny{$1$}};
\end{tikzpicture}
\]
\item Jones projections:
\[ e_{2i+1}= d^{-1}\ 
\begin{tikzpicture}[baseline=.4cm]
\filldraw[shaded] (.8,1) arc (-180:0:.3cm);
\filldraw[shaded] (.8,0) arc (180:0:.3cm);

\draw[blue,thick] (0.2,0) -- (0.2,1);
\draw (.5,0) -- (.5,1);

\draw[dashed] (0,0) rectangle (1.7,1);
\node at (.5,-.2) {\tiny{$2i$}};
\end{tikzpicture}\in M_{2i+2}
\qquad
e_{2i+2} = d^{-1}\ 
\begin{tikzpicture}[baseline=.4cm]
\fill[shaded] (0.5,0) rectangle (1.7,1);
\filldraw[unshaded] (.8,1) arc (-180:0:.3cm);
 \filldraw[unshaded] (.8,0) arc (180:0:.3cm);

  \draw[blue,thick] (0.2,0) -- (0.2,1);
 \draw (.5,0) -- (.5,1);

\draw[dashed] (0,0) rectangle (1.7,1);
 \node at (.5,-.2) {\tiny{$2i\+ 1$}};
\end{tikzpicture}\in M_{2i+3}
\]
\item Conditional expectation $E_n:M_n\to M_{n-1}$ and  $E_n|_{A_{0,n}}=E^r_{0,n}$:
\[E_n(x)= d^{-1}\ 
\begin{tikzpicture}[baseline=-.1cm]
\draw (.5,-.7) -- (.5,.7);
\draw (.8,-.5) -- (.8,.5);
\draw[thick,blue] (.2,-.7) -- (.2,.7);
\dbox{unshaded}{(.4,0)}{.3}{.05}{.3}{$x$};

\draw[orange,thick] (1.1,-.5) -- (1.1,.5);
\draw[orange,thick] (.8,.5) arc (180:0:.15cm);
\draw[orange,thick] (.8,-.5) arc (-180:0:.15cm);

\node at (.5,-.9) {\tiny{$n\sm1$}};
\node at (.8,.8) {\tiny{$1$}};
\end{tikzpicture}
\ ,x\in M_n
\qquad
E_n(x)=E^r_{0,n}(x)= d^{-1}
\begin{tikzpicture}[baseline=-.1cm]
\draw (.5,-.7) -- (.5,.7);
\draw (.8,-.5) -- (.8,.5);
\draw[thick,blue] (.2,-.7) -- (.2,.7);
\nbox{unshaded}{(.7,0)}{.3}{0.05}{0}{$x$};

\draw[orange,thick] (1.1,-.5) -- (1.1,.5);
\draw[orange,thick] (.8,.5) arc (180:0:.15cm);
\draw[orange,thick] (.8,-.5) arc (-180:0:.15cm);

\node at (.5,-.9) {\tiny{$n\sm1$}};
\node at (.8,.8) {\tiny{$1$}};
\end{tikzpicture}
\ ,x\in A_{0,n}
\]

\item Pull down condition: For $x\in M_{n+1}$, $xe_n=dE_{n+1}(xe_n)e_n$.
\[
\begin{tikzpicture}[baseline=.9cm]
\draw[thick,blue] (0.4,0) -- (0.4,2);
\draw (0.6,1) -- (0.6,2);
\draw (1,1) -- (1,2);
\dbox{unshaded}{(.5,1.5)}{.3}{0.0}{0.0}{$x$};

\draw (.6,1) arc (-180:0:.2cm);
\draw (.6,0) arc (180:0:.2cm);
 
\draw[dashed] (0,1) -- (1.2,1);
\draw[dashed] (0,0) rectangle (1.2,2);
\node at (0.4,-.2) {\tiny{$j\sm 1$}};
\node at (.5,2.2) {\tiny{$j$}};
\node at (1,2.2) {\tiny{$1$}};
\end{tikzpicture}
\ =\
\begin{tikzpicture}[baseline=.9cm]
\draw[thick,blue] (0.4,-1) -- (0.4,2);
\draw (0.6,1) -- (0.6,2);
\draw (1,1) -- (1,2);
\draw (1.5,0) -- (1.5,2);
\dbox{unshaded}{(.5,1.5)}{.3}{0.0}{0.0}{$x$};

\draw (.6,1) arc (-180:0:.2cm);
\draw (.6,0) arc (180:0:.2cm);
\draw (.6,0) arc (-180:0:.45cm);
\draw (.6,-1) arc (180:0:.45cm);
\draw[orange,thick] (1.3,0) -- (1.3,2);
\draw[orange,thick] (1,2) arc (180:0:.15cm);
\draw[orange,thick] (1,0) arc (-180:0:.15cm);
 
\draw[dashed] (0,1) -- (1.7,1);
\draw[dashed] (0,0) -- (1.7,0);
\draw[dashed] (0,-1) rectangle (1.7,2);
\node at (0.4,-1.2) {\tiny{$j\sm1$}};
\node at (.5,2.2) {\tiny{$j$}};
\node at (1.5,2.2) {\tiny{$1$}}; 
\end{tikzpicture}
\]

\item Standard condition: For $f\in M_i$, $x\in A_{k,l}$ with $k\ge i$, then we regard $\phi,x$ as elements in $M_l$, $f x=x f$.
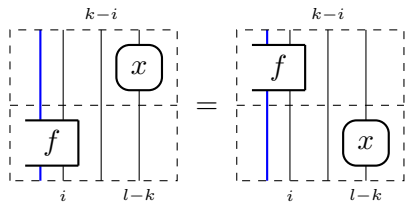
\[
\begin{tikzpicture}[baseline={([yshift=-\the\dimexpr\fontdimen22\textfont2\relax]current bounding box.center)}]
 
\draw[blue,thick] (0.2,0) -- (0.2,2);
\draw (0.5,0) -- (0.5,2);
\draw (1,0) -- (1,2);
\draw (1.5,0) -- (1.5,2);
\dbox{unshaded}{(0.4,0.5)}{.3}{0.1}{0.0}{$f$};
\roundNbox{unshaded}{(1.5,1.5)}{.3}{0.0}{0.0}{$x$};
 
\draw[dashed] (-.2,1) -- (2,1);
\draw[dashed] (-.2,0) rectangle (2,2);
\node at (0.5,-.2) {\tiny{$i$}};
\node at (1,2.2) {\tiny{$k\sm i$}};
\node at (1.5,-.2) {\tiny{$l\sm k$}};
\end{tikzpicture}
\ =\
\begin{tikzpicture}[baseline={([yshift=-\the\dimexpr\fontdimen22\textfont2\relax]current bounding box.center)}]
\draw[blue,thick] (0.2,0) -- (0.2,2);
\draw (0.5,0) -- (0.5,2);
\draw (1,0) -- (1,2);
\draw (1.5,0) -- (1.5,2);
\dbox{unshaded}{(0.4,1.5)}{.3}{0.1}{0.0}{$f$};
\roundNbox{unshaded}{(1.5,0.5)}{.3}{0.0}{0.0}{$x$};
 
\draw[dashed] (-.2,1) -- (2,1);
\draw[dashed] (-.2,0) rectangle (2,2);
\node at (0.5,-.2) {\tiny{$i$}};
\node at (1,2.2) {\tiny{$k\sm i$}};
\node at (1.5,-.2) {\tiny{$l\sm k$}};
\end{tikzpicture}
\]
\end{compactenum}

\subsection{From Markov tower as a standard module to planar module category}\label{MT to planar mod cat}

\subsubsection{Planar module category over planar tensor category}
\begin{definition}
Let $\mathcal{A}_0$ be a planar tensor category defined in \S\ref{planar tensor category}. 
Let $\mathcal{M}_0$ be an indecomposable semisimple ${\rm C}^*$ right $\mathcal{A}_0-$module category with following properties:
\begin{compactenum}[(a)]
\item Object: The objects of $\mathcal{M}_0$ are $[n]=[n]_{\mathcal{M}_0}$, $n\in\mathbb{Z}_{\ge 0}$, where $[0]$ is simple.
\item The tensor product of objects are 
$$[m]_{\mathcal{M}_0}\lhd [n,+]_{\mathcal{A}_0}=[m+n]_{\mathcal{M}_0},\qquad [m]_{\mathcal{M}_0}\lhd[n,-]_{\mathcal{A}_0}=0.$$
\item Only $\mathcal{M}_0([n]\to [n\pm 2i])$ is non-empty, $n,i\in\mathbb{Z}_{\ge0}$. The module product of morphism in $\Hom(\mathcal{M}_0)$ and $\Hom(\mathcal{A}_0)$ should match the shading types.
\item $\mathcal{M}_0$ is a strict right $\mathcal{A}_0-$module category, i.e., the module associator is identity. For $x_1,x_2\in \mathcal{A}_0$ and $f\in \mathcal{M}_0$, 
$$(f\lhd x_1)\lhd x_2=f\lhd (x_1\otimes x_2).$$
\item $\mathcal{M}_0$ is a ${\rm C}^*$ category with a natural dagger structure such that $\lhd$ is a dagger functor, i.e., for $x\in \Hom(\mathcal{A}_0)$ and $f\in \Hom(\mathcal{M}_0)$,
$$(f\lhd x)^\dagger = f^\dagger \lhd x^\dagger.$$
\end{compactenum}

Such module category is called a \textbf{planar module category}.
\end{definition}

\begin{remark}\label{main idea remark endo,domain,range for module}
Similar to Remark \ref{main idea remark endo,domain,range}, the morphisms in $\mathcal{M}_0$ is determined by its representation as an endomorphism and its domain and range.

There is a canonical isomorphism $\phi:\mathcal{M}_0([m]\to [m+2i])\to \mathcal{M}_0([m+i]\to [m+i])$ by using the rigid structure on $\mathcal{A}_0$.
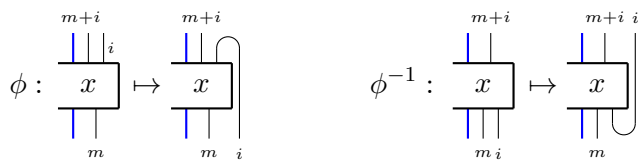
\[\phi:
\begin{tikzpicture}[baseline=-.1cm]
\draw[blue,thick] (-.3,-.7) -- (-.3,.7);
\draw (0,-.7) -- (0,0);
\draw (-.1,0) -- (-.1,.7);
\draw (.1,0) -- (.1,.7);

\dbox{unshaded}{(0,0)}{.3}{.2}{0}{$x$};
\node at (0,-.9) {\tiny{$m$}};
\node at (-.2,.9) {\tiny{$m\+ i$}};
\node at (.2,.5) {\tiny{$i$}};
\end{tikzpicture}
\mapsto
\begin{tikzpicture}[baseline=-.1cm]
\draw[blue,thick] (-.3,-.7) -- (-.3,.7);
\draw (0,-.7) -- (0,0);
\draw (-.1,0) -- (-.1,.7);
\draw (.1,0) -- (.1,.5);
\draw (.1,.5) arc (180:0:.15);
\draw (.4,-.7) -- (.4,.5);

\dbox{unshaded}{(0,0)}{.3}{.2}{0}{$x$};
\node at (0,-.9) {\tiny{$m$}};
\node at (-.1,.9) {\tiny{$m\+ i$}};
\node at (.4,-.9) {\tiny{$i$}};
\end{tikzpicture}
\qquad\qquad \phi^{-1}:\ 
\begin{tikzpicture}[baseline=-.1cm]
\draw[blue,thick] (-.3,-.7) -- (-.3,.7);
\draw (0,0) -- (0,.7);
\draw (-.1,-.7) -- (-.1,0);
\draw (.1,-.7) -- (.1,0);
\dbox{unshaded}{(0,0)}{.3}{0.2}{0}{$x$};
\node at (0,.9) {\tiny{$m\+i$}};
\node at (-.1,-.9) {\tiny{$m$}};
\node at (.1,-.9) {\tiny{$i$}};
\end{tikzpicture}
\mapsto
\begin{tikzpicture}[baseline=-.1cm]
\draw[blue,thick] (-.3,-.7) -- (-.3,.7);
\draw (0,0) -- (0,.7);
\draw (-.1,-.7) -- (-.1,0);
\draw (.1,-.5) -- (.1,0);
\draw (.1,-.5) arc (-180:0:.15);
\draw (.4,-.5) -- (.4,.7);
\dbox{unshaded}{(0,0)}{.3}{0.2}{0}{$x$};
\node at (0,.9) {\tiny{$m\+i$}};
\node at (-.1,-.9) {\tiny{$m$}};
\node at (.4,.9) {\tiny{$i$}};
\end{tikzpicture}
\]
For morphism $x\in \mathcal{M}_0([m],[n])$, we can write a triple $(\phi(x);[m],[n])$ to represent $x$, where $\phi(x)\in \End([\frac{m+n}{2}])$, which is called the \textbf{endormophism representation part} of $x$. In the following context, we simply write $x$ instead of $\phi(x)$ in the triple $(x;[m],[n])$. 
\end{remark}

\subsubsection{From Markov tower as a standard module to planar module category}
\label{Markov tower to planar module cat}

Define the multi-step conditional expectation $E^m_n=E_{n-m+1}\circ \cdots \circ E_n$, for $m\le n$. Similar to Definition \ref{Def:std lambda lattice to planar tensor category}, we may regard the elements in $M_n$ as endomorphisms in the category, we can construct a planar module category from a given Markov tower as a standard module over a standard $\lambda$-lattice.

\begin{definition}\label{Def:planar module category}
Let $M=(M_n)_{n\ge 0}$ be a Markov tower as a standard right module over standard $\lambda$-lattice $A=(A_{i,j})$ with $\dim(M_0)=1$. We define a planar module category $\mathcal{M}_0$ from $M$ as follows.
\begin{compactenum}[(a)]
\item The objects of $\mathcal{M}_0$ are the symbols $[n]$ for $n\in \mathbb{Z}_{\ge 0}$.
\item Given $n\ge 0$, define $\mathcal{M}_0([n]\to [n]):=M_n$.
\item The identity morphism in $\mathcal{M}_0([n]\to [n])$ is $1_{M_{n}}$.
\item For $(f;[m],[n])$ with $2\mid m+n$, we define $(f;[m],[n])^\dagger:=(f^*;[n],[m])$, where $f,f^*\in M_{\frac{m+n}{2}}$.
\item We define composition in three cases.
\begin{compactenum}[{(C}1{)}]
\item $(g;[n+2i],[n+2i+2j])\circ (f;[n],[n+2i]) = (d^i E^i_{n+2i+j}(gfe^n_{j,i});[n],[n+2i+2j])$, where $f\in M_{n+i},g\in M_{n+2i+j}$ and $d^i E^i_{n+2i+j}(gfe^n_{j,i})\in M_{n+i+j}$.
\item $(g;[n+2i+2j],[n+2i])\circ (f;[n],[n+2i+2j]) = (d^iE^{i+j}_{n+2i+j}(gfe^{n,*}_{j,i});[n],[n+2i])$, where $f\in M_{n+i+j}, g\in M_{n+2i+j}$ and $d^iE^{i+j}_{n+2i+j}(gfe^{n,*}_{j,i})\in M_{n+i}$.
\item $(g;[n],[n+2i+2j])\circ (f;[n+2i],[n]) = (d^ige^{n,*}_{j,i}f;[n+2i],[n+2i+2j])$, where $f\in M_{n+i}, g\in M_{n+i+j}$ and $d^ige^{n,*}_{j,i}f\in M_{n+2i+j}$.
\end{compactenum}
For the other cases, we can use the dagger structure $f^\dagger \circ g^\dagger := (g\circ f)^\dagger$ to define.
\end{compactenum}
\end{definition}

Similarly, the composition and the dagger structure are well defined, and $\mathcal{M}_0$ is ${\rm C}^*$ \cite[\S3.4]{CHPS18}.

\begin{center}
\begin{tabular}{ c c c c c }
\begin{tikzpicture}[baseline=0.5cm]
\draw[dashed] (0.9,1) -- (2,1);
\draw[dashed] (.1,0) -- (2,0);
\draw[dashed] (.1,-1) rectangle (2,2);
\draw[thick,blue] (0.1,-1) -- (0.1,2);
\draw (0.4,-1) -- (0.4,2);
\draw (0.6,0) -- (0.6,2);
\draw (1,0) -- (1,2);  
\draw (1.4,0) -- (1.4,2);

\draw (1,-1) arc (180:0:.2cm);
\draw (.6,0) arc (-180:0:.2cm);
\draw (.6,-1)  .. controls ++(90:.35cm) and ++(270:.35cm) .. (1.4,0);

\draw[orange,thick] (1.7,-1) -- (1.7,2);
\draw[orange,thick] (1.4,2) arc (180:0:.15cm);
\draw[orange,thick] (1.4,-1) arc (-180:0:.15cm);

\dbox{unshaded}{(.9,1.5)}{.3}{0.7}{0.4}{$g$};
\dbox{unshaded}{(.5,.5)}{.3}{0.3}{0}{$f$};
  
\node at (0.4,2.2) {\tiny{$n$}};
\node at (0.8,2.2) {\tiny{$i\+j$}};
\node at (.7,1) {\tiny{$i$}};
\node at (.4,-1.2) {\tiny{$n$}};
\node at (.3,1) {\tiny{$n$}};
\node at (.6,-.2) {\tiny{$i$}};
\node at (1,-1.2) {\tiny{$i$}};
\node at (.6,-1.2) {\tiny{$j$}};
\end{tikzpicture} 
& &
\begin{tikzpicture}[baseline=0.5cm]
\draw[dashed] (1,1) -- (2,1);
\draw[dashed] (.1,0) -- (2,0);
\draw[dashed] (.1,-1) rectangle (2,2);
 
\draw[thick,blue] (0.1,-1) -- (0.1,2);
\draw (0.4,-1) -- (0.4,2);
\draw (0.6,0) -- (0.6,2);
\draw (1,0) -- (1,2);  
\draw (1.4,0) -- (1.4,2);

\draw (.6,-1) arc (180:0:.2cm);
\draw (1,0) arc (-180:0:.2cm);
\draw (1.4,-1)  .. controls ++(90:.35cm) and ++(270:.35cm) .. (0.6,0);

\draw[orange,thick] (1.7,-1) -- (1.7,2);
\draw[orange,thick] (1.4,2) arc (180:0:.15cm);
\draw[orange,thick] (1.4,-1) arc (-180:0:.15cm);
 
\draw[orange,thick] (1.8,-1) -- (1.8,2);
\draw[orange,thick] (1,2) arc (180:0:.4cm);
\draw[orange,thick] (1,-1) arc (-180:0:.4cm);

\dbox{unshaded}{(.9,1.5)}{.3}{0.7}{0.4}{$g$};
\dbox{unshaded}{(.7,.5)}{.3}{0.5}{0.2}{$f$};
 
\node at (0.4,2.2) {\tiny{$n$}};
\node at (0.6,2.2) {\tiny{$i$}};
\node at (.8,1) {\tiny{$j\+ i$}};
\node at (.4,-1.2) {\tiny{$n$}};
\node at (.3,1) {\tiny{$n$}};
\node at (.6,-.2) {\tiny{$j$}};
\node at (1.4,-.2) {\tiny{$i$}};
\node at (.6,-1.2) {\tiny{$i$}};
\end{tikzpicture}
& &
\begin{tikzpicture}[baseline=0.5cm]
\draw[dashed] (0.6,1) -- (1.6,1);
\draw[dashed] (.1,0) -- (1.6,0);
\draw[dashed] (.1,-1) rectangle (1.6,2);

\draw[thick,blue] (0.1,-1) -- (0.1,2);
\draw (0.4,-1) -- (0.4,2);
\draw (0.6,0) -- (0.6,-1);
\draw (0.6,1) -- (0.6,2);
\draw (1,1) -- (1,2);  
\draw (1,0) -- (1,-1);  
\draw (1.4,0) -- (1.4,-1);
\draw (1.4,2) -- (1.4,1);

\draw (.6,0) arc (180:0:.2cm);
\draw (1,1) arc (-180:0:.2cm);
\draw (1.4,0) .. controls ++(90:.35cm) and ++(270:.35cm) .. (.6,1);

\dbox{unshaded}{(.7,1.5)}{.3}{0.5}{0.2}{$g$};
\dbox{unshaded}{(.5,-.5)}{.3}{.3}{0}{$f$};
  
\node at (0.4,2.2) {\tiny{$n$}};
\node at (0.8,2.2) {\tiny{$i\+j$}};
\node at (1.4,-1.2) {\tiny{$j$}};
\node at (.4,-1.2) {\tiny{$n$}};
\node at (.3,1) {\tiny{$n$}};
\node at (1.4,2.2) {\tiny{$i$}};
\node at (1,-1.2) {\tiny{$i$}};
\node at (.6,-1.2) {\tiny{$i$}};
\end{tikzpicture}\\
(C1) & & (C2) & & (C3)
\end{tabular}
\end{center}

\begin{remark}
Readers can observe the similarity between the diagrammatic explanation of elements in $M_n$ and $A_{i,n}$, difference only appears on the leftmost. Moreover, the similar version of Lemma \ref{lemma 3} and Lemma \ref{lemma 4} is also true for Markov tower case. 
\end{remark}

Now we define the module action of morphisms. 
\begin{definition}\label{1 otimes x second times}
$f\lhd 1$ and $1\lhd x$, $f\in\Hom(\mathcal{M}_0)$ and $x\in \Hom(\mathcal{A}_0)$. The idea is the same as in Definition \ref{x otimes 1, 1 otimes y}.

First, we define $f\lhd 1$ as 
\begin{center}
\small
\begin{tabular}{ |c|c| } 
 \hline
 $f$ & $f\lhd 1_j$  \\
 \hline
  $(f;[m], [m+2i]),\ i\le j$ & $(fe^m_{j-i,i};[m+j],[m+2i+j])$  \\
\hline
  $(f;[m], [m+2i]),\ i>j$ & $(fe^{m,*}_{i-j,j};[m+j],[m+2i+j])$  \\
\hline
\end{tabular}
\end{center}

The definition of $1\lhd x$ will be the same as $1\otimes x$ by using the 2-shift maps in Definition \ref{x otimes 1, 1 otimes y}.

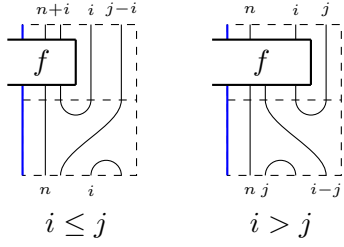
\begin{center}
\begin{tabular}{ c c c }
\begin{tikzpicture}[baseline=1cm]
\draw[dashed] (.1,1) -- (1.6,1);
\draw[dashed] (.1,0) rectangle (1.6,2);

\draw[thick,blue] (0.1,0) -- (0.1,2);
\draw (0.4,0) -- (0.4,2);
\draw (0.6,1) -- (0.6,2);
\draw (1,1) -- (1,2);
\draw (1.4,1) -- (1.4,2);
\dbox{unshaded}{(.5,1.5)}{.3}{0.3}{0.0}{$f$};

\draw (.6,1) arc (-180:0:.2cm);
\draw (1,0) arc (180:0:.2cm);
\draw (.6,0)  .. controls ++(90:.35cm) and ++(270:.35cm) .. (1.4,1);
 
\node at (0.4,-.2) {\tiny{$n$}};
\node at (1.4,2.2) {\tiny{$j\sm i$}};
\node at (1,-.2) {\tiny{$i$}};
\node at (.5,2.2) {\tiny{$n\+i$}};
\node at (1,2.2) {\tiny{$i$}};
\end{tikzpicture}
& &
\begin{tikzpicture}[baseline=1cm]
\draw[dashed] (.1,1) -- (1.6,1);
\draw[dashed] (.1,0) rectangle (1.6,2);

\draw[thick,blue] (0.1,0) -- (0.1,2);
\draw (0.4,0) -- (0.4,2);
\draw (0.6,1) -- (0.6,1.5);
\draw (1,1) -- (1,2);
\draw (1.4,1) -- (1.4,2);
\dbox{unshaded}{(.7,1.5)}{.3}{0.5}{0.2}{$f$};

\draw (1,1) arc (-180:0:.2cm);
\draw (.6,0) arc (180:0:.2cm);
\draw (1.4,0)  .. controls ++(90:.35cm) and ++(270:.35cm) .. (.6,1);
 
\node at (0.4,-.2) {\tiny{$n$}};
\node at (0.6,-.2) {\tiny{$j$}};
\node at (1.4,-.2) {\tiny{$i\sm j$}};
\node at (.4,2.2) {\tiny{$n$}};
\node at (1,2.2) {\tiny{$i$}};
\node at (1.4,2.2) {\tiny{$j$}};
\end{tikzpicture}\\
$i\le j$ & & $i>j$ 
\end{tabular}
\end{center}
\end{definition}

The proof of following propositions are the same as in Proposition \ref{tensor product def prop}, \ref{strictness prop} and \ref{circ tensor 1}. 
\begin{proposition}
For $f\in \Hom(\mathcal{M}_0)$, $x\in \Hom(\mathcal{A}_0)$, $(f\lhd 1)\circ (1\lhd x)=(1\lhd x)\circ (f\lhd 1)$.
\end{proposition}

\begin{definition}
Define $f\lhd x:=(f\lhd 1)\circ (1\lhd x)$.
\end{definition}

The following propositions guarantee the module action defined above is well-defined.

\begin{proposition}
For $f\in \Hom(\mathcal{M}_0)$, $x,y\in \Hom(\mathcal{A}_0)$, $(f\lhd x)\lhd y=f\lhd(x\otimes y)$.
\end{proposition}

\begin{proposition}
For $f,g\in \Hom(\mathcal{M}_0)$, $(f\circ g)\lhd 1=(f\lhd 1)\circ (g\lhd 1)$ and $1\lhd (x\otimes y)=(1\lhd x)\circ (1\lhd y)$.
\end{proposition}

\subsection{Indecomposable semisimple ${\rm C}^*$ $\mathcal{A}-$module categories and planar $\mathcal{A}_0-$module categories}
\subsubsection{Indecomposable semisimple ${\rm C}^*$ $\mathcal{A}-$module category}\label{cat M to M_0}
Let $\mathcal{A}$ be a 2-shaded rigid ${\rm C}^*$ multitensor category with a generator $X=1^+\otimes X\otimes 1^-$ with a canonical unitary dual functor $\overline{(\cdot)}$. Let $\mathcal{M}$ be a Cauchy complete indecomposable semisimple ${\rm C}^*$ $\mathcal{A}-$module category. Note that there is a natural dagger structure on $\mathcal{M}$, and the module action $\lhd$ is a dagger functor, namely, for morphism $f\in\Hom(\mathcal{M})$ and $x\in \Hom(\mathcal{A})$,
$$(f\lhd x)^\dagger = f^\dagger \lhd x^\dagger.$$

We call a module category $\mathcal{M}$ \textbf{indecomposable} if for any two simple objects $P,Q\in\mathcal{M}$, $Q$ is a direct summand of $P\lhd X^{\alt\otimes n}$ if $P=P\lhd 1^+$ ($P\lhd \overline{X}^{\alt\otimes n}$ if $P=P\lhd 1^-$) for some $n\in\mathbb{Z}_{\ge 0}$. 

\begin{construction}
Let $\mathcal{A}_0$ be a planar tensor category obtained from $(\mathcal{A},X)$ via the construction in \S\ref{cat A to A_0}. 
By MacLane's coherence theorem, $\mathcal{M}_{\mathcal{A}}$ is unitary equivalent to a strict one, i.e., $\mathcal{M}$ and $\mathcal{A}$ are strict and the right module associator is trivial. 
Then $\mathcal{M}$ is also a strict right $\mathcal{A}_0-$module category. 

We construct the planar $\mathcal{A}_0-$module category $\mathcal{M}_0$ as follows:
\begin{compactenum}[(a)]
\item Objects: Pick a simple object $Z=Z\lhd 1^+\in \mathcal{M}$, define $[0]:=Z$, and 
$$[n+1]:=[n]\lhd [1,?],$$
where $[1,?]=[1,+]$ if $2\mid n$ and $[1,?]=[1,-]$ if $2\nmid n$.
\item Morphisms: $\mathcal{M}_0$ is a full subcategory of $\mathcal{M}$ with above objects.
\end{compactenum}
\end{construction}

Given $\mathcal{M}_0$ to be a planar $\mathcal{A}_0-$module category, then its Cauchy completion $\widehat{M_0}$ is an $\widehat{\mathcal{A}_0}-$module, compatible with the dagger structure. The proof is left to the reader as an exercise.

\begin{remark}
Suppose $\mathcal{M}_0$ is a planar $\mathcal{A}_0-$module category constructed from $(\mathcal{M},Z)$ over $(\mathcal{A},X)$, then there is a unitary equivalence between $\mathcal{M}$ as $\mathcal{A}-$module and $\widehat{\mathcal{M}_0}$ as $\widehat{\mathcal{A}_0}-$module, which sends base object to base object.
\end{remark}

\subsubsection{From planar module category to Markov tower as a standard module over a standard $\lambda$-lattice}\label{M_0 to MT M}
\begin{construction}
Let $\mathcal{M}_0$ be a planar $\mathcal{A}_0-$module category with modulus $d$ and $A$ is a standard $\lambda$-lattice constructed from $\mathcal{A}_0$ as in \S\ref{A_0 to std A}. Define $M_j=\End([j])$, $j\in\mathbb{Z}_{\ge 0}$. Then we check $M=(M_j)_{j\ge 0}$ to be a Markov tower as a standard $A-$module.
\begin{compactenum}[(a)]
\item The horizontal inclusion $M_j\subset M_{j+1}$ sends $x\in M_j$ to $x\lhd \id_{[1,?]}\in M_{j+1}$, where $?=+$ if $2\mid j$ and $?=-$ if $2\nmid j$. The vertical inclusion $A_{0,j}\subset M_{j}$ sends $x\in A_{0,j}$ to $\id_{[0]}\lhd x\in M_j$.
\item Conditional expectation: Define $E^M_j:M_j\to M_{j-1}$ by
\begin{align*}
    & E^M_{2k}(x)=d^{-1}(\id_{[2k-1]}\otimes \ev_{\overline{[1,+]}})\circ (x\lhd [1,+])\circ (\id_{[2k-1]}\otimes \coev_{[1,+]}),\\
    & E^M_{2k+1}(x)=d^{-1}(\id_{[2k]}\otimes \ev_{\overline{[1,-]}})\circ (x\lhd [1,-])\circ (\id_{[2k]}\otimes \coev_{[1,-]}).
\end{align*}
\item Jones projections: the same Jones projections in $A$ and identify $e_n\in A_{0,n+1}$ with $1\lhd e_n\in M_{n+1}$.
\end{compactenum}
The check that $M$ is a Markov tower and a standard $A-$module is left to the reader. In particular, we have $E_{n+1}(e_n)=d^{-2}\cdot 1$.
\end{construction}

In this section, we show that the class of unitary equivalent pairs $(\mathcal{M},Z)$ with $\mathcal{M}$ an indecomposable right $\mathcal{A}-$module category and $Z$ a simple base point induces the equivalent class of planar module categories; according to \S\ref{Markov tower to planar module cat}, the class of equivalent planar module categories is one to one corresponding to the class of isomorphic Markov towers as standard module over isomorphic standard $\lambda$-lattices.

Combining above discussion, we can deduce the equivalence between $(\mathcal{M},Z)$ as $\mathcal{A}-$module category and Markov tower $M$ as standard $A-$module.
\begin{theorem}There is a bijective correspondence between equivalence classes of the following:
\[\left\{\, \parbox{5.1cm}{\rm 
Traceless Markov tower $M=(M_i)_{i\ge 0}$ with $\dim(M_0)=1$ 
as a standard right module over a standard $\lambda$-lattice $A$} \,\right\}\ \, \cong\ \, \left\{\, \parbox{7cm}{\rm Pairs $(\mathcal{M}, Z)$ with $\mathcal{M}$ an indecomposable semisimple ${\rm C}^*$ right $\mathcal{A}-$module category together with a choice of simple object $Z= Z\lhd 1^+_\mathcal{A}$} \,\right\}\]
Equivalence on the left hand side is $*$-isomorphism of traceless Markov towers as standard $A-$modules; 
equivalence on the right hand side is unitary $\mathcal{A}-$module equivalence on their Cauchy completions which maps the simple base object to simple base object.
\end{theorem}

\begin{corollary}\label{corollary: TLJ module}
Any Markov tower $M$ with modulus $d$ and $\dim(M_0)=1$ is naturally a standard right $\mathrm{TLJ}(d)-$module, where $\mathrm{TLJ}(d)$ is a Temperley-Lieb-Jones standard $\lambda$-lattice as in Example \ref{TLJ lambda lattice}, which corresponds to an indecomposable semisimple ${\rm C}^*$ right $\mathcal{TLJ}(d)-$module category with a simple base object.
\end{corollary}

\begin{remark}
The tracial case will be discussed in \S\ref{Tracial Markov tower}.
\end{remark}

\section{Markov lattices as standard bimodule over two standard $\lambda$-lattices and bimodule categories}
\label{Cpt 3}
In this chapter, we extend the discussion into the bimodule case. We give the notion Markov lattices and Markov lattices as bimodule over two standard $\lambda$-lattices, by using the similar method, which correspond to bimodule categories.

\subsection{Markov lattice and basic properties}
\begin{definition}[Markov lattice]
A tuple $M=(M_{i,j},E^{M,l}_{i,j},E^{M,r}_{i,j},e_i,f_j)_{i,j\ge 0}$ is called a Markov lattice if the following conditions hold.
$$\begin{matrix}
M_{i+1,j} & {\subset} & M_{i+1,j+1}\\
\cup & & \cup \\
M_{i,j} & {\subset} & M_{i,j+1} 
\end{matrix}$$

\begin{compactenum}[(a)]
\item $M_{i,j}\subset M_{i,j+1}$ and $M_{i,j}\subset M_{i+1,j}$ are unital inclusions.
\item $M_{-,j}=(M_{i,j},E^{M,l}_{i,j},e_i)_{i\ge 0}$ are Markov towers with the same modulus $d_0$ and $e_i\in M_{i+1,j}$ for all $j$; $M_{i,-}=(M_{i,j},E^{M,r}_{i,j},f_j)_{j\ge 0}$ are Markov towers with the same modulus $d_1$ and $f_j\in M_{i,j+1}$ for all $i$. We call $M$ of modulus $(d_0,d_1)$.
\item The commuting square condition:
$$\xymatrix@+0.5pc{
M_{i+1,j} \ar[d]_{E^{M,l}_{i+1,j}} &M_{i+1,j+1}  \ar[l]_{E^{M,r}_{i+1,j+1}}\ar[d]^{E^{M,l}_{i+1,j+1}}\\
M_{i,j} & M_{i,j+1}\ar[l]^{E^{M,r}_{i,j+1}}
}$$
is a commuting square, i.e., $E^{M,r}_{i,j+1}\circ E^{M,l}_{i,j}=E^{M,l}_{i,j+1}\circ E^{M,r}_{i+1,j+1}$.
\end{compactenum}
\end{definition}

Here are some properties of Markov lattice.

\begin{proposition} \label{Markov lattice prop 1} Let $M=(M_{i,j},E^{M,l}_{i,j},E^{M,r}_{i,j},e_i,f_j)_{i,j\ge 0}$ be a Markov lattice.
\begin{compactenum}[(1)]
\item $E^{M,r}_{i+1,j+1}(e_i)=e_i$ and $E^{M,l}_{i+1,j+1}(f_j)=f_j$ for each $i,j=1,2,\cdots$.
\item $[f_j,e_i]=0$ for each $i,j=1,2,3,\cdots$.
\end{compactenum}
\end{proposition}
\begin{proof} 
\item[(1)] Note that $e_i\in M_{i+1,j}\subset M_{i+1,j+1}$ and $E^r_{i+1,j+1}:M_{i+1,j+1}\to M_{i+1,j}$ is a conditional expectation, we have $E^r_{i+1,j+1}(e_i)=e_i$. Similarly, $E^{M,l}_{i+1,j+1}(f_j)=f_j$.

\item[(2)] By Proposition \ref{Markov tower properties}(1). 
\end{proof}

\begin{remark}
If there is a faithful normal trace on $\bigcup_{i,j\ge 0}M_{i,j}$ and $E^{M,r}_{i,j},E^{M,l}_{i,j}$ are the canonical faithful normal trace-preserving conditional expectations for $i,j=0,1,2,\cdots$, then $M$ is called a \textbf{tracial Markov lattice}.
\end{remark}

In the rest of this Chapter, we only consider the traceless Markov lattice with $\dim (M_{0,0})=1$ unless stated.

\subsection{Markov lattice as a standard bimodule over two standard $\lambda$-lattices}

\begin{definition}[Markov lattice as a standard bimodule over two standard $\lambda$-lattices]\label{def ML std bimod}
$$\begin{matrix}
\cup & & \cup & & \cup & & \cup & & \cup & & \cup & & \\
A_{3,1} & \subset & A_{3,0} & \subset & M_{3,0} & \subset & M_{3,1} & \subset & M_{3,2} & \subset & M_{3,3} & \subset \\
\cup & & \cup & & \cup & & \cup & & \cup & & \cup & & \\
A_{2,1} & \subset & A_{2,0} & \subset & M_{2,0} & \subset & M_{2,1} & \subset & M_{2,2} & \subset & M_{2,3} & \subset \\
\cup & & \cup & & \cup & & \cup & & \cup & & \cup & & \\
A_{1,1} & \subset & A_{1,0} & \subset & M_{1,0} & \subset & M_{1,1} & \subset & M_{1,2} & \subset & M_{1,3} & \subset \\
 & & \cup & & \cup & & \cup & & \cup & & \cup & & \\
& & A_{0,0} & \subset & M_{0,0} & \subset & M_{0,1} & \subset & M_{0,2} & \subset & M_{0,3} & \subset \\
 & & & & \cup & & \cup & & \cup & & \cup & & \\
& & & & B_{0,0} & \subset & B_{0,1} & \subset & B_{0,2} & \subset & B_{0,3} & \subset\\
 & & & & & & \cup & & \cup & & \cup & & \\
& & & & & & B_{1,1} & \subset & B_{1,2} & \subset & B_{1,3} & \subset\\
\end{matrix}$$

Let $A^{\text{op}}=(A_{i,j})_{0\le j\le i<\infty}$ $B=(B_{i,j})_{0\le i\le j<\infty}$ be two standard $\lambda$-lattices with Jones projection $e_i\in A_{i+1,j}$, $f_j\in B_{i,j+1}$ respectively and compatible conditional expectations.
Here, $A$ and $M$ share the same Jones projections $e_i$; $B$ and $M$ share the same Jones projections $f_j$. 
(Warning: here we use the opposite $\lambda$-lattice $A^{\text{op}}$, see Definition \ref{Def:opposite Aop})

Let $M=(M_{i,j}, e_i,f_j)_{i,j\ge 0}$ be a Markov lattice with conditional expectation $E^{M,r},E^{M,l}$. We call a Markov lattice $M$ a standard $A-B$ bimodule where the left action is the opposite action, if it satisfies the following three conditions.
\begin{compactenum}[(a)]
\item $A_{i,0}\subset M_{i,0}$, $B_{0,j}\subset M_{0,j}$ are unital inclusions, $i,j=0,1,2,\cdots$.
\item $E^{M,l}_{i,0}|_{A_{i,0}}=E^{A,l}_{i,0}$, $E^{M,r}_{0,j}|_{B_{0,j}}=E^{B,r}_{0,j}$ $i=1,2,\cdots$.
\item (standard condition) $[M_{i,j},A_{p,q}]=0$ for $i\le q\le p$; $[M_{i,j},B_{k,l}]=0$, for $j\le k\le l$.
\end{compactenum}
\end{definition}

\begin{remark}\label{everything commute, for bimodule proof}
The standard condition implies that $[A_{p,q},B_{k,l}]=0$ for all $q\le p,k\le l$ since $A_{p,q}\subset A_{p,0}\subset M_{p,0}$ and $B_{k,l}\subset B_{0,l}\subset M_{0,l}$. Moreover, $E^{M,r}_{i,j}|_{A_{k,l}}=\id$, $E^{M,l}_{i,j}|_{B_{k,l}}=\id$. In particular, we have $E^{M,r}_{i,j}(e_k)=e_k$, $E^{M,l}_{i,j}(f_l)=f_l$ for Jones projections.
\end{remark}

\subsection{String diagram explanation}
We now provide the string diagram explanation of the element, conditional expectation, Jones projection and their relations in a Markov lattice with the same spirit in \S\ref{TLJ diagram explanation, module}.
\begin{compactenum}[({ML}1)]
\item Element $x\in M_{i,j}$:
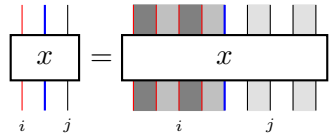
\[
\begin{tikzpicture}[baseline=-.1cm]

\draw[thick,blue] (0,-.7) -- (0,.7);
\draw (.3,-.7) -- (.3,.7);
\draw[red] (-.3,-.7) -- (-.3,.7);
\nbox{unshaded}{(0,0)}{.3}{.15}{.15}{$x$};

\node at (-.3,-.9) {\tiny{$i$}};
\node at (.3,-.9) {\tiny{$j$}};
\end{tikzpicture}
=
\begin{tikzpicture}[baseline=-.1cm]
\filldraw[gray1] (.3,-.7) rectangle (.6,.7);
\filldraw[gray1] (.9,-.7) rectangle (1.2,.7);

\filldraw[gray2] (0,-.7) rectangle (-.3,.7);
\filldraw[gray2] (-.6,-.7) rectangle (-.9,.7);
\filldraw[gray3] (-.3,-.7) rectangle (-.6,.7);
\filldraw[gray3] (-.9,-.7) rectangle (-1.2,.7);

\draw[thick,blue] (0,-.7) -- (0,.7);

\draw (.3,-.7) -- (.3,.7);
\draw (.6,-.7) -- (.6,.7);
\draw (.9,-.7) -- (.9,.7);
\draw (1.2,-.7) -- (1.2,.7);

\draw[red] (-.3,-.7) -- (-.3,.7);
\draw[red] (-.6,-.7) -- (-.6,.7);
\draw[red] (-.9,-.7) -- (-.9,.7);
\draw[red] (-1.2,-.7) -- (-1.2,.7);

\nbox{unshaded}{(0,0)}{.3}{1.05}{1.05}{$x$};

\node at (-.6,-.9) {\tiny{$i$}};
\node at (.6,-.9) {\tiny{$j$}};
\end{tikzpicture}
\]
\item Horizontal inclusion $x\in M_{i,j}\subset M_{i,j+1}$ and $x\in A_{i,0}\subset M_{i,j}$:
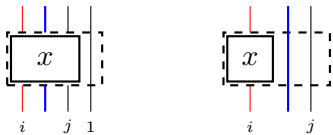
\[
\begin{tikzpicture}[baseline=-.1cm]

\draw[thick,blue] (0,-.7) -- (0,.7);
\draw (.3,-.7) -- (.3,.7);
\draw[red] (-.3,-.7) -- (-.3,.7);

\nbox{unshaded,dashed}{(0,0)}{.35}{.15}{.4}{};
\draw (.6,-.7) -- (.6,.7);
\nbox{unshaded}{(0,0)}{.3}{.15}{.15}{$x$};

\node at (-.3,-.9) {\tiny{$i$}};
\node at (.3,-.9) {\tiny{$j$}};
\node at (.6,-.9) {\tiny{$1$}};
\end{tikzpicture}
\qquad\qquad
\begin{tikzpicture}[baseline=-.1cm]
\draw[red] (-.5,-.7) -- (-.5,.7);

\nbox{unshaded,dashed}{(-.5,0)}{.35}{0}{.7}{};
\draw[thick,blue] (0,-.7) -- (0,.7);
\draw (.3,-.7) -- (.3,.7);

\nbox{unshaded}{(-.5,0)}{.3}{0}{0}{$x$};

\node at (-.5,-.9) {\tiny{$i$}};
\node at (.3,-.9) {\tiny{$j$}};
\end{tikzpicture}
\]
\item Vertical inclusion $x\in M_{i,j}\subset M_{i+1,j}$ and $x\in B_{0,j}\subset M_{i,j}$:
\[
\begin{tikzpicture}[baseline=-.1cm]
\draw[thick,blue] (0,-.7) -- (0,.7);
\draw (.3,-.7) -- (.3,.7);
\draw[red] (-.3,-.7) -- (-.3,.7);

\nbox{unshaded,dashed}{(0,0)}{.35}{.4}{.15}{};
\draw[red] (-.6,-.7) -- (-.6,.7);
\nbox{unshaded}{(0,0)}{.3}{.15}{.15}{$x$};

\node at (-.3,-.9) {\tiny{$i$}};
\node at (-.6,-.9) {\tiny{$1$}};
\node at (.3,-.9) {\tiny{$j$}};
\end{tikzpicture}
\qquad\qquad
\begin{tikzpicture}[baseline=-.1cm]
\draw (.5,-.7) -- (.5,.7);

\nbox{unshaded,dashed}{(.5,0)}{.35}{.7}{0}{};
\draw[thick,blue] (0,-.7) -- (0,.7);
\draw[red] (-.3,-.7) -- (-.3,.7);
\nbox{unshaded}{(.5,0)}{.3}{0}{0}{$x$};

\node at (-.3,-.9) {\tiny{$i$}};
\node at (.5,-.9) {\tiny{$j$}};
\end{tikzpicture}
\]
\item Horizontal conditional expectation $E^{M,r}_{i,j}:M_{i,j}\to M_{i,j-1}$ and $E^{M,r}_{i,j}|_{B_{0,j}}=E^{B,r}_{0,j}$:
\[
E^{M,r}_{i,j}(x)=d_1^{-1}
\begin{tikzpicture}[baseline=-.1cm]
\draw[thick,blue] (0,-.7) -- (0,.7);
\draw (.3,-.7) -- (.3,.7);
\draw (.6,-.5) -- (.6,.5);
\draw[red] (-.3,-.7) -- (-.3,.7);

\nbox{unshaded}{(0,0)}{.3}{.15}{.45}{$x$};

\draw[orange,thick] (.6,.5) arc (180:0:.15);
\draw[orange,thick] (.6,-.5) arc (-180:0:.15);
\draw[orange,thick] (.9,-.5) -- (.9,.5);

\node at (-.3,-.9) {\tiny{$i$}};
\node at (.3,-.9) {\tiny{$j\sm 1$}};
\node at (.6,.8) {\tiny{$1$}};
\end{tikzpicture}
\ ,x\in M_{i,j}
\qquad
E^{M,r}_{i,j}(x)=E^{B,r}_{0,j}(x)=d_1^{-1}
\begin{tikzpicture}[baseline=-.1cm]

\draw[thick,blue] (0,-.7) -- (0,.7);
\draw (.3,-.7) -- (.3,.7);
\draw (.6,-.5) -- (.6,.5);
\draw[red] (-.3,-.7) -- (-.3,.7);

\nbox{unshaded}{(.45,0)}{.3}{0}{0}{$x$};

\draw[orange,thick] (.6,.5) arc (180:0:.15);
\draw[orange,thick] (.6,-.5) arc (-180:0:.15);
\draw[orange,thick] (.9,-.5) -- (.9,.5);

\node at (-.3,-.9) {\tiny{$i$}};
\node at (.3,-.9) {\tiny{$j\sm 1$}};
\node at (.6,.8) {\tiny{$1$}};
\end{tikzpicture}
\ ,x\in B_{0,j}
\]
\item Vertical conditional expectation $E^{M,l}_{i,j}:M_{i,j}\to M_{i-1,j}$ and $E^{M,l}_{i,j}|_{A_{i,0}}=E^{A,l}_{i,0}$:

\[
E^{M,l}_{i,j}(x)=d_0^{-1}
\begin{tikzpicture}[baseline=-.1cm]
\draw[thick,blue] (0,-.7) -- (0,.7);
\draw (.3,-.7) -- (.3,.7);
\draw[red] (-.6,-.5) -- (-.6,.5);
\draw[red] (-.3,-.7) -- (-.3,.7);

\nbox{unshaded}{(0,0)}{.3}{.45}{.15}{$x$};

\draw[green1,thick] (-.6,.5) arc (0:180:.15);
\draw[green1,thick] (-.6,-.5) arc (0:-180:.15);
\draw[green1,thick] (-.9,-.5) -- (-.9,.5);

\node at (-.3,-.9) {\tiny{$i\sm 1$}};
\node at (.3,-.9) {\tiny{$j$}};
\node at (-.6,.8) {\tiny{$1$}};
\end{tikzpicture}
\ ,x\in M_{i,j}
\qquad
E^{M,l}_{i,j}(x)=E^{A,l}_{i,0}(x)=d_0^{-1}
\begin{tikzpicture}[baseline=-.1cm]
\draw[thick,blue] (0,-.7) -- (0,.7);
\draw (.3,-.7) -- (.3,.7);
\draw[red] (-.6,-.5) -- (-.6,.5);
\draw[red] (-.3,-.7) -- (-.3,.7);

\nbox{unshaded}{(-.45,0)}{.3}{0}{0}{$x$};

\draw[green1,thick] (-.6,.5) arc (0:180:.15);
\draw[green1,thick] (-.6,-.5) arc (0:-180:.15);
\draw[green1,thick] (-.9,-.5) -- (-.9,.5);

\node at (-.3,-.9) {\tiny{$i\sm 1$}};
\node at (.3,-.9) {\tiny{$j$}};
\node at (-.6,.8) {\tiny{$1$}};
\end{tikzpicture}
,x\in A_{i,0}
\]
\item Commuting square of conditional expectation $E^{M,r}_{i,j+1}\circ E^{M,l}_{i,j}=E^{M,l}_{i,j+1}\circ E^{M,r}_{i+1,j+1}:M_{i+1,j+1}\to M_{i,j}$, $x\in M_{i+1,j+1}$:
\[
E^{M,r}_{i,j+1}\circ E^{M,l}_{i,j}(x)=E^{M,l}_{i,j+1}\circ E^{M,r}_{i+1,j+1}(x)=d_0^{-1}d_1^{-1}
\begin{tikzpicture}[baseline=-.1cm]
\draw[thick,blue] (0,-.7) -- (0,.7);
\draw (.3,-.7) -- (.3,.7);
\draw (.6,-.5) -- (.6,.5);
\draw[red] (-.3,-.7) -- (-.3,.7);
\draw[red] (-.6,-.5) -- (-.6,.5);

\nbox{unshaded}{(0,0)}{.3}{.45}{.45}{$x$};

\draw[orange,thick] (.6,.5) arc (180:0:.15);
\draw[orange,thick] (.6,-.5) arc (-180:0:.15);
\draw[orange,thick] (.9,-.5) -- (.9,.5);

\draw[green1,thick] (-.6,.5) arc (0:180:.15);
\draw[green1,thick] (-.6,-.5) arc (0:-180:.15);
\draw[green1,thick] (-.9,-.5) -- (-.9,.5);

\node at (-.3,-.9) {\tiny{$i$}};
\node at (.3,-.9) {\tiny{$j$}};
\node at (.6,-.9) {\tiny{$1$}};
\node at (-.6,-.9) {\tiny{$1$}};
\end{tikzpicture}
\]
\item Horizontal Jones projections $f_j\in M_{i,j+1}$ and vertical Jones projections $e_i\in M_{i+1,j}$: \small
\[\hspace*{-1.8cm} f_{2j+1}= d_1^{-1}
\begin{tikzpicture}[baseline=.4cm]
\filldraw[shaded] (.6,1) arc (-180:0:.3cm);
\filldraw[shaded] (.6,0) arc (180:0:.3cm);

\draw[blue,thick] (0,0) -- (0,1);
\draw (.3,0) -- (.3,1);
\draw[red] (-.3,0) -- (-.3,1);

\draw[dashed] (-.6,0) rectangle (1.5,1);
\node at (.3,-.2) {\tiny{$2j$}};
\node at (-.3,-.2) {\tiny{$i$}};
\end{tikzpicture}
\quad
f_{2j+2} = d_1^{-1}
\begin{tikzpicture}[baseline=.4cm]
\fill[shaded] (0.3,0) rectangle (1.5,1);
\filldraw[unshaded] (.6,1) arc (-180:0:.3cm);
\filldraw[unshaded] (.6,0) arc (180:0:.3cm);

\draw[blue,thick] (0,0) -- (0,1);
\draw (.3,0) -- (.3,1);
\draw[red] (-.3,0) -- (-.3,1);

\draw[dashed] (-.6,0) rectangle (1.5,1);
\node at (.3,-.2) {\tiny{$2j\+1$}};
\node at (-.3,-.2) {\tiny{$i$}};
\end{tikzpicture}
\quad
e_{2i+1}= d_0^{-1}
\begin{tikzpicture}[baseline=.4cm]
\fill[gray2] (-1.5,0) rectangle (-.3,1);
\filldraw[gray3] (-.6,1) arc (0:-180:.3cm);
\filldraw[gray3] (-.6,0) arc (0:180:.3cm);

\draw[red] (-.6,1) arc (0:-180:.3cm);
\draw[red] (-.6,0) arc (0:180:.3cm);

\draw[blue,thick] (0,0) -- (0,1);
\draw (.3,0) -- (.3,1);
\draw[red] (-.3,0) -- (-.3,1);

\draw[dashed] (-1.5,0) rectangle (.6,1);
\node at (-.3,-.2) {\tiny{$2i$}};
\node at (.3,-.2) {\tiny{$j$}};
\end{tikzpicture}
\quad
e_{2i+2} = d_0^{-1}
\begin{tikzpicture}[baseline=.4cm]
\fill[gray3] (-1.5,0) rectangle (-.3,1);
\filldraw[gray2] (-.6,1) arc (0:-180:.3cm);
\filldraw[gray2] (-.6,0) arc (0:180:.3cm);

\draw[red] (-.6,1) arc (0:-180:.3cm);
\draw[red] (-.6,0) arc (0:180:.3cm);

\draw[blue,thick] (0,0) -- (0,1);
\draw (.3,0) -- (.3,1);
\draw[red] (-.3,0) -- (-.3,1);

\draw[dashed] (-1.5,0) rectangle (.6,1);
\node at (-.3,-.2) {\tiny{$2i\+1$}};
\node at (.3,-.2) {\tiny{$j$}};
\end{tikzpicture}
\]\normalsize
\item Standard condition: 
\begin{compactenum}[$\bullet$]
\item $[M_{i,j},A_{p,q}]=0$ for $i\le q\le p$. For $g\in M_{i,j},x\in A_{p,q}$, regard them as elements in $M_{p,j}$, then $g x=x g$;
\item $[M_{i,j},B_{k,l}]=0$, for $j\le k\le l$. For $g\in M_{i,j},y\in B_{k,l}$, regard them as elements in $M_{i,l}$, then $g y=y g$:
\end{compactenum}

\[
\begin{tikzpicture}[baseline=.9cm]
\draw[thick,blue] (0,0) -- (0,2);
\draw[red] (-.3,0) -- (-.3,2);
\draw (.3,0) -- (.3,2);
\draw[red] (-.6,0) -- (-.6,2);
\draw[red] (-1.1,0) -- (-1.1,2);

\nbox{unshaded}{(0,.5)}{.3}{.15}{.15}{$g$};
\roundNbox{unshaded}{(-1.1,1.5)}{.3}{0.0}{0.0}{$x$};

\draw[dashed] (-1.6,1) -- (.6,1);
\draw[dashed] (-1.6,0) rectangle (.6,2);

\node at (-.3,-.2) {\tiny{$i$}};
\node at (.3,-.2) {\tiny{$j$}};
\node at (-.6,2.2) {\tiny{$q\sm i$}};
\node at (-1.1,-.2) {\tiny{$p\sm q$}};
\end{tikzpicture}
=
\begin{tikzpicture}[baseline=.9cm]
\draw[thick,blue] (0,0) -- (0,2);
\draw[red] (-.3,0) -- (-.3,2);
\draw (.3,0) -- (.3,2);
\draw[red] (-.6,0) -- (-.6,2);
\draw[red] (-1.1,0) -- (-1.1,2);

\nbox{unshaded}{(0,1.5)}{.3}{.15}{.15}{$g$};
\roundNbox{unshaded}{(-1.1,.5)}{.3}{0.0}{0.0}{$x$};

\draw[dashed] (-1.6,1) -- (.6,1);
\draw[dashed] (-1.6,0) rectangle (.6,2);

\node at (-.3,-.2) {\tiny{$i$}};
\node at (.3,-.2) {\tiny{$j$}};
\node at (-.6,2.2) {\tiny{$q\sm i$}};
\node at (-1.1,-.2) {\tiny{$p\sm q$}};
\end{tikzpicture}
\qquad\qquad
\begin{tikzpicture}[baseline=.9cm]
\draw[thick,blue] (0,0) -- (0,2);
\draw[red] (-.3,0) -- (-.3,2);
\draw (.3,0) -- (.3,2);
\draw (.6,0) -- (.6,2);
\draw (1.1,0) -- (1.1,2);

\nbox{unshaded}{(0,.5)}{.3}{.15}{.15}{$g$};
\roundNbox{unshaded}{(1.1,1.5)}{.3}{0.0}{0.0}{$y$};

\draw[dashed] (-.6,1) -- (1.6,1);
\draw[dashed] (-.6,0) rectangle (1.6,2);

\node at (-.3,-.2) {\tiny{$i$}};
\node at (.3,-.2) {\tiny{$j$}};
\node at (.6,2.2) {\tiny{$k\sm j$}};
\node at (1.1,-.2) {\tiny{$l\sm k$}};
\end{tikzpicture}
=
\begin{tikzpicture}[baseline=.9cm]
\draw[thick,blue] (0,0) -- (0,2);
\draw[red] (-.3,0) -- (-.3,2);
\draw (.3,0) -- (.3,2);
\draw (.6,0) -- (.6,2);
\draw (1.1,0) -- (1.1,2);

\nbox{unshaded}{(0,1.5)}{.3}{.15}{.15}{$g$};
\roundNbox{unshaded}{(1.1,.5)}{.3}{0.0}{0.0}{$y$};

\draw[dashed] (-.6,1) -- (1.6,1);
\draw[dashed] (-.6,0) rectangle (1.6,2);

\node at (-.3,-.2) {\tiny{$i$}};
\node at (.3,-.2) {\tiny{$j$}};
\node at (.6,2.2) {\tiny{$k\sm j$}};
\node at (1.1,-.2) {\tiny{$l\sm k$}};
\end{tikzpicture}
\]
\end{compactenum}

\subsection{From Markov lattice as standard bimodule to planar bimodule category}
\subsubsection{Planar bimodule category}

Let $\mathcal{A}_0$ and $\mathcal{B}_0$ be planar tensor categories. Let $\mathcal{M}_0$ be a ${\rm C}^*$ $\mathcal{A}_0-\mathcal{B}_0$ bimodule category with following properties:
\begin{compactenum}[(a)]
\item Object: The objects of $\mathcal{M}_0$ are $[m,n]=[m,n]_{\mathcal{M}_0}$, $m,n\in\mathbb{Z}_{\ge 0}$, where $[0,0]:=1_{\mathcal{M}_0}$ is simple.
\item The module tensor product of objects are
\begin{align*}
    [i,+]_{\mathcal{A}_0}\rhd [m,n]_{\mathcal{M}_0}=[m+i,n]_{\mathcal{M}_0},&\qquad [i,-]_{\mathcal{A}_0}\rhd [m,n]_{\mathcal{M}_0}=\text{none}\\
    [m,n]_{\mathcal{M}_0}\lhd [j,+]_{\mathcal{B}_0}=[i,n+j]_{\mathcal{M}_0},&\qquad
    [m,n]_{\mathcal{M}_0}\lhd [j,-]_{\mathcal{B}_0}=\text{none}\\
    ([i,+]_{\mathcal{A}_0}\rhd [m,n]_{\mathcal{M}_0})\lhd [j,+]_{\mathcal{B}_0} = [m+i,& n+j]_{\mathcal{M}_0} = [i,+]_{\mathcal{A}_0}\rhd( [m,n]_{\mathcal{M}_0}\lhd [j,+]_{\mathcal{B}_0})
\end{align*}
\item Only $\mathcal{M}_0([m,n]\to [m\pm 2i,n\pm 2j])$ is non-empty, $m,n,i,j\in \mathbb{Z}_{\ge 0}$. The module tensor product of morphisms in $\Hom(\mathcal{A}_0)$, $\Hom(\mathcal{M}_0)$ and $\Hom(\mathcal{M}_0)$ should match the shading types. 
\item $\mathcal{M}_0$ is a strict $\mathcal{A}_0-\mathcal{B}_0$ bimodule category, i.e., the left/right module associator and bimodule associator are trivial. For $x,x_1,x_2\in \Hom(\mathcal{A}_0)$, $g\in \Hom(\mathcal{M}_0)$ and $y,y_1,y_2\in \Hom(\mathcal{B}_0)$,
\begin{align*}
    & x_2\rhd(x_1\rhd g)=(x_2\otimes x_1)\rhd g \qquad (g\lhd y_1)\lhd y_2 = g\lhd (y_1\otimes y_2) \\
    & (x\rhd g)\lhd y = x\rhd (g\lhd y).
\end{align*}
\item $\mathcal{M}_0$ is a ${\rm C}^*$ category with a natural dagger structure such that $\lhd$ and $\rhd$ are dagger functors, i.e., for $x\in \Hom(\mathcal{A}_0),g\in\Hom(\mathcal{M}_0)$ and $y\in\Hom(\mathcal{B}_0)$,
$$(x\rhd g\lhd y)^\dagger =x^\dagger \rhd g^\dagger \lhd y^\dagger.$$
\end{compactenum}

Such bimodule category is called a \textbf{planar bimodule category}.

\begin{remark}
As in Remark \ref{main idea remark endo,domain,range for module}, the morphisms in $\mathcal{M}_0$ is determined by its representation as an endomorphism and its domain and range.

There is a canonical isomorphism $\phi:\mathcal{M}_0([m,n]\to [m+2i,n+2j])\to \mathcal{M}_0([m+i,n+j]\to [m+i,n+j])$ by using the rigid structure on $\mathcal{A}_0$ and $\mathcal{B}_0$.
\[
\begin{tikzpicture}[baseline=-.1cm]
\draw[thick,blue] (0,-.7) -- (0,.7);
\draw (.3,-.7) -- (.3,0);
\draw (.4,.7) -- (.4,0);
\draw (.2,.7) -- (.2,0);
\draw[red] (-.2,.7) -- (-.2,0);
\draw[red] (-.3,-.7) -- (-.3,0);
\draw[red] (-.4,.7) -- (-.4,0);
\nbox{unshaded}{(0,0)}{.3}{.3}{.3}{$x$};

\node at (-1,0) {$\phi:$};

\node at (-.3,.9) {\tiny{$m\+ i$}};
\node at (.25,.9) {\tiny{$n\+ j$}};
\node at (-.3,-.9) {\tiny{$m$}};
\node at (.3,-.9) {\tiny{$n$}};
\node at (-.5,.5) {\tiny{$i$}};
\node at (.5,.5) {\tiny{$j$}};
\end{tikzpicture}
\mapsto
\begin{tikzpicture}[baseline=-.1cm]
\draw[thick,blue] (0,-.7) -- (0,.7);
\draw (.3,-.7) -- (.3,0);
\draw (.4,.5) -- (.4,0);
\draw (.2,.7) -- (.2,0);
\draw[red] (-.2,.7) -- (-.2,0);
\draw[red] (-.3,-.7) -- (-.3,0);
\draw[red] (-.4,.5) -- (-.4,0);

\draw (.4,.5) arc (180:0:.15);
\draw (.7,.5) -- (.7,-.7);
\draw[red] (-.4,.5) arc (0:180:.15);
\draw[red] (-.7,.5) -- (-.7,-.7);

\nbox{unshaded}{(0,0)}{.3}{.3}{.3}{$x$};

\node at (-.3,.9) {\tiny{$m\+ i$}};
\node at (.25,.9) {\tiny{$n\+ j$}};
\node at (-.3,-.9) {\tiny{$m$}};
\node at (-.7,-.9) {\tiny{$i$}};
\node at (.7,-.9) {\tiny{$j$}};
\node at (.3,-.9) {\tiny{$n$}};
\end{tikzpicture}
\]
\end{remark}

\begin{remark}
Let $\mathcal{M}_0$ and $\mathcal{N}_0$ be planar bimodule categories over the same planar tensor category. If they are unitary monoidal equivalent, then they are unitary isomorphic.
\end{remark}

\subsubsection{From Markov lattice as standard bimodule to planar bimodule category}
Use the similar notion as we define the planar module category in Definition \ref{Def:planar module category}. 

Define the multi-step conditional expectations $E^{l,i}_{m,n}:=E^{M,l}_{m-i+1,n}\circ\cdots\circ E^{M,l}_{m,n}$ and $E^{r,k}_{m,n}:=E^{M,r}_{m,n-k+1}\circ\cdots\circ E^{M,r}_{m,n}$. 

\begin{definition}
Let $A,B$ be standard $\lambda$-lattices and $M=(M_{m,n})_{m,n\ge 0}$ be a Markov lattice as a standard $A-B$ bimodule with $\dim(M_{0,0})=1$. We define a planar bimodule category $\mathcal{M}_0$ from $M$ as follows.
\begin{compactenum}[(a)]
\item The objects of $\mathcal{M}_0$ are the symbols $[m,n]$ for $m,n\in\mathbb{Z}_{\ge 0}$.
\item Given $m,n\ge 0$, define $\mathcal{M}_0([m,n]\to [m,n]):=M_{m,n}$.
\item The identity morphism in $\mathcal{M}_0([m,n]\to [m,n])$ is $1_{M_{m,n}}$.
\item For $(f;[m_1,n_1],[m_2,n_2])$ with $2\mid m_1+m_2$ and $2\mid n_1+n_2$, define $(f;[m_1,n_1],[m_2,n_2])^\dagger = (f^*[m_2,n_2],[m_1,n_1])$, where $f,f^*\in M_{\frac{m_1+m_2}{2},\frac{n_1+n_2}{2}}$.
\item Define the composition in nine cases.
\begin{compactenum}[{(C1}1{)}]
\item $(h;[m+2i,n+2k],[m+2i+2j,n+2k+2t])\circ (g;[m,n],[m+2i,n+2k])\\ = (d_0^id_1^k E^{l,i}_{m+2i+j,n+k+t}(E^{r,k}_{m+2i+j,n+2k+t}(hgf^n_{t,k}e^m_{j,i}));[m,n],[m+2i+2j,n+2k+2t])$, where $g\in M_{m+i,n+k}$, $h\in M_{m+2i+j,n+2k+t}$ and $d_0^id_1^k E^{l,i}_{m+2i+j,n+k+t}(E^{r,k}_{m+i+j,n+2k+t}(hgf^n_{t,k}e^m_{j,i}))\in M_{m+i+j,n+k+t}$.
\item $(h;[m+2i,n+2k+2t],[m+2i+2j,n+2k])\circ (g;[m,n],[m+2i,n+2k+2t])\\ =(d_0^id_1^k E^{l,i}_{m+2i+j,n+k}(E^{r,k+t}_{m+2i+j,n+2k+t}(hgf^{n,*}_{t,k}e^m_{j,i}));[m,n],[m+2i+2j,n+2k])$, where $g\in M_{m+i,n+k+t}, h\in M_{m+2i+j,n+2k+t}$ and $d_0^id_1^k E^{l,i}_{m+2i+j,n+k}(E^{r,k+t}_{m+2i+j,n+2k+t}(hgf^{n,*}_{t,k}e^m_{j,i}))\in M_{m+i+j,n+k}$.
\item $(h;[m+2i,n][m+2i+2j,n+2k+2t])\circ (g;[m,n+2k],[m+2i,n])\\ = (d_0^id_1^k E^{l,i}_{m+2i+j,n+2k+l}(hf^{n,*}_{l,k}ge^m_{j,i});[m,n+2k],[m+2i+2j,n+2k+2t])$, where $g\in M_{m+i,n+k}, h\in M_{m+2i+j,n+k+t}$ and $d_0^id_1^k E^{l,i}_{m+2i+j,n+2k+l}(hf^{n,*}_{l,k}ge^m_{j,i}) \in M_{m+i+j,n+2k+t}$.
\end{compactenum}
\begin{compactenum}[{(C2}1{)}]
\item $(h;[m+2i+2j,n+2k],[m+2i,n+2k+2t])\circ (g;[m,n],[m+2i+2j,n+2k])\\ = (d_0^id_1^k E^{l,i+j}_{m+2i+j,n+k+t}(E^{r,k}_{m+2i+j,n+2k+t}(hgf^n_{t,k}e^{m,*}_{j,i}));[m,n],[m+2i,n+2k+2t])$, where $g\in M_{m+i+j,n+k}, h\in M_{m+2i+j,n+2k+t}$ and $d_0^id_1^k E^{l,i+j}_{m+2i+j,n+k+t}(E^{r,k}_{m+2i+j,n+2k+t}(hgf^n_{t,k}e^{m,*}_{j,i}))\in M_{m+i,n+2k+t}$.
\item $(h;[m+2i+2j,n+2k+2t],[m+2i,n+2k])\circ (g;[m,n],[m+2i+2j,n+2k+2t])\\ = (d_0^id_1^k E^{l,i+j}_{m+2i+j,n+k}(E^{r,k+t}_{m+2i+j,n+2k+t}(hgf^{n,*}_{t,k}e^{m,*}_{j,i}));[m,n],[m+2i,n+2k])$, where $g\in M_{m+i+j,n+k+t}, h\in M_{m+2i+j,n+2k+t}$ and $d_0^id_1^k E^{l,i+j}_{m+2i+j,n+k}(E^{r,k+t}_{m+2i+j,n+2k+t}(hgf^{n,*}_{t,k}e^{m,*}_{j,i}))\in M_{m+i,n+k}$.
\item $(h;[m+2i+2j,n],[m+2i,n+2k+2t])\circ (g;[m,n+2k],[m+2i+2j,n])\\ = (d_0^if_2^k E^{l,i+j}_{m+2i+j,n+2k+t}(hf^{n,*}_{t,k}ge^{m,*}_{j,i});[m,n+2k],[m+2i,n+2k+2t])$, where $g\in M_{m+i+j,n+k}, h\in M_{m+2i+j,n+k+t}$ and $d_0^if_2^k E^{l,i+j}_{m+2i+j,n+2k+t}(hf^{n,*}_{t,k}ge^{m,*}_{j,i})\in M_{m+i,n+2k+t}$.
\end{compactenum}
\begin{compactenum}[{(C3}1{)}]
\item $(h;[m,n+2k],[m+2i+2j,n+2k+2t])\circ (g;[m+2i,n],[m,n+2k])\\ = (d_0^id_1^k E^{r,k}_{m+2i+j,n+2k+t}(he^{m,*}_{j,i}gf^n_{t,k});[m+2i,n],[m+2i+2j,n+2k+2t])$, where $g\in M_{m+i,n+k}, h\in M_{m+i+j,n+2k+t}$ and $d_0^id_1^k E^{r,k}_{m+2i+j,n+2k+t}(he^{m,*}_{j,i}gf^n_{t,k})\in M_{m+2i+j,n+k+t}$.
\item $(h;[m,n+2k+2t],[m+2i+2j,n+2k])\circ (g;[m+2i,n],[m,n+2k+2t])\\ = (d_0^id_1^k E^{r,k+t}_{m+2i+j,n+2k+t}(he^{m,*}_{j,i}gf^{n,*}_{t,k});[m+2i,n],[m+2i+2j,n+2k])$, where $g\in M_{m+i,n+k+t}, h\in M_{m+i+j,n+2k+t}$ and $d_0^id_1^k E^{r,k+t}_{m+2i+j,n+2k+t}(he^{m,*}_{j,i}gf^{n,*}_{t,k})\in M_{m+2i+j,n+k}$.
\item $(h;[m,n],[m+2i+2j,n+2k+2t])\circ (g;[m+2i,n+2k],[m,n])\\ = (d_0^id_1^k hf^{n,*}_{t,k}e^{m,*}_{j,i}g;[m+2i,n+2k],[m+2i+2j,n+2k+2t])$, where $g\in M_{m+i,n+k}, h\in M_{m+i+j,n+k+t}$ and $d_0^id_1^k hf^{n,*}_{t,k}e^{m,*}_{j,i}g\in M_{m+2i+j,n+2k+t}$.
\end{compactenum}
For the other cases, we can use the dagger structure $g^\dagger \circ h^\dagger := (h\circ g)^\dagger$ to define.
\end{compactenum}
\end{definition}

Similarly, we use the string diagrams to explain the composition.
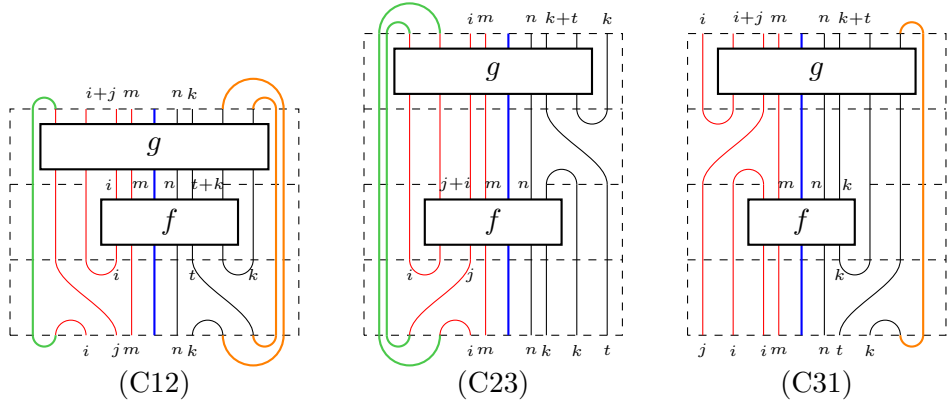
\begin{center}
\begin{tabular}{ c c c c c }
\begin{tikzpicture}[baseline=0.5cm]
\draw[dashed] (-1.9,1) -- (-.9,1);
\draw[dashed] (0.9,1) -- (1.9,1);
\draw[dashed] (-1.9,0) -- (1.9,0);
\draw[dashed] (-1.9,-1) rectangle (1.9,2);
 
\draw[thick,blue] (0,-1) -- (0,2);
\draw[red] (-0.3,-1) -- (-0.3,2);
\draw[red] (-0.5,0) -- (-0.5,2);
\draw[red] (-.9,0) -- (-.9,2);  
\draw[red] (-1.3,0) -- (-1.3,2);

\draw[red] (-1.3,-1) arc (180:0:.2cm);
\draw[red] (-.9,0) arc (-180:0:.2cm);
\draw[red] (-.5,-1)  .. controls ++(90:.35cm) and ++(270:.35cm) .. (-1.3,0);

\draw (0.3,-1) -- (0.3,2);
\draw (0.5,0) -- (0.5,2);
\draw (.9,0) -- (.9,2);  
\draw (1.3,0) -- (1.3,2);

\draw (.5,-1) arc (180:0:.2cm);
\draw (.9,0) arc (-180:0:.2cm);
\draw (1.3,-1)  .. controls ++(90:.35cm) and ++(270:.35cm) .. (0.5,0);

\draw[orange,thick] (1.6,-1) -- (1.6,2);
\draw[orange,thick] (1.3,2) arc (180:0:.15cm);
\draw[orange,thick] (1.3,-1) arc (-180:0:.15cm);
 
\draw[orange,thick] (1.7,-1) -- (1.7,2);
\draw[orange,thick] (.9,2) arc (180:0:.4cm);
\draw[orange,thick] (.9,-1) arc (-180:0:.4cm);

\draw[green1,thick] (-1.6,-1) -- (-1.6,2);
\draw[green1,thick] (-1.6,2) arc (180:0:.15cm);
\draw[green1,thick] (-1.6,-1) arc (-180:0:.15cm);

\nbox{unshaded}{(0,1.5)}{.3}{1.2}{1.2}{$g$};
\nbox{unshaded}{(0,.5)}{.3}{0.4}{0.8}{$f$};
  
\node at (-0.3,2.2) {\tiny{$m$}};
\node at (-0.7,2.2) {\tiny{$i\+j$}};
\node at (-.6,1) {\tiny{$i$}};
\node at (-.3,-1.2) {\tiny{$m$}};
\node at (-.18,1) {\tiny{$m$}};
\node at (-.5,-.2) {\tiny{$i$}};
\node at (-.9,-1.2) {\tiny{$i$}};
\node at (-.5,-1.2) {\tiny{$j$}};

\node at (0.3,2.2) {\tiny{$n$}};
\node at (0.5,2.2) {\tiny{$k$}};
\node at (.7,1) {\tiny{$t\+ k$}};
\node at (.3,-1.2) {\tiny{$n$}};
\node at (.2,1) {\tiny{$n$}};
\node at (.5,-.2) {\tiny{$t$}};
\node at (1.3,-.2) {\tiny{$k$}};
\node at (.5,-1.2) {\tiny{$k$}};
\end{tikzpicture} 
& &
\begin{tikzpicture}[baseline=0.5cm]
\draw[dashed] (-1.9,2) -- (1.5,2);
\draw[dashed] (-1.9,1) -- (-.9,1);
\draw[dashed] (.5,1) -- (1.5,1);
\draw[dashed] (-1.9,0) -- (1.5,0);
\draw[dashed] (-1.9,-1) rectangle (1.5,3);
 
\draw[thick,blue] (0,-1) -- (0,3);
\draw[red] (-0.3,-1) -- (-0.3,3);
\draw[red] (-0.5,0) -- (-0.5,3);
\draw[red] (-.9,0) -- (-.9,3);  
\draw[red] (-1.3,0) -- (-1.3,3);

\draw[red] (-.9,-1) arc (180:0:.2cm);
\draw[red] (-1.3,0) arc (-180:0:.2cm);
\draw[red] (-1.3,-1)  .. controls ++(90:.35cm) and ++(270:.35cm) .. (-0.5,0);

\draw (0.3,-1) -- (0.3,3);
\draw (0.5,1) -- (0.5,-1);
\draw (0.5,2) -- (0.5,3);
\draw (.9,2) -- (.9,3);  
\draw (.9,1) -- (.9,-1);  
\draw (1.3,1) -- (1.3,-1);
\draw (1.3,2) -- (1.3,3);

\draw (.5,1) arc (180:0:.2cm);
\draw (.9,2) arc (-180:0:.2cm);
\draw (1.3,1) .. controls ++(90:.35cm) and ++(270:.35cm) .. (.5,2);

\draw[green1,thick] (-1.6,-1) -- (-1.6,3);
\draw[green1,thick] (-1.6,3) arc (180:0:.15cm);
\draw[green1,thick] (-1.6,-1) arc (-180:0:.15cm);
 
\draw[green1,thick] (-1.7,-1) -- (-1.7,3);
\draw[green1,thick] (-1.7,3) arc (180:0:.4cm);
\draw[green1,thick] (-1.7,-1) arc (-180:0:.4cm);

\nbox{unshaded}{(0,2.5)}{.3}{1.2}{.8}{$g$};
\nbox{unshaded}{(0,.5)}{.3}{0.8}{0.4}{$f$};
 
\node at (-0.3,3.2) {\tiny{$m$}};
\node at (-0.5,3.2) {\tiny{$i$}};
\node at (-.7,1) {\tiny{$j\+ i$}};
\node at (-.3,-1.2) {\tiny{$m$}};
\node at (-.2,1) {\tiny{$m$}};
\node at (-.5,-.2) {\tiny{$j$}};
\node at (-1.3,-.2) {\tiny{$i$}};
\node at (-.5,-1.2) {\tiny{$i$}};

\node at (0.3,3.2) {\tiny{$n$}};
\node at (0.7,3.2) {\tiny{$k\+t$}};
\node at (1.3,-1.2) {\tiny{$t$}};
\node at (.3,-1.2) {\tiny{$n$}};
\node at (.2,1) {\tiny{$n$}};
\node at (1.3,3.2) {\tiny{$k$}};
\node at (.9,-1.2) {\tiny{$k$}};
\node at (.5,-1.2) {\tiny{$k$}};
\end{tikzpicture}
& &
\begin{tikzpicture}[baseline=0.5cm]
\draw[dashed] (-1.5,2) -- (1.9,2);
\draw[dashed] (-1.5,1) -- (-.5,1);
\draw[dashed] (.9,1) -- (1.9,1);
\draw[dashed] (-1.5,0) -- (1.9,0);
\draw[dashed] (-1.5,-1) rectangle (1.9,3);

\draw[thick,blue] (0,-1) -- (0,3);

\draw[red] (-0.3,-1) -- (-0.3,3);
\draw[red] (-0.5,1) -- (-0.5,-1);
\draw[red] (-0.5,2) -- (-0.5,3);
\draw[red] (-.9,2) -- (-.9,3);  
\draw[red] (-.9,1) -- (-.9,-1);  
\draw[red] (-1.3,1) -- (-1.3,-1);
\draw[red] (-1.3,2) -- (-1.3,3);

\draw[red] (-.9,1) arc (180:0:.2cm);
\draw[red] (-1.3,2) arc (-180:0:.2cm);
\draw[red] (-1.3,1) .. controls ++(90:.35cm) and ++(270:.35cm) .. (-.5,2);

\draw (0.3,-1) -- (0.3,3);
\draw (0.5,0) -- (0.5,3);
\draw (.9,0) -- (.9,3);  
\draw (1.3,0) -- (1.3,3);

\draw (.9,-1) arc (180:0:.2cm);
\draw (.5,0) arc (-180:0:.2cm);
\draw (.5,-1)  .. controls ++(90:.35cm) and ++(270:.35cm) .. (1.3,0);

\draw[orange,thick] (1.6,-1) -- (1.6,3);
\draw[orange,thick] (1.3,3) arc (180:0:.15cm);
\draw[orange,thick] (1.3,-1) arc (-180:0:.15cm);

\nbox{unshaded}{(0,2.5)}{.3}{.8}{1.2}{$g$};
\nbox{unshaded}{(0,.5)}{.3}{0.4}{0.4}{$f$};

\node at (-0.3,3.2) {\tiny{$m$}};
\node at (-0.7,3.2) {\tiny{$i\+j$}};
\node at (-1.3,-1.2) {\tiny{$j$}};
\node at (-.3,-1.2) {\tiny{$m$}};
\node at (-.2,1) {\tiny{$m$}};
\node at (-1.3,3.2) {\tiny{$i$}};
\node at (-.9,-1.2) {\tiny{$i$}};
\node at (-.5,-1.2) {\tiny{$i$}};

\node at (0.3,3.2) {\tiny{$n$}};
\node at (0.7,3.2) {\tiny{$k\+t$}};
\node at (.6,1) {\tiny{$k$}};
\node at (.3,-1.2) {\tiny{$n$}};
\node at (.2,1) {\tiny{$n$}};
\node at (.5,-.2) {\tiny{$k$}};
\node at (.9,-1.2) {\tiny{$k$}};
\node at (.5,-1.2) {\tiny{$t$}};
\end{tikzpicture}\\
(C12) & & (C23) & & (C31)
\end{tabular}
\end{center}
The composition is well-defined and $\mathcal{M}_0$ is a ${\rm C}^*$ category as before.

\begin{remark}
The composition is well-defined, because of the commuting square of left/right conditional expectation condition and Proposition \ref{Markov lattice prop 1}.
\end{remark}

The definition of $x\rhd 1$ and $1\lhd y$ for $x\in \Hom(\mathcal{A}_0)$ and $y\in \Hom(\mathcal{B}_0)$ are the same as in

\begin{definition} $1\rhd g\lhd 1$, $x\rhd 1$ and $1\lhd y$, $g\in \Hom(\mathcal{M}_0)$, $x\in \Hom(\mathcal{A}_0)$ and $y\in \Hom(\mathcal{B}_0)$. 

The idea is the same as in Definition \ref{1 otimes x second times}. First, we define $1\rhd g\lhd 1$ as
\begin{center}
\small
\begin{tabular}{ |c|c| } 
 \hline
 $g$ & $1_j\rhd g\lhd 1_t$  \\
 \hline
  $(g;[m,n], [m+2i,n+2k]),\ i\le j,k\le t$ & $(ge^m_{j-i,i}f^n_{t-k,k};[m+j,n+t],[m+2i+j,m+2k+t])$  \\
\hline
  $(g;[m,n], [m+2i,n+2k]),\ i>j, k\le t$ & $(ge^{m,*}_{i-j,j}f^n_{t-k,k};[m+j,n+t],[m+2i+j,m+2k+t])$  \\
\hline
$(g;[m,n], [m+2i,n+2k]),\ i\le j,k> t$ & $(ge^m_{j-i,i}f^{n,*}_{k-t,t};[m+j,n+t],[m+2i+j,m+2k+t])$  \\
\hline
  $(g;[m,n], [m+2i,n+2k]),\ i>j, k> t$ & $(ge^{m,*}_{i-j,j}f^{n,*}_{k-t,t};[m+j,n+t],[m+2i+j,m+2k+t])$  \\
\hline
\end{tabular}
\end{center}
Note that here we use the fact that the Jones projection $[e_i,f_k]=0$ for all $i,k\ge 1$ and hence $(1\rhd g)\lhd 1 = 1\rhd (g\lhd 1) =:1\rhd g\lhd 1$.

The definitions of $x\rhd 1$ and $1\lhd y$ will be the same as $x\otimes 1$ and $1\otimes y$ in Definition \ref{x otimes 1, 1 otimes y} by using the shift maps.

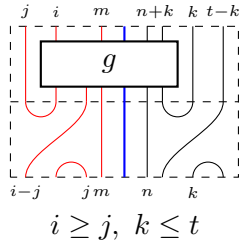
\begin{center}
\begin{tabular}{ c }
\begin{tikzpicture}[baseline=1cm]
\draw[dashed] (-1.5,1) -- (1.5,1);
\draw[dashed] (-1.5,0) rectangle (1.5,2);

\draw[thick,blue] (0,0) -- (0,2);
\draw (0.3,0) -- (0.3,2);
\draw (0.5,1) -- (0.5,2);
\draw (.9,1) -- (.9,2);
\draw (1.3,1) -- (1.3,2);

\draw (.5,1) arc (-180:0:.2cm);
\draw (.9,0) arc (180:0:.2cm);
\draw (.5,0)  .. controls ++(90:.35cm) and ++(270:.35cm) .. (1.3,1);

\draw[red] (-0.3,0) -- (-0.3,2);
\draw[red] (-0.5,1) -- (-0.5,1.5);
\draw[red] (-.9,1) -- (-.9,2);
\draw[red] (-1.3,1) -- (-1.3,2);

\draw[red] (-1.3,1) arc (-180:0:.2cm);
\draw[red] (-.9,0) arc (180:0:.2cm);
\draw[red] (-1.3,0)  .. controls ++(90:.35cm) and ++(270:.35cm) .. (-.5,1);

\nbox{unshaded}{(0,1.5)}{.3}{0.8}{0.4}{$g$};
 
\node at (0.3,-.2) {\tiny{$n$}};
\node at (1.3,2.2) {\tiny{$t\sm k$}};
\node at (.9,-.2) {\tiny{$k$}};
\node at (.4,2.2) {\tiny{$n\+k$}};
\node at (.9,2.2) {\tiny{$k$}};

\node at (-0.3,-.2) {\tiny{$m$}};
\node at (-0.5,-.2) {\tiny{$j$}};
\node at (-1.3,-.2) {\tiny{$i\sm j$}};
\node at (-.3,2.2) {\tiny{$m$}};
\node at (-.9,2.2) {\tiny{$i$}};
\node at (-1.3,2.2) {\tiny{$j$}};
\end{tikzpicture}
\\
$i\ge j,\ k\le t$  
\end{tabular}
\end{center}
\end{definition}

The proof of the following propositions are the same as in the Markov tower case with the fact in Remark \ref{everything commute, for bimodule proof}. To be precise, the diagrammatic proof can be split as left-hand-side and right-hand-side independently, and the proof on each side is the same as the Markov tower case.

\begin{proposition}
$\mathcal{M}_0$ is a left $\mathcal{A}_0-$module.
That is,
\begin{compactenum}[\rm (1)]
\item
For $g\in \Hom(\mathcal{M}_0)$, $x\in \Hom(\mathcal{A}_0)$, $(1\lhd g)\circ (x\lhd 1)=(x\lhd 1)\circ (1\lhd g)$.
\item
For $g\in \Hom(\mathcal{M}_0)$, $x_1,x_2\in \Hom(\mathcal{A}_0)$, $x_2\rhd (x_1\rhd g)=(x_2\otimes x_1)\rhd g$.
\item
For $g_1,g_2\in \Hom(\mathcal{M}_0)$, $x_1,x_2\in \Hom(\mathcal{A}_0)$, $1\rhd (g_1\circ g_2)=(1\rhd g_1)\circ (1\rhd g_2)$ and $(x_1\circ x_2)\rhd 1 = (x_1\rhd 1)\circ (x_2\rhd 1)$.
\end{compactenum}
\end{proposition}

\begin{proposition}
Similarly, $\mathcal{M}_0$ is a right $\mathcal{B}_0-$module.
That is,
\begin{compactenum}[\rm (1)]
\item
For $g\in \Hom(\mathcal{M}_0)$, $y\in \Hom(\mathcal{B}_0)$, $(g\lhd 1)\circ (1\lhd y)=(1\lhd y)\circ (g\lhd 1)$.
\item
For $g\in \Hom(\mathcal{M}_0)$, $y_1,y_2\in \Hom(\mathcal{B}_0)$, $(g\lhd y_1)\lhd y_2=g\lhd (y_1\otimes y_2)$.
\item
For $g_1,g_2\in \Hom(\mathcal{M}_0)$, $y_1,y_2\in \Hom(\mathcal{B}_0)$, $(g_1\circ g_2)\lhd 1=(g_1\lhd 1)\circ (g_2\lhd 1)$ and $1\lhd (x_1\circ x_2) = (1\lhd x_1)\circ (1\lhd x_2)$.
\end{compactenum}
\end{proposition}

\begin{proposition}
$\mathcal{M}_0$ is a $\mathcal{A}_0-\mathcal{B}_0$ bimodule.
That is,
for $g\in \Hom(\mathcal{M}_0)$, $x\in \Hom(\mathcal{A}_0)$, $y\in \Hom(\mathcal{B}_0)$, $(x\rhd 1)\circ (1\lhd y)\circ (1\rhd g \lhd 1)=(1\lhd y)\circ (x\rhd 1)\circ (1\rhd g \lhd 1)$.
\end{proposition}
\begin{proof} 
By Remark \ref{everything commute, for bimodule proof}. 
\end{proof}

\begin{definition}
Define $x\rhd g\lhd y:=(x\rhd 1)\circ (1\lhd y)\circ (1\rhd g\lhd 1)$.
\end{definition}

\subsection{Indecomposable semisimple ${\rm C}^*$ $\mathcal{A}-\mathcal{B}$ bimodules and planar $\mathcal{A}_0-\mathcal{B}_0$ bimodule categories}
\subsubsection{Indecomposable semisimple ${\rm C}^*$ $\mathcal{A}-\mathcal{B}$ bimodule category}
Let $\mathcal{A}$ and $\mathcal{B}$ be 2-shaded rigid ${\rm C}^*$ multitensor categories with generators $X=1^+_{\mathcal{A}}\otimes X\otimes 1^-_{\mathcal{A}}$ and $Y=1^+_{\mathcal{B}}\otimes Y\otimes 1^-_{\mathcal{B}}$. Let $\mathcal{M}$ be a Cauchy complete indecomposable semisimple ${\rm C}^*$ $\mathcal{A}-\mathcal{B}$ bimodule category. Note that there is a natural dagger structure on $\mathcal{M}$, and the left/right module actions are dagger functors, i.e., for morphism $g\in \Hom(\mathcal{M})$, $x\in \Hom(\mathcal{A})$ and $y\in \Hom(\mathcal{B})$,
$$(x\rhd g)^\dagger =x^\dagger \rhd g^\dagger,\qquad (f\lhd y)^\dagger =f^\dagger \lhd y^\dagger.$$

We call $\mathcal{M}$ \textbf{indecomposable} if for any two simple objects $P,Q\in\mathcal{M}$ (WLOG, $P=1_{\mathcal{A}}^+\rhd P\lhd 1_{\mathcal{B}}^+$), $Q$ is a direct summand of $(X^{\alt\otimes m}\rhd P)\lhd Y^{\alt\otimes n}$ for some $m,n\in\mathbb{Z}_{\ge 0}$.

Let $\mathcal{A}_0,\mathcal{B}_0$ be planar tensor categories constructed from $(\mathcal{A},X)$ and $(\mathcal{B},Y)$ respectively. 
By MacLane's coherence theorem, ${}_\mathcal{A}\mathcal{M}_{\mathcal{B}}$ is unitary equivalent to a strict one, i.e., $\mathcal{A},\mathcal{B}$ are strict, the right/left module associators and the bimodule associator are trivial. 
This strict category is also a strict $\mathcal{A}_0-\mathcal{B}_0$ bimodule category. WLOG, we also denote it as $\mathcal{M}$.

Pick a simple object $Z=1_{\mathcal{A}}^+\rhd Z\lhd 1_{\mathcal{B}}^+\in\mathcal{M}$, then we construct a planar $\mathcal{A}_0-\mathcal{B}_0$ bimodule category $\mathcal{M}_0$ as follows:
\begin{compactenum}[(a)]
\item Objects: Define $[0,0]:=Z$, and
$$[m+1,0]:= [1,?]_{\mathcal{A}_0}\rhd [m,0], \qquad [m,n+1]:= [m,n]\lhd [1,?]_{\mathcal{B}_0},$$
where $[1,?]_{\mathcal{A}_0}=[1,+]_{\mathcal{A}_0}$ if $2\mid m$ and $[1,?]_{\mathcal{A}_0}=[1,-]_{\mathcal{A}_0}$ if $2\nmid m$; $[1,?]_{\mathcal{B}_0}=[1,+]_{\mathcal{B}_0}$ if $2\nmid n$ and $[1,?]_{\mathcal{B}_0}=[1,-]_{\mathcal{B}_0}$ if $2\mid n$. 
\item $\mathcal{M}_0$ is a full subcategory of $\mathcal{M}$ with above objects.
\end{compactenum}

Given $\mathcal{M}_0$ to be a planar $\mathcal{A}_0-\mathcal{B}_0$ bimodule category, for the similar reason, its Cauchy completion $\widehat{\mathcal{M}_0}$ is a $\widehat{\mathcal{A}_0}-\widehat{\mathcal{B}_0}$ bimodule category, compatible with the dagger structure.

\begin{remark}
Suppose $\mathcal{M}_0$ is a planar $\mathcal{A}_0-\mathcal{B}_0$ bimodule category constructed from $\mathcal{M}$ over $(\mathcal{A},X)$ and $(\mathcal{B},Y)$, then there is a unitary equivalence between $\mathcal{M}$ as $\mathcal{A}-\mathcal{B}$ bimodule category and $\widehat{\mathcal{M}_0}$ as $\widehat{\mathcal{A}_0}-\widehat{\mathcal{B}_0}$ bimodule category, which maps base object to base object.
\end{remark}

\subsubsection{From planar bimodule to Markov lattice as standard bimodule}
\begin{construction}
Now let $M_{i,j}=\End([i,j])$, $i,j\in\mathbb{Z}_{\ge 0}$. After identifying $f\in M_{i,j}$ with $\id_{[1,?]}\rhd f\in M_{i+1,j}$ and $f\lhd \id_{[1,?]}\in M_{i,j+1}$ and identifying $x\in A_{i,0}=\End([i,+]_{\mathcal{A}_0})$ with $x\lhd \id_{[0,j]}\in M_{i,j}$ and $y\in B_{0,j}=\End([j,+]_{\mathcal{B}_0})$ with $\id_{[i,0]}\lhd y\in M_{i,j}$. 
It is easy to show that $M=(M_{i,j})_{i,j\ge 0}$ is a Markov lattice as a standard $A-B$ bimodule with modulus $(d_0,d_1)$.
\end{construction}

Similar to the module case, combining above discussion, we have the following theorem.

\begin{theorem}
There is a bijective correspondence between equivalence classes of the following:
\[\left\{\, \parbox{5.7cm}{\rm 
Traceless Markov lattice $M=(M_{i,j})_{i,j\ge 0}$ 
with $\dim(M_{0,0})=1$
as a standard $A-B$ bimodule over standard $\lambda$-lattices $A,B$} \,\right\}\ \, \cong\ \, \left\{\, \parbox{6.5cm}{\rm Pairs $(\mathcal{M}, Z)$ with $\mathcal{M}$ an indecomposable semisimple ${\rm C}^*$  $\mathcal{A}-\mathcal{B}$ bimodule category together with a choice of simple object $Z= 1^+_\mathcal{A}\rhd Z\lhd 1^+_\mathcal{B}$} \,\right\}\]
Equivalence on the left hand side is the $*$-isomorphism on the traceless Markov lattice as standard $A-B$ bimodule; the equivalence on the right hand side is the unitary $\mathcal{A}-\mathcal{B}$ bimodule equivalence between their Cauchy completions which maps the simple base object to simple base object.
\end{theorem}

\begin{corollary}
Any Markov lattice $M$ with modulus $(d_0,d_1)$ and $\dim(M_{00})=1$ is naturally a standard $\mathrm{TLJ}(d_0)-\mathrm{TLJ}(d_1)$ bimodule, which corresponds to an indecomposable semisimple ${\rm C}^*$ $\mathcal{TLJ}(d_0)-\mathcal{TLJ}(d_1)$ bimodule category with a simple base object.
\end{corollary}

\begin{remark}
The tracial case will be discussed in \S\ref{Tracial Markov lattice}.
\end{remark}

\section{Markov towers, 
bigraded Hilbert spaces,
and balanced fair graphs
}
\label{Cpt 4}

In this Chapter, as an application, we are going to classify all indecomposable semisimple $\mathcal{TLJ}-$modules (see Corollary \ref{corollary: TLJ module}) to get Markov tower, which are also the same as balanced $d$-fair bipartite graphs \cite{DY15}. We will explain exactly how these two classifications agree by directly constructing the correspondence passing through the 2-category \textsf{BigHilb} \cite{FP19}. Although this is known \cite{DY15,FP19}, we explain in detail here so that we are able to do the bimodules in \S\ref{Cpt 5} below.

\subsection{Balanced $d$-fair bipartite graph}
In \cite{DY15}, the authors classify unshaded unoriented $\mathcal{TLJ}(d)-$modules in terms of the combinatorial data of fair and balanced graphs.
This classification was generalized to $\mathcal{TLJ}(\Gamma)-$modules in \cite{FP19}, where $\mathcal{TLJ}(\Gamma)$ is a generalized Temperley-Lieb-Jones category associated to a weighted bidirected graph $\Gamma$.
We will be interested in the special case of 2-shaded $\mathcal{TLJ}(d)-$modules.

\begin{notation}
Let $\Lambda$ be a graph where $V(\Lambda)$ is the set of vertices and $E(\Lambda)$ is the set of edges.
Let $s,t: E(\Lambda) \to V(\Lambda)$ be the source and target functions respectively.
\end{notation}

\begin{definition} \label{Def:d-fairness and balanced}
Let $\Lambda$ be a 
bipartite graph with vertices $V(\Lambda)=V_0\sqcup V_1$ and $\{e|s(e),t(e)\in V_i\}=\varnothing$, $i=0,1$. Let $\omega:E(\Lambda)\to (0,\infty)$ be the weighting on the edges of graph \cite{FP19}.

We call $(\Lambda,\omega)$ a $d$\textbf{-fair} graph if for each $P\in V_0$, $Q\in V_1$
$$\sum_{\{e|s(e)=P\}}w(e)=\sum_{\{e|s(e)=Q\}}w(e)=d.$$

We call $(\Lambda, \omega)$ a \textbf{balanced} graph if there exists an involution $(\overline{\cdot})$ on $E(\Lambda)$ that switches sources and targets for each $e\in E(\Lambda)$ and 
$$\omega(e)\omega(\overline{e})=1.$$
\end{definition}

\begin{proposition}\label{uniform boundedness for graph}
Suppose $(\Lambda,\omega)$ is a balanced $d$-fair bipartite graph. 
Then the graph is locally finite, i.e., the number of edges coming in or out of any vertex is uniformly bounded:
$$\#\{e:s(e)=P\}=\#\{e:t(e)=P\}\le d^2 \qquad\text{ for any vertex }P.$$
\end{proposition}
\begin{proof}
Suppose $P$ has $N$ edges, then there exists an edge $e_0:P\to Q$ such that $\omega(e_0)\le \frac{d}{N}$ and hence $\omega(\overline{e_0})=\frac{1}{\omega(e_0)}\ge \frac{N}{d}$. Note that $$d=\sum_{\{e|s(e)=Q\}}\omega(e)\ge \omega(\overline{e_0})\ge \frac{N}{d},$$
which follows that $N\le d^2<\infty$.
\end{proof} 

\begin{definition}\label{Def:graph automorphism}
We call $\theta:(\Lambda,\omega)\to(\Lambda',\omega')$ an 
isomorphism of edge-weighted graphs
if
$\theta$ is a graph isomorphism and $\omega'(\theta(e))=\omega(e)$ for each $e\in E(\Lambda)$.
\end{definition}

\subsection{\textsf{BigHilb} and 2-subcategory $\mathcal{C}(K,\ev_K)$}

\begin{definition}\label{Def:Hilb UxV}
Let $U,V$ be countable sets. Define a category $\textsf{Hilb}^{U\times V}_f$ as follows:
\begin{compactenum}[(a)]
\item Object: $U\times V-$bigraded Hilbert spaces 
$$H=\bigoplus_{\substack{u\in U \\ v\in V}}H_{uv},$$
where $H_{uv}$ is finite dimensional for each pair $(u,v)$, and only finite many $H_{uv}$ is non-trivial for each fixed $u\in U$ or each fixed $v\in V$. 
\item Morphism: The morphisms are defined as uniformly bounded operators 
$$f=\bigoplus_{\substack{u\in U \\ v\in V}}f_{uv}:H\to G,$$
where $f_{uv}:H_{uv}\to G_{uv}$ are morphisms in $\textsf{Hilb}_f$, the category of finitely dimensional Hilbert spaces. Uniformly boundedness means
$$\sup_{\substack{u\in U \\ v\in V}}\|f_{uv}\|<\infty.$$
\item The composition: For morphisms $f,g$, define the composition entry-wisely as 
$$g\circ f:=\bigoplus_{\substack{u\in U \\ v\in V}}g_{uv}\circ f_{uv}.$$
\item The identity morphism: Define the identity morphism $\id_H:H\to H$ as 
$$\id_H:=\bigoplus_{\substack{u\in U \\ v\in V}}\id_{H_{uv}},$$
where $\id_{H,uv}=\id_{H_{uv}}$ is the identity map on $H_{uv}$. 
\end{compactenum}
\end{definition}

\begin{definition}\label{Def:BigHilb}
Let $\textsf{BigHilb}$ be a dagger 2-category defined as follows:
\begin{compactenum}[(a)]
\item Object: Countable sets.
\item For objects $U,V$, $\Hom(U,V)=\textsf{Hilb}^{U\times V}_f$.
\item The composition of 1-morphisms: For 1-morphisms $H:U\to V$, $G:V\to W$, the composition of $U,V$ denoted by $\otimes$ is defined as 
$$G\circ H=H\otimes G:=\bigoplus_{\substack{u\in U \\ w\in W}}\bigoplus_{v\in V}H_{uv}\otimes G_{vw}:U\to W,$$
where the $\otimes$ on the right hand side is the tensor product of Hilbert spaces. The operator is analogous to matrix multiplication, the product is replaced by tensor product and the sum is replaced by direct sum. Clearly, $(H\otimes G)\otimes L=H\otimes (G\otimes L).$
\item The identity 1-morphism: For an object $U$, the identity 1-morphism $\mathbb{C}^{|U|}\in \Hom(U,U)$ is defined as 
$$\mathbb{C}^{|U|}:=\bigoplus_{u,v\in U}\delta_{u=v}\cdot \mathbb{C}.$$
\item The dual 1-morphism: For 1-morphism $H=\bigoplus\limits_{\substack{u\in U \\ v\in V}}H_{uv}:U\to V$, define its dual as 
$$\overline{H}:=\bigoplus\limits_{\substack{v\in V \\ u\in U}}\overline{H}_{vu}:V\to U,$$
where $\overline{H}_{vu}:=\overline{H_{uv}}$ and $\overline{H_{uv}}$ is the complex conjugate Hilbert space of $H_{uv}$. 
\item Tensor product of 2-morphisms. Let $H_1,H_2:U\to V$, $G_1,G_2:V\to W$, and $f:H_1\to H_2$, $g:G_1\to G_2$, define $f\otimes g$ as
$$(f\otimes g)_{uw}:=\bigoplus_{v\in V}f_{uv}\otimes g_{vw}:\bigoplus_{v\in V}H_{1,uv}\otimes G_{1,vw}\to \bigoplus_{v\in V}H_{2,uv}\otimes G_{2,vw}.$$
Clearly, $(f\otimes g)\otimes h=f\otimes (g\otimes h)$.

\item Dagger structure: For a 2-morphism $f=\bigoplus_{u,v}f_{uv}:H\to G$, define its adjoint $f^\dagger:=\bigoplus_{u,v}f_{uv}^*:G\to H$, where $f_{uv}^*$ is the adjoint of $f_{uv}$ as a bounded linear map. Clearly, $(f^\dagger)^\dagger=f$.
\end{compactenum}

\end{definition}

\begin{definition}
We call a 1-morphism $H:U\to V$ \textbf{dualizable}, if there exist evaluation and coevaluation 2-morphisms $\ev_H:\overline{H}\otimes H\to \mathbb{C}^{|V|}$ and $\coev_H:\mathbb{C}^{|U|}\to H\otimes \overline{H}$ meeting the \textbf{zigzag condition}:
\begin{align*}
    (\id_H\otimes \ev_{H})\circ (\coev_{H}\otimes \id_H) &= \id_H\\
    (\ev_{H}\otimes \id_{\overline{H}})\circ (\id_{\overline{H}}\otimes \coev_{H}) &= \id_{\overline{H}}.
\end{align*}
\end{definition}

We are going to discuss the evaluation and coevaluation $\ev_H$ and $\coev_H$ in more details.

\begin{definition}\label{Def:C,D for ev,coev} Note that $\ev_{H,uv}:\bigoplus_{w}\overline{H}_{uw}\otimes H_{wv}=(\overline{H}\otimes H)_{uv}\to (\mathbb{C}^{|V|})_{uv}=\delta_{u=v}\cdot\mathbb{C}$, only $\ev_{H,vv}$ is nonzero for $v\in V$.
Let $C_{H,vu}:\overline{H}_{vu}\otimes H_{uv}=\overline{H_{uv}}\otimes H_{uv}\to \mathbb{C}$ such that $\ev_{H,vv}=\bigoplus_{u\in U}C_{H,vu}$.
Similarly, only $\coev_{H,uu}:\mathbb{C}\to (H\otimes \overline{H})_{uu}=\bigoplus_{v\in V}H_{uv}\otimes \overline{H}_{vu}$ is nonzero for $u\in U$. Let $D_{H,uv}:\mathbb{C}\to H_{uv}\otimes \overline{H}_{vu}=H_{uv}\otimes \overline{H_{uv}}$ such that $\coev_{H,uu}=\bigoplus_{v\in V}D_{H,uv}$.

Then
\begin{align*}
   \id_{H,uv} & = ((\id_H\otimes \ev_{H})\circ (\coev_{H}\otimes \id_H))_{uv} \\
   & = (\id_H\otimes \ev_{H})_{uv} \circ (\coev_{H}\otimes \id_H)_{uv}\\
   & = \left(\bigoplus_{w\in V}\id_{H,uw}\otimes\ev_{H,wv}\right)\circ \left(\bigoplus_{t\in U}\coev_{H,ut}\otimes\id_{H,tv}\right) \\
   & = (\id_{H,uv}\otimes \ev_{H,vv})\circ (\coev_{H,uu}\otimes \id_{H,uv})\\
   & = (\id_{H,uv}\otimes C_{H,vu})\circ (D_{H,uv}\otimes \id_{H,uv})
\end{align*}
for $u\in U,\ v\in V$. 
Similarly,
$$\id_{\overline{H},vu}=(\ev_{H,vv}\otimes \id_{\overline{H},vu})\circ (\id_{\overline{H},vu}\otimes \coev_{H,uu})=(C_{H,vu}\otimes \id_{\overline{H},vu})\circ (\id_{\overline{H},vu}\otimes D_{H,uv}),$$
for $v\in V,u\in U$.
\end{definition}

\begin{remark}
$\ev_H$ and $\coev_H$ are completely determined by $C_{H,uv}$ and $D_{H,uv}$.
\end{remark}

\begin{definition}\label{C(K,w) def}
Let $\mathcal{C}(K,\ev_K)=\mathcal{C}(K,\ev_K,\coev_K)$ be a 2-subcategory of \textsf{BigHilb} with a 1-morphism generator $K:V_0\to V_1$ and distinguished 2-morphisms evaluation and coevaluation $\ev_K,\coev_K$. We require that 
\begin{compactenum}[(a)]
\item $K$ is dualizable. 
\item The evaluation and coevaluation for the dual $\overline{K}$: $$\ev_{\overline{K}}:= (\coev_K)^\dagger\qquad\text{and}\qquad \coev_{\overline{K}}:=(\ev_K)^\dagger.$$
\item They satisfy the \textbf{$d-$fairness condition}, namely,
\begin{align*}
    \ev_{\overline{K}}\circ \coev_K = d\cdot \id_{\mathbb{C}^{|V_0|}} &\qquad \ev_{K}\circ \coev_{\overline{K}} = d\cdot \id_{\mathbb{C}^{|V_1|}}. 
\end{align*}
\end{compactenum}
In other words,
$$C_{\overline{K},uv}=(D_{K,uv})^\dagger \qquad D_{\overline{K},vu}=(C_{K,vu})^\dagger,$$
and 
\begin{align*}
    \text{For each }P\in V_0,&\ \sum_{Q\in V_1}C_{\overline{K},PQ}\circ D_{K,PQ} = d\cdot \id_{\mathbb{C}} \\
    \text{For each }Q\in V_1,&\ \sum_{P\in V_0}C_{K,QP}\circ D_{\overline{K},QP} = d\cdot\id_{\mathbb{C}},
\end{align*}

Here, the 1-morphism generator means all the 1-morphism is Cauchy generated by $K$ and $\overline{K}$. 
\end{definition}

\begin{remark}
$\coev_K,\ev_{\overline{K}}$ and $\coev_{\overline{K}}$ are determined by $\ev_K$ in $\mathcal{C}(K,\ev_K)$.
\end{remark}

\begin{proposition}The followings are some properties of $\mathcal{C}(K,\ev_K)$.
\begin{compactenum}[\rm (1)]
\item Let $V=V_0\sqcup V_1$, then all the 1-morphisms in $\mathcal{C}(K,\ev_K)$, including $K,\overline{K}$, can be regarded as $V\times V-$bigraded Hilbert spaces. So we can regard $\mathcal{C}(K,\ev_K)$ as a 2-category with one object $V$.
Then all the 2-morphisms can be regarded as $V\times V-$bigraded uniformly bounded operators.

If $(P,Q)\not\in V_0\times V_1$, then $K_{PQ}=\overline{K}_{QP}=0$, which follows that $C_{K,QP}=D_{K,PQ}=0$. The zigzag condition between them still hold. 

\item All the 1-morphisms in $\mathcal{C}(K,\ev_K)$ are dualizable.
\item $\sup_{P\in V_0,Q\in V_1}\dim (K_{PQ})<\infty$. In fact, we will see $\sup_{P\in V_0,Q\in V_1}\dim (K_{PQ})\le d^2$ in the next section \S\ref{2-Hilb and edge weighting} together with Proposition \ref{uniform boundedness for graph}.

\item There exist standard spherical evaluation and coevaluation in 2-morphisms: 
\begin{align*}
    \ev_K^{\text{st}}:\overline{K}\otimes K\to \mathbb{C}^{|V_1|}  & \qquad \coev_K^{\text{st}}:\mathbb{C}^{|V_0|}\to K\otimes \overline{K}\\
    \ev_{\overline{K}}^{\text{st}}:=(\coev_K)^\dagger  & \qquad \coev_{\overline{K}}^{\text{st}}:=(\ev_K)^\dagger.
\end{align*}

In more details, 
Let $\{\epsilon_i\}_{i=1}^k$ be the orthonormal basis $($ONB\,$)$ of $K_{PQ}$ and $\{\epsilon_i^*\}$ be the dual basis of $\overline{K_{PQ}}$, $P\in V_0,\ Q\in V_1$ then
\begin{align*}
    C_{K,QP}^{\text{st}}:\overline{K}_{QP}\otimes K_{PQ}=\overline{K_{PQ}}\otimes K_{PQ}\to \mathbb{C} & \qquad D_{K,ab}^{\text{st}}:\mathbb{C}\to K_{PQ}\otimes \overline{K}_{QP}=K_{PQ}\otimes \overline{K_{PQ}}\\
    C_{\overline{K},PQ}^{\text{st}}:= (D_{K,PQ}^{\text{st}})^\dagger & \qquad D_{\overline{K},QP}^{\text{st}}:= (C_{K,QP}^{\text{st}})^\dagger
\end{align*}
are defined as
\begin{align*}
    C_{K,QP}^{\text{st}} : \epsilon_i^*\otimes\epsilon_j\mapsto \delta_{i=j} \qquad
    D_{K,PQ}^{\text{st}} : 1\mapsto \sum_{i=1}^k \epsilon_i\otimes \epsilon_i^*.
\end{align*}
 
Note that $\ev_K^{\text{st}}$ and $\coev_K^\text{st}$ are well-defined 2-morphisms because of $(3)$, and the definitions of $\ev^{\text{st}}_K$ and $\coev^{\text{st}}_K$ do not depend on the choice of ONB on each $K_{PQ}$ and they also meet the zigzag condition.. 
\end{compactenum}
\end{proposition}

\begin{notation}

Now, we use the graphic calculus to describe $\mathcal{C}(K,\ev_K)$. The idea is from the graphical calculus for 2-\textsf{Hilb} \cite{RV16}. However, in their paper, they only care about the case when $\ev=\ev^\text{st}$ and $\coev=\coev^\text{st}$, which is not necessarily true in our context.

First we provide the single object version:
\begin{compactenum}[(1)]
\item For $P\in V_0,Q\in V_1$, $C_{\overline{K},PQ}$, $D_{\overline{K},QP}$, $C^{\text{st}}_{\overline{K},PQ}$ and $D^{\text{st}}_{\overline{K},QP}$.
\[
\begin{tikzpicture}[baseline=0cm]
\draw (0,1) arc (180:0:.5cm);

\node at (0,.8) {\tiny{$K_{PQ}$}};
\node at (1,.8) {\tiny{$\overline{K}_{QP}$}};
\node at (-.6,1.1) {$P$};
\node at (.5,1.1) {$Q$};
\node at (0.2,0) {\tiny{$C_{\overline{K},PQ}:K_{PQ}\otimes \overline{K}_{QP}\to\mathbb{C}$}};
\end{tikzpicture}  
\qquad
\begin{tikzpicture}[baseline=0cm]   
\draw (0,1.3) arc (-180:0:.5cm);  

\node at (0,1.5) {\tiny{$\overline{K}_{QP}$}}; 
\node at (1,1.5) {\tiny{$K_{PQ}$}}; 
\node at (-.6,1.1) {$Q$};
\node at (.5,1.1) {$P$};
\node at (.2,0) {\tiny{$D_{\overline{K},QP}:\mathbb{C}\to \overline{K}_{QP}\otimes K_{PQ}$}};
\end{tikzpicture}
\qquad
\begin{tikzpicture}[baseline=0cm]
\draw[green1,thick] (0,1) arc (180:0:.5cm);

\node at (0,.8) {\tiny{$K_{PQ}$}};
\node at (1,.8) {\tiny{$\overline{K}_{QP}$}};
\node at (-.6,1.1) {$P$};
\node at (.5,1.1) {$Q$};
\node at (0.2,0) {\tiny{$C^{\text{st}}_{\overline{K},PQ}:K_{PQ}\otimes \overline{K}_{QP}\to\mathbb{C}$}};
\end{tikzpicture}  
\qquad
\begin{tikzpicture}[baseline=0cm]   
\draw[green1,thick] (0,1.3) arc (-180:0:.5cm);  

\node at (0,1.5) {\tiny{$\overline{K}_{QP}$}}; 
\node at (1,1.5) {\tiny{$K_{PQ}$}}; 
\node at (-.6,1.1) {$Q$};
\node at (.5,1.1) {$P$};
\node at (.2,0) {\tiny{$D^{\text{st}}_{\overline{K},QP}:\mathbb{C}\to \overline{K}_{QP}\otimes K_{PQ}$}};
\end{tikzpicture}
\]
\item Rigidity: 
\[
\begin{tikzpicture}[baseline=0cm]
\draw (0,.3) arc (-180:0:.3cm);
\draw (0,0.3) -- (0,1.6);
\draw (.6,.3) -- (.6,1.3);
\draw (1.2,0) -- (1.2,1.3);
\draw (.6,1.3) arc (180:0:.3cm);

\node at (-.5,.8) {$P$};
\node at (1.5,.8) {$Q$};
\node at (2.5,.8) {$P$};
\node at (3.5,.8) {$Q$};
\node at (4.5,.8) {$P$};
\node at (6.5,.8) {$Q$};

\node at (2,.8) {$=$};

\draw (3,0) -- (3,1.6);

\node at (4,.8) {$=$};

\draw (5.6,.3) arc (-180:0:.3cm);
\draw (5,0) -- (5,1.3);
\draw (5.6,.3) -- (5.6,1.3);
\draw (6.2,0.3) -- (6.2,1.6);
\draw (5,1.3) arc (180:0:.3cm);

\end{tikzpicture}
\qquad
\begin{tikzpicture}[baseline=0cm]
\draw[green1,thick] (0,.3) arc (-180:0:.3cm);
\draw[green1,thick] (0,0.3) -- (0,1.6);
\draw[green1,thick] (.6,.3) -- (.6,1.3);
\draw[green1,thick] (1.2,0) -- (1.2,1.3);
\draw[green1,thick] (.6,1.3) arc (180:0:.3cm);

\node at (-.5,.8) {$P$};
\node at (1.5,.8) {$Q$};
\node at (2.5,.8) {$P$};
\node at (3.5,.8) {$Q$};
\node at (4.5,.8) {$P$};
\node at (6.5,.8) {$Q$};

\node at (2,.8) {$=$};

\draw[green1,thick] (3,0) -- (3,1.6);

\node at (4,.8) {$=$};

\draw[green1,thick] (5.6,.3) arc (-180:0:.3cm);
\draw[green1,thick] (5,0) -- (5,1.3);
\draw[green1,thick] (5.6,.3) -- (5.6,1.3);
\draw[green1,thick] (6.2,0.3) -- (6.2,1.6);
\draw[green1,thick] (5,1.3) arc (180:0:.3cm);
\end{tikzpicture}
\]
\item $d$-fairness. For $P\in V$,
\[
\begin{tikzpicture}[baseline=0cm]
\node at (5,1.3) {$\sum\limits_{Q\in V}$}; 
   
\draw (6,1.5) arc (180:0:.5cm);  
\draw (6,1.5) arc (-180:0:.5cm); 
\draw[dashed] (8.3,1) rectangle (9.3,2);

\node at (5.6,1.5) {$P$}; 
\node at (6.5,1.5) {$Q$}; 
\node at (7.8,1.5) {= $d\cdot$};
\node at (8.8,1.5) {$P$};
\end{tikzpicture}
\]
\end{compactenum}

Then the graphical calculus version: In the $n$-category setting, $n$-morphisms are n-morphisms are used to label codimension $n$ cells of an $n$-manifold.
So here, 0-morphisms in \textsf{BigHilb} label regions of the plane, 1-morphisms label strings from left to right, and 2-morphisms label tickets (including ev and coev) from bottom to top. Shading is just shorthand for the labelling. 
The unshaded region indicates the object $V_0$ and the shaded region indicates $V_1$.

\begin{compactenum}[(1)]
\item $\coev_K$, $\ev_{K}$, $\coev_{\overline{K}}^\text{st}$ and $\ev_{\overline{K}}^\text{st}$.
\[
\begin{tikzpicture}
\fill[white] (0,0) rectangle (1.8,1);
\filldraw[gray1] (.4,1) arc (-180:0:.5);

\draw (.4,1) arc (-180:0:.5);

\draw[dashed] (0,0) rectangle (1.8,1);
\node at (.9,-.4) {\tiny{$\coev_K:\mathbb{C}^{|V_0|}\to K\otimes\overline{K}$}};
\end{tikzpicture}
\qquad
\begin{tikzpicture}
\fill[gray1] (9,0) rectangle (10.8,1);
\filldraw[white] (9.4,0) arc (180:0:.5);

 \draw (9.4,0) arc (180:0:.5);

\draw[dashed] (9,0) rectangle (10.8,1);
\node at (9.9,-.4) {\tiny{$\ev_{K}:\overline{K}\otimes K\to\mathbb{C}^{|V_1|}$}};
\end{tikzpicture}
\qquad
\begin{tikzpicture}
\fill[gray1] (6,0) rectangle (7.8,1);
\filldraw[white] (6.4,1) arc (-180:0:.5);

\draw[green1,thick] (6.4,1) arc (-180:0:.5);

\draw[dashed] (6,0) rectangle (7.8,1);
\node at (6.9,-.4) {\tiny{$\coev_{\overline{K}}^\text{st}:\mathbb{C}^{|V_1|}\to \overline{K}\otimes K$}};
\end{tikzpicture}
\qquad
\begin{tikzpicture}
\fill[white] (3,0) rectangle (4.8,1);
\filldraw[gray1] (3.4,0) arc (180:0:.5);

\draw (3.4,0)[green1,thick] arc (180:0:.5);

\draw[dashed] (3,0) rectangle (4.8,1);
\node at (3.9,-.4) {\tiny{$\ev_{\overline{K}}^\text{st}:K\otimes\overline{K}\to\mathbb{C}^{|V_0|}$}};
\end{tikzpicture}
\]
\item Rigidity: 
\[
\begin{tikzpicture}
\fill[white] (0,0) rectangle (0.3,1.2);
\fill[white] (0.3,0.6) rectangle (1.5,1.2);
\fill[gray1] (0.3,0.6) rectangle (1.5,0);
\fill[gray1] (1.5,0) rectangle (1.8,1.2);

\filldraw[white] (.9,0.6) arc (-180:0:.3);
\filldraw[gray1] (.3,0.6) arc (180:0:.3);

\draw (.9,0.6) arc (-180:0:.3);
\draw (.3,0.6) arc (180:0:.3);
\draw (.3,0) -- (.3,.6);
\draw (1.5,.6) -- (1.5,1.2);

\draw[dashed] (0,0) rectangle (1.8,1.2);
\node at (2.1,.6) {$=$};
\fill[white] (2.4,0) rectangle (3.3,1.2);
\fill[gray1] (3.3,0) rectangle (4.2,1.2);
\draw (3.3,0) -- (3.3,1.2);
\draw[dashed] (2.4,0) rectangle (4.2,1.2);
\node at (4.5,.6) {$=$};
\fill[white] (4.8,0) rectangle (5.1,1.2);
\fill[gray1] (5.1,0.6) rectangle (6.3,1.2);
\fill[white] (5.1,0.6) rectangle (6.3,0);
\fill[gray1] (6.3,0) rectangle (6.6,1.2);

\filldraw[gray1] (5.1,0.6) arc (-180:0:.3);
\filldraw[white] (5.7,0.6) arc (180:0:.3);

\draw (5.1,0.6) arc (-180:0:.3);
\draw (5.7,0.6) arc (180:0:.3);
\draw (6.3,0) -- (6.3,.6);
\draw (5.1,.6) -- (5.1,1.2);

\draw[dashed] (4.8,0) rectangle (6.6,1.2);
\end{tikzpicture}
\qquad
\begin{tikzpicture}
\fill[white] (0,0) rectangle (0.3,1.2);
\fill[white] (0.3,0.6) rectangle (1.5,1.2);
\fill[gray1] (0.3,0.6) rectangle (1.5,0);
\fill[gray1] (1.5,0) rectangle (1.8,1.2);

\filldraw[white] (.9,0.6) arc (-180:0:.3);
\filldraw[gray1] (.3,0.6) arc (180:0:.3);

\draw[green1,thick] (.9,0.6) arc (-180:0:.3);
\draw[green1,thick] (.3,0.6) arc (180:0:.3);
\draw[green1,thick] (.3,0) -- (.3,.6);
\draw[green1,thick] (1.5,.6) -- (1.5,1.2);

\draw[dashed] (0,0) rectangle (1.8,1.2);
\node at (2.1,.6) {$=$};
\fill[white] (2.4,0) rectangle (3.3,1.2);
\fill[gray1] (3.3,0) rectangle (4.2,1.2);
\draw[green1,thick] (3.3,0) -- (3.3,1.2);
\draw[dashed] (2.4,0) rectangle (4.2,1.2);
\node at (4.5,.6) {$=$};
\fill[white] (4.8,0) rectangle (5.1,1.2);
\fill[gray1] (5.1,0.6) rectangle (6.3,1.2);
\fill[white] (5.1,0.6) rectangle (6.3,0);
\fill[gray1] (6.3,0) rectangle (6.6,1.2);

\filldraw[gray1] (5.1,0.6) arc (-180:0:.3);
\filldraw[white] (5.7,0.6) arc (180:0:.3);

\draw[green1,thick] (5.1,0.6) arc (-180:0:.3);
\draw[green1,thick] (5.7,0.6) arc (180:0:.3);
\draw[green1,thick] (6.3,0) -- (6.3,.6);
\draw[green1,thick] (5.1,.6) -- (5.1,1.2);

\draw[dashed] (4.8,0) rectangle (6.6,1.2);
\end{tikzpicture}
\]
\item $d$-fairness:
\[
\begin{tikzpicture}
\fill[white] (0,0) rectangle (1.8,1.2);
\filldraw[gray1] (.6,.6) arc (0:360:-.3);
\draw (.6,.6) arc (0:360:-.3);
\draw[dashed] (0,0) rectangle (1.8,1.2);

\node at (2.3,.6) {$=d\cdot$};

\fill[white] (2.7,0) rectangle (4.5,1.2);
\draw[dashed] (2.7,0) rectangle (4.5,1.2);
\fill[gray1] (6,0) rectangle (7.8,1.2);
\filldraw[white] (6.6,.6) arc (0:360:-.3);
\draw (6.6,.6) arc (0:360:-.3);
\draw[dashed] (6,0) rectangle (7.8,1.2);

\node at (8.3,.6) {$=d\cdot$};

\fill[gray1] (8.7,0) rectangle (10.5,1.2);
\draw[dashed] (8.7,0) rectangle (10.5,1.2);

\end{tikzpicture}
\]
\item Dagger structure on $\ev$ and $\ev^{\text{st}}$.
\[ \left(\,\begin{tikzpicture}[baseline={([yshift=-\the\dimexpr\fontdimen22\textfont2\relax]current bounding box.center)}]
\fill[white] (0,0) rectangle (1.8,1);
\filldraw[gray1] (.4,1) arc (-180:0:.5);

\draw (.4,1) arc (-180:0:.5);

\draw[dashed] (0,0) rectangle (1.8,1);
\end{tikzpicture} \,\right)^\dagger
=\ 
\begin{tikzpicture}[baseline={([yshift=-\the\dimexpr\fontdimen22\textfont2\relax]current bounding box.center)}]
\fill[white] (0,0) rectangle (1.8,1);
\filldraw[gray1] (0.4,0) arc (180:0:.5);

 \draw (0.4,0) arc (180:0:.5);

\draw[dashed] (0,0) rectangle (1.8,1);
\end{tikzpicture} \qquad
\left(\,\begin{tikzpicture}[baseline={([yshift=-\the\dimexpr\fontdimen22\textfont2\relax]current bounding box.center)}]
\fill[gray1] (0,0) rectangle (1.8,1);
\filldraw[white] (.4,1) arc (-180:0:.5);

\draw[green1,thick] (.4,1) arc (-180:0:.5);

\draw[dashed] (0,0) rectangle (1.8,1);
\end{tikzpicture} \,\right)^\dagger
=\ 
\begin{tikzpicture}[baseline={([yshift=-\the\dimexpr\fontdimen22\textfont2\relax]current bounding box.center)}]
\fill[gray1] (0,0) rectangle (1.8,1);
\filldraw[white] (0.4,0) arc (180:0:.5);

\draw[green1,thick] (0.4,0) arc (180:0:.5);

\draw[dashed] (0,0) rectangle (1.8,1);
\end{tikzpicture}\]
\end{compactenum}
\end{notation}

\subsection{The 2-subcategory of \textsf{BigHilb} generated by a balanced $d$-fair bipartite graph}\label{2-Hilb and edge weighting}
In this section, we show the relation between 2-categories $\mathcal{C}(K,\ev_K)$ and $d$-fair bipartite graphs $(\Lambda,\omega)$. Then we may regard the generator $K$ as a \textsf{Hilb}-enriched graph, and the edge-weighting $\omega$ giving the interesting dual pair.

\begin{construction}\label{construction: graph to BigHilb}
First, we construct a $\rm W^*$ 2-subcategory $\mathcal{C}(\Lambda,\omega)$ of \textsf{BigHilb} from a balanced $d$-fair bipartite graph $(\Lambda,\omega)$ as follows:
\begin{compactenum}[(a)]
\item Object is $V=V(\Lambda)=V_0\sqcup V_1$, which is a countable set.
\item The 1-morphism generator $K=K_{\Lambda}$: At $(P,Q)\in V_0\times V_1$, $K_{PQ}$ is the Hilbert space with ONB $\{|e\rangle: e\in E(\Lambda), s(e)=P,t(e)=Q\}$ and other entries are 0. The uniform boundedness condition follows from Proposition \ref{uniform boundedness for graph}.

As for the dual 1-morphism $\overline{K}$, at entry $(Q,P)\in V_1\times V_0$, $\overline{K}_{QP}$ is the Hilbert space with ONB $\{|e\rangle: e\in E(\Lambda), s(e)=Q,t(e)=P\}=\{|\overline{e}\rangle: e\in E(\Lambda), s(e)=P,t(e)=Q\}$, where $\overline{(\cdot)}$ is the involution of edge.

So we may regard $K$ as a \textsf{Hilb}-enriched graph.

\item All the 1-morphisms are Cauchy generated by $K$ and $\overline{K}$.

\item 2-morphisms are $V\times V$-bigraded uniformly bounded operators between those 1-morphisms.

\item The edge-weighting gives the distinguished evaluation and coevaluation  $\ev$ and $\coev$. 
Note that $K_{PQ}$ is a Hilbert space with orthonormal basis $\{|e\rangle:e\in E(\Lambda), s(e)=P,t(e)=Q\}$, then $\{|\bar{e}\rangle : e\in E(\Lambda), s(e)=P,t(e)=Q\}$ is an orthonormal basis for $\overline{K}_{QP}$. Define 
\begin{align*}
    C_{\overline{K},PQ}: K_{PQ}\otimes \overline{K}_{QP}\to \mathbb{C} \quad &\text{by}\quad |e\rangle\otimes |\overline{e'}\rangle \mapsto \delta_{e=e'}w(e)^{\frac{1}{2}}, \ e:P\to Q\\
    D_{K,PQ}: \mathbb{C}\to K_{PQ}\otimes \overline{K}_{QP} \quad &\text{by}\quad 1\mapsto \sum_{e:P\to Q} w(e)^{\frac{1}{2}}|e\rangle\otimes |\overline{e}\rangle = \sum_{e:Q\to P} w(\overline{e})^{\frac{1}{2}}|\overline{e}\rangle\otimes |e\rangle.\\
    C_{K,QP}: \overline{K}_{QP}\otimes K_{PQ}\to \mathbb{C} \quad &\text{by}\quad |e\rangle\otimes |\overline{e'}\rangle \mapsto \delta_{e=e'}w(e)^{\frac{1}{2}},\ e:Q\to P\\
    D_{\overline{K},QP}: \mathbb{C}\to \overline{K}_{QP}\otimes K_{PQ} \quad &\text{by}\quad 1\mapsto \sum_{e:Q\to P} w(e)^{\frac{1}{2}}|e\rangle\otimes |\overline{e}\rangle = \sum_{e:P\to Q} w(\overline{e})^{\frac{1}{2}}|\overline{e}\rangle\otimes |e\rangle.
\end{align*}
\end{compactenum}
\end{construction}

\begin{proposition}
$\mathcal{C}(\Lambda,\omega)$ satisfies the condition in Definition \ref{C(K,w) def}.
\end{proposition}
\begin{proof} We shall prove that $C(\Lambda,\omega)$ is rigid and $d$-fair.
\begin{compactenum}[(a)]
\item Rigidity: For each $P,Q\in V$, $e:P\to Q$,
\begin{align*}
    (C_{\overline{K},PQ}\otimes \id_{K,{PQ}})\circ (\id_{K,PQ}\otimes D_{\overline{K},QP})(|e\rangle\otimes 1) & = (C_{\overline{K},PQ}\otimes \id_{K,{PQ}})\left(|e\rangle \otimes \sum_{e:P\to Q}w(\bar{e})^{\frac{1}{2}}|\bar{e}\rangle\otimes |e\rangle\right)\\
    & = w(e)^{\frac{1}{2}}w(\bar{e})^{\frac{1}{2}}|e\rangle = |e\rangle,\\
    (\id_{K,{PQ}}\otimes C_{K,QP})\circ (D_{K,PQ}\otimes \id_{K,{QP}}) (1\otimes |e\rangle) & = (\id_{K,{PQ}}\otimes C_{K,QP})\left(\sum_{e:P\to Q}w(e)^{\frac{1}{2}}|e\rangle\otimes |\bar{e}\rangle\otimes |e\rangle\right)\\
    & = w(e)^{\frac{1}{2}}w(\bar{e})^{\frac{1}{2}}|e\rangle = |e\rangle.
\end{align*}
\item $d$-fairness: 
\begin{align*}
    \sum_{Q\in V_1}C_{\overline{K},PQ}\circ D_{K,PQ}(1) &=\sum_{Q\in V}C_{\overline{K},PQ}\left(\sum_{e:P\to Q}w(e)^{\frac{1}{2}}|e\rangle\otimes |\bar{e}\rangle\right)=\sum_{\{e|s(e)=P\}}w(e)^{\frac{1}{2}}w(e)^{\frac{1}{2}}=d,;\\
    \sum_{P\in V_0}C_{K,QP}\circ D_{\overline{K},QP}(1) &=\sum_{a\in V}C_{K,QP}\left(\sum_{e:Q\to P}w(e)^{\frac{1}{2}}|e\rangle\otimes |\bar{e}\rangle\right)=\sum_{\{e|s(e)=Q\}}w(e)^{\frac{1}{2}}w(e)^{\frac{1}{2}}=d.
\end{align*}
\end{compactenum}
\end{proof}

\begin{remark}\label{Rmk:graph automorphism and C(Lambda,omega)}
Suppose $\theta:(\Lambda,\omega)\to (\Lambda',\omega')$ is an isomorphism of edge-weighted graphs (see Definition \ref{Def:graph automorphism}). 
We construct a unitary equivalence between $\mathcal{C}(\Lambda,\omega)$ and $\mathcal{C}(\Lambda',\omega')$.
For the 1-morphism generators $K_\Lambda$ and $K_{\Lambda'}$, we have
$$K_{\Lambda,PQ}\cong K_{\Lambda',\theta(P)\theta(Q)}$$
as finite dimensional Hilbert spaces, via the bijection of ONBs given by $|e\rangle\mapsto |\theta(e)\rangle$. 
Denote by $u_{\theta}:K_{\Lambda}\to K_{\Lambda'}$ this unitary isomorphism.

As for the evaluation $\ev_{K_\Lambda}$ and $\ev_{K_{\Lambda'}}$, we look at $C_{K_{\Lambda},PQ}$ and $C_{K_{\Lambda'},\theta(P)\theta(Q)}$ (see Definition \ref{Def:C,D for ev,coev}). Note that 
$C_{K_{\Lambda'},\theta(P)\theta(Q)}:\overline{K}_{\Lambda',\theta(Q)\theta(P)}\otimes {K}_{\Lambda',\theta(P)\theta(Q)}\to \mathbb{C}$ by 
$$|\theta(e)\rangle\otimes |\theta(e')\rangle \mapsto \delta_{\theta(e)=\theta(e')}\omega'(\theta(e))=\delta_{e=e'}\omega(e),\qquad\qquad \forall\, e:Q\to P\in E(\Lambda).$$ 
We have
$$C_{K_{\Lambda'},\theta(P)\theta(Q)}=C_{K_{\Lambda},PQ}\circ (\overline{u_{\theta}}_{QP}^\dagger\otimes {u_\theta}_{PQ}^\dagger).$$
In other words, 
$$\ev_{K_{\Lambda'}} = \ev_{K_\Lambda}\circ (\overline{u_\theta}^\dagger\otimes u_\theta^\dagger).$$
Therefore, $\mathcal{C}(\Lambda,\omega)$ and $\mathcal{C}(\Lambda',\omega')$ are unitary equivalent up to the unitary 2-morphism $u_\theta$.
\end{remark}

Next, start with a 2-category $\mathcal{C}(K,\ev_K)$, we construct a balanced $d$-fair bipartite graph $(\Lambda,\omega)$.

\begin{definition}
For $P\in V_0,\ Q\in V_1$, let $v_{PQ}:K_{PQ}\to \overline{K_{PQ}}=\overline{K}_{QP}$ be the canonical dual map that $\xi\mapsto \xi^*$ and $v^\dagger_{PQ}:\overline{K}_{QP}\to K_{PQ}$ defined by $\xi^*\to \xi^{**}=\xi$. Then $v^\dagger_{PQ}\circ v_{PQ}=\id_{K,{PQ}}$ and $v_{PQ}\circ v^\dagger_{PQ}=\id_{\overline{K},{QP}}$. Define 
\begin{align*}
    \varphi_{K,PQ}:\overline{K}_{QP}\to K_{PQ} &\text{ by }\varphi_{K,PQ}=(\id_{K,{PQ}}\otimes C_{K,QP}^\text{st})\circ (D_{K,PQ}\otimes v^\dagger_{PQ})\\
    \varphi_{\overline{K},QP}:K_{PQ}\to \overline{K}_{QP} &\text{ by }\varphi_{\overline{K},QP}=(\id_{\overline{K},{QP}}\otimes C_{\overline{K},PQ}^\text{st})\circ (D_{\overline{K},QP}\otimes v^\dagger_{PQ}).
\end{align*}
\end{definition}

\begin{proposition} \label{get edge weight from C}
Here are some properties for $\varphi_K$ and $\varphi_{\overline{K}}$.
\begin{compactenum}[\rm (1)]
\item $\varphi_{K,PQ}\circ \varphi_{\overline{K},QP}=\id_{K,{PQ}}$.
\item $\sum_{Q\in V_1}\Tr(\varphi_{K,PQ}^\dagger\circ\varphi_{K,PQ})=\sum_{P\in V_0}\Tr(\varphi_{\overline{K},QP}^\dagger\circ\varphi_{\overline{K},QP})=d$.
\end{compactenum}
\end{proposition}
\begin{proof} 
See \cite[Prop.~1.8]{DY15}, \cite[Prop.~3.10]{FP19}.
\end{proof}

\begin{construction}\label{graph Lambda to edge weighting}
Define the graph $\Lambda$ to be $V(\Lambda):=V$ and the number of edges from $P\in V_0$ to $Q\in V_1$ to be $\dim K_{PQ}$. Define edge-weighting function $\omega:E(\Lambda)\to(0,\infty)$ as the multiset
\begin{align*}
    \{\omega(e)\}_{e:P\to Q}&:=\{\text{eigenvalues of } \varphi_{K,PQ}\circ\varphi_{K,PQ}^\dagger\}\\
    \{\omega(e)\}_{e:Q\to P}&:=\{\text{eigenvalues of } \varphi_{\overline{K},QP}\circ\varphi_{\overline{K},QP}^\dagger\}.
\end{align*}
\end{construction}

From above Proposition \ref{get edge weight from C}, $(\Lambda,\omega)$ is a $d$-fair and balanced bipartite graph. 
To be precise, (1) gives the balance condition and (2) gives the $d$-fairness. 
In fact, 
\begin{align*}
    \varphi_{K,PQ}\circ \varphi^\dagger_{K,PQ} = & (\id_{K,{PQ}}\otimes C_{K,QP}^\text{st})\circ (D_{K,PQ}\otimes \id_{K,{PQ}})\circ \\ 
    & (C_{\overline{K},PQ}\otimes \id_{K,{PQ}})\circ (\id_{K,{PQ}}\otimes D_{\overline{K},QP}^\text{st})\\
    \varphi_{\overline{K},QP}\circ \varphi^\dagger_{\overline{K},QP} = & (\id_{\overline{K},{QP}}\otimes C_{\overline{K},PQ}^\text{st})\circ (D_{\overline{K},QP}\otimes \id_{\overline{K},{QP}})\circ \\ 
    & (C_{K,QP}\otimes \id_{\overline{K},{QP}})\circ (\id_{\overline{K},{QP}}\otimes D_{K,PQ}^\text{st}).
\end{align*}

\[
\begin{tikzpicture}[baseline=0cm]
\draw (0,.6) -- (0,.9);
\draw (0,-.6) -- (0,-.9);
\draw[green1,thick] (1.2,.6) -- (1.2,-.6);

\draw (0,.6) arc (-180:0:.3);
\draw[green1,thick] (.6,.6) arc (180:0:.3);
\draw (0,-.6) arc (180:0:.3);
\draw[green1,thick] (.6,-.6) arc (-180:0:.3);

\node at (-.3,0) {$P$};
\node at (1.5,0) {$Q$};
\end{tikzpicture}
\qquad\qquad
\begin{tikzpicture}[baseline=0cm]
\draw (0,.6) -- (0,.9);
\draw (0,-.6) -- (0,-.9);
\draw[green1,thick] (1.2,.6) -- (1.2,-.6);

\draw (0,.6) arc (-180:0:.3);
\draw[green1,thick] (.6,.6) arc (180:0:.3);
\draw (0,-.6) arc (180:0:.3);
\draw[green1,thick] (.6,-.6) arc (-180:0:.3);

\node at (-.3,0) {$Q$};
\node at (1.5,0) {$P$};
\end{tikzpicture}
\]

\begin{remark}\label{Rmk:choice of ONB}
For a given 2-category $\mathcal{C}(K,\ev_K)$, let $(\Lambda,\omega)$ be the balanced $d$-fair bipartite graph obtained from Construction \ref{graph Lambda to edge weighting}. 
When we construct the 1-morphism generator $K=K_\Lambda$ in $\mathcal{C}(\Lambda,\omega)$ from the bipartite graph $\Lambda$, 
we secretly make a choice of ONB for each $(K_\Lambda)_{PQ}$, 
so there is a unitary 2-morphism $\alpha:K\to K_\Lambda$ such that $\ev_K=\ev_{K_\Lambda}\circ (\overline{\alpha}\otimes \alpha)$. 
Therefore, $\mathcal{C}(K,\ev_K)$ and $\mathcal{C}(\Lambda,\omega)$ are unitary equivalent up to a unitary 2-morphism $\alpha$. 
\end{remark}

\subsection{From $\mathcal{C}(K,\ev_K)$ to Markov tower}

\begin{construction}\label{From C(K,w) to MT}
Here, we are going to build a tower of algebra from the 2-category $\mathcal{C}(K,\ev_K)$ discussed above with a chosen point, say $P_0\in V_0$. Let $\mathbb{C}^{|P_0|}$ be a 1-morphism with all the entry being 0 except $(\mathbb{C}^{|P_0|})_{P_0P_0}=\mathbb{C}$. 

Note that $\mathbb{C}^{|P_0|}\otimes K^{\alt\otimes n}$ is a 1-morphism for each $n\in\mathbb{Z}_{\ge 0}$.

Let $M_n=\End\left(\mathbb{C}^{|P_0|}\otimes K^{\alt\otimes n}\right)$ and identify $M_n\ni x$ with $x\otimes \id_{K^?}\in M_{n+1}$, where $K^?=K$ if $2\mid n$, $K^?=\overline{K}$ if $2\nmid n$. We use the graphical calculus to show $M=(M_n)_{n\ge 0}$ is a Markov tower.

\begin{compactenum}[(1)]
\item Element $x\in M_n$:
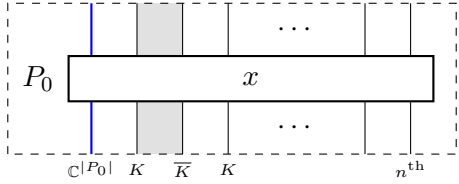
\[ 
\begin{tikzpicture}[baseline={([yshift=-\the\dimexpr\fontdimen22\textfont2\relax]current bounding box.center)}]
\fill[white] (0,0) rectangle (1.2,2);
\fill[gray1] (1.2,0) rectangle (1.8,2);
\fill[white] (1.8,0) rectangle (2.4,2);

\draw (4.8,0) -- (4.8,2);
\draw (4.2,0) -- (4.2,2);
\draw (2.4,0) -- (2.4,2);
\draw (1.8,0) -- (1.8,2);
\draw (1.2,0) -- (1.2,2);
\draw[blue,thick] (0.6,0) -- (.6,2);

\nbox{unshaded}{(2.7,1)}{.3}{2.1}{2.1}{$x$};

\draw[dashed] (-.5,0) rectangle (5.4,2);
\node at (-.1,1) {$P_0$};
\node at (3.3,.35) {$\cdots$};
\node at (3.3,1.65) {$\cdots$};

\node at (.6,-.2) {\tiny{$\mathbb{C}^{|P_0|}$}};
\node at (1.2,-.2) {\tiny{$K$}};
\node at (1.8,-.2) {\tiny{$\overline{K}$}};
\node at (2.4,-.2) {\tiny{$K$}};
\node at (4.8,-.2) {\tiny{$n^{\text{th}}$}};
\end{tikzpicture}
\]
\item Inclusion $x\in M_n\subset M_{n+1}$:
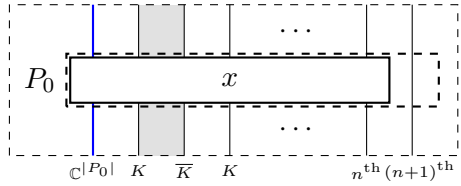
\[
\begin{tikzpicture}[baseline={([yshift=-\the\dimexpr\fontdimen22\textfont2\relax]current bounding box.center)}]
\fill[white] (0,0) rectangle (1.2,2);
\fill[gray1] (1.2,0) rectangle (1.8,2);
\fill[white] (1.8,0) rectangle (2.4,2);

\draw (4.2,0) -- (4.2,2);
\draw (2.4,0) -- (2.4,2);
\draw (1.8,0) -- (1.8,2);
\draw (1.2,0) -- (1.2,2);
\draw[blue,thick] (0.6,0) -- (.6,2);

\nbox{dashed}{(2.7,1)}{.35}{2.1}{2.1}{};
\nbox{unshaded}{(2.7,1)}{.3}{2.1}{1.5}{$x$};
\draw (4.8,0) -- (4.8,2);

\draw[dashed] (-.5,0) rectangle (5.4,2);
\node at (-.1,1) {$P_0$};
\node at (3.3,.35) {$\cdots$};
\node at (3.3,1.65) {$\cdots$};

\node at (.6,-.2) {\tiny{$\mathbb{C}^{|P_0|}$}};
\node at (1.2,-.2) {\tiny{$K$}};
\node at (1.8,-.2) {\tiny{$\overline{K}$}};
\node at (2.4,-.2) {\tiny{$K$}};
\node at (4.2,-.2) {\tiny{$n^{\text{th}}$}};
\node at (4.9,-.2) {\tiny{$(n\+1)^{\text{th}}$}};
\end{tikzpicture}
\]
\item Conditional expectation $E_{n+1}:M_{n+1}\to M_n$, $x\in M_n$:
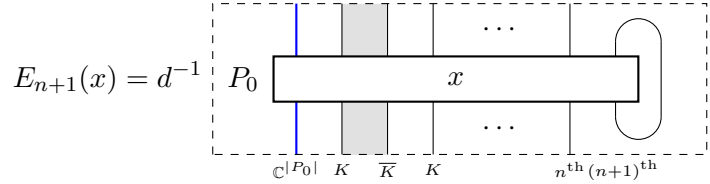
\[ E_{n+1}(x) = d^{-1}\
\begin{tikzpicture}[baseline=.9cm]
\fill[white] (0,0) rectangle (1.2,2);
\fill[gray1] (1.2,0) rectangle (1.8,2);
\fill[white] (1.8,0) rectangle (2.4,2);

\draw (4.8,.5) -- (4.8,1.5);
\draw (4.2,0) -- (4.2,2);
\draw (2.4,0) -- (2.4,2);
\draw (1.8,0) -- (1.8,2);
\draw (1.2,0) -- (1.2,2);
\draw[blue,thick] (0.6,0) -- (.6,2);

\nbox{unshaded}{(2.7,1)}{.3}{2.1}{2.1}{$x$};

\draw (4.8,1.5) arc (180:0:.3);
\draw (4.8,.5) arc (-180:0:.3);
\draw (5.4,.5) -- (5.4,1.5);

\draw[dashed] (-.5,0) rectangle (6,2);
\node at (-.1,1) {$P_0$};
\node at (3.3,.35) {$\cdots$};
\node at (3.3,1.65) {$\cdots$};

\node at (.6,-.2) {\tiny{$\mathbb{C}^{|P_0|}$}};
\node at (1.2,-.2) {\tiny{$K$}};
\node at (1.8,-.2) {\tiny{$\overline{K}$}};
\node at (2.4,-.2) {\tiny{$K$}};
\node at (4.2,-.2) {\tiny{$n^{\text{th}}$}};
\node at (4.9,-.2) {\tiny{$(n\+1)^{\text{th}}$}};
\end{tikzpicture}
\]
Here, the choice of the duality pair $(\coev_K,(\coev_K)^\dagger)$ or $(\ev_K,(\ev_K)^\dagger)$ depends on the shading.\\

\item Jones projection $e_n\in M_{n+1}$:
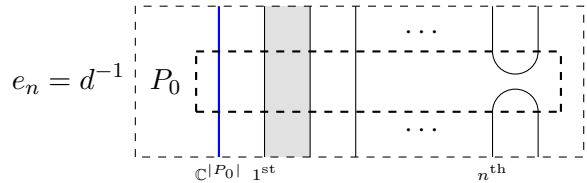
\[
e_n = d^{-1}\
\begin{tikzpicture}[baseline=.9cm]
\fill[white] (0,0) rectangle (1.2,2);
\fill[gray1] (1.2,0) rectangle (1.8,2);
\fill[white] (1.8,0) rectangle (2.4,2);

\draw (4.8,1.4) -- (4.8,2);
\draw (4.2,1.4) -- (4.2,2);
\draw (4.8,0) -- (4.8,.6);
\draw (4.2,0) -- (4.2,.6);
\draw (2.4,0) -- (2.4,2);
\draw (1.8,0) -- (1.8,2);
\draw (1.2,0) -- (1.2,2);
\draw[blue,thick] (0.6,0) -- (.6,2);

\draw (4.2,1.4) arc (-180:0:.3);
\draw (4.2,.6) arc (180:0:.3);

\draw[dashed,thick] (.3,.6) -- (.3,1.4);
\draw[dashed,thick] (.3,.6) -- (5.1,.6);
\draw[dashed,thick] (5.1,.6) -- (5.1,1.4);
\draw[dashed,thick] (.3,1.4) -- (5.1,1.4);

\draw[dashed] (-.5,0) rectangle (5.4,2);
\node at (-.1,1) {$P_0$};
\node at (3.3,.35) {$\cdots$};
\node at (3.3,1.65) {$\cdots$};

\node at (.6,-.2) {\tiny{$\mathbb{C}^{|P_0|}$}};
\node at (1.2,-.2) {\tiny{$1^{\text{st}}$}};
\node at (4.2,-.2) {\tiny{$n^{\text{th}}$}};
\end{tikzpicture}
\]

\item The pull down property is true automatically in this setting. See the diagram \ref{TLJ diagram explanation, module}(MT6).
\end{compactenum}
\end{construction}

\subsection{More properties of Markov tower}
Here, we are going to explore more properties of Markov tower. The tracial version has been proved in \cite[Thm.~4.1.4, Thm.~4.6.3]{GHJ89}\cite[Prop.~3.4]{CHPS18}. For convenience, here we will prove those properties for the traceless case.

\begin{lemma}\label{otb lemma for f.d. inclusion}
Suppose $A\subset B$ is a unital inclusion of finite dimensional ${\rm C}^*$-algebras and $E:B\to A$ is a faithful conditional expectation. 
Then there is an orthonormal basis $\{u_i\}_{i\in I}$ such that $\sum_{i\in I}u_i E(u_i^*x)=x$ for all $x\in B$, where  $|I|<\infty$.
\end{lemma}
\begin{proof} 
Regard $B$ as a right $A$-module equipped with an $A$-valued inner product $\langle x|y\rangle_A:=E(x^*y)$. Note that $A$ and $B$ are finite dimensional, so $B$ is a finitely generated projective Hilbert $A$-module. 
By \cite[Thm.~4.1]{FL02}\cite[Lemma.~1.7]{KW00}, 
there exists an orthonormal basis $\{u_i\}_{i\in I}\subset B$ such that $x=\sum_{i\in I}u_i\langle u_i|x\rangle_A =\sum_{i\in I}u_i E(u_i^*x)$ for all $x\in B$ and $|I|<\infty$. 
\end{proof}

\begin{proposition}\label{Markov Tower prop 2}
\mbox{}
\begin{compactenum}[\rm (1)]
\item $X_{n+1}:=M_ne_nM_n$ is a 2-sided ideal of $M_{n+1}$ and hence $M_{n+1}$  splits as a direct sum of von Neumann algebras $X_{n+1}\oplus Y_{n+1}$. We also define $Y_0=M_0,Y_1=M_1$ so that $X_0=X_1=0$. $X_{n+1}$ is called the old stuff and $Y_{n+1}$ is called the new stuff.
\item $X_{n+1}$ is isomorphic to $M_n\otimes_{M_{n-1}}M_n$, which is the basic construction from $E_n:M_n\to M_{n-1}$. Denote this isomorphism as $\phi$. Here, $M_n\otimes_{M_{n-1}}M_n$ is a $*$-algebra with multiplication $(x_1\otimes y_1)(x_2\otimes y_2)=x_1E_n(y_1x_2)\otimes y_2$ and adjoint $(x\otimes y)^*=y^*\otimes x^*$.
\item If $y\in Y_{n+1}$ and $x\in X_n$, then $yx=0$ in $M_{n+1}$. Hence $E_{n+1}(Y_{n+1})\subset Y_n$, which means the new stuff comes from the old new stuff.
\item If $Y_n=0$, then $Y_k=0$ for all $k\ge n$.
\end{compactenum}
\end{proposition}
\begin{proof}
\item[(1)]
Note that $M_{n+1}e_n=M_ne_n$, then $M_{n+1}M_ne_nM_n\subset M_{n+1}e_nM_n=M_ne_nM_n$ and $M_ne_nM_nM_{n+1}=(M_{n+1}M_ne_nM_n)^*\subset (M_ne_nM_n)^*=M_ne_nM_n$.

\item[(2)]
See Watatani index theory \cite[\S1]{Wa90} with Lemma \ref{otb lemma for f.d. inclusion}.

\item[(3)]
Note that as a finite dimensional von Neumann algebra, $M_{n+1}=\bigoplus_i M_{n+1}p_i$, where $p_i$ are the minimum central projections. So if $y\in Y_{n+1}$, then $y=\sum_j m_jp_j$, where $[p_j,e_n]=0$. 

For $ae_{n-1}b\in X_n$ and $m_jp_j\in Y_{n+1}$, by Jones projection property,
$$m_jp_jae_{n-1}b=d^{-2}m_jp_jae_{n-1}e_ne_{n-1}b=d^{-2}m_jae_{n-1}p_je_ne_{n-1}b=0,$$
so $yx=0$ for any $x\in X_n,\ y\in Y_{n+1}$.

Let $X_n=\bigoplus_k M_nq_k$, where $q_k$ are the minimum central projections. For any $y\in Y_{n+1}$, $q_kE_{n+1}(y)=E_{n+1}(q_ky)=0$ for all $k$, which implies that $E_{n+1}(y)\in Y_n$. 

\item[(4)]
By (3) and faithfulness of $E_n$. 
\end{proof}

\subsection{From Markov tower to $\mathcal{C}(\Lambda,\omega)$}\label{MT to Cb}
Now we are able to extract the so-called principal graph data from the Markov tower, which is similar to the classical tracial Markov tower 
\cite{Oc88}\cite[\S4.2]{JS97}.

If $A$ is a finite dimensional ${\rm C}^*$-algebra, we write $\pi(A)$ to be the set of minimal central projections of $A$. If $A\subset B$ is a unital inclusion of finite dimensional ${\rm C}^*$-algebras, then the inclusion matrix is the $\pi(A)\times \pi(B)$ matrix, with $(p,q)$-th entry being $(\dim_\mathbb{C}(pqA'pq\cap pqBpq))^{\frac{1}{2}}$. If $A\subset B\subset B_1$ is a basic construction, then the inclusion matrix of $B\subset B_1$ is the transpose of the inclusion matrix of $A\subset B$ \cite[\S2]{GHJ89}\cite{JS97}.

The inclusion matrix of $A\subset B$ can be described as the \textbf{Bratteli diagram} of $A\subset B$, whose vertices are the minimal central projections and the number of edges between $p$ and $q$ is the $(p,q)$-th entry. 

The Bratteli diagram $\Delta$ of the Markov tower $M=(M_n)_{n\ge 0}$ contains all the Bratteli diagram $\Delta_n$ of $M_n\subset M_{n+1}$. Then by the property of  inclusion matrix of basic construction and Proposition \ref{Markov Tower prop 2}(2), the Bratteli diagram for $M_n\subset M_{n+1}$ contains the reflection of the Bratteli diagram of $M_{n-1}\subset M_n$ and new part, which is called the \textbf{principal part}. A vertex in the new part is called a \textbf{new vertex}, otherwise, called an \textbf{old vertex}. The reflected vertex from a new vertex is called a \textbf{new old vertex}. Moreover, for a new vertex $p\in Y_n$, denote $p'$ to be the new old vertex of $p$ in $M_{n+2}$.

The \textbf{principal graph} $\Lambda$ contains the new part in the Bratteli diagram $\Delta$, so its vertices are new vertices. To be precise, $V(\Lambda)$ contains all the minimal central projections $p$ in the new stuff. By Proposition \ref{Markov Tower prop 2}(4), the new stuff comes from the old new stuff, then for $p,q\in \Lambda$, $E(\Lambda)$ contains all the edges between $p$ and $q$. 

It is clear that both the Bratteli diagram and the principal graph are  bipartite. We can also use the principal graph to construct the Bratteli diagram by doing the reflection at each level.

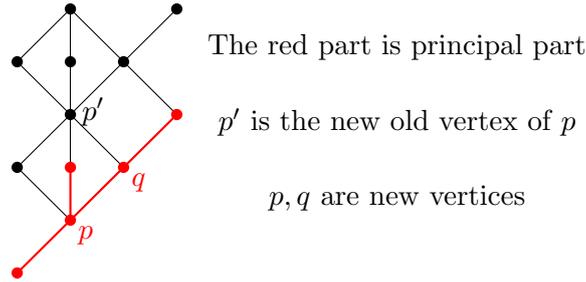
\[
\begin{tikzpicture}
\draw[red,thick] (0,0) -- (2.1,2.1);
\draw[red,thick] (.7,.7) -- (.7,1.4);
\draw (.7,.7) -- (0,1.4);
\draw (0,1.4) -- (2.1,3.5);
\draw (.7,1.4) -- (.7,3.5);
\draw (1.4,1.4) -- (0,2.8);
\draw (2.1,2.1) -- (.7,3.5);
\draw (0,2.8) -- (.7,3.5);

\node at (0,0)[circle,fill,red,inner sep=1.5pt]{};
\node at (0.7,0.7)[circle,fill,red,inner sep=1.5pt]{};
\node at (1.4,1.4)[circle,fill,red,inner sep=1.5pt]{};
\node at (2.1,2.1)[circle,fill,red,inner sep=1.5pt]{};
\node at (.7,1.4)[circle,fill,red,inner sep=1.5pt]{};

\node at (0,1.4)[circle,fill,inner sep=1.5pt]{};
\node at (0,2.8)[circle,fill,inner sep=1.5pt]{};
\node at (0.7,2.1)[circle,fill,inner sep=1.5pt]{};
\node at (0.7,2.8)[circle,fill,inner sep=1.5pt]{};
\node at (0.7,3.5)[circle,fill,inner sep=1.5pt]{};
\node at (1.4,2.8)[circle,fill,inner sep=1.5pt]{};
\node at (2.1,3.5)[circle,fill,inner sep=1.5pt]{};

\node at (.9,.5)[red] {$p$};
\node at (1.6,1.2)[red] {$q$};
\node at (1.0,2.14) {$p'$};

\node at (5,1) {$p,q$ are new vertices};
\node at (5,2) {$p'$ is the new old vertex of $p$};
\node at (5,3) {The red part is principal part};
\end{tikzpicture}
\]

Let us then compute the edge weighting $w:E(\Lambda)\to (0,\infty)$. Before that, we first give a lemma:

\begin{lemma}\label{Lemma: relative commutant in BigHilb}
The follows are some properties for the relative commutant in $\mathsf{BigHilb}$:
\begin{compactenum}[\rm (1)]

\item Let $H_1,H_2,\cdots,H_n,G_1,G_2,\cdots,G_n$ be finite dimensional Hilbert spaces. We identify $B(H_i)$ with $B(H_i)\otimes \id_{G_i}$ and $B(G_i)$ with $\id_{H_i}\otimes B(G_i)$ as subalgebras in $B(\bigoplus_{i=1}^n H_i\otimes G_i)$ for each $i=1,\cdots, n$, then the relative commutant 
\begin{align*}
    \bigcap_{i=1}^n \left(B(H_i)'\cap B\left(\bigoplus_{i=1}^n H_i\otimes G_i\right)\right)=\bigoplus_{i=1}^n B(G_i).\tag{$*$}
\end{align*}
\item Let $H$ be a 1-morphism in $\mathsf{BigHilb}$, then the center  $Z(\End(H))$ is the linear span of all the direct summands of $\id_H$.
\item Let $G$ be another 1-morphism in $\mathsf{BigHilb}$ such that $H\otimes G$ is nondegenerate, i.e., for each nonzero $H_{pq}$, there is a nonzero $G_{qr}$ and vice versa. We identify $\End(H)$ with $\End(H)\otimes \id_G$ and $\End(G)$ with $\id_H\otimes \End(G)$ as subalgebras in $\End(H\otimes G)$. Then the relative commutant 
$$\End(H)'\cap \End(H\otimes G)=Z(\End(H))\otimes \End(G).$$
\item Moreover, if $H_{pq}$ is nonzero only when $p=p_0\in V$, then the relative commutant can be represented as 
$$\End(H)'\cap \End(H\otimes G)=\id_H\otimes \End(G).$$
\end{compactenum}
\textbf{Warning}: the tensor product in (1) is the tensor product of Hilbert spaces and bounded operators; the tensor product in (3) and (4) is the tensor product of 1-morphisms/2-morphisms in $\mathsf{BigHilb}$, see Definition \ref{Def:BigHilb}.
\end{lemma}
\begin{proof}
\item[(1)] $\supset$ is clear. We show $\subset$.

For $f\in B(\bigoplus_{i=1}^n H_i\otimes G_i)$, $f=\bigoplus_{i,j=1}^n f_{i,j}$, where $f_{i,j}\in B(H_i\otimes G_i,H_j\otimes G_j)$. We shall prove that $f_{i,j}=0$ for $i\ne j$ and $f_{i,i}\in \id_{H_i}\otimes B(G_i)$ if $f\in$ LHS of equation ($*$). Let $x_i\in B(H_i)$, then 
$$f(x_i\otimes \id_{G_i})=\bigoplus_{j=1}^n f_{i,j}(x_i\otimes \id_{G_i})=\bigoplus_{k=1}^n (x_i\otimes \id_{G_i})f_{k,i}=(x_i\otimes \id_{G_i})f,$$
which implies that $f_{i,j}(x_i\otimes \id_{G_i})=(x_i\otimes \id_{G_i})f_{k,i}=0$ for $k\ne i,\ j\ne i$ and $f_{i,i}(x_i\otimes \id_{G_i})=(x_i\otimes \id_{G_i})f_{i,i}$. 

From the first half, if we choose $x_i=\id_{H_i}$, we obtain $f_{i,j}=f_{k,i}=0$, $j\ne i,\ k\ne i$; from the second half, from a well-known statement that $B(H_i)'\cap B(H_i\otimes G_i)=B(G_i)$, so that $f_{i,i}\in \id_{H_i}\otimes G_i$.

\item[(2)] Clear, see Definition \ref{Def:Hilb UxV}(d).

\item[(3)] $\supset$ is clear. We show $\subset$.

For $f\in \End(H)'\cap \End(H\otimes G)$, we shall prove that $f_{pq}\in \bigoplus_{r\in V} \id_{H_{pr}}\otimes B(H_{rq})$.

Note that 
$$(\End(H\otimes G))_{pq}=\End((H\otimes G)_{pq})=B\left(\bigoplus_{r\in V}H_{pr}\otimes G_{rq}\right)$$

For $f\in \End(H)'\cap \End(H\otimes G)$, $f_{pq}$ commute with $B(H_{pr})\otimes \id_{G_{rq}}$ for all $r\in V$. By (1), we have $f_{pq}\in \bigoplus_{r\in V} \id_{H_{pr}}\otimes B(H_{rq})$. Together with (2), we prove this statement.

\item[(4)] From (3), for $f\in \End(H)'\cap \End(H\otimes G)$, 
$$f=\bigoplus_{q\in V} \id_{H_{p_0q}}\otimes g^{(q)},$$
where $g^{(q)}\in \End(G)$.

Now we define $g\in \End(G)$ by $g_{ij}:=g^{(i)}_{ij}.$
Then $f=\id_H\otimes g$.
\end{proof}  

By \S\ref{2-Hilb and edge weighting}, we are able to construct a $\rm W^*$ 2-subcategory $\mathcal{C}(\Lambda)$ without providing the distinguished evaluation and coevaluation given by the edge weighting, though we still have the canonical evaluation and coevaluation denoted by $\ev^{\text{st}}$ and $\coev^{\text{st}}$, which are drawn in green below. 
We denote the generators by $K=K_\Lambda$ and $\overline{K}$. 
From Construction \ref{From C(K,w) to MT}, let $N_n:=\text{End}(\mathbb{C}^{|p_0|}\otimes K^{\text{alt}\otimes n})$. 

\begin{notation} \textbf{and Observation}\label{Notation:Kn}
Denote $\Lambda_n$ to be the subgraph of $\Lambda$ with vertices depth $\le n$ and the corresponding \textsf{Hilb}-enriched graph to be $K_n:=K_{\Lambda_n}$ and $\overline{K}_n$ the dual space in the sense of Construction \ref{construction: graph to BigHilb}. As a convention, $p_0$ is of depth $0$.
Observe that 
$$N_n=\End(K_1\otimes \overline{K}_2\otimes K_3\otimes \overline{K}_4\otimes \cdots\otimes K^?_{n}).$$
where $K^?_n=K_n$ if $2\nmid n$, $K^?_n=\overline{K}_n$ if $2\mid n$.
\end{notation}

\begin{example}\label{Ex:A_5 Markov tower}
Let us take $A_5$ graph for example. We label the vertices as follows.
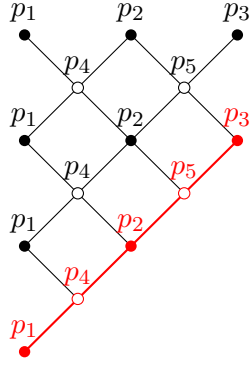
\[
\begin{tikzpicture}
\draw[red,thick] (0,0) -- (2.8,2.8);
\draw (.7,.7) -- (0,1.4);
\draw (0,1.4) -- (2.8,4.2);
\draw (1.4,1.4) -- (0,2.8);
\draw (2.1,2.1) -- (0,4.2);
\draw (0,2.8) -- (1.4,4.2);
\draw (2.8,2.8) -- (1.4,4.2);

\node at (0,0)[circle,fill,red,inner sep=1.5pt]{};
\node at (0.7,0.7)[draw,red,circle,fill=white,inner sep=1.5pt]{};
\node at (1.4,1.4)[circle,fill,red,inner sep=1.5pt]{};
\node at (2.1,2.1)[draw,red,circle,fill=white,inner sep=1.5pt]{};
\node at (2.8,2.8)[circle,fill,red,inner sep=1.5pt]{};

\node at (0,1.4)[circle,fill,inner sep=1.5pt]{};
\node at (0,2.8)[circle,fill,inner sep=1.5pt]{};
\node at (0,4.2)[circle,fill,inner sep=1.5pt]{};
\node at (0.7,2.1)[draw,circle,fill=white,inner sep=1.5pt]{};
\node at (0.7,3.5)[draw,circle,fill=white,inner sep=1.5pt]{};
\node at (1.4,2.8)[circle,fill,inner sep=1.5pt]{};
\node at (1.4,4.2)[circle,fill,inner sep=1.5pt]{};
\node at (2.1,3.5)[draw,circle,fill=white,inner sep=1.5pt]{};
\node at (2.8,4.2)[circle,fill,inner sep=1.5pt]{};

\node at (0,.3)[red] {$p_1$};
\node at (.7,1)[red] {$p_4$};
\node at (1.4,1.7)[red] {$p_2$};
\node at (2.1,2.4)[red] {$p_5$};
\node at (2.8,3.1)[red] {$p_3$};

\node at (0,1.7) {$p_1$};
\node at (0,3.1) {$p_1$};
\node at (0,4.5) {$p_1$};
\node at (.7,2.4) {$p_4$};
\node at (.7,3.8) {$p_4$};
\node at (1.4,3.1) {$p_2$};
\node at (1.4,4.5) {$p_2$};
\node at (2.8,4.5) {$p_3$};
\node at (2.1,3.8) {$p_5$};

\end{tikzpicture}
\]

Then 
$$K_1=\begin{bmatrix}0 & 0 & 0 & \mathbb{C} & 0 \\ 0 & 0 & 0 & 0 & 0 \\ 0 & 0 & 0 & 0 & 0 \\ 0 & 0 & 0 & 0 & 0 \\ 0 & 0 & 0 & 0 & 0\end{bmatrix}\ 
\overline{K}_2=\begin{bmatrix}0 & 0 & 0 & 0 & 0 \\ 0 & 0 & 0 & 0 & 0 \\ 0 & 0 & 0 & 0 & 0 \\ \mathbb{C} & \mathbb{C} & 0 & 0 & 0 \\ 0 & 0 & 0 & 0 & 0\end{bmatrix}\ 
K_3=\begin{bmatrix}0 & 0 & 0 & \mathbb{C} & 0 \\ 0 & 0 & 0 & \mathbb{C} & \mathbb{C} \\ 0 & 0 & 0 & 0 & 0 \\ 0 & 0 & 0 & 0 & 0 \\ 0 & 0 & 0 & 0 & 0\end{bmatrix} 
$$

$$\overline{K}_4=\begin{bmatrix}0 & 0 & 0 & 0 & 0 \\ 0 & 0 & 0 & 0 & 0 \\ 0 & 0 & 0 & 0 & 0 \\ \mathbb{C} & \mathbb{C} & 0 & 0 & 0 \\ 0 & \mathbb{C} & \mathbb{C} & 0 & 0\end{bmatrix}=\overline{K}_{4+2k},\quad 
K_5=\begin{bmatrix}0 & 0 & 0 & \mathbb{C} & 0 \\ 0 & 0 & 0 & \mathbb{C} & \mathbb{C} \\ 0 & 0 & 0 & 0 & \mathbb{C} \\ 0 & 0 & 0 & 0 & 0 \\ 0 & 0 & 0 & 0 & 0\end{bmatrix}=K_{5+2k},\quad k=0,1,2,\cdots$$

$$K_1\otimes \overline{K}_2 \otimes K_3=\begin{bmatrix}0 & 0 & 0 & \mathbb{C}^2 & \mathbb{C} \\ 0 & 0 & 0 & 0 & 0 \\ 0 & 0 & 0 & 0 & 0 \\ 0 & 0 & 0 & 0 & 0 \\ 0 & 0 & 0 & 0 & 0\end{bmatrix}
\qquad\qquad
K_1\otimes \overline{K}_2 \otimes K_3\otimes \overline{K_4} = \begin{bmatrix}\mathbb{C}^2 & \mathbb{C}^3 & \mathbb{C} & 0 & 0 \\ 0 & 0 & 0 & 0 & 0 \\ 0 & 0 & 0 & 0 & 0 \\ 0 & 0 & 0 & 0 & 0 \\ 0 & 0 & 0 & 0 & 0\end{bmatrix}
$$
For this example, observe that $\End(K_1\otimes \overline{K}_2\otimes \cdots\otimes K_n^?)$ is the semisimple quotient of $TLJ_{n}(\sqrt{3})$.

One can regard $\Lambda_n$ as the subgraph of the Bratteli diagram between depth $n-1$ and $n$, and $K_n$ is the \textsf{Hilb}-enriched graph of $\Lambda_n$.
The entry $(i,j)$ in $K_1\otimes \overline{K}_2\otimes \cdots \otimes K^?_n$ indicates the number of paths from the vertex $p_i$ at depth $0$ to the vertex $p_j$ at depth $n$.
Note that the base point is a single vertex $p_1$, so  entry only at $(1,j)$ can be nonzero.
\end{example}

\begin{proposition}\label{Prop: Nn-1' cap Nn+1}
$$N_{n-1}'\cap N_{n+1}=
\begin{cases} 
\id_{K_1\otimes \overline{K}_2\otimes \cdots \otimes K_{2k-1}}\otimes \End(\overline{K}_{2k}\otimes K_{2k+1}) & n=2k\\ 
\id_{K_1\otimes \overline{K}_2\otimes \cdots \otimes \overline{K}_{2k}}\otimes \End(K_{2k+1}\otimes \overline{K}_{2k+2}) & n=2k+1. 
\end{cases}$$
\end{proposition}
\begin{proof}
Note that $K_1\otimes \overline{K}_2\otimes \cdots \otimes K_n^?$ satisfies the condition in Lemma \ref{Lemma: relative commutant in BigHilb}(3) and (4).
\end{proof}

The idea is to transport the Jones projections from the Markov tower $(M_n)$ to the endomorphism algebras $(N_n)$ in order to obtain the edge weighting $\omega$. 
Let $\psi_n:M_n\to N_n$ be a $*$-algebra isomorphism for each $n\ge 0$ with $\psi_{n+1}|_{M_n}=\psi_n$.

Let us consider the image of Jones projection $\psi(e_n)\in N_{n+1}$. 
Note that $e_n\in M_{n-1}'\cap M_{n+1}$, so $\psi(e_n)\in N_{n-1}'\cap N_{n+1}$.

\begin{proposition}
WLOG, let $n=2k$. There exists a projection $\varepsilon_{2k}\in\End(\overline{K}_{2k}\otimes K_{2k+1})$ such that $\psi(e_{2k})=\id_{K_1\otimes \overline{K}_2\otimes \cdots \otimes K_{2k-1}}\otimes \varepsilon_{2k}$.
\end{proposition}
\begin{proof}
By proposition \ref{Prop: Nn-1' cap Nn+1}, there exists  $\varepsilon_{2k}\in\End(\overline{K}_{2k}\otimes K_{2k+1})$ such that $\psi(e_{2k})=\id_{K_1\otimes \overline{K}_2\otimes \cdots \otimes K_{2k-1}} \otimes \varepsilon_{2k}$. Note that $e_{2k}$ is a projection, so is $\varepsilon_{2k}$.
\end{proof}

\begin{lemma}\label{Lemma:Jones projection split}
Let $H$ be a Hilbert space and $p\ne 0$ be a projection on $H$. Suppose $pfp\in \mathbb{C} p$ for all $f\in B(H)$, then $p=r^*r$, where $r:H\to \mathbb{C}$ and $rr^*=1$. 

Similarly, let $H$ be a 1-morphism in $\mathsf{BigHilb}$ and $p\ne 0$ be a projection on $H$. Suppose $pfp\in\mathbb{C}p$ for all $f\in\End(H)$, then $p=r^*r$, where $r:H\to \mathbb{C}^{|V|}$ and $rr^*=\mathbb{C}^{|V|}$. 
\end{lemma}
\begin{proof}
For the Hilbert space case: Note that $\text{Im}(fp)$ can be any subspace of $H$ and $\text{Im}(p(fp))=\text{Im}(p)$, 
so $\text{Im}(p)$ does not depend on the input, i.e., $p$ facts through $\mathbb{C}$. Let $r:H\to \mathbb{C}$ and $p=r^*r$ with $rr^*=1$, since $p^*=p=p^*p$.

The similar argument on 1-morphisms in \textsf{BigHilb}.
\end{proof}

As we see the construction of Jones projection in Construction \ref{From C(K,w) to MT}(4), we shall prove that the Jones projection splits into two pieces.

By Proposition \ref{Markov tower properties}(6), $e_nM_{n+1}e_n=M_{n-1}e_n$, so $\psi(e_n)N_{n+1}\psi(e_n)=N_{n-1}e_n$. 
WLOG, let $n=2k$.
For each $f\in \End(\overline{K}_{2k}\otimes K_{2k+1})$, $\id_{K_1\otimes \overline{K}_2\otimes \cdots \otimes K_{2k-1}}\otimes f\in N_{2k+1}$, there exists $x\in N_{2k-1}$ such that 
$$\id_{K_1\otimes \overline{K}_2\otimes \cdots \otimes K_{2k-1}}\otimes (\varepsilon_{2k} f \varepsilon_{2k})=(x\otimes \id_{\overline{K}_{2k}\otimes K_{2k+1}})(\id_{K_1\otimes \overline{K}_2\otimes \cdots \otimes K_{2k-1}}\otimes \varepsilon_{2k})=x\otimes \varepsilon_{2k},$$
which follows that $\varepsilon_{2k} f \varepsilon_{2k} \in \mathbb{C}\varepsilon_{2k}$.

By Lemma \ref{Lemma:Jones projection split}, there exists $r_{2k}: \overline{K}_{2k}\otimes K_{2k+1}\to \mathbb{C}^{|V_{1, 2k-1}|}$ such that $$\varepsilon_{2k}=r_{2k}^\dagger r_{2k}\qquad\text{and}\qquad r_{2k}r_{2k}^\dagger = \mathbb{C}^{|V_{1, 2k-1}|},$$
where $V_{1, 2k+1}$ contains all the simple objects in $\Lambda_{2k+1}$ with odd depth.

Similarly, we can define $\varepsilon_{2k+1}\in\End(K\otimes \overline{K})$ corresponding to Jones projection $e_{2k+1}$ and there exists $r_{2k+1}:K_{2k+1}\otimes \overline{K}_{2k+2}\to \mathbb{C}^{|V_{0,2k}|}$ such that 
$$\varepsilon_{2k+1}=r_{2k+1}^\dagger r_{2k+1}\qquad \text{and}\qquad r_{2k+1}r_{2k+1}^\dagger=\mathbb{C}^{|V_{0,2k}|},$$
where $V_{0, 2k}$ contains all the simple objects in $\Lambda_{2k}$ with even depth.

Now consider $u_{2k}:=d(\id_{\overline{K}}\otimes r_{2k+1})\circ (r_{2k}^\dagger\otimes \id_K)\in\End(\overline{K})$. Note that $e_{2k}e_{2k+1}e_{2k}=d^{-2}e_{2k}$ and $e_{2k+1}e_{2k}e_{2k+1}=d^{-2}e_{2k+1}$, we have $u_{2k}^\dagger u_{2k}=\id_{\overline{K}_{2k}}$ and $u_{2k}u_{2k}^\dagger =\id_{\overline{K}_{2k+2}}$, so $u_{2k}$ is a unitary. 

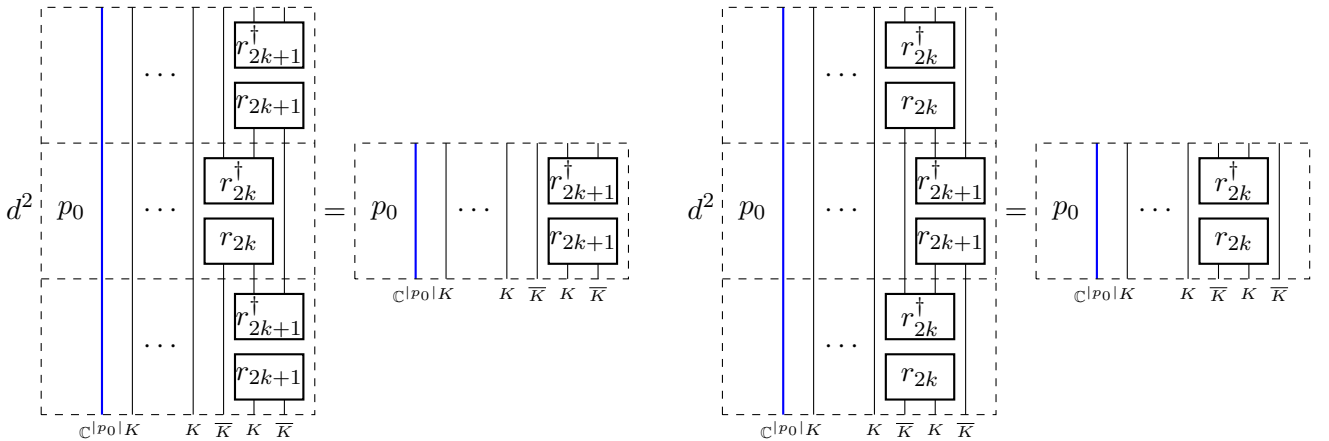
\[\hspace*{-.3cm}
d^2\
\begin{tikzpicture}[baseline = -.1cm]
\draw[blue,thick] (0,-2.7) -- (0,2.7);
\draw (.4,-2.7) -- (.4,2.7);
\draw (1.2,-2.7) -- (1.2,2.7);
\draw (1.6,-2.7) -- (1.6,-.4);
\draw (1.6,.4) -- (1.6,2.7);
\draw (2,-2.7) -- (2,-2.2);
\draw (2,-1.4) -- (2,-.4);
\draw (2,.4) -- (2,1.4);
\draw (2,2.2) -- (2,2.7);
\draw (2.4,-2.7) -- (2.4,-2.2);
\draw (2.4,-1.4) -- (2.4,1.4);
\draw (2.4,2.2) -- (2.4,2.7);

\nbox{unshaded}{(1.8,.4)}{.3}{.15}{.15}{$r_{2k}^\dagger$};
\nbox{unshaded}{(1.8,-.4)}{.3}{.15}{.15}{$r_{2k}$};

\nbox{unshaded}{(2.2,2.2)}{.3}{.15}{.15}{$r_{2k+1}^\dagger$};
\nbox{unshaded}{(2.2,1.4)}{.3}{.15}{.15}{$r_{2k+1}$};

\nbox{unshaded}{(2.2,-1.4)}{.3}{.15}{.15}{$r_{2k+1}^\dagger$};
\nbox{unshaded}{(2.2,-2.2)}{.3}{.15}{.15}{$r_{2k+1}$};

\draw[dashed] (-.8,-.9) -- (2.8,-.9);
\draw[dashed] (-.8,.9) -- (2.8,.9);
\draw[dashed] (-.8,-2.7) rectangle (2.8,2.7);

\node at (.8,0) {$\cdots$};
\node at (.8,1.8) {$\cdots$};
\node at (.8,-1.8) {$\cdots$};
\node at (-.4,0) {$p_0$};

\node at (0,-2.9) {\tiny{$\mathbb{C}^{|p_0|}$}};
\node at (.4,-2.9) {\tiny{$K$}};
\node at (1.2,-2.9) {\tiny{$K$}};
\node at (1.6,-2.9) {\tiny{$\overline{K}$}};
\node at (2,-2.9) {\tiny{$K$}};
\node at (2.4,-2.9) {\tiny{$\overline{K}$}};
\end{tikzpicture}
=
\begin{tikzpicture}[baseline = -.1cm]
\draw[blue,thick] (0,-.9) -- (0,.9);
\draw (.4,-.9) -- (.4,.9);
\draw (1.2,-.9) -- (1.2,.9);
\draw (1.6,-.9) -- (1.6,.9);
\draw (2,-.9) -- (2,-.4);
\draw (2,.4) -- (2,.9);
\draw (2.4,-.9) -- (2.4,-.4);
\draw (2.4,.4) -- (2.4,.9);

\nbox{unshaded}{(2.2,.4)}{.3}{.15}{.15}{$r_{2k+1}^\dagger$};
\nbox{unshaded}{(2.2,-.4)}{.3}{.15}{.15}{$r_{2k+1}$};

\draw[dashed] (-.8,-.9) rectangle (2.8,.9);

\node at (.8,0) {$\cdots$};
\node at (-.4,0) {$p_0$};

\node at (0,-1.1) {\tiny{$\mathbb{C}^{|p_0|}$}};
\node at (.4,-1.1) {\tiny{$K$}};
\node at (1.2,-1.1) {\tiny{$K$}};
\node at (1.6,-1.1) {\tiny{$\overline{K}$}};
\node at (2,-1.1) {\tiny{$K$}};
\node at (2.4,-1.1) {\tiny{$\overline{K}$}};
\end{tikzpicture}
\qquad
d^2\
\begin{tikzpicture}[baseline = -.1cm]
\draw[blue,thick] (0,-2.7) -- (0,2.7);
\draw (.4,-2.7) -- (.4,2.7);
\draw (1.2,-2.7) -- (1.2,2.7);
\draw (1.6,-2.7) -- (1.6,-2.2);
\draw (1.6,-1.4) -- (1.6,1.4);
\draw (1.6,2.2) -- (1.6,2.7);
\draw (2,-2.7) -- (2,-2.2);
\draw (2,-1.4) -- (2,-.4);
\draw (2,.4) -- (2,1.4);
\draw (2,2.2) -- (2,2.7);
\draw (2.4,-2.7) -- (2.4,-.4);
\draw (2.4,.4) -- (2.4,2.7);

\nbox{unshaded}{(1.8,2.2)}{.3}{.15}{.15}{$r_{2k}^\dagger$};
\nbox{unshaded}{(1.8,1.4)}{.3}{.15}{.15}{$r_{2k}$};

\nbox{unshaded}{(1.8,-1.4)}{.3}{.15}{.15}{$r_{2k}^\dagger$};
\nbox{unshaded}{(1.8,-2.2)}{.3}{.15}{.15}{$r_{2k}$};

\nbox{unshaded}{(2.2,.4)}{.3}{.15}{.15}{$r_{2k+1}^\dagger$};
\nbox{unshaded}{(2.2,-.4)}{.3}{.15}{.15}{$r_{2k+1}$};

\draw[dashed] (-.8,-.9) -- (2.8,-.9);
\draw[dashed] (-.8,.9) -- (2.8,.9);
\draw[dashed] (-.8,-2.7) rectangle (2.8,2.7);

\node at (.8,0) {$\cdots$};
\node at (.8,1.8) {$\cdots$};
\node at (.8,-1.8) {$\cdots$};
\node at (-.4,0) {$p_0$};

\node at (0,-2.9) {\tiny{$\mathbb{C}^{|p_0|}$}};
\node at (.4,-2.9) {\tiny{$K$}};
\node at (1.2,-2.9) {\tiny{$K$}};
\node at (1.6,-2.9) {\tiny{$\overline{K}$}};
\node at (2,-2.9) {\tiny{$K$}};
\node at (2.4,-2.9) {\tiny{$\overline{K}$}};
\end{tikzpicture}
=
\begin{tikzpicture}[baseline = -.1cm]
\draw[blue,thick] (0,-.9) -- (0,.9);
\draw (.4,-.9) -- (.4,.9);
\draw (1.2,-.9) -- (1.2,.9);
\draw (1.6,-.9) -- (1.6,-.4);
\draw (1.6,.4) -- (1.6,.9);
\draw (2,-.9) -- (2,-.4);
\draw (2,.4) -- (2,.9);
\draw (2.4,-.9) -- (2.4,.9);

\nbox{unshaded}{(1.8,.4)}{.3}{.15}{.15}{$r_{2k}^\dagger$};
\nbox{unshaded}{(1.8,-.4)}{.3}{.15}{.15}{$r_{2k}$};

\draw[dashed] (-.8,-.9) rectangle (2.8,.9);

\node at (.8,0) {$\cdots$};
\node at (-.4,0) {$p_0$};

\node at (0,-1.1) {\tiny{$\mathbb{C}^{|p_0|}$}};
\node at (.4,-1.1) {\tiny{$K$}};
\node at (1.2,-1.1) {\tiny{$K$}};
\node at (1.6,-1.1) {\tiny{$\overline{K}$}};
\node at (2,-1.1) {\tiny{$K$}};
\node at (2.4,-1.1) {\tiny{$\overline{K}$}};
\end{tikzpicture}
\]

For adjacent simple objects $p,q\in \Lambda$ with $p$ at depth $n$ and $q$ at depth $n+1$, we shall compute the edge weighting on the edges $e:p\to q$ and $e:q\to p$. 
WLOG, $n=2k$.

Define $\varphi_{2k}$ and $\varphi_{2k+1}$ as follows:
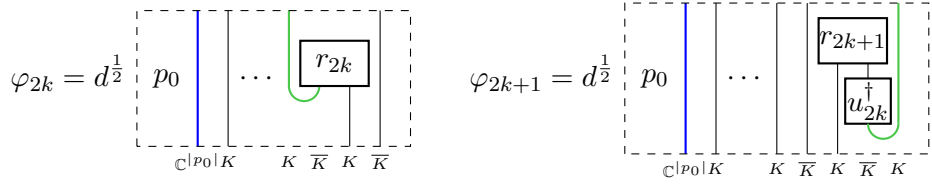
\[
\varphi_{2k} = d^{\frac{1}{2}}\ 
\begin{tikzpicture}[baseline = -.1cm]
\draw[blue,thick] (0,-.9) -- (0,.9);
\draw (.4,-.9) -- (.4,.9);
\draw (2,-.9) -- (2,.2);
\draw (2.4,-.9) -- (2.4,.9);

\draw[green1,thick] (1.6,-.1) arc (0:-180:.2);
\draw[green1,thick] (1.2,-.1) -- (1.2,.9);

\nbox{unshaded}{(1.8,.2)}{.3}{.15}{.15}{$r_{2k}$};

\draw[dashed] (-.8,-.9) rectangle (2.8,.9);

\node at (.8,0) {$\cdots$};
\node at (-.4,0) {$p_0$};

\node at (0,-1.1) {\tiny{$\mathbb{C}^{|p_0|}$}};
\node at (.4,-1.1) {\tiny{$K$}};
\node at (1.2,-1.1) {\tiny{$K$}};
\node at (1.6,-1.1) {\tiny{$\overline{K}$}};
\node at (2,-1.1) {\tiny{$K$}};
\node at (2.4,-1.1) {\tiny{$\overline{K}$}};
\end{tikzpicture}
\qquad
\varphi_{2k+1} = d^{\frac{1}{2}}\
\begin{tikzpicture}[baseline = -.1cm]
\draw[blue,thick] (0,-1) -- (0,1);
\draw (.4,-1) -- (.4,1);
\draw (1.2,-1) -- (1.2,1);
\draw (1.6,-1) -- (1.6,1);
\draw (2,-1) -- (2,.5);
\draw (2.4,-.3) -- (2.4,.5);

\nbox{unshaded}{(2.2,.5)}{.3}{.15}{.15}{$r_{2k+1}$};
\nbox{unshaded}{(2.4,-.3)}{.3}{0}{0}{$u_{2k}^\dagger$};

\draw[green1,thick] (2.8,-.6) arc (0:-180:.2);
\draw[green1,thick] (2.8,-.6) -- (2.8,1);

\draw[dashed] (-.8,-1) rectangle (3.2,1);

\node at (.8,0) {$\cdots$};
\node at (-.4,0) {$p_0$};

\node at (0,-1.2) {\tiny{$\mathbb{C}^{|p_0|}$}};
\node at (.4,-1.2) {\tiny{$K$}};
\node at (1.2,-1.2) {\tiny{$K$}};
\node at (1.6,-1.2) {\tiny{$\overline{K}$}};
\node at (2,-1.2) {\tiny{$K$}};
\node at (2.4,-1.2) {\tiny{$\overline{K}$}};
\node at (2.8,-1.2) {\tiny{$K$}};
\end{tikzpicture}
\]
and we have following properties:
\begin{compactenum}[(1)]
\item $\varphi_{2k+1}\circ \varphi_{2k}^\dagger =\id$.
\item $\Tr(\varphi_{2k}^\dagger\circ \varphi_{2k})=d\Tr(r_{2k}^\dagger r_{2k})=d\Tr(r_{2k}r_{2k}^\dagger)=d$.
\item $\Tr(\varphi_{2k+1}^\dagger\circ \varphi_{2k+1})=d\Tr(u_{2k}r_{2k+1}^\dagger r_{2k+1}u_{2k}^\dagger)=d\Tr(r_{2k+1}r_{2k+1}^\dagger)=d$.
\end{compactenum}

\begin{definition}
Define the edge-weighting function $\omega$ as the multiset:
\begin{align*}
    \{\omega(e)\}_{e:p\to q} &:=\{\text{eigenvalues of }(\varphi_{2k}^\dagger\circ \varphi_{2k})_{pq}\}\\
    \{\omega(e)\}_{e:q\to p} &:=\{\text{eigenvalues of }(\varphi_{2k+1}^\dagger\circ \varphi_{2k+1})_{pq}\}
\end{align*}
\end{definition}

Combining Construction \ref{graph Lambda to edge weighting} and our definition with properties for $\varphi_{2k},\varphi_{2k+1}$, the edge weighting $\omega$ we obtained for bipartite graph $\Lambda$ is $d$-fair and balanced.

\subsection{$\mathcal{C}(K,\ev_K)$ and $\End^\dag_0(\mathcal{M},F)$}\label{C(K) End(M)}

In this section, $\mathcal{TLJ}(d)$ means the 2-shaded pivotal rigid ${\rm C}^*$ multitensor category 
from Definition \ref{TLJ(d) category}
with 
endomorphism spaces the Temperley-Lieb algebras and simple generator $X=1^+\otimes X\otimes 1^-$.

We have already seen the ways to construct a Markov tower from $\mathcal{C}(K,\ev_K)$ in this chapter or from $\mathcal{M}$ in \S\ref{Cpt 2} with a simple base point $Z$, where $\mathcal{M}$ is an indecomposable semisimple ${\rm C}^*$  $\mathcal{TLJ}(d)-$module category. 
In this section, we will show their relation to each other. 


\begin{definition}[Endofunctor monoidal category] \label{End(M)}
Define $\End^\dagger(\mathcal{M})$ to be a $\rm C^*$ tensor category as follows:
\begin{compactenum}[(a)]
\item Objects: The objects are all the dagger endofunctors of $\mathcal{M}$.
\item Morphisms: The morphisms are the uniformly bounded natural transformations between these dagger endofunctors which compatible with the dagger structure.
\item Tensor structure: The tensor product is given by the composition of endofunctors, i.e., $F_1\otimes F_2:=F_2\circ F_1$ for endofunctors $F_1,F_2$.
\end{compactenum}
\end{definition}

\begin{definition}\label{(End(M),F)}
Define $F:=-\lhd X,\ \overline{F}:=-\lhd \overline{X}$, which are endofunctors of $\mathcal{M}$. 
Note that $F$ and $\overline{F}$ are adjoint functors, with unit $\ev_F$ and counit $\coev_F$ induced by $\ev_X$ and $\coev_X$. 

Define $\End^\dag_0(\mathcal{M},F)$ to be the full category Cauchy generated by $F$ and $\overline{F}$. Since the generators are dualizable, the category is rigid.

We warn the reader that $\End^\dag_0(\mathcal{M},F)$ will only be multitensor ($\dim(\End(\operatorname{id}_{\mathcal{M}}))<\infty$) when $\mathcal{M}$ is finitely semisimple.
Moreover, the dual functor on $\End^\dag_0(\mathcal{M},F)$ given by $\ev_F$ and $\coev_F$ is not a unitary dual functor.
\end{definition}


We can give an alternative description of $\End_0^\dag(\mathcal{M},F)$ using the following remark.

\begin{remark}
Let $\mathcal{A}$ be a 2-shaded rigid $\rm C^*$ multitensor category with generator $X$.
The follows are equivalent \cite{GMPPS18}:
\begin{compactenum}[(1)]
\item $\mathcal{M}$ is an indecomposable semisimple $\rm C^*$ right $\mathcal{A}$-module category;
\item there is a faithful dagger tensor functor $\phi:\mathcal{A}\to \End^\dagger(\mathcal{M})$, where $\End^\dagger(\mathcal{M})$ is a tensor category with all the dagger endofunctors being objects and uniformly bounded natural transformations being morphisms.
\end{compactenum}

We see that under this equivalence, $\End^\dagger_0(\mathcal{M},F):=\phi(\mathcal{A})$ is
the $\rm C^*$ category Cauchy tensor generated by the image 
of the tensor functor $\mathcal{TLJ} \to \End^\dag(\mathcal{M})$, where $F=-\lhd X$. 
Then $\End_0^\dagger(\mathcal{M},F)$ is clearly a rigid $\rm C^*$ tensor category.
\end{remark}


As the end of this chapter, we are going to show that the tensor category $\End^\dag_0(\mathcal{M},F)$ and 2-category $\mathcal{C}(K,\ev_K)$ are unitarily equivalent. 


\begin{construction} 
\label{construction:CatFromEndM}
We construct $\mathcal{C}(K,\ev_K)$ from $\End^\dag_0(\mathcal{M},F)$ functorially.
\begin{compactenum}[(a)]
\item Object: Let $V_0$ be a set of representatives of all isomorphism classes of simple objects $P\in\mathcal{M}$ such that $P=P\lhd 1^+$ and $V_1$ a set of representatives of all isomorphism classes of simple objects $Q\in\mathcal{M}$ such that $Q=Q\lhd 1^-$. Then the object is the set $V=V_0\sqcup V_1$.
\item 1-morphism: Let $G\in \End^\dag_0(\mathcal{M},F)$ be an object with adjoint $\overline{G}$. Define the $V\times V-$bigraded Hilbert space $H_G$ by
$$H_{G,PQ}:=\Hom(Q,G(P)),$$
with inner product $\langle f|g\rangle_{G,PQ}$ for $f,g\in \Hom(Q,G(P))$ defined by
$$f^\dagger \circ g=\langle f|g\rangle_{G,PQ} \cdot \id_Q,$$
since $Q$ is simple and $f^\dagger \circ g\in\End(Q)\cong \mathbb{C}\cdot\id_Q$. Note that $\Hom(Q,G(P))\cong \Hom(\overline{G}(Q),P)$ 
is a natural isomorphism, so $H_{\overline{G},QP}$ and $H_{G,PQ}$ are dual Hilbert spaces.
\item Composition of 1-morphisms: 
\begin{proposition}
For $G_1,G_2\in \End^\dag_0(\mathcal{M},F)$, we have $H_{G_1\circ G_2}\cong H_{G_1}\circ H_{G_2}$ as $V\times V-$bigraded Hilbert spaces, i.e., 
$$H_{G_1\circ G_2,PQ}\cong (H_{G_1}\circ H_{G_2})_{PQ}=(H_{G_2}\otimes H_{G_1})_{PQ}=\bigoplus_{R} H_{G_2,PR}\otimes H_{G_1,RQ}.$$
is a unitary isomorphism between Hilbert spaces for each pair $(P,Q)\in V\times V$.
\end{proposition}
\begin{proof} Note that the direct sum contains finite many components. For each nonzero component with respect to $R$, define $\theta_R:H_{G_2,PR}\otimes H_{G_1,RP}\to H_{G_1\circ G_2,PQ}$ by
$$\theta_R(f_2\otimes f_1):= G_1(f_2)\circ f_1.$$
First, we prove that $\theta_R$ is an isometry, i.e., 
$$\langle \theta(f_2\otimes f_1)|\theta(g_2\otimes g_1)\rangle_{G_1\circ G_2,PQ}=\langle f_2\otimes f_1|g_2\otimes g_1\rangle=\langle f_2|g_2\rangle_{G_2,PR}\cdot \langle f_1|g_1\rangle_{G_1,RQ}$$
for $f_2,g_2\in H_{G_2,PR},\ f_1,g_1\in H_{G_1,RQ}$.
\begin{align*}
    LHS & = \langle G_1(f_2)\circ f_1|G_1(g_2)\circ g_1\rangle_{G_1\circ G_2,PQ}\\
    & = (G_1(f_2)\circ f_1)^\dagger\circ (G_1(g_2)\circ g_1)\\
    & = f_1^\dagger \circ G_1(f_2^\dagger \circ g_2)\circ g_1 \tag{$G_1$ is a dagger functor}\\
    & = f_1^\dagger \circ G_1(\langle f_2|g_2\rangle_{G_2,PR}\cdot \id_{R})\circ g_1\\
    & = \langle f_2|g_2\rangle_{G_2,PR}\cdot f_1^\dagger \circ \id_{G_1(R)}\circ g_1 \tag{$G_1$ is a functor}\\
    & = \langle f_2|g_2\rangle_{G_2,PR}\cdot f_1^\dagger \circ g_1\\
    & = RHS.
\end{align*}
It follows that $\bigoplus_R \theta_R:\bigoplus_{R} H_{G_2,PR}\otimes H_{G_1,RQ} \to H_{G_1\circ G_2,PQ}$ is an isometry.

Note that for a semisimple rigid ${\rm C}^*$ category,
\begin{align*}
    \dim H_{G_1\circ G_2,PQ} & = \dim \Hom(Q,G_1\circ G_2(P))\\
    & = \dim \Hom(\overline{G_1}(Q),G_2(P))\\
    & = \dim \bigoplus_R \Hom(\overline{G_1}(Q),R)\otimes \Hom(R,G_2(P))\\
    & = \dim \bigoplus_R \Hom(Q,G_1(R))\otimes \Hom(R,G_2(P))\\
    & = \dim \bigoplus_R H_{G_1,RQ}\otimes H_{G_2,PR}\\
    & = \dim \bigoplus_{R} H_{G_2,PR}\otimes H_{G_1,RQ}.
\end{align*}
Note that $\bigoplus_R \theta_R$ is an isometry and hence injective, so $\bigoplus_R \theta_R:\bigoplus_{R} H_{G_2,PR}\otimes H_{G_1,RQ} \to H_{G_1\circ G_2,PQ}$ is a bijection and hence a unitary.
\end{proof}

It follows that 
$$H_{G_1\circ G_2}\circ H_{G_3}\cong H_{G_1\circ G_2\circ G_3}\cong H_{G_1}\circ H_{G_2\circ G_3}$$ 
as $V\times V-$bigraded Hilbert space.

\item 1-morphism generator: Define $K:=H_F$ and $\overline{K}:=H_{\overline{F}}$. It is clear that $\mathbb{C}^{|V_0|}=H_{I^+}$ and $\mathbb{C}^{|V_1|}=H_{I^-}$.

\item 2-morphism: The 2-morphism of $\mathcal{C}(K)$ is the morphism of $\End^\dag_0(\mathcal{M},F)$. Let $\alpha:G_1\to G_2$ be a uniformly bounded natural transformation. Then $\alpha(P):G_1(P)\to G_2(P)$ and hence 
$$\alpha_{PQ}:=\alpha_P\circ -:H_{G_1,PQ}=\Hom(Q,G_1(P))\to \Hom(Q,G_2(P))=H_{G_2,PQ}$$
is a uniformly bounded linear map.

\item Composition of 2-morphisms: Let $\alpha_1:G_1\to G_2$, $\alpha_2:G_2\to G_3$ be uniformly bounded natural transformations. Then $G_1(P)\xrightarrow{\alpha_1(P)} G_2(P)\xrightarrow{\alpha_2}G_3(P)$, then
$$(\alpha_2\circ \alpha_1)_{PQ}=(\alpha_2\circ \alpha_1)_P\circ -=\alpha_{2,P}\circ \alpha_{1,P}\circ - = \alpha_{2,PQ}\circ \alpha_{1,PQ}:H_{G_1,PQ}\to H_{G_2,PQ}\to H_{G_3,PQ}.$$

\item Tensor product of 2-morphisms: Let $\alpha_1:G_1\to G_2$, $\alpha_2:G_3\to G_4$ be uniformly bounded natural transformation. Then $\alpha_1\otimes \alpha_2:G_3\circ G_1= G_1\otimes G_3\to G_2\otimes G_4= G_4\circ G_2$ defined as 
\[
\begin{tikzpicture}[baseline=.7cm]
\draw[->] (0,1.2) -- (0,.3);
\draw[->] (.8,0) -- (1.7,0);
\draw[->] (2.5,1.2) -- (2.5,.3);
\draw[->] (.8,1.5) -- (1.7,1.5);

\node at (2.5,1.5) {$G_3\circ G_2$};
\node at (0,1.5) {$G_3\circ G_1$};
\node at (2.5,0) {$G_4\circ G_2$};
\node at (0,0) {$G_4\circ G_1$};
\end{tikzpicture}
\quad
\Longrightarrow
\quad
\begin{tikzpicture}[baseline=.7cm]
\draw[->] (0,1.2) -- (0,.3);
\draw[->] (1.1,0) -- (2.1,0);
\draw[->] (3.2,1.2) -- (3.2,.3);
\draw[->] (1.1,1.5) -- (2.1,1.5);

\node at (3.2,1.5) {$H_{G_2}\otimes H_{G_3}$};
\node at (0,1.5) {$H_{G_1}\otimes H_{G_3}$};
\node at (3.2,0) {$H_{G_2}\otimes H_{G_4}$};
\node at (0,0) {$H_{G_1}\otimes H_{G_4}$};
\end{tikzpicture}
\]
Clearly, the tensor product is strict.

\item $\ev_K$ and $\coev_K$: Define $\ev_K$ to be the unit of adjoint pair $(F,\overline{F})$ and $\coev_K$ to be the counit of $(F,\overline{F})$. 
Note that the duality is a property, not an extra structure. 
The dual functor is generated by the duality of generator, which is not necessarily a unitary dual functor.
\end{compactenum}
\end{construction}

\begin{definition}[{\cite[Def.~7.2.1]{EGNO15}}]\label{Def:C-module functor}
Let $\mathcal{M}$ and $\mathcal{N}$ be two semisimple $\rm C^*$ module category categories over a semisimple rigid $\rm C^*$ (multi)tensor category $\mathcal{C}$. A $\mathcal{C}$-module functor from $\mathcal{M}$ to $\mathcal{N}$ consists of a functor $\psi:\mathcal{M}\to \mathcal{N}$ and a natural isomorphism $s_{X,M}:\psi(M\lhd X)\to \psi(M)\lhd X$ for all $X\in \mathcal{C}$, $M\in\mathcal{M}$ which satisfies the pentagon equation. 

We call that $\mathcal{M}$ and $\mathcal{N}$ are $\mathcal{C}$-module equivalent if $\psi$ is an equivalence of categories. 
\end{definition}

Let $\mathcal{C}=\mathcal{TLJ}(d)$. Now we discuss the relation between the equivalence on $\mathcal{TLJ}(d)-$module category and the equivalence on $\End^\dagger_0(\mathcal{M},F)$, where $F=-\lhd X$, and the corresponding 2-category $\mathcal{C}(K,\ev_K)$.

\begin{remark}\label{Rmk:TLJ equivalence}
Let $\mathcal{M}$ be an indecomposable semisimple  $\mathcal{TLJ}(d)-$module $\rm C^*$ categories and $(\psi,s):\mathcal{M}\to \mathcal{M}$ is an $\mathcal{TLJ}(d)-$module equivalence. Then $\psi\in\End^\dagger(M)$ is an object.
Since $\mathcal{TLJ}(d)$ is generated by $X$, $s_{-,-}$ in above Definition \ref{Def:C-module functor} is determined by $s_{X,-}$. 
Note that
$$s_{X,-}:\psi(F(-))=\psi(-\lhd X)\to \psi(-)\lhd X=F(\psi(-))$$
is a unitary natural isomorphism. 
Note that as an equivalence, $\psi$ maps simple objects in $\mathcal{M}$ to simple objects. 
Then we have
$$H_{F,\psi(P)\psi(Q)}=\Hom(\psi(Q),F(\psi(P)))\xrightarrow[-\circ s^{-1}]{\sim} \Hom(\psi(Q),\psi(F(P)))\cong \Hom(Q,F(P))= H_{F,PQ}.$$
It follows that the 1-morphism generator $K=H_{F}$ indexed by $V$ and $H_{F}$ indexed by $\psi(V)$ are unitary equivalent.  

Comparing the discussion here with Remark \ref{Rmk:graph automorphism and C(Lambda,omega)}, the $\mathcal{TLJ}(d)-$module equivalence corresponds to the unitary equivalence on $\mathcal{C}(K,\ev_K)$, which corresponds to isomorphism of edge-weighted graphs $(\Lambda,\omega)$.  
\end{remark}

\begin{theorem}
There is a bijective correspondence between equivalence classes of the following:
\[\left\{\, \parbox{5.4cm}{\rm Indecomposable semisimple $\rm C^*$
$\mathcal{TLJ}(d)-$module categories $\mathcal{M}$} \,\right\}\ \, \cong\ \, \left\{\, \parbox{7cm}{\rm $\rm W^*$ 2-subcategories $\mathcal{C}(\Lambda,\omega)$ of \textsf{BigHilb}, where $\Lambda$ is a balanced $d$-fair bipartite graph with  edge-weighting $\omega$
} \,\right\}\]
Equivalence on the left hand side is unitary equivalence; equivalence on the right hand side is isomorphism of edge-weighted graphs.
\end{theorem}
\begin{proof} We can prove this correspondence for the version with base point by passing through the Markov tower. 
According to Construction \ref{construction:CatFromEndM}, the correspondence holds without fixing the base point. 
As for the equivalence, see Remark \ref{Rmk:TLJ equivalence}.
\end{proof}

\begin{remark}
Given a semisimple $\rm C^*$ category $\mathcal{C}$,
similar to Construction \ref{construction:CatFromEndM}, we get a dagger tensor functor from $\End^\dagger(\mathcal{C})$ to the tensor category
$\textsf{Hilb}^{\operatorname{Irr}(\mathcal{C})\times \operatorname{Irr}(\mathcal{C})}_f$, which is the endomorphism tensor category of the object $\operatorname{Irr}(\mathcal{C})$ in \textsf{BigHilb}.
One should view this as a concrete version of $\End^\dagger(\mathcal{C})$. 
Note that dualizable endofunctors always map to dualizable 1-morphisms. 
\end{remark}

\section{Markov lattices and biunitary connections
}
\label{Cpt 5}
\subsection{Balanced $(d_0,d_1)$-fair square-partite graph}

\begin{definition}\label{Def:square partite graph} Let $\Gamma$ be an oriented square-partite graph with vertices $V(\Gamma)=V_{00}\sqcup V_{01}\sqcup V_{10}\sqcup V_{11}$.
$$\xymatrix{
V_{10} \ar@{-}_{\Lambda_0}[d] &V_{11} \ar@[red]@{-}_{\Omega_0}[l]\ar@{-}^{\Lambda_0}[d]\\
V_{00} & V_{01}\ar@[red]@{-}^{\Omega_1}[l]
}$$
We call that $\Gamma$ \textbf{associative} if for any two vertices on opposite corners of $\Gamma$, there are the same number of length 2 paths going either way around $\Gamma$. In more details,
\begin{itemize}
\item for any $P\in V_{00}$ and $R\in V_{11}$, there are the same number of length 2 paths from $P$ to $R$ (or $R$ to $P$) through vertices $Q\in V_{01}$ and through vertices $S\in V_{01}$;
\item for any $Q\in V_{01}$ and $S\in V_{10}$, there are the same number of length 2 paths from $Q$ to $S$ (or $S$ to $Q$) through vertices $P\in V_{00}$ and through vertices $R\in V_{11}$.
\end{itemize}

Let $\omega:E(\Gamma)\to (0,\infty)$ be a weighting on the edges of graph $\Gamma$. 

Let $\Lambda_i$ denote the full subgraph of $\Gamma$ on $V_{0i}\sqcup V_{1i}$, $i=0,1$; let $\Omega_i$ denote the full subgraph of $\Gamma$ on $V_{i0}\sqcup V_{i1}$, $i=0,1$. Then $\Lambda_1,\Lambda_2,\Omega_1,\Omega_2$ are oriented bipartite graphs.

We call $(\Gamma,\omega)$ a \textbf{balanced $(d_0,d_1)$-fair} square-partite graph if $\Lambda_0,\Lambda_1$ are balanced $d_0$-fair bipartite graphs and $\Omega_0,\Omega_1$ are balanced $d_1$-fair bipartite graphs.
\end{definition}

\begin{remark}
We can define the edge-weighting preserving graph isomorphism literally the same as in Definition \ref{Def:graph automorphism} for balanced $(d_0,d_1)$-fair square partite graph.
\end{remark}







\subsection{2-subcategory $\mathcal{C}(K_0,K_1,L_0,L_1,\ev)$ of \textsf{BigHilb}  and biunitary connection $\Phi$}
\begin{definition}
Let $\mathcal{C}(K_0,K_1,L_0,L_1,\ev)$ be a $\rm W^*$ 2-subcategory of \textsf{BigHilb} with four 1-morphism generators $K_i:V_{0i}\to V_{1i}$, $L_i:V_{i0}\to V_{i1}$, $i=0,1$ and a chosen evaluation and coevaluation for each generator. 
We require that
\begin{compactenum}[(a)]
\item 
$K_i,L_i$ are dualizable, $i=0,1$.
\item The evaluation and coevaluation for the dual:
$$\ev_{\overline{?}}:=(\coev_?)^\dagger\qquad\text{and}\qquad\coev_{\overline{?}}:=(\ev_?)^\dagger,$$
where $?=K_i,L_i$, $i=0,1$.

\item 
They satisfy the $(d_0,d_1)-$fairness condition, namely,
\begin{align*}
    \ev_{\overline{K_0}}\circ \coev_{K_0}= d_0\cdot\id_{\mathbb{C}^{|V_{00}|}} &\qquad
    \ev_{K_0}\circ \coev_{\overline{K_0}}= d_0\cdot\id_{\mathbb{C}^{|V_{10}|}}\\
    \ev_{\overline{K_1}}\circ \coev_{K_1}= d_0\cdot\id_{\mathbb{C}^{|V_{01}|}} &\qquad
    \ev_{K_1}\circ \coev_{\overline{K_1}}= d_0\cdot\id_{\mathbb{C}^{|V_{11}|}}\\
    \ev_{\overline{L_0}}\circ \coev_{L_0}= d_1\cdot\id_{\mathbb{C}^{|V_{00}|}} &\qquad
    \ev_{L_0}\circ \coev_{\overline{L_0}}= d_1\cdot\id_{\mathbb{C}^{|V_{01}|}}\\
    \ev_{\overline{L_1}}\circ \coev_{L_1}= d_1\cdot\id_{\mathbb{C}^{|V_{10}|}} &\qquad
    \ev_{L_1}\circ \coev_{\overline{L_1}}= d_1\cdot\id_{\mathbb{C}^{|V_{11}|}}
\end{align*}
\end{compactenum}
\end{definition}

\begin{notation}
Now, we provide the graphical calculus to describe $\mathcal{C}(K_0,K_1,L_0,L_1,\ev)$. The white region indicates the object $V_{00}$, the lightest gray for $V_{10}$, the medium gray for $V_{11}$ and the darkest gray for $V_{01}$; the black edge indicates $K_0,K_1$ and red for $L_0,L_1$, so white and medium gray, lightest gray and darkest gray will not be adjacent.

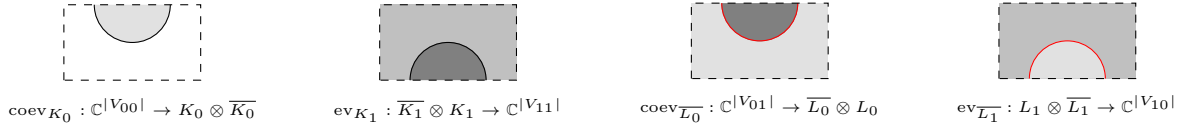
\[
\begin{tikzpicture}
\fill[white] (0,0) rectangle (1.8,1);
\filldraw[gray1] (.4,1) arc (-180:0:.5);

\draw (.4,1) arc (-180:0:.5);

\draw[dashed] (0,0) rectangle (1.8,1);
\node at (.9,-.4) {\tiny{$\coev_{K_0}:\mathbb{C}^{|V_{00}|}\to K_0\otimes\overline{K_0}$}};
\end{tikzpicture}
\qquad
\begin{tikzpicture}
\fill[gray2] (9,0) rectangle (10.8,1);
\filldraw[gray3] (9.4,0) arc (180:0:.5);

\draw (9.4,0) arc (180:0:.5);

\draw[dashed] (9,0) rectangle (10.8,1);
\node at (9.9,-.4) {\tiny{$\ev_{K_1}:\overline{K_1}\otimes K_1\to\mathbb{C}^{|V_{11}|}$}};
\end{tikzpicture}
\qquad
\begin{tikzpicture}
\fill[gray1] (6,0) rectangle (7.8,1);
\filldraw[gray3] (6.4,1) arc (-180:0:.5);

\draw[red] (6.4,1) arc (-180:0:.5);

\draw[dashed] (6,0) rectangle (7.8,1);
\node at (6.9,-.4) {\tiny{$\coev_{\overline{L_0}}:\mathbb{C}^{|V_{01}|}\to \overline{L_0}\otimes L_0$}};
\end{tikzpicture}
\qquad
\begin{tikzpicture}
\fill[gray2] (3,0) rectangle (4.8,1);
\filldraw[gray1] (3.4,0) arc (180:0:.5);

\draw (3.4,0)[red] arc (180:0:.5);

\draw[dashed] (3,0) rectangle (4.8,1);
\node at (3.9,-.4) {\tiny{$\ev_{\overline{L_1}}:L_1\otimes\overline{L_1}\to\mathbb{C}^{|V_{10}|}$}};
\end{tikzpicture}
\]
\end{notation}

\begin{remark}\label{C(K0,K1,L0,L1;ev) and C(Gamma,omega)}
Similar to the discussion in \S\ref{2-Hilb and edge weighting}, from a given balanced $(d_0,d_1)$-fair square-partite graph $(\Gamma,\omega)$, we can construct a 2-subcategory $\mathcal{C}(\Gamma,\omega)$ of \textsf{BigHilb}; on the other hand, if we start with $\mathcal{C}(K_0,K_1,L_0,L_1,\ev)$, we can obtain the $(\Gamma,\omega)$. Moreover, $\mathcal{C}(K_0,K_1,L_0,L_1,\ev)$ and $\mathcal{C}(\Gamma,\omega)$ are unitary equivalent.

Similar to the discussion in Remark \ref{Rmk:graph automorphism and C(Lambda,omega)}, the edge-weighting preserving graph automorphism will result in the unitary equivalence on $\mathcal{C}(\Gamma,\omega)$.
\end{remark}
\vspace{.2cm}

In the rest of this section, we define a special 2-morphism $\Phi$ in $\mathcal{C}(K_0,K_1,L_0,L_1,\ev)$, called \textbf{biunitary connection}.


\begin{definition}
[Biunitary connection]
A \textbf{biunitary connection} $\Phi:K_0\otimes L_1\to L_0\otimes K_1$ is a 2-morphism which is a vertical unitary and a horizontal unitary, as defined as follows. 
Here is the graphical calculus of $\Phi$.
\begin{compactenum}[(1)]
\item The biunitary connection $\Phi$:
\[
\begin{tikzpicture}
\fill[white] (0,0) rectangle (1,2);
\fill[gray2] (1,0) rectangle (2,2);

\filldraw[gray1] (.6,0)  .. controls ++(90:.2cm) and ++(180:.2cm) .. (1,1);
\filldraw[gray1] (1.4,0)  .. controls ++(90:.2cm) and ++(0:.2cm) .. (1,1);
\filldraw[gray1] (.6,0) -- (1,1) -- (1.4,0);

\filldraw[gray3] (.6,2)  .. controls ++(-90:.2cm) and ++(180:.2cm) .. (1,1);
\filldraw[gray3] (1.4,2) .. controls ++(-90:.2cm) and ++(0:.2cm) .. (1,1);
\filldraw[gray3] (.6,2) -- (1,1) -- (1.4,2);

\draw (.6,0)  .. controls ++(90:.2cm) and ++(180:.2cm) .. (1,1);
\draw[red] (1.4,0)  .. controls ++(90:.2cm) and ++(0:.2cm) .. (1,1);
\draw[red] (.6,2)  .. controls ++(-90:.2cm) and ++(180:.2cm) .. (1,1);
\draw (1.4,2) .. controls ++(-90:.2cm) and ++(0:.2cm) .. (1,1);

\nbox{unshaded}{(1,1)}{.3}{0}{0}{$\Phi$};
\draw[dashed] (0,0) rectangle (2,2);
\end{tikzpicture}
\]
\item Vertical unitary: $\Phi^\dagger\circ \Phi=\id_{K_0}\otimes \id_{L_1}$ and $\Phi\circ \Phi^\dagger = \id_{L_0}\otimes \id_{K_1}$.
\[
\begin{tikzpicture}
\fill[white] (0,0) rectangle (1,4);
\fill[gray2] (1,0) rectangle (2,4);

\filldraw[gray1] (.6,0)  .. controls ++(90:.2cm) and ++(180:.2cm) .. (1,1);
\filldraw[gray1] (1.4,0)  .. controls ++(90:.2cm) and ++(0:.2cm) .. (1,1);
\filldraw[gray1] (.6,0) -- (1,1) -- (1.4,0);

\filldraw[gray3] (1,3) arc (90:-90: .3cm and 1cm);
\filldraw[gray3] (1,3) arc (90:-90: -.3cm and 1cm);
\filldraw[gray3] (.8,1) -- (.8,3) -- (1.2,3) -- (1.2,1);

\filldraw[gray1] (.6,4)  .. controls ++(-90:.2cm) and ++(180:.2cm) .. (1,3);
\filldraw[gray1] (1.4,4) .. controls ++(-90:.2cm) and ++(0:.2cm) .. (1,3);
\filldraw[gray1] (.6,4) -- (1,3) -- (1.4,4);

\draw (.6,0)  .. controls ++(90:.2cm) and ++(180:.2cm) .. (1,1);
\draw[red] (1.4,0)  .. controls ++(90:.2cm) and ++(0:.2cm) .. (1,1);
\draw (.6,4)  .. controls ++(-90:.2cm) and ++(180:.2cm) .. (1,3);
\draw[red] (1.4,4) .. controls ++(-90:.2cm) and ++(0:.2cm) .. (1,3);
\draw (1,3) arc (90:-90: .3cm and 1cm);
\draw[red] (1,3) arc (90:-90: -.3cm and 1cm);

\nbox{unshaded}{(1,1)}{.3}{0}{0}{$\Phi$};
\nbox{unshaded}{(1,3)}{.3}{0}{0}{$\Phi^\dagger$};
\draw[dashed] (0,0) rectangle (2,4);

\node at (2.3,2) {$=$};

\fill[white] (2.6,0) rectangle (3.2,4);
\fill[gray1] (3.2,0) rectangle (4,4);
\fill[gray2] (4,0) rectangle (4.6,4);
\draw (3.2,0) -- (3.2,4);
\draw[red] (4,0) -- (4,4);
\draw[dashed] (2.6,0) rectangle (4.6,4);
\end{tikzpicture}
\qquad\qquad
\begin{tikzpicture}
\fill[white] (0,0) rectangle (1,4);
\fill[gray2] (1,0) rectangle (2,4);

\filldraw[gray3] (.6,0)  .. controls ++(90:.2cm) and ++(180:.2cm) .. (1,1);
\filldraw[gray3] (1.4,0)  .. controls ++(90:.2cm) and ++(0:.2cm) .. (1,1);
\filldraw[gray3] (.6,0) -- (1,1) -- (1.4,0);

\filldraw[gray1] (1,3) arc (90:-90: .3cm and 1cm);
\filldraw[gray1] (1,3) arc (90:-90: -.3cm and 1cm);
\filldraw[gray1] (.8,1) -- (.8,3) -- (1.2,3) -- (1.2,1);

\filldraw[gray3] (.6,4)  .. controls ++(-90:.2cm) and ++(180:.2cm) .. (1,3);
\filldraw[gray3] (1.4,4) .. controls ++(-90:.2cm) and ++(0:.2cm) .. (1,3);
\filldraw[gray3] (.6,4) -- (1,3) -- (1.4,4);

\draw[red] (.6,0)  .. controls ++(90:.2cm) and ++(180:.2cm) .. (1,1);
\draw (1.4,0)  .. controls ++(90:.2cm) and ++(0:.2cm) .. (1,1);
\draw (.6,4)  .. controls ++(-90:.2cm) and ++(180:.2cm) .. (1,3);
\draw[red] (1.4,4) .. controls ++(-90:.2cm) and ++(0:.2cm) .. (1,3);
\draw[red] (1,3) arc (90:-90: .3cm and 1cm);
\draw (1,3) arc (90:-90: -.3cm and 1cm);

\nbox{unshaded}{(1,3)}{.3}{0}{0}{$\Phi$};
\nbox{unshaded}{(1,1)}{.3}{0}{0}{$\Phi^\dagger$};
\draw[dashed] (0,0) rectangle (2,4);

\node at (2.3,2) {$=$};

\fill[white] (2.6,0) rectangle (3.2,4);
\fill[gray3] (3.2,0) rectangle (4,4);
\fill[gray2] (4,0) rectangle (4.6,4);
\draw[red] (3.2,0) -- (3.2,4);
\draw (4,0) -- (4,4);
\draw[dashed] (2.6,0) rectangle (4.6,4);
\end{tikzpicture}
\]
\item Horizontal unitary: 
\begin{align*}
    (\id_{L_0}\otimes \ev_{\overline{K_1}}\otimes \id_{\overline{L_0}})\circ (\Phi\otimes \overline{\Phi}^\dagger)\circ (\id_{K_0}\otimes \coev_{L_1}\otimes \id_{\overline{K_0}}) & = \coev_{L_0}\circ \ev_{\overline{K_0}}\\
    (\id_{\overline{K_1}}\otimes \ev_{L_0}\otimes \id_{K_1})\circ (\overline{\Phi}^\dagger\otimes \Phi)\circ (\id_{\overline{L_1}}\otimes \coev_{\overline{K_0}}\otimes \id_{L_1}) & = \coev_{\overline{K_1}}\circ \ev_{L_1}.
\end{align*}
\[
\begin{tikzpicture}
\filldraw[white] (0,0) rectangle (.8,3);
\filldraw[white] (2.2,0) rectangle (3,3);
\filldraw[gray3] (0.6,3)  .. controls ++(-90:.2cm) and ++(180:.2cm) .. (1,1.5);
\filldraw[gray3] (2.4,3)  .. controls ++(-90:.2cm) and ++(0:.2cm) .. (2,1.5);
\filldraw[gray3] (.8,1.8) -- (.6,3) -- (2.4,3) -- (2.2,1.8);
\filldraw[gray1] (0.6,0)  .. controls ++(90:.2cm) and ++(180:.2cm) .. (1,1.5);
\filldraw[gray1] (2.4,0) .. controls ++(90:.2cm) and ++(0:.2cm) .. (2,1.5);
\filldraw[gray1] (.8,1.2) -- (.6,0) -- (2.4,0) -- (2.2,1.2);
\filldraw[gray2] (1.2,1) rectangle (1.8,2);
\filldraw[gray2] (1.2,2) arc (180:0:.3);
\filldraw[gray2] (1.2,1) arc (-180:0:.3);

\draw[red] (1.2,1) -- (1.2,1.5);
\draw (1.2,1.5) -- (1.2,2);
\draw[red] (1.8,1) -- (1.8,1.5);
\draw (1.8,1.5) -- (1.8,2);
\draw (1.2,2) arc (180:0:.3);
\draw[red] (1.2,1) arc (-180:0:.3);
\draw (0.6,0)  .. controls ++(90:.2cm) and ++(180:.2cm) .. (1,1.5); 
\draw (2.4,0)  .. controls ++(90:.2cm) and ++(0:.2cm) .. (2,1.5); 
\draw[red] (0.6,3)  .. controls ++(-90:.2cm) and ++(180:.2cm) .. (1,1.5); 
\draw[red] (2.4,3)  .. controls ++(-90:.2cm) and ++(0:.2cm) .. (2,1.5); 

\nbox{unshaded}{(1,1.5)}{.3}{0}{0}{$\Phi$};
\nbox{unshaded}{(2,1.5)}{.3}{0}{0}{$\overline{\Phi}^\dagger$};
\draw[dashed] (0,0) rectangle (3,3);

\node at (3.3,1.5) {$=$};
\filldraw[white] (3.6,.5) rectangle (5.6,2.5);
\filldraw[gray3] (4.1,2.5) arc (-180:0:.5);
\filldraw[gray1] (4.1,.5) arc (180:0:.5);

\draw[red] (4.1,2.5) arc (-180:0:.5);
\draw (4.1,.5) arc (180:0:.5);
\draw[dashed] (3.6,.5) rectangle (5.6,2.5);
\end{tikzpicture}
\qquad
\begin{tikzpicture}
\filldraw[gray2] (0,0) rectangle (.8,3);
\filldraw[gray2] (2.2,0) rectangle (3,3);
\filldraw[gray3] (0.6,3)  .. controls ++(-90:.2cm) and ++(180:.2cm) .. (1,1.5);
\filldraw[gray3] (2.4,3)  .. controls ++(-90:.2cm) and ++(0:.2cm) .. (2,1.5);
\filldraw[gray3] (.8,1.8) -- (.6,3) -- (2.4,3) -- (2.2,1.8);
\filldraw[gray1] (0.6,0)  .. controls ++(90:.2cm) and ++(180:.2cm) .. (1,1.5);
\filldraw[gray1] (2.4,0) .. controls ++(90:.2cm) and ++(0:.2cm) .. (2,1.5);
\filldraw[gray1] (.8,1.2) -- (.6,0) -- (2.4,0) -- (2.2,1.2);
\filldraw[white] (1.2,1) rectangle (1.8,2);
\filldraw[white] (1.2,2) arc (180:0:.3);
\filldraw[white] (1.2,1) arc (-180:0:.3);

\draw (1.2,1) -- (1.2,1.5);
\draw[red] (1.2,1.5) -- (1.2,2);
\draw (1.8,1) -- (1.8,1.5);
\draw[red] (1.8,1.5) -- (1.8,2);
\draw[red] (1.2,2) arc (180:0:.3);
\draw (1.2,1) arc (-180:0:.3);
\draw[red] (0.6,0)  .. controls ++(90:.2cm) and ++(180:.2cm) .. (1,1.5); 
\draw[red] (2.4,0)  .. controls ++(90:.2cm) and ++(0:.2cm) .. (2,1.5); 
\draw (0.6,3)  .. controls ++(-90:.2cm) and ++(180:.2cm) .. (1,1.5); 
\draw (2.4,3)  .. controls ++(-90:.2cm) and ++(0:.2cm) .. (2,1.5); 

\nbox{unshaded}{(2,1.5)}{.3}{0}{0}{$\Phi$};
\nbox{unshaded}{(1,1.5)}{.3}{0}{0}{$\overline{\Phi}^\dagger$};
\draw[dashed] (0,0) rectangle (3,3);

\node at (3.3,1.5) {$=$};
\filldraw[gray2] (3.6,.5) rectangle (5.6,2.5);
\filldraw[gray3] (4.1,2.5) arc (-180:0:.5);
\filldraw[gray1] (4.1,.5) arc (180:0:.5);

\draw (4.1,2.5) arc (-180:0:.5);
\draw[red] (4.1,.5) arc (180:0:.5);
\draw[dashed] (3.6,.5) rectangle (5.6,2.5);
\end{tikzpicture}
\]
Here $\overline{\Phi}$ is defined as the dual of $\Phi$  in the sense of Definition \ref{Def:bar action}.
\end{compactenum}
\end{definition}

\begin{definition}
$\mathcal{C}(K_0,K_1,L_0,L_1,\ev)$ equipped with a biunitary connection $\Phi$ is written as $\mathcal{C}(K_0,K_1,L_0,L_1,\ev;\Phi)$ or simply $\mathcal{C}(\Phi)$.
\end{definition}

\begin{remark}
The existence of $\Phi$ implies that
\begin{align*}
    \dim(K_0\otimes L_1)_{uv} &=\dim(L_0\otimes K_1)_{uv}\\
    \dim(\overline{K_0}\otimes L_0)_{uv} &=\dim(L_1\otimes \overline{K_1})_{uv},
\end{align*}
for each pair $(u,v)\in V\times V$. In other word, the corresponding square-partite graph is associative.
\end{remark}

We are going to discuss some properties of biunitary connection.
\begin{definition}[Rotation by $90^\circ$]
Define the rotation by $90^\circ$ to be
$$\Phi^r:=(\id_{\overline{K_0}}\otimes\id_{L_0}\otimes \ev_{\overline{K_1}})\circ (\id_{\overline{K_0}}\otimes \Phi\otimes \id_{\overline{K_1}})\circ (\coev_{\overline{K_0}}\otimes \id_{L_1}\otimes\id_{\overline{K_1}}).$$
Similarly, 
$$\Phi^{r^2}:=(\id_{\overline{L_1}}\otimes\id_{\overline{K_0}}\otimes \ev_{\overline{L_0}})\circ (\id_{\overline{L_1}}\otimes \Phi^r\otimes \id_{\overline{L_0}})\circ (\coev_{\overline{L_1}}\otimes \id_{\overline{K_1}}\otimes\id_{\overline{L_0}})=\overline{\Phi}.$$
\[
\begin{tikzpicture}[baseline=1cm]
\fill[gray1] (0,0) rectangle (1,2);
\fill[gray3] (1,0) rectangle (2,2);

\filldraw[gray2] (.6,0)  .. controls ++(90:.2cm) and ++(180:.2cm) .. (1,1);
\filldraw[gray2] (1.4,0)  .. controls ++(90:.2cm) and ++(0:.2cm) .. (1,1);
\filldraw[gray2] (.6,0) -- (1,1) -- (1.4,0);

\filldraw[white] (.6,2)  .. controls ++(-90:.2cm) and ++(180:.2cm) .. (1,1);
\filldraw[white] (1.4,2) .. controls ++(-90:.2cm) and ++(0:.2cm) .. (1,1);
\filldraw[white] (.6,2) -- (1,1) -- (1.4,2);

\draw[red] (.6,0)  .. controls ++(90:.2cm) and ++(180:.2cm) .. (1,1);
\draw (1.4,0)  .. controls ++(90:.2cm) and ++(0:.2cm) .. (1,1);
\draw (.6,2)  .. controls ++(-90:.2cm) and ++(180:.2cm) .. (1,1);
\draw[red] (1.4,2) .. controls ++(-90:.2cm) and ++(0:.2cm) .. (1,1);

\nbox{unshaded}{(1,1)}{.3}{0}{0}{$\Phi^r$};
\draw[dashed] (0,0) rectangle (2,2);
\end{tikzpicture}
:=
\begin{tikzpicture}[baseline=1cm]
\filldraw[gray1] (0,0) rectangle (.3,2);
\filldraw[gray1] (0.3,0) rectangle (1.1,.7);
\filldraw[gray3] (1.7,0) rectangle (2,2);
\filldraw[gray3] (.9,1.3) rectangle (1.7,2);
\filldraw[white] (0.9,0.7) arc (0:-180:.3);
\filldraw[white] (.3,.7) rectangle (.9,2);
\filldraw[gray2] (1.1,1.3) arc (180:0:.3);
\filldraw[gray2] (1.1,0) rectangle (1.7,1.3);

\draw[red] (.9,0.7) -- (.9,2);
\draw[red] (1.1,1.3) -- (1.1,0);
\draw (1.1,1.3) arc (180:0:.3);
\draw (0.9,0.7) arc (0:-180:.3);
\draw (.3,.7) -- (.3,2);
\draw (1.7,1.3) -- (1.7,0);

\nbox{unshaded}{(1,1)}{.3}{0}{0}{$\Phi$};
\draw[dashed] (0,0) rectangle (2,2);
\end{tikzpicture}
\qquad\qquad
\begin{tikzpicture}[baseline=1cm]
\fill[gray2] (0,0) rectangle (1,2);
\fill[white] (1,0) rectangle (2,2);

\filldraw[gray3] (.6,0)  .. controls ++(90:.2cm) and ++(180:.2cm) .. (1,1);
\filldraw[gray3] (1.4,0)  .. controls ++(90:.2cm) and ++(0:.2cm) .. (1,1);
\filldraw[gray3] (.6,0) -- (1,1) -- (1.4,0);

\filldraw[gray1] (.6,2)  .. controls ++(-90:.2cm) and ++(180:.2cm) .. (1,1);
\filldraw[gray1] (1.4,2) .. controls ++(-90:.2cm) and ++(0:.2cm) .. (1,1);
\filldraw[gray1] (.6,2) -- (1,1) -- (1.4,2);

\draw (.6,0)  .. controls ++(90:.2cm) and ++(180:.2cm) .. (1,1);
\draw[red] (1.4,0)  .. controls ++(90:.2cm) and ++(0:.2cm) .. (1,1);
\draw[red] (.6,2)  .. controls ++(-90:.2cm) and ++(180:.2cm) .. (1,1);
\draw (1.4,2) .. controls ++(-90:.2cm) and ++(0:.2cm) .. (1,1);

\nbox{unshaded}{(1,1)}{.3}{0}{0}{$\Phi^{r^2}$};
\draw[dashed] (0,0) rectangle (2,2);
\end{tikzpicture}
:=
\begin{tikzpicture}[baseline=1cm]
\filldraw[gray2] (0,0) rectangle (.3,2);
\filldraw[gray2] (0.3,0) rectangle (1.1,.7);
\filldraw[white] (1.7,0) rectangle (2,2);
\filldraw[white] (.9,1.3) rectangle (1.7,2);
\filldraw[gray1] (0.9,0.7) arc (0:-180:.3);
\filldraw[gray1] (.3,.7) rectangle (.9,2);
\filldraw[gray3] (1.1,1.3) arc (180:0:.3);
\filldraw[gray3] (1.1,0) rectangle (1.7,1.3);

\draw (.9,0.7) -- (.9,2);
\draw (1.1,1.3) -- (1.1,0);
\draw[red] (1.1,1.3) arc (180:0:.3);
\draw[red] (0.9,0.7) arc (0:-180:.3);
\draw[red] (.3,.7) -- (.3,2);
\draw[red] (1.7,1.3) -- (1.7,0);

\nbox{unshaded}{(1,1)}{.3}{0}{0}{$\Phi^r$};
\draw[dashed] (0,0) rectangle (2,2);
\end{tikzpicture}
\]
\end{definition}

\begin{remark} Here are some properties for biunitary connections and rotation.
\begin{compactenum}[(1)]
\item The group $\langle r,\dagger\rangle=\langle r,\dagger|r^4=\dagger^2=\id,r\dagger=\dagger r^3 \rangle$ for the biunitary connection is isomorphic to the dihedral group $D_4$.
\item $\Phi$ is a biunitary connection if and only if $\Phi^g$ is both vertical unitary and horizontal unitary, where $g\in \langle r,\dagger\rangle$. 
\end{compactenum}
\end{remark}

\begin{definition}[{\cite[\S4]{RV16}}]\label{Def:Gauge equivalence}
We call biunitary connections $\Phi:K_0\otimes L_1\to L_0\otimes K_1$ and $\Phi':K_0'\otimes L_1'\to L_0'\otimes K_1'$ \textbf{gauge equivalent}, if there exist unitaries $u_1:K_0'\to K_0$, $u_2:L_0\to L_0'$, $u_3:K_1\to K_1'$ and $u_4:L_1'\to L_1$ such that $\Phi_2=(u_2\otimes u_3)\circ \Phi_1\circ (u_1\otimes u_4)$.
\[
\begin{tikzpicture}[baseline=.9cm]
\fill[white] (0,0) rectangle (1,2);
\fill[gray2] (1,0) rectangle (2,2);

\filldraw[gray1] (.6,0)  .. controls ++(90:.2cm) and ++(180:.2cm) .. (1,1);
\filldraw[gray1] (1.4,0)  .. controls ++(90:.2cm) and ++(0:.2cm) .. (1,1);
\filldraw[gray1] (.6,0) -- (1,1) -- (1.4,0);

\filldraw[gray3] (.6,2)  .. controls ++(-90:.2cm) and ++(180:.2cm) .. (1,1);
\filldraw[gray3] (1.4,2) .. controls ++(-90:.2cm) and ++(0:.2cm) .. (1,1);
\filldraw[gray3] (.6,2) -- (1,1) -- (1.4,2);

\draw (.6,0)  .. controls ++(90:.2cm) and ++(180:.2cm) .. (1,1);
\draw[red] (1.4,0)  .. controls ++(90:.2cm) and ++(0:.2cm) .. (1,1);
\draw[red] (.6,2)  .. controls ++(-90:.2cm) and ++(180:.2cm) .. (1,1);
\draw (1.4,2) .. controls ++(-90:.2cm) and ++(0:.2cm) .. (1,1);

\nbox{unshaded}{(1,1)}{.3}{0}{0}{$\Phi'$};
\draw[dashed] (0,0) rectangle (2,2);
\end{tikzpicture}
=
\begin{tikzpicture}[baseline=1.9cm]
\fill[white] (0,0) rectangle (1,4);
\fill[gray2] (1,0) rectangle (2,4);

\filldraw[gray1] (.6,0)  .. controls ++(90:.2cm) and ++(180:.2cm) .. (1,2);
\filldraw[gray1] (1.4,0)  .. controls ++(90:.2cm) and ++(0:.2cm) .. (1,2);
\filldraw[gray1] (.6,0) -- (1,2) -- (1.4,0);

\filldraw[gray3] (.6,4)  .. controls ++(-90:.2cm) and ++(180:.2cm) .. (1,2);
\filldraw[gray3] (1.4,4) .. controls ++(-90:.2cm) and ++(0:.2cm) .. (1,2);
\filldraw[gray3] (.6,4) -- (1,2) -- (1.4,4);

\draw (.6,0)  .. controls ++(90:.2cm) and ++(180:.2cm) .. (1,2);
\draw[red] (1.4,0)  .. controls ++(90:.2cm) and ++(0:.2cm) .. (1,2);
\draw[red] (.6,4)  .. controls ++(-90:.2cm) and ++(180:.2cm) .. (1,2);
\draw (1.4,4) .. controls ++(-90:.2cm) and ++(0:.2cm) .. (1,2);

\nbox{unshaded}{(1,2)}{.3}{0}{0}{$\Phi$};
\nbox{unshaded}{(.65,1)}{.2}{0}{0}{$u_1$};
\nbox{unshaded}{(.65,3)}{.2}{0}{0}{$u_2$};
\nbox{unshaded}{(1.35,3)}{.2}{0}{0}{$u_3$};
\nbox{unshaded}{(1.35,1)}{.2}{0}{0}{$u_4$};

\draw[dashed] (0,0) rectangle (2,4);
\end{tikzpicture}
\]
\end{definition}

\begin{notation}\textbf{and Observation}

Observe that once we know the color of region and the color of edge, the biunitary connection in the circle is determined. So we can simplify the graphical calculus of biunitary connection as follows.
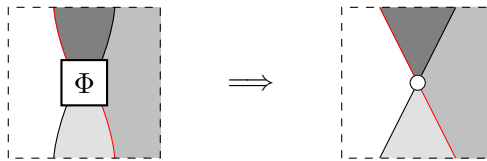
\[
\begin{tikzpicture}[baseline=.9cm]
\fill[white] (0,0) rectangle (1,2);
\fill[gray2] (1,0) rectangle (2,2);

\filldraw[gray1] (.6,0)  .. controls ++(90:.2cm) and ++(180:.2cm) .. (1,1);
\filldraw[gray1] (1.4,0)  .. controls ++(90:.2cm) and ++(0:.2cm) .. (1,1);
\filldraw[gray1] (.6,0) -- (1,1) -- (1.4,0);

\filldraw[gray3] (.6,2)  .. controls ++(-90:.2cm) and ++(180:.2cm) .. (1,1);
\filldraw[gray3] (1.4,2) .. controls ++(-90:.2cm) and ++(0:.2cm) .. (1,1);
\filldraw[gray3] (.6,2) -- (1,1) -- (1.4,2);

\draw (.6,0)  .. controls ++(90:.2cm) and ++(180:.2cm) .. (1,1);
\draw[red] (1.4,0)  .. controls ++(90:.2cm) and ++(0:.2cm) .. (1,1);
\draw[red] (.6,2)  .. controls ++(-90:.2cm) and ++(180:.2cm) .. (1,1);
\draw (1.4,2) .. controls ++(-90:.2cm) and ++(0:.2cm) .. (1,1);

\nbox{unshaded}{(1,1)}{.3}{0}{0}{$\Phi$};
\draw[dashed] (0,0) rectangle (2,2);
\end{tikzpicture}
\qquad\Longrightarrow\qquad
\begin{tikzpicture}[baseline=.9cm]
\filldraw[white] (0,0) rectangle (1,2);
\filldraw[gray2] (1,0) rectangle (2,2);
\filldraw[gray3] (.5,2) -- (1,1) -- (1.5,2);
\filldraw[gray1] (.5,0) -- (1,1) -- (1.5,0);

\draw (.5,0) -- (1.5,2);
\draw[red] (.5,2) -- (1.5,0);

\filldraw[white] (1.1,1) arc (0:360:.1);
\draw (1.1,1) arc (0:360:.1);
\draw[dashed] (0,0) rectangle (2,2);
\end{tikzpicture}
\]
Moreover, if the color of the leftmost region and the color of each edge are determined, then the color of the rest of the regions will be determined. The 4 colors on the leftmost region and 2 colors on the edge (8 cases) can represent all $\Phi^g$, $g\in \langle r,\dagger\rangle$.

Here are the simplified graphical calculus of vertical unitarity and horizontal unitarity. In the following context, We require that the leftmost regions in the uncolored equality have the same color.
\[
\begin{tikzpicture}[baseline=1cm]
\draw (.2,0) -- (.8,1);
\draw[red] (.8,0) -- (.2,1);

\draw (.8,1) -- (.2,2);
\draw[red] (.2,1) -- (.8,2);

\draw[dashed] (0,1) -- (1,1);

\filldraw[white] (.6,1.5) arc (0:360:.1);
\draw (.6,1.5) arc (0:360:.1);
\filldraw[white] (.6,.5) arc (0:360:.1);
\draw (.6,.5) arc (0:360:.1);
\draw[dashed] (0,0) rectangle (1,2);
\end{tikzpicture}
=
\begin{tikzpicture}[baseline=1cm]
\draw (.3,0) -- (.3,2);
\draw[red] (.7,0) -- (.7,2);

\draw[dashed] (0,0) rectangle (1,2);
\end{tikzpicture}
\qquad\qquad\qquad\qquad
\begin{tikzpicture}[baseline=1cm]
\draw[red] (.2,0) -- (.8,1);
\draw (.8,0) -- (.2,1);

\draw[red] (.8,1) -- (.2,2);
\draw (.2,1) -- (.8,2);

\draw[dashed] (0,1) -- (1,1);

\filldraw[white] (.6,1.5) arc (0:360:.1);
\draw (.6,1.5) arc (0:360:.1);
\filldraw[white] (.6,.5) arc (0:360:.1);
\draw (.6,.5) arc (0:360:.1);
\draw[dashed] (0,0) rectangle (1,2);
\end{tikzpicture}
=
\begin{tikzpicture}[baseline=1cm]
\draw (.3,0) -- (.3,2);
\draw[red] (.7,0) -- (.7,2);

\draw[dashed] (0,0) rectangle (1,2);
\end{tikzpicture}
\]
\[
\begin{tikzpicture}[baseline=1cm]
\draw (.2,0) -- (.7,1.25);
\draw[red] (.2,2) -- (.7,.75);
\draw[red] (1.8,2) -- (1.3,.75);
\draw (1.8,0) -- (1.3,1.25);

\draw (.7,1.25) arc (180:0:.3);
\draw[red] (.7,.75) arc (-180:0:.3);

\draw[dashed] (0,1.25) -- (2,1.25);
\draw[dashed] (0,.75) -- (2,.75);

\filldraw[white] (.7,1) arc (0:360:.1);
\draw (.7,1) arc (0:360:.1);
\filldraw[white] (1.5,1) arc (0:360:.1);
\draw (1.5,1) arc (0:360:.1);
\draw[dashed] (0,0) rectangle (2,2);
\end{tikzpicture}
=
\begin{tikzpicture}[baseline=1cm]
\draw[red] (.4,2) arc (-180:0:.6);
\draw (.4,0) arc (180:0:.6);
\draw[dashed] (0,0) rectangle (2,2);
\end{tikzpicture}
\qquad\qquad
\begin{tikzpicture}[baseline=1cm]
\draw[red] (.2,0) -- (.7,1.25);
\draw (.2,2) -- (.7,.75);
\draw (1.8,2) -- (1.3,.75);
\draw[red] (1.8,0) -- (1.3,1.25);

\draw[red] (.7,1.25) arc (180:0:.3);
\draw (.7,.75) arc (-180:0:.3);

\draw[dashed] (0,1.25) -- (2,1.25);
\draw[dashed] (0,.75) -- (2,.75);

\filldraw[white] (.7,1) arc (0:360:.1);
\draw (.7,1) arc (0:360:.1);
\filldraw[white] (1.5,1) arc (0:360:.1);
\draw (1.5,1) arc (0:360:.1);
\draw[dashed] (0,0) rectangle (2,2);
\end{tikzpicture}
=
\begin{tikzpicture}[baseline=1cm]
\draw (.4,2) arc (-180:0:.6);
\draw[red] (.4,0) arc (180:0:.6);
\draw[dashed] (0,0) rectangle (2,2);
\end{tikzpicture}
\]
\end{notation}

\begin{proposition}\label{prop of BC}
Here are some properties that will be used in the next section and the proof is left to the reader.
\begin{compactenum}[\rm (1)]
\item 
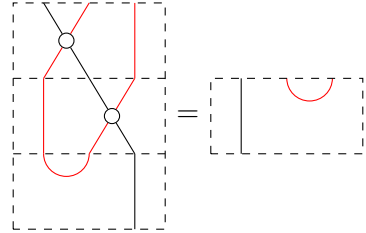
\[
\begin{tikzpicture}
\draw[red] (.4,2) -- (1,3);
\draw (1,2) -- (.4,3);
\draw[red] (1.6,2) -- (1.6,3);
\filldraw[white] (.8,2.5) arc (0:360:.1);
\draw (.8,2.5) arc (0:360:.1);

\draw[red] (.4,1) -- (.4,2);
\draw[red] (1,1) -- (1.6,2);
\draw (1.6,1) -- (1,2);
\filldraw[white] (1.4,1.5) arc (0:360:.1);
\draw (1.4,1.5) arc (0:360:.1);

\draw[red] (.4,1) arc (-180:0:.3);
\draw (1.6,0) -- (1.6,1);

\draw[dashed] (0,2) -- (2,2);
\draw[dashed] (0,1) -- (2,1);
\draw[dashed] (0,0) rectangle (2,3);

\node at (2.3,1.5) {$=$};

\draw (3,1) -- (3,2);
\draw[red] (3.6,2) arc (-180:0:.3);

\draw[dashed] (2.6,1) rectangle (4.6,2);
\end{tikzpicture}
\]
\item For 2-morphism $x\in\End(F\otimes K_0\otimes L_1)$, where $F$ is a proper $2$-morphism, we have
\[
\begin{tikzpicture}[baseline=.9cm]
\draw (.2,0) -- (.2,2);
\draw (.5,.5) -- (.5,1.5);
\draw[red] (.8,.5) -- (.8,1.5);

\draw (.5,1.5) arc (180:0:.45);
\draw (.5,.5) arc (-180:0:.45);
\draw (1.4,.5) -- (1.4,1.5);

\draw[red] (.8,1.5) arc (180:0:.2);
\draw[red] (.8,.5) arc (-180:0:.2);
\draw[red] (1.2,.5) -- (1.2,1.5);

\nbox{unshaded}{(.5,1)}{.2}{.3}{.3}{$x$};
\node at (.2,-.2) {\tiny{$F$}};
\node at (.48,-.23) {\tiny{$K_0$}};
\node at (.82,-.23) {\tiny{$L_1$}};
\end{tikzpicture}
=
\begin{tikzpicture}[baseline=1.4cm]
\draw (.2,.5) -- (.2,2.5);
\draw (.5,1) -- (.5,2);
\draw (.5,2) -- (1.1,2.5);
\draw (.5,1) -- (1.1,.5);
\draw[red] (.8,.5) -- (.8,2.5);

\draw (1.1,2.5) arc (180:0:.15);
\draw (1.1,.5) arc (-180:0:.15);
\draw (1.4,.5) -- (1.4,2.5);

\draw[red] (.8,2.5) arc (180:0:.4);
\draw[red] (.8,.5) arc (-180:0:.4);
\draw[red] (1.6,.5) -- (1.6,2.5);

\filldraw[white] (.9,2.25) arc (0:360:.1);
\draw (.9,2.25) arc (0:360:.1);
\filldraw[white] (.9,.75) arc (0:360:.1);
\draw (.9,.75) arc (0:360:.1);

\nbox{unshaded}{(.5,1.5)}{.2}{.3}{.3}{$x$};
\node at (.2,.3) {\tiny{$F$}};
\node at (.48,.77) {\tiny{$K_0$}};
\node at (1,1) {\tiny{$L_1$}};
\node at (.65,.3) {\tiny{$L_0$}};
\node at (1.01,.35) {\tiny{$K_1$}};
\end{tikzpicture}
\]
\end{compactenum}
\end{proposition}

\subsection{From $\mathcal{C}(\Phi)$ to Markov lattice}
\begin{construction}\label{C(Phi) to ML}
Here we are going to construct a Markov lattice from the 2-category $\mathcal{C}(\Phi)$ discussed above with a chosen point, say $P_{0}\in V_{00}$. Let $\mathbb{C}^{|P_{0}|}$ be a 1-morphism with all the entry being $0$ except $(\mathbb{C}^{|P_0|})_{P_{0}P_{0}}=\mathbb{C}$.

Note that $\mathbb{C}^{|P_0|}\otimes K_0^{\alt\otimes i}\otimes L_?^{\alt\otimes j}$ is a 1-morphism for each $i,j\in\mathbb{Z}_{\ge 0}$.

Let $M_{i,j}=\End\left(\mathbb{C}^{|P_0|}\otimes K_0^{\alt\otimes i}\otimes L_?^{\alt\otimes j}\right)$, where $L_?=L_0$ if $2\mid i$ and $L_?=L_1$ if $2\nmid j$. We use the graphical calculus to show $M=(M_{i,j})_{i,j\ge 0}$ is a Markov lattice.
\begin{compactenum}[(1)]
\item Element $x\in M_{i,j}$:
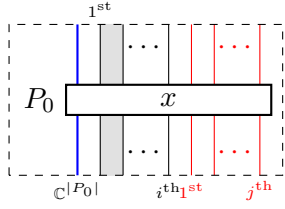
\[
\begin{tikzpicture}
\filldraw[white] (0,0) rectangle (.6,2);
\filldraw[gray1] (.6,0) rectangle (.9,2);

\draw[red] (1.8,0) -- (1.8,2);
\draw[red] (2.1,0) -- (2.1,2);
\draw[red] (2.7,0) -- (2.7,2);
\node[red] at (2.4,1.7) {$\cdots$};
\node[red] at (2.4,.3) {$\cdots$};

\draw (1.5,0) -- (1.5,2);
\draw (.9,0) -- (.9,2);
\draw (.6,0) -- (.6,2);
\node at (1.2,1.7) {$\cdots$};
\node at (1.2,.3) {$\cdots$};

\draw[thick,blue] (.3,0) -- (.3,2);

\nbox{unshaded}{(1.5,1)}{.2}{1.15}{1.15}{$x$};

\node at (.3,-.2) {\tiny{$\mathbb{C}^{|P_0|}$}};
\node at (.6,2.2) {\tiny{$1^\text{st}$}};
\node at (1.5,-.2) {\tiny{$i^\text{th}$}};
\node[red] at (1.8,-.2) {\tiny{$1^\text{st}$}};
\node[red] at (2.7,-.2) {\tiny{$j^\text{th}$}};
\draw[dashed] (-.6,0) rectangle (3,2);
\node at (-.2,1) {$P_0$};
\end{tikzpicture}
\]
\item Horizontal inclusion $x\in M_{i,j}\subset M_{i,j+1}$:
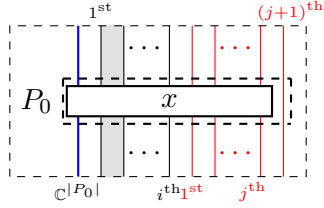
\[
\begin{tikzpicture}
\filldraw[white] (0,0) rectangle (.6,2);
\filldraw[gray1] (.6,0) rectangle (.9,2);

\draw[red] (1.8,0) -- (1.8,2);
\draw[red] (2.1,0) -- (2.1,2);
\draw[red] (2.7,0) -- (2.7,2);
\draw[red] (3,0) -- (3,2);
\node[red] at (2.4,1.7) {$\cdots$};
\node[red] at (2.4,.3) {$\cdots$};

\draw (1.5,0) -- (1.5,2);
\draw (.9,0) -- (.9,2);
\draw (.6,0) -- (.6,2);
\node at (1.2,1.7) {$\cdots$};
\node at (1.2,.3) {$\cdots$};

\draw[thick,blue] (.3,0) -- (.3,2);

\nbox{unshaded}{(1.5,1)}{.2}{1.15}{1.15}{$x$};

\draw[dashed,thick] (.1,1.3) -- (.1,.7);
\draw[dashed,thick] (.1,.7) -- (3.1,.7);
\draw[dashed,thick] (3.1,1.3) -- (3.1,.7);
\draw[dashed,thick] (.1,1.3) -- (3.1,1.3);

\node at (.3,-.2) {\tiny{$\mathbb{C}^{|P_0|}$}};
\node at (.6,2.2) {\tiny{$1^\text{st}$}};
\node at (1.5,-.2) {\tiny{$i^\text{th}$}};
\node[red] at (1.8,-.2) {\tiny{$1^\text{st}$}};
\node[red] at (2.6,-.2) {\tiny{$j^\text{th}$}};
\node[red] at (3.1,2.2) {\tiny{$(j\+1)^\text{th}$}};
\draw[dashed] (-.6,0) rectangle (3.3,2);
\node at (-.25,1) {$P_0$};
\end{tikzpicture}
\]
\item Vertical inclusion $x\in M_{i,j}\subset M_{i+1,j}$:
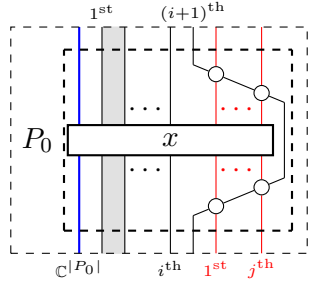
\[
\begin{tikzpicture}
\filldraw[white] (0,0) rectangle (.6,3);
\filldraw[gray1] (.6,0) rectangle (.9,3);

\draw[red] (2.1,0) -- (2.1,3);
\draw[red] (2.7,0) -- (2.7,3);
\node[red] at (2.4,1.9) {$\cdots$};
\node[red] at (2.4,1.1) {$\cdots$};

\draw (1.8,0) -- (1.8,.5);
\draw (1.8,2.5) -- (1.8,3);
\draw (1.5,0) -- (1.5,3);
\draw (.9,0) -- (.9,3);
\draw (.6,0) -- (.6,3);
\node at (1.2,1.9) {$\cdots$};
\node at (1.2,1.1) {$\cdots$};
\draw (3,1) -- (3,2);
\draw (3,1) -- (1.8,.5);
\draw (3,2) -- (1.8,2.5);

\draw[thick,blue] (.3,0) -- (.3,3);

\nbox{unshaded}{(1.5,1.5)}{.2}{1.15}{1.15}{$x$};

\filldraw[white] (2.8,2.125) arc (0:360:.1);
\draw (2.8,2.125) arc (0:360:.1);
\filldraw[white] (2.8,.875) arc (0:360:.1);
\draw (2.8,.875) arc (0:360:.1);

\filldraw[white] (2.2,2.375) arc (0:360:.1);
\draw (2.2,2.375) arc (0:360:.1);
\filldraw[white] (2.2,.625) arc (0:360:.1);
\draw (2.2,.625) arc (0:360:.1);

\draw[dashed,thick] (.1,.3) -- (.1,2.7);
\draw[dashed,thick] (.1,.3) -- (3.1,.3);
\draw[dashed,thick] (3.1,.3) -- (3.1,2.7);
\draw[dashed,thick] (.1,2.7) -- (3.1,2.7);

\node at (.3,-.2) {\tiny{$\mathbb{C}^{|P_0|}$}};
\node at (.6,3.2) {\tiny{$1^\text{st}$}};
\node at (1.5,-.2) {\tiny{$i^\text{th}$}};
\node[red] at (2.1,-.2) {\tiny{$1^\text{st}$}};
\node[red] at (2.7,-.2) {\tiny{$j^\text{th}$}};
\node at (1.8,3.2) {\tiny{$(i\+1)^\text{th}$}};
\draw[dashed] (-.6,0) rectangle (3.3,3);
\node at (-.25,1.5) {$P_0$};
\end{tikzpicture}
\]
\item Horizontal conditional expectation $E^{M,r}_{i,j}:M_{i,j}\to M_{i,j-1}$, $,x\in M_{i,j}$:
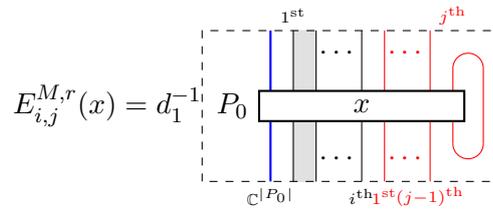
\[
E^{M,r}_{i,j}(x)=d_1^{-1}
\begin{tikzpicture}[baseline=.9cm]
\filldraw[white] (0,0) rectangle (.6,2);
\filldraw[gray1] (.6,0) rectangle (.9,2);

\draw[red] (1.8,0) -- (1.8,2);
\draw[red] (2.4,0) -- (2.4,2);
\draw[red] (2.7,.5) -- (2.7,1.5);
\draw[red] (3.1,.5) -- (3.1,1.5);
\node[red] at (2.1,1.7) {$\cdots$};
\node[red] at (2.1,.3) {$\cdots$};

\draw (1.5,0) -- (1.5,2);
\draw (.9,0) -- (.9,2);
\draw (.6,0) -- (.6,2);
\node at (1.2,1.7) {$\cdots$};
\node at (1.2,.3) {$\cdots$};

\draw[red] (2.7,1.5) arc (180:0:.2);
\draw[red] (2.7,.5) arc (-180:0:.2);

\draw[thick,blue] (.3,0) -- (.3,2);

\nbox{unshaded}{(1.5,1)}{.2}{1.15}{1.15}{$x$};

\node at (.3,-.2) {\tiny{$\mathbb{C}^{|P_0|}$}};
\node at (.6,2.2) {\tiny{$1^\text{st}$}};
\node at (1.5,-.2) {\tiny{$i^\text{th}$}};
\node[red] at (1.8,-.2) {\tiny{$1^\text{st}$}};
\node[red] at (2.4,-.2) {\tiny{$(j\sm 1)^\text{th}$}};
\node[red] at (2.7,2.2) {\tiny{$j^\text{th}$}};

\draw[dashed] (-.6,0) rectangle (3.3,2);
\node at (-.2,1) {$P_0$};
\end{tikzpicture}
\]
\item Vertical conditional expectation $E^{M,l}_{i,j}:M_{i,j}\to M_{i-1,j}$, $x\in M_{i,j}$:
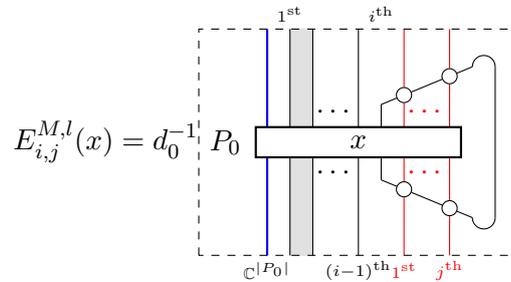
\[
E^{M,l}_{i,j}(x)=d_0^{-1}
\begin{tikzpicture}[baseline=1.4cm]
\filldraw[white] (0,0) rectangle (.6,3);
\filldraw[gray1] (.6,0) rectangle (.9,3);

\draw[red] (2.1,0) -- (2.1,3);
\draw[red] (2.7,0) -- (2.7,3);
\node[red] at (2.4,1.9) {$\cdots$};
\node[red] at (2.4,1.1) {$\cdots$};

\draw (1.5,0) -- (1.5,3);
\draw (.9,0) -- (.9,3);
\draw (.6,0) -- (.6,3);
\node at (1.2,1.9) {$\cdots$};
\node at (1.2,1.1) {$\cdots$};
\draw (1.8,1) -- (1.8,2);
\draw (3,.5) -- (1.8,1);
\draw (3,2.5) -- (1.8,2);

\draw (3,2.5) arc (180:0:.15);
\draw (3,.5) arc (-180:0:.15);
\draw (3.3,.5) -- (3.3,2.5);

\draw[thick,blue] (.3,0) -- (.3,3);

\nbox{unshaded}{(1.5,1.5)}{.2}{1.15}{1.15}{$x$};

\filldraw[white] (2.8,2.375) arc (0:360:.1);
\draw (2.8,2.375) arc (0:360:.1);
\filldraw[white] (2.8,.625) arc (0:360:.1);
\draw (2.8,.625) arc (0:360:.1);

\filldraw[white] (2.2,2.125) arc (0:360:.1);
\draw (2.2,2.125) arc (0:360:.1);
\filldraw[white] (2.2,.875) arc (0:360:.1);
\draw (2.2,.875) arc (0:360:.1);

\node at (.3,-.2) {\tiny{$\mathbb{C}^{|P_0|}$}};
\node at (.6,3.2) {\tiny{$1^\text{st}$}};
\node at (1.5,-.2) {\tiny{$(i\sm1)^\text{th}$}};
\node[red] at (2.1,-.2) {\tiny{$1^\text{st}$}};
\node[red] at (2.7,-.2) {\tiny{$j^\text{th}$}};
\node at (1.8,3.2) {\tiny{$i^\text{th}$}};
\draw[dashed] (-.6,0) rectangle (3.6,3);
\node at (-.25,1.5) {$P_0$};
\end{tikzpicture}
\]
\item Commuting square of conditional expectations
$E^{M,r}_{i-1,j}\circ E^{M,l}_{i-1,j-1}=E^{M,l}_{i-1,j}\circ E^{M,r}_{i,j}:M_{i,j}\to M_{i-1,j-1}$, $x\in M_{i,j}$:
\[\hspace*{-1.8cm}
E^{M,r}_{i-1,j}\circ E^{M,l}_{i-1,j-1}(x)= d_0^{-1}d_1^{-1}
\begin{tikzpicture}[baseline=1.4cm]
\filldraw[white] (0,0) rectangle (.6,3);
\filldraw[gray1] (.6,0) rectangle (.9,3);

\draw[red] (2.1,0) -- (2.1,3);
\draw[red] (2.7,0) -- (2.7,3);
\draw[red] (3,.3) -- (3,2.7);
\node[red] at (2.4,1.9) {$\cdots$};
\node[red] at (2.4,1.1) {$\cdots$};

\draw (1.5,0) -- (1.5,3);
\draw (.9,0) -- (.9,3);
\draw (.6,0) -- (.6,3);
\node at (1.2,1.9) {$\cdots$};
\node at (1.2,1.1) {$\cdots$};
\draw (1.8,1) -- (1.8,2);
\draw (3.3,.5) -- (1.8,1);
\draw (3.3,2.5) -- (1.8,2);

\draw (3.3,2.5) arc (180:0:.15);
\draw (3.3,.5) arc (-180:0:.15);
\draw (3.6,.5) -- (3.6,2.5);

\draw[red] (3,2.7) arc (180:0:.4);
\draw[red] (3,.3) arc (-180:0:.4);
\draw[red] (3.8,.3) -- (3.8,2.7);

\draw[thick,blue] (.3,0) -- (.3,3);

\nbox{unshaded}{(1.5,1.5)}{.2}{1.15}{1.45}{$x$};

\filldraw[white] (3.1,2.4) arc (0:360:.1);
\draw (3.1,2.4) arc (0:360:.1);
\filldraw[white] (3.1,.6) arc (0:360:.1);
\draw (3.1,.6) arc (0:360:.1);

\filldraw[white] (2.8,2.3) arc (0:360:.1);
\draw (2.8,2.3) arc (0:360:.1);
\filldraw[white] (2.8,.7) arc (0:360:.1);
\draw (2.8,.7) arc (0:360:.1);

\filldraw[white] (2.2,2.1) arc (0:360:.1);
\draw (2.2,2.1) arc (0:360:.1);
\filldraw[white] (2.2,.9) arc (0:360:.1);
\draw (2.2,.9) arc (0:360:.1);

\node at (.3,-.2) {\tiny{$\mathbb{C}^{|P_0|}$}};
\node at (.6,3.2) {\tiny{$1^\text{st}$}};
\node at (1.5,-.2) {\tiny{$(i\sm1)^\text{th}$}};
\node[red] at (2.1,-.2) {\tiny{$1^\text{st}$}};
\node[red] at (2.7,-.2) {\tiny{$(j\sm1)^\text{th}$}};
\node[red] at (3,3.2) {\tiny{$j^\text{th}$}};
\node at (1.8,3.2) {\tiny{$i^\text{th}$}};
\node at (-.25,1.5) {$P_0$};
\draw[dashed] (-.6,0) rectangle (4,3);
\end{tikzpicture}
=d_0^{-1}d_1^{-1}
\begin{tikzpicture}[baseline=1.4cm]
\filldraw[white] (0,0) rectangle (.6,3);
\filldraw[gray1] (.6,0) rectangle (.9,3);

\draw[red] (2.1,0) -- (2.1,3);
\draw[red] (2.7,0) -- (2.7,3);
\draw[red] (3,1) -- (3,2);
\node[red] at (2.4,1.9) {$\cdots$};
\node[red] at (2.4,1.1) {$\cdots$};

\draw (1.5,0) -- (1.5,3);
\draw (.9,0) -- (.9,3);
\draw (.6,0) -- (.6,3);
\node at (1.2,1.9) {$\cdots$};
\node at (1.2,1.1) {$\cdots$};
\draw (1.8,1) -- (1.8,2);
\draw (3,.5) -- (1.8,1);
\draw (3,2.5) -- (1.8,2);

\draw (3,2.5) arc (180:0:.3);
\draw (3,.5) arc (-180:0:.3);
\draw (3.6,.5) -- (3.6,2.5);

\draw[red] (3,2) arc (180:0:.2);
\draw[red] (3,1) arc (-180:0:.2);
\draw[red] (3.4,1) -- (3.4,2);

\draw[thick,blue] (.3,0) -- (.3,3);

\nbox{unshaded}{(1.5,1.5)}{.2}{1.15}{1.45}{$x$};

\filldraw[white] (2.8,2.375) arc (0:360:.1);
\draw (2.8,2.375) arc (0:360:.1);
\filldraw[white] (2.8,.625) arc (0:360:.1);
\draw (2.8,.625) arc (0:360:.1);

\filldraw[white] (2.2,2.125) arc (0:360:.1);
\draw (2.2,2.125) arc (0:360:.1);
\filldraw[white] (2.2,.875) arc (0:360:.1);
\draw (2.2,.875) arc (0:360:.1);

\node at (.3,-.2) {\tiny{$\mathbb{C}^{|P_0|}$}};
\node at (.6,3.2) {\tiny{$1^\text{st}$}};
\node at (1.5,-.2) {\tiny{$(i\sm1)^\text{th}$}};
\node[red] at (2.1,-.2) {\tiny{$1^\text{st}$}};
\node[red] at (2.7,-.2) {\tiny{$(j\sm1)^\text{th}$}};
\node[red] at (3,3.2) {\tiny{$j^\text{th}$}};
\node at (1.8,3.2) {\tiny{$i^\text{th}$}};
\node at (-.25,1.5) {$P_0$};
\draw[dashed] (-.6,0) rectangle (3.8,3);
\end{tikzpicture}
=E^{M,l}_{i-1,j}\circ E^{M,r}_{i,j}(x)
\]

\item Vertical Jones projections $e_i\in M_{i+1,j}$ and horizontal Jones projection $f_j\in M_{i,j+1}$:
\[e_i=d_0^{-1}
\begin{tikzpicture}[baseline=.9cm]
\filldraw[white] (0,0) rectangle (.6,2);
\filldraw[gray1] (.6,0) rectangle (.9,2);

\draw[red] (2.1,0) -- (2.1,2);
\draw[red] (2.7,0) -- (2.7,2);
\node[red] at (2.4,1.7) {$\cdots$};
\node[red] at (2.4,.3) {$\cdots$};

\draw (1.8,1.2) -- (1.8,2);
\draw (1.8,0) -- (1.8,.8);
\draw (1.5,1.2) -- (1.5,2);
\draw (1.5,0) -- (1.5,.8);
\draw (.9,0) -- (.9,2);
\draw (.6,0) -- (.6,2);
\node at (1.2,1.7) {$\cdots$};
\node at (1.2,.3) {$\cdots$};

\draw (1.5,1.2) arc (-180:0:.15);
\draw (1.5,.8) arc (180:0:.15);

\draw[thick,blue] (.3,0) -- (.3,2);

\draw[dashed,thick] (.15,.8) -- (2.85,.8);
\draw[dashed,thick] (2.85,.8) -- (2.85,1.2);
\draw[dashed,thick] (.15,1.2) -- (2.85,1.2);
\draw[dashed,thick] (.15,.8) -- (.15,1.2);

\node at (.3,-.2) {\tiny{$\mathbb{C}^{|P_0|}$}};
\node at (.6,2.2) {\tiny{$1^\text{st}$}};
\node at (1.5,-.2) {\tiny{$i^\text{th}$}};
\node[red] at (2.1,-.2) {\tiny{$1^\text{st}$}};
\node[red] at (2.7,-.2) {\tiny{$j^\text{th}$}};
\node at (-.25,1) {$P_0$};
\draw[dashed] (-.6,0) rectangle (3,2);
\end{tikzpicture}
\qquad\qquad
f_j=d_1^{-1}
\begin{tikzpicture}[baseline=.9cm]
\filldraw[white] (0,0) rectangle (.6,2);
\filldraw[gray1] (.6,0) rectangle (.9,2);

\draw[red] (1.8,0) -- (1.8,2);
\draw[red] (2.4,1.2) -- (2.4,2);
\draw[red] (2.4,0) -- (2.4,.8);
\draw[red] (2.7,1.2) -- (2.7,2);
\draw[red] (2.7,0) -- (2.7,.8);
\node[red] at (2.1,1.7) {$\cdots$};
\node[red] at (2.1,.3) {$\cdots$};

\draw (1.5,0) -- (1.5,2);
\draw (.9,0) -- (.9,2);
\draw (.6,0) -- (.6,2);
\node at (1.2,1.7) {$\cdots$};
\node at (1.2,.3) {$\cdots$};

\draw[red] (2.4,1.2) arc (-180:0:.15);
\draw[red] (2.4,.8) arc (180:0:.15);

\draw[thick,blue] (.3,0) -- (.3,2);

\draw[dashed,thick] (.15,.8) -- (2.85,.8);
\draw[dashed,thick] (2.85,.8) -- (2.85,1.2);
\draw[dashed,thick] (.15,1.2) -- (2.85,1.2);
\draw[dashed,thick] (.15,.8) -- (.15,1.2);

\node at (.3,-.2) {\tiny{$\mathbb{C}^{|P_0|}$}};
\node at (.6,2.2) {\tiny{$1^\text{st}$}};
\node at (1.5,-.2) {\tiny{$i^\text{th}$}};
\node[red] at (1.8,-.2) {\tiny{$1^\text{st}$}};
\node[red] at (2.4,-.2) {\tiny{$j^\text{th}$}};
\node at (-.25,1) {$P_0$};
\draw[dashed] (-.6,0) rectangle (3,2);
\end{tikzpicture}
\]

\item It is clear that $M_j=(M_{i,j},E^{M,l}_{i,j},e_i)_{i\ge 0}$ are Markov towers with the same modulus $d_0$ and $e_i\in M_{i+1,j}$ for all $i$, $i,j=0,1,2,\cdots$; $M_i=(M_{i,j},E^{M,r}_{i,j},f_j)_{j\ge 0}$ are Markov towers with the same modulus $d_1$ and $f_j\in M_{i,j+1}$ for all $j$.
\end{compactenum}
\end{construction}

\begin{remark}
A gauge equivalence $\Phi\sim \Phi'$ will result in an isomorphism of the corresponding Markov lattices.
\end{remark}

\subsection{From Markov lattice to $\mathcal{C}(\Gamma,\omega;\Phi)$}

First, we are going to explore more properties of Markov lattice.
\begin{proposition}\label{Markov lattice prop 2}
\mbox{}
\begin{compactenum}[\rm (a)]
\item $X_{i+1,j+1}:=\langle e_i,f_j\rangle$ is a 2-sided ideal of $M_{i+1,j+1}$ and hence $M_{i+1,j+1}$ can split as a direct sum of von Neumann algebras $X_{i+1,j+1}\oplus Y_{i+1,j+1}$. We also define $Y_{0,0}=M_{0,0},Y_{1,0}=M_{1,0},Y_{0,1}=M_{0,1},Y_{1,1}=M_{1,1}$ so that $X_{0,0}=X_{1,0}=X_{0,1}=X_{1,1}=0$. $X_{i+1,j+1}$ is called the old stuff and $Y_{i+1,j+1}$ is called the new stuff.
\item If $y\in Y_{i+1,j+1}$ and $x\in X_{i+1,j}$ or $x\in X_{i,j+1}$, then $yx=0$ in $M_{i+1,j+1}$. Hence $E^r_{i+1,j+1}(Y_{i+1,j+1})\subset Y_{i+1,j}$ and $E^l_{i+1,j+1}(Y_{i+1,j+1})\subset Y_{i,j+1}$, which means the new stuff comes from the old new stuff.
\item If $Y_{i,j}=0$, then $Y_{k,l}=0$ for all $k\ge i,\ l\ge j$.
\end{compactenum}
\end{proposition}
\begin{proof} Similar to Proposition \ref{Markov Tower prop 2}.
\end{proof}

Now we are going to construct $\mathcal{C}(\Gamma,\omega;\Phi)$ from a given Markov lattice $M$. 

\begin{construction}\label{Markov lattice to (Gamma,omega)}
The square partite graph and the edge weighting $(\Gamma,\omega)$:

From Markov lattice $M$, since each row and column is a Markov tower, we can obtain a Bratteli diagram $\Delta$ as in \S\ref{MT to Cb} (which can be viewed as a `lattice-partite' graph). 
After taking only the new vertices in $\Delta\cap Y_{i,j}$ and the edges between them, we obtain the principal graph $\Gamma_0$ because of Proposition \ref{Markov lattice prop 2}(2). Here, $\Gamma_0$ is not necessary a square-partite graph, so we have to do some identification.

For the new vertices $p_1\in \Gamma_0\cap Y_{i,j}$ and $p_2\in \Gamma_0\cap Y_{i+2,j-2}$, as in \S\ref{MT to Cb}, let $p_1'$ be the new old vertex of $p_1$ in $M_{i+2,j}$ and $p_2'$ be the new old vertex of $p_2$ in $M_{i+2,j}$. 
We identify $p_1$ with $p_2$ if $p_2'\in M_{i+2,j}p_1'$ (or equivalently $p_1'\in M_{i+2,j}p_2'$). 

For the pairs of new vertices $p_1\in \Gamma_0\cap Y_{i,j}$ and $q_1\in \Gamma_0\cap Y_{i+1,j}$, and the pairs of new vertices $p_2\in \Gamma_0\cap Y_{i+2,j-2}$ and $q_2\in \Gamma_0\cap Y_{i+3,j-2}$, suppose $p_1$ and $p_2$ are identified in $M_{i+2,j}$, $q_1$ and $q_2$ are identified in $M_{i+3,j}$ on above sense, then the  numbers of edges between $p_1,q_1$ and $p_2,q_2$ are equal, since they both equal to $$(\dim_\mathbb{C}(p_1'q_1'M_{i+2,j}'p_1'q_1'\cap p_1'q_1'M_{i+3,j}p_1'q_1'))^{\frac{1}{2}},$$
see the discussion in \S\ref{MT to Cb}.
Then we can also identify the edges between $p_1,q_1$ and $p_2,q_2$. Similar statement for $p_1\in \Gamma_0\cap Y_{i,j}$ and $r_1\in \Gamma_0\cap Y_{i,j+1}$, and the pairs of new vertices $p_2\in \Gamma_0\cap Y_{i+2,j-2}$ and $r_2\in \Gamma_0\cap Y_{i+2,j-1}$. After above identification as well as the edges between those identified vertices (see following example), we obtain a graph $\Gamma$, which is a square-partite graph.

Then $V_{ij}\subset V(\Gamma)$ contains all the vertices in $V(\Gamma_0)\cap M_{i+2m,j+2n}$, $i,j=0,1$, $m,n\in\mathbb{Z}_{\ge 0}$.

The edge-weighting $\omega$ can be obtained the same way as in \S\ref{MT to Cb}.
\end{construction}

\begin{example}\label{Ex:square partite A3}
Here we provide an example to see the difference between the square-partite graph and the principal graph of a Markov lattice. 
In the diagram below, if $p_1$ is at depth zero, then $p_2$ is at depth $2$ of the principal graph. 
Therefore, as a new vertex, $p_2$ will appear in two places $M_{0,2}$ and $M_{2,0}$, but their reflections/new old vertices coincide in $M_{2,2}$.
\[
\begin{tikzpicture}[baseline = .6cm]
\draw[red,thick] (0,-.1) -- (.7,.7);
\draw[red,thick] (0,-.1) -- (-.7,.7);
\draw[red,thick] (0,.1) -- (.7,.7);
\draw[red,thick] (0,.1) -- (-.7,.7);
\draw[red,thick] (.7,.7) -- (0,1.3);
\draw[red,thick] (-.7,.7) -- (0,1.3);
\draw[red,thick] (.7,.7) -- (0,1.5);
\draw[red,thick] (-.7,.7) -- (0,1.5);

\node at (0,.1)[circle,fill,red,inner sep=1.5pt]{};
\node at (0,-.1)[circle,fill,red,inner sep=1.5pt]{};
\node at (.7,.7)[draw,red,diamond, thick,fill=white,inner sep=1.5pt]{};
\node at (-.7,.7)[draw,red,circle,fill=white,inner sep=1.5pt] {};
\node at (0,1.3)[rectangle,fill,red,inner sep=1.7pt]{};
\node at (0,1.5)[rectangle,fill,red,inner sep=1.7pt]{};

\node[red] at (0,.4) {$p_1$};
\node[red] at (0,-.4) {$p_2$};
\node[red] at (-.7,1) {$p_3$};
\node[red] at (0,1) {$p_4$};
\node[red] at (0,1.8) {$p_5$};
\node[red] at (.7,1) {$p_6$};

\node[red] at (0,-1.9) {square-partite graph};
\end{tikzpicture}
\qquad \Longrightarrow \qquad
\begin{tikzpicture}[baseline = 2cm]
\draw[red,thick] (0,0) -- (.7,.7);
\draw[red,thick] (0,0) -- (-.7,.7);
\draw[red,thick] (.7,.7) -- (0,1.3);
\draw[red,thick] (-.7,.7) -- (0,1.3);
\draw[red,thick] (.7,.7) -- (0,1.5);
\draw[red,thick] (-.7,.7) -- (0,1.5);

\draw[red,thick] (.7,.7) -- (1.4,1.3);
\draw (.7,.7) -- (1.4,1.5);
\draw[red,thick] (-.7,.7) -- (-1.4,1.3);
\draw (-.7,.7) -- (-1.4,1.5);

\draw (-1.4,1.3) -- (-.7,2.1);
\draw (-1.4,1.5) -- (-.7,2.1);
\draw (1.4,1.3) -- (.7,2.1);
\draw (1.4,1.5) -- (.7,2.1);

\draw (0,1.3) -- (.7,2.1);
\draw (0,1.5) -- (.7,2.1);
\draw (0,1.3) -- (-.7,2.1);
\draw (0,1.5) -- (-.7,2.1);

\draw (-.7,2.1) -- (0,2.7);
\draw (-.7,2.1) -- (0,2.9);
\draw (.7,2.1) -- (0,2.7);
\draw (.7,2.1) -- (0,2.9);

\draw (1.4,1.3) -- (2.1,2.1);
\draw (1.4,1.5) -- (2.1,2.1);
\draw (.7,2.1) -- (1.4,2.7);
\draw (.7,2.1) -- (1.4,2.9);
\draw (1.4,2.7) -- (2.1,2.1);
\draw (1.4,2.9) -- (2.1,2.1);

\draw (-1.4,1.3) -- (-2.1,2.1);
\draw (-1.4,1.5) -- (-2.1,2.1);
\draw (-.7,2.1) -- (-1.4,2.7);
\draw (-.7,2.1) -- (-1.4,2.9);
\draw (-1.4,2.7) -- (-2.1,2.1);
\draw (-1.4,2.9) -- (-2.1,2.1);

\draw (0,2.7) -- (.7,3.5);
\draw (0,2.9) -- (.7,3.5);
\draw (1.4,2.7) -- (.7,3.5);
\draw (1.4,2.9) -- (.7,3.5);

\draw (0,2.7) -- (-.7,3.5);
\draw (0,2.9) -- (-.7,3.5);
\draw (-1.4,2.7) -- (-.7,3.5);
\draw (-1.4,2.9) -- (-.7,3.5);

\draw (.7,3.5) -- (0,4.1);
\draw (.7,3.5) -- (0,4.3);
\draw (-.7,3.5) -- (0,4.1);
\draw (-.7,3.5) -- (0,4.3);

\node at (0,0)[circle,fill,red,inner sep=1.5pt]{};
\node at (.7,.7)[draw,red,diamond, thick,fill=white,inner sep=1.5pt]{};
\node at (-.7,.7)[draw,red,circle,fill=white,inner sep=1.5pt]{};
\node at (0,1.3)[rectangle,fill,red,inner sep=1.7pt]{};
\node at (0,1.5)[rectangle,fill,red,inner sep=1.7pt]{};
\node at (1.4,1.3)[circle,fill,red,inner sep=1.5pt]{};
\node at (-1.4,1.3)[circle,fill,red,inner sep=1.5pt]{};

\node at (1.4,1.5)[circle,fill,inner sep=1.5pt]{};
\node at (-1.4,1.5)[circle,fill,inner sep=1.5pt]{};
\node at (.7,2.1)[draw,circle,fill=white,inner sep=1.5pt]{};
\node at (-.7,2.1)[draw,diamond, thick,fill=white,inner sep=1.5pt]{};
\node at (0,2.9)[circle,fill,inner sep=1.5pt]{};
\node at (0,2.7)[circle,fill,inner sep=1.5pt]{};
\node at (.7,3.5)[draw,diamond, thick,fill=white,inner sep=1.5pt]{};
\node at (1.4,2.7)[rectangle,fill,inner sep=1.7pt]{};
\node at (1.4,2.9)[rectangle,fill,inner sep=1.7pt]{};
\node at (2.1,2.1)[draw,diamond, thick,fill=white,inner sep=1.5pt]{};
\node at (-.7,3.5)[draw,circle,fill=white,inner sep=1.5pt]{};
\node at (-1.4,2.7)[rectangle,fill,inner sep=1.7pt]{};
\node at (-1.4,2.9)[rectangle,fill,inner sep=1.7pt]{};
\node at (-2.1,2.1)[draw,circle,fill=white,inner sep=1.5pt]{};
\node at (0,4.1)[rectangle,fill,inner sep=1.7pt]{};
\node at (0,4.3)[rectangle,fill,inner sep=1.7pt]{};

\node[red] at (0,.3) {$p_1$};
\node[red] at (-1.4,1) {$p_2$};
\node[red] at (1.4,1) {$p_2$};
\node[red] at (-.7,1) {$p_3$};
\node[red] at (0,1) {$p_4$};
\node[red] at (0,1.8) {$p_5$};
\node[red] at (.7,1) {$p_6$};

\node at (-1.4,1.8) {$p_1$};
\node at (1.4,1.8) {$p_1$};
\node at (0,2.4) {$p_1$};
\node at (0,3.2) {$p_2$};
\node at (.7,2.4) {$p_3$};
\node at (-.7,2.4) {$p_6$};
\node at (2.1,2.4) {$p_6$};
\node at (-2.1,2.4) {$p_3$};
\node at (1.4,2.4) {$p_5$};
\node at (1.4,3.2) {$p_4$};
\node at (-1.4,2.4) {$p_5$};
\node at (-1.4,3.2) {$p_4$};
\node at (.7,3.8) {$p_6$};
\node at (-.7,3.8) {$p_3$};
\node at (0,3.8) {$p_4$};
\node at (0,4.6) {$p_5$};

\node[red] at (0,-.5) {principal graph with base point $p_1$};
\node at (0,.-1) {and Bratteli diagram};
\end{tikzpicture}
\]

\end{example}

\begin{remark} \label{Rmk:vertex first appear in SQ graph}
Suppose vertex $q\in V_{00}$ is at depth $2n$ of the principal graph, then $q$ will first appear in $M_{2i,2n-2i}$, $i=0,1,\cdots,n$; if $q\in V_{10}$ is at depth $2n+1$, then $q$ will first appear in $M_{2i+1,2n-2i}$, $i=0,1,\cdots,n$; if $q\in V_{01}$ is at depth $2n+1$, then $q$ will first appear in $M_{2i,2n+1-2i}$, $i=0,1,\cdots,n$; if $q\in V_{11}$ is at depth $2n+2$, then $q$ will first appear in $M_{2i+1,2n+1-2i}$, $i=0,1,\cdots,n$.
\end{remark}



Next, we compute the biunitary connection $\Phi$.
\begin{notation}\textbf{and Observation} We choose $p_0\in V_{00}$ as the base point, which is at depth 0. 
Similar to Observation \ref{Notation:Kn}, denote $\Lambda_{0,n}$ to be the subgraph of $\Lambda_0$ with vertices depth $\le n$, similar definition for $\Omega_{0,n},\ \Lambda_{1,n}$ and $\Omega_{1,n}$, see Definition \ref{Def:square partite graph}. 
The corresponding \textsf{Hilb}-enriched graphs are $K_{i,n}:=K_{\Lambda_{i,n}}$, $L_{i,n}:=L_{\Omega_{i,n}}$. 
From Construction \ref{C(Phi) to ML}, $N_{i,j}:= \End(\mathbb{C}^{|p_0|}\otimes K_0^{\alt\otimes i}\otimes L_?^{\alt\otimes j})$. WLOG, let $2\nmid i$. Observe that 
$$N_{i,j}=\End(K_{0,1}\otimes \overline{K}_{0,2}\otimes \cdots \overline{K}_{0,i}\otimes L_{1,i+1}\otimes \overline{L}_{1,i+2}\otimes\cdots\otimes L^?_{1,i+j}),$$
where $L^?_{1,j}=L_{1,j}$ if $2\nmid j$, $L^?_{1,j}=\overline{L}_{1,j}$ if $2\mid j$.
\end{notation}

\begin{example}
Following Example \ref{Ex:square partite A3}, 
\[
\begin{tikzpicture}
\draw[red] (0,0) -- (2.1,2.1);
\draw (0,0) -- (-2.1,2.1);
\draw[red] (-.7,.7) -- (1.4,2.8);
\draw (.7,.7) -- (-1.4,2.8);
\draw[red] (-1.4,1.4) -- (.7,3.5);
\draw (1.4,1.4) -- (-.7,3.5);
\draw (2.1,2.1) -- (0,4.2);
\draw[red] (-2.1,2.1) -- (0,4.2);

\node at (-.7,.3) {\tiny{$K_{0,1}$}};
\node at (-1.4,1.0) {\tiny{$\overline{K}_{0,2}$}};
\node at (-2.1,1.7) {\tiny{$K_{0,3}$}};
\node at (.7,1.7) {\tiny{$K_{0,3}$}};
\node at (0,2.4) {\tiny{$\overline{K}_{0,4}$}};
\node at (-.7,3.1) {\tiny{$K_{0,5}$}};

\node at (0,1) {\tiny{$K_{1,2}$}};
\node at (-.7,1.7) {\tiny{$\overline{K}_{1,3}$}};
\node at (-1.4,2.4) {\tiny{$K_{1,4}$}};
\node at (1.4,2.4) {\tiny{$K_{1,4}$}};
\node at (.7,3.1) {\tiny{$\overline{K}_{1,5}$}};
\node at (0,3.8) {\tiny{$K_{1,6}$}};

\node[red] at (0,.4) {\tiny{$L_{0,1}$}};
\node[red] at (.7,1.1) {\tiny{$\overline{L}_{0,2}$}};
\node[red] at (1.4,1.8) {\tiny{$L_{0,3}$}};
\node[red] at (-1.4,1.8) {\tiny{$L_{0,3}$}};
\node[red] at (-.7,2.5) {\tiny{$\overline{L}_{0,4}$}};
\node[red] at (0,3.2) {\tiny{$L_{0,5}$}};

\node[red] at (-.7,1.1) {\tiny{$L_{1,2}$}};
\node[red] at (0,1.8) {\tiny{$\overline{L}_{1,3}$}};
\node[red] at (.7,2.5) {\tiny{$L_{1,4}$}};
\node[red] at (-2.1,2.5) {\tiny{$L_{1,4}$}};
\node[red] at (-1.4,3.2) {\tiny{$\overline{L}_{1,5}$}};
\node[red] at (-.7,3.9) {\tiny{$L_{1,6}$}};

\node at (0,-.2) {\tiny{$N_{0,0}$}};
\node at (-2.5,2.1) {\tiny{$N_{3,0}$}};
\node at (-1.1,2.1) {\tiny{$N_{2,1}$}};
\end{tikzpicture}
\]
we have
$$K_{0,1}=\begin{bmatrix}0 & 0 & \mathbb{C} & 0 & 0 & 0 \\ 0 & 0 & 0 & 0 & 0 & 0 \\ 0 & 0 & 0 & 0 & 0 & 0 \\ 0 & 0 & 0 & 0 & 0 & 0 \\ 0 & 0 & 0 & 0 & 0 & 0 \\ 0 & 0 & 0 & 0 & 0 & 0 \end{bmatrix}\ 
\overline{K}_{0,2}=\begin{bmatrix}0 & 0 & 0 & 0 & 0 & 0 \\ 0 & 0 & 0 & 0 & 0 & 0 \\ \mathbb{C} & \mathbb{C} & 0 & 0 & 0 & 0 \\ 0 & 0 & 0 & 0 & 0 & 0 \\ 0 & 0 & 0 & 0 & 0 & 0 \\ 0 & 0 & 0 & 0 & 0 & 0 \end{bmatrix}\ 
K_{0,3}=\begin{bmatrix}0 & 0 & \mathbb{C} & 0 & 0 & 0 \\ 0 & 0 & \mathbb{C} & 0 & 0 & 0 \\ 0 & 0 & 0 & 0 & 0 & 0 \\ 0 & 0 & 0 & 0 & 0 & 0 \\ 0 & 0 & 0 & 0 & 0 & 0 \\ 0 & 0 & 0 & 0 & 0 & 0 \end{bmatrix}$$

$$K_{1,4}=\begin{bmatrix}0 & 0 & 0 & 0 & 0 & 0 \\ 0 & 0 & 0 & 0 & 0 & 0 \\ 0 & 0 & 0 & 0 & 0 & 0 \\ 0 & 0 & 0 & 0 & 0 & 0 \\ 0 & 0 & 0 & 0 & 0 & 0 \\ 0 & 0 & 0 & \mathbb{C} & \mathbb{C} & 0 \end{bmatrix}\ 
L_{0,1}=\begin{bmatrix}0 & 0 & 0 & 0 & 0 & \mathbb{C} \\ 0 & 0 & 0 & 0 & 0 & 0 \\ 0 & 0 & 0 & 0 & 0 & 0 \\ 0 & 0 & 0 & 0 & 0 & 0 \\ 0 & 0 & 0 & 0 & 0 & 0 \\ 0 & 0 & 0 & 0 & 0 & 0 \end{bmatrix}\ 
\overline{L}_{1,3}=\begin{bmatrix}0 & 0 & 0 & 0 & 0 & 0 \\ 0 & 0 & 0 & 0 & 0 & 0 \\ 0 & 0 & 0 & 0 & 0 & 0 \\ 0 & 0 & \mathbb{C} & 0 & 0 & 0 \\ 0 & 0 & \mathbb{C} & 0 & 0 & 0 \\ 0 & 0 & 0 & 0 & 0 & 0 \end{bmatrix}$$

$$K_{0,1}\otimes \overline{K}_{0,2}\otimes L_{0,3} = \begin{bmatrix}0 & 0 & 0 & 0 & 0 & \mathbb{C}^6 \\ 0 & 0 & 0 & 0 & 0 & 0 \\ 0 & 0 & 0 & 0 & 0 & 0 \\ 0 & 0 & 0 & 0 & 0 & 0 \\ 0 & 0 & 0 & 0 & 0 & 0 \\ 0 & 0 & 0 & 0 & 0 & 0 \end{bmatrix} \cong  K_{0,1}\otimes L_{1,2}\otimes \overline{K}_{1,3} \cong L_{0,1}\otimes K_{1,2}\otimes \overline{K}_{1,3}$$

Similar to Example \ref{Ex:A_5 Markov tower}, 
the entry $(i,j)$ in $N_{m,n}$ indicates number of paths from the vertex $p_i$ at depth $0$ to the vertex $p_j$ at depth $m+n$.
Note that the base point is a single vertex $p_1$, so only at entry $(1,j)$ can be nonzero.
\end{example}

\begin{remark}\label{Rmk:inner automorphism and unitary}
Any automorphism of $M_n(\mathbb{C})$ is inner. To be precise, if $\alpha\in \text{Aut}(M_n(\mathbb{C}))$, then there exists a unitary $u\in M_n(\mathbb{C})$, such that $\alpha(x)=uxu^*=\text{Ad}(u)(x)$, for any $x\in M_n(\mathbb{C})$. Moreover, this unitary $u$ is unique up to a unit scalar. 
Indeed, if $uxu^*=u_1xu_1^*$ for all $x\in M_n(\mathbb{C})$, then $x(u^*u_1)=(u^*u_1)x$, which implies that $u^*u_1$ is in the center of $M_n(\mathbb{C})$. Thus, $u^*u_1=a\in\mathbb{C}$ with $|a|=1$ and hence $u_1=au$.

As a corollary, for 1-morphisms $H,G$, if $\alpha:\End(H)\cong \End(G)$ is a $*$-isomorphism, then there exists a unitary 2-morphism $u:H\to G$ such that $\alpha=\text{Ad}(u)$. 

\noindent \textbf{Warning}: the unitary $u$ is obtained by taking a unitary $u_{i,j}$ in each entry.
Thus any two choices of implementing unitary $u=(u_{i,j})$ and $v=(v_{i,j})$ differ by a matrix of scalars $(a_{i,j})$ which may be distinct.
Hence the unitary $u$ is unique up to a matrix of scalars.
\end{remark}

\begin{construction}\label{Construction:biunitary connection in Markov lattice}
The biunitary connection $\Phi$: The construction (for the tracial case) has been written in \cite[\S5.5]{JS97} in the language of path algebras. 
For convenience, we will construct it here using our language.

From Construction \ref{Markov lattice to (Gamma,omega)} and Remark \ref{C(K0,K1,L0,L1;ev) and C(Gamma,omega)}, the 2-category $\mathcal{C}(\Gamma,\omega)$ can be constructed.

In order to obtain the biunitary connection $\Phi$,  we shall compute it componentwise, which is similar to the idea to compute the edge-weighting in \S\ref{MT to Cb}. The goal is to compute $\Phi_{pr}:(K_0\otimes L_1)_{pr}=\bigoplus_{q\in V_{10}}K_{0,pq}\otimes L_{1,qr}\to \bigoplus_{s\in V_{01}}L_{0,ps}\otimes K_{1,sr}=(L_0\otimes K_1)_{pr}$ for each pair $(p,r)\in V_{00}\times V_{11}$. 

Suppose $p$ is at depth $2n$ of the principal graph and $r$ is at depth $2n+2$. By Remark \ref{Rmk:vertex first appear in SQ graph}, $p$ first appear in $M_{0,2n}$ and $r$ first appears in $M_{1,2n+1}$.    

Consider two path models $M_{0,0}\subset M_{0,1}\subset \cdots \subset M_{0,2n}\subset M_{0,2n+1}\subset M_{1,2n+1}$ and $M_{0,0}\subset M_{0,1}\subset \cdots \subset M_{0,2n}\subset M_{1,2n}\subset M_{1,2n+1}$. 

Similar to Proposition \ref{Prop: Nn-1' cap Nn+1}, we have 
\begin{align*}
    N_{0,2n}'\cap N_{1,2n+1} &= \id_{K_{0,1}\otimes \overline{K}_{0,2}\otimes\cdots\otimes \overline{K}_{0,2n}}\otimes \End(K_{0,2n+1}\otimes L_{1,2n+1})\quad\text{ for the first model}\\
    N_{0,2n}'\cap N_{1,2n+1} &= \id_{K_{0,1}\otimes \overline{K}_{0,2}\otimes\cdots\otimes K_{0,2n-1}}\otimes \End(L_{0,2n}\otimes K_{1,2n+1})\quad\text{ for the second model}.
\end{align*}

Let $\psi:M_{1,2n+1}\to N_{1,2n+1}$ denote the $*$-isomorphism onto the first model and $\psi':M_{1,2n+1}\to N_{1,2n+1}$ denote the $*$-isomorphism onto the second model, then
\begin{align*}
    \psi&: M_{0,2n}'\cap M_{1,2n+1}\to N_{0,2n}'\cap N_{1,2n+1}\cong \End(K_{0,2n+1}\otimes L_{1,2n+1})\\
    \psi'&: M_{0,2n}'\cap M_{1,2n+1}\to N_{0,2n}'\cap N_{1,2n+1}\cong \End(L_{0,2n}\otimes K_{1,2n+1}).
\end{align*}
are $*$-isomorphisms. Then $\psi'\circ \psi^{-1}:\End(K_{0,2n+1}\otimes L_{1,2n+1})\to \End(L_{0,2n}\otimes K_{1,2n+1})$ is a $*$ isomorphism between two $1$-morphisms. By Remark \ref{Rmk:inner automorphism and unitary}, their exists a unique unitary $u$ up to a matrix of scalars such that $\psi'\circ \psi^{-1}=\text{Ad}(u)$. We define $\Phi_{pr}:=u_{pr}$. 

Similar to Remark \ref{Rmk:choice of ONB}, we secretly make a choice of ONB when we construct the generators $K_i,L_j$ from the square-partite graph $\Gamma$, $i,j=0,1$.
Different choice results in multiplying a unitary on each generator.  
Combining Definition \ref{Def:Gauge equivalence} of gauge equivalence and above discussion, the biunitary connection $\Phi$ we construct here is unique up to gauge equivalence.

\end{construction}

\subsection{$\mathcal{C}(\Phi)$ and $\End_0^\dagger(\mathcal{M},F,G)$} \label{C(Phi) and End(M,F,G)}

We have already seen the method to construct a Markov lattice from $\mathcal{C}(\Phi)$ above or from $\mathcal{M}$ in \S\ref{Cpt 3} with a simple base point, where $\mathcal{M}$ is an indecomposable semisimple ${\rm C}^*$ $\mathcal{A}-\mathcal{B}$ bimodule category. 
In this section, by using the similar technique as in \S\ref{C(K) End(M)}, we will show their relation to each other.

\begin{definition}
Suppose $\mathcal{M}$ is an indecomposable semisimple ${\rm C}^*$ $\mathcal{TLJ}(d_0)-\mathcal{TLJ}(d_1)$ bimodule category, where $X=1^+\otimes X\otimes 1^-,Y=1^+\otimes Y\otimes 1^-$ are the generators of $\mathcal{TLJ}(d_0)$ and $\mathcal{TLJ}(d_1)$ respectively. 
Define $F=X\rhd -$, $\overline{F}=\overline{X}\rhd -$, $G=-\lhd Y$, $\overline{G}=-\lhd \overline{Y}$, which are endofunctors on $\mathcal{M}$. 
Note that $(F,\overline{F})$ and $(G,\overline{G})$ are adjoint pairs, with unit $\ev_F,\ev_G$ induced by $\ev_X,\ev_Y$ and counit $\coev_F,\coev_G$ induced by $\coev_{\overline{X}},\coev_{\overline{Y}}$.

Define $\End_0^\dagger(\mathcal{M},F,G)$ to be the full subcategory of $\End^\dagger(\mathcal{M})$ Cauchy tensor generated by $F,\overline{F},G,\overline{G}$, so it is a rigid $\rm C^*$ tensor category.

We warn the reader that $\End^\dag_0(\mathcal{M},F,G)$ will only be multitensor ($\dim(\End(\id_{\mathcal{M}}))<\infty$) when $\mathcal{M}$ is finitely semisimple.
\end{definition}

\begin{definition}[Biunitary connection in $\End_0^\dagger(\mathcal{M},F,G)$]
Note that the bimodule associator 
$\alpha_{X,-,Y}:(X\rhd -)\lhd Y\to X\rhd (-\lhd Y)$ is a unitary, which induces a natural isomorphism $\Phi_{F,G}:F\otimes G\to G\otimes F$, where $F\otimes G:=G\circ F$. Then $\Phi_{G,\overline{F}}:G\otimes\overline{F}\to \overline{F}\otimes G$ is equal to the $90^\circ$ rotation $\Phi_{F,G}^r$ defined as follows:
$$\Phi_{F,G}^r: = (\id_{\overline{F}}\otimes \id_G\otimes \ev_F)\circ (\id_{\overline{F}}\otimes \Phi_{F,G}\otimes \id_{\overline{F}})\circ (\coev_F\otimes \id_G\otimes \id_{\overline{F}}).$$

It is easy to show that $\Phi_{F,G}$ is vertical and horizontal unitary and so is $\Phi_{G,\overline{F}}$.
\end{definition}

Similar to \S\ref{C(K) End(M)}, we will show that the tensor category $\End_0^\dagger(\mathcal{M},F,G)$ and 2-category $\mathcal{C}(\Phi)$ are unitarily equivalent.

\begin{construction}\label{Construction:End(M,F,G) to C(Phi)}
We construct $\mathcal{C}(\Phi)$ from $\End_0^\dagger(\mathcal{M},F,G)$ functorially.
\begin{compactenum}[(a)]
\item 
Let $V_{00}$ be a set of representatives of all simple objects $P\in\mathcal{M}$ such that $P=1^+\rhd P\lhd 1^+$; $V_{10}$ be the set of representatives of all simple objects $Q\in\mathcal{M}$ such that $Q=1^-\rhd Q\lhd 1^+$; $V_{11}$ be the set of representatives of all simple objects $R\in\mathcal{M}$ such that $R=1^-\rhd R\lhd 1^-$; $V_{01}$ be the set of representatives of all simple objects $S\in\mathcal{M}$ such that $S=1^+\rhd S\lhd 1^-$.
Then the objects are the sets $V_{i,j}$, $i,j=0,1$ and their union $V=V_{00}\sqcup V_{01}\sqcup V_{11}\sqcup V_{10}$.
\item 1-morphism: The 1-morphism of $\mathcal{C}(\Phi)$ is the object of $\End_0^\dagger(\mathcal{M},F,G)$. The way to construct the corresponding $V\times V$-bigraded Hilbert space from an endofunctor is the same as in Construction \ref{construction:CatFromEndM}. The same for the dual 1-morphism and tensor structure/composition.
\item 2-morphism: The 2-morphism of $\mathcal{C}(\Phi)$ is the morphism of $\End_0^\dagger(\mathcal{M},F,G)$.
\item 1-morphism generator: Define
\begin{align*}
    K_0:=H_{J^+}\otimes H_{F} & \qquad \overline{K_0}=H_{J^+}\otimes H_{\overline{F}} \qquad
    K_1:=H_{J^-}\otimes H_F  \qquad \overline{K_1}=H_{J^-}\otimes H_{\overline{F}}\\
    L_0:=H_{I^+}\otimes H_G & \qquad \overline{L_0}=H_{I^+}\otimes H_{\overline{G}} \qquad
    L_1:=H_{I^-}\otimes H_G  \qquad \overline{L_0}=H_{I^-}\otimes H_{\overline{G}}
\end{align*}
\item $\ev$ and $\coev$. The same as in Construction \ref{construction:CatFromEndM}(h).
\item Biunitary connection: $\Phi:K_0\otimes L_1\to L_0\otimes K_1$ is defined as $\Phi_{F,G}:F\otimes G\to G\otimes F$. The check that $\Phi$ is vertical and horizontal unitary is left to the reader.
\end{compactenum}
\end{construction}

\begin{construction}
For the convenience to the reader, we also provide the construction from $\mathcal{C}(\Phi)$ to $\End_0^\dagger(\mathcal{M},F,G)$:
\begin{compactenum}[(a)]
\item Object: The object are the 1-morphisms in $\mathcal{C}(\Phi)$. In particular, the generator $F=K_0\oplus K_1$, $\overline{F}=\overline{K_0}\oplus \overline{K_1}$, $G=L_0\oplus L_1$ and $\overline{G}=\overline{L_0}\oplus\overline{L_1}$; the unit $I^+=1^+\rhd - = \mathbb{C}^{|V_{00}\sqcup V_{01}|}$, $I^-=1^-\rhd - = \mathbb{C}^{|V_{10}\sqcup V_{11}|}$, $J^+=-\lhd 1^+ = \mathbb{C}^{|V_{00}\sqcup V_{10}|}$ and $J^-=-\lhd 1^-=\mathbb{C}^{|V_{01}\sqcup V_{11}|}$.

\item Morphism: The morphisms are the 2-morphisms in $\mathcal{C}(\Phi)$.

\item The associator: Note that $F\otimes G=(K_0\oplus K_1)\otimes (L_0\oplus L_1)=K_0\otimes L_1$ and $G\otimes F=(L_0\oplus L_1)\otimes (K_0\oplus K_1) =L_0\otimes K_1$, the associator $\Phi_{F,G}:F\otimes G\to G\otimes F$ is defined as the biunitary connection $\Phi:K_0\otimes L_1\to L_0\otimes K_1$. All the 8 cases of associators are defined as $\Phi^g$, where $g\in \langle r,\dagger\rangle$. 
\end{compactenum}
\end{construction}

\begin{theorem}
There is a bijective correspondence between equivalence classes of the following:
\[\left\{\, \parbox{5.3cm}{\rm Indecomposable semisimple $\rm C^*$
$\mathcal{TLJ}(d_0)-\mathcal{TLJ}(d_1)$ bimodule categories $\mathcal{M}$} \,\right\}\ \, \cong\ \, \left\{\, \parbox{7.4cm}{\rm $\rm W^*$ 2-subcategories $\mathcal{C}(\Gamma,\omega;\Phi)$ of \textsf{BigHilb}, where $\Gamma$ is a balanced $(d_0,d_1)$-fair square partite graph with edge-weighting $\omega$ and $\Phi$ a biunitary connection} \,\right\}\]
Equivalence on the left hand side is unitary equivalence; equivalence on the right hand side is isomorphism on the edge-weighted square-partite graph and gauge equivalence on biunitary connection.
\end{theorem}
\begin{proof}
We can prove this correspondence for the version with base point by using the Markov lattice. According to Construction \ref{Construction:End(M,F,G) to C(Phi)}, the correspondence holds without fixing the base point. 

As for the equivalence, combining  Remark \ref{C(K0,K1,L0,L1;ev) and C(Gamma,omega)},  Definition \ref{Def:Gauge equivalence} and the last paragraph in Construction \ref{Construction:biunitary connection in Markov lattice}, the isomorphism on the edge-weighted graph $(\Gamma,\omega)$ and gauge equivalence on $\Phi$ corresponds to the unitary equivalence on $\mathcal{C}(\Phi)$, which corresponds to the unitary equivalence on $\mathcal{TLJ}(d_0)-\mathcal{TLJ}(d_1)$ bimodule category $\mathcal{M}$ based on Construction \ref{Construction:biunitary connection in Markov lattice} and Remark \ref{Rmk:TLJ equivalence}.
\end{proof}

\section{The tracial case}
In this chapter, we finally discuss the tracial/pivotal case for (bi)module categories. As an application, we prove the module embedding theorem for (infinite depth) graph planar algebra.
\subsection{Tracial Markov towers and pivotal module categories}\label{Tracial Markov tower}

\begin{definition}\cite{Sc13} \label{pivotal module category}
Let $\mathcal{C}$ be a rigid C* (multi)tensor category with the canonical spherical unitary dual functor.
We call $\mathcal{M}$ a semisimple pivotal C* $\mathcal{C}-$module category, 
if there exists a pivotal trace $\tr^{\mathcal{M}}$ compatible with the spherical structure on $\mathcal{C}$, i.e.,
$$\tr^{\mathcal{M}}_{m\lhd c}(f)=\tr^{\mathcal{M}}_{m}((\id_m\lhd \coev^\dagger_c)\circ(f\lhd \id_{\overline{c}})\circ (\id_m\lhd \coev_c)),$$
for all $f\in \End(m\lhd c)$, where $m\in\mathcal{M},\ c\in\mathcal{C}$.
\end{definition}

\begin{remark}
If $f\in \End(c),\ c\in \mathcal{C}$ and $m\in\mathcal{M}$, 
\begin{align*}
   \tr^{\mathcal{M}}_{m\lhd c}(\id_m\lhd f)&= \tr^\mathcal{M}_{m}((\id_m\lhd \coev^\dagger_c)\circ((\id_m\lhd f)\lhd \id_{\overline{c}})\circ (\id_m\lhd \coev_c))\\
   &=\tr^\mathcal{M}_{m}(\id_m\lhd (\coev^\dagger_c\circ (f\lhd \id_{\overline{c}}\circ \coev_c)))\\
   &=\tr^\mathcal{M}_{m}(\id_m\lhd \tr^{\mathcal{A}}_c(f))\\
   &= \tr^\mathcal{M}_{m}(\id_m)\cdot\tr^{\mathcal{A}}_c(f).
\end{align*}
Here we call $\tr^\mathcal{M}_m(\id_m)$ the dimension of object $m$.
\end{remark}

\begin{remark}\cite[\S4.1]{Sc13}
If $\mathcal{C}$ is fusion and $\mathcal{M}$ is indecomposable, then the pivotal trace $\tr^\mathcal{M}$ is unique up to scalar.
\end{remark}

\begin{definition}[Tracial Markov tower]
We call $M$ a tracial Markov tower if $M$ a Markov tower equipped with a unital trace $\tr$ on $\bigcup_{n\ge 0}M_n$ and the conditional expectation $E_n$ are trace-preserving, i.e.,
$$\tr\circ E_n=\tr$$
on $M_n$, $n\ge 0$.
\end{definition}

\begin{definition}
We call $M$ a tracial standard $A-$module, where $A$ is a standard $\lambda$-lattice, if $\tr_M|_A=\tr_A$ and $M$ is a standard $A-$module, see Definition \ref{Def:Markov tower as standard module}. 
\end{definition}

Let $A$ be a standard $\lambda$-lattice. If we start with a tracial standard $A$-module $M$, combining the construction in \S\ref{MT to planar mod cat} and the proof in proposition \ref{A_0 is pivotal from A}, we are able to construct a pivotal planar $\mathcal{A}_0-$module category. Furthermore, from this pivotal planar $\mathcal{A}_0-$module category, we can construct an indecomposable semisimple pivotal C* $\mathcal{A}-$module category with a choice of simple base object. The following is the theorem.

\begin{theorem}\label{Thm:tracial MT and pivotal module category} There is a bijective correspondence between equivalence classes of the following:
\[\left\{\, \parbox{5.5cm}{\rm Tracial Markov towers $M=(M_i)_{i\ge 0}$ with $\dim(M_0)=1$ as standard right modules over a standard $\lambda$-lattice $A$} \,\right\}\ \, \cong\ \, \left\{\, \parbox{7.4cm}{\rm Pairs $(\mathcal{M}, Z)$ with $\mathcal{M}$ an indecomposable semisimple pivotal C* right $\mathcal{A}-$module category together with a choice of simple base object $Z= Z\lhd 1^+_\mathcal{A}$} \,\right\}\]
Equivalence on the left hand side is trace-preserving $*$-isomorphism on the tracial Markov tower as standard $A-$module; equivalence on the right hand side is the pivotal unitary $\mathcal{A}-$module equivalence on their Cauchy completions which maps simple base object to simple base object.
\end{theorem}


Let us look at the balanced $d$-fair bipartite graph $(\Lambda,\omega)$ from the tracial Markov tower $M$. Since the evaluation and coevaluation are compatible with the trace, the edge-weighting comes from a vertex-weighting, see \cite[Prop. 6.8]{JP19}. 
To be precise,
\begin{definition}[Vertex weighting]\label{Def:vertex  weighting}
Let $\Lambda$ be a bipartite graph. 
Let $\nu:V(\Lambda)\to (0,\infty)$ be a weighting on the vertices of $\Lambda$ 
which satisfies the Frobenius-Perron condition: for each $P\in V(\Lambda)$,
$$\sum_{\{Q\in V(\Lambda):P,Q\text{ adjacent}\}}\nu(Q)=d\cdot\nu(P).$$
In the sum on the left hand side, $\nu(Q)$ has number of edges between $P\to Q$ copies.

From an undirected bipartite graph, one can obtain a directed graph with involution\cite[Def.~2.20]{HP17}. 
Then for $e:P\to Q$, define $w(e):=\frac{\nu(P)}{\nu(Q)}$. 
The $d$-fairness and balance condition in Definition \ref{Def:d-fairness and balanced} follows automatically.
\end{definition}

\begin{remark}\label{Rmk:fusion graph}
Suppose $\mathcal{M}$ is an indecomposable semisimple $\rm C^*$ pivotal $\mathcal{A}-$module category with fusion/principal graph $\Lambda$ whose vertices are simple objects of $\mathcal{M}$. We can define the vertex weighting for simple object $P$ as $\nu(P):=\Tr_P(\id_P)$. 
\end{remark}  

\begin{remark}\label{Remark:C(Lambda,nu)}
Note that $\mathcal{M}$ being a pivotal $\mathcal{A}-$module is equivalent to the dagger tensor functor $\mathcal{A}\to \End^\dagger(\mathcal{M})$ being pivotal \cite[Thm.~3.70]{GMPPS18},
so that its essential image $\End^\dagger_0(\mathcal{M},F)$ has a unitary pivotal structure from the pivotal structure in $\mathcal{A}$, 
where $F=-\lhd X$ is the generator. 
We also denote the corresponding 2-subcategory of \textsf{BigHilb} as $\mathcal{C}(K,\phi)$ or $\mathcal{C}(\Lambda,\nu)$.
\end{remark}
 
\subsection{The module embedding theorem}
Jones' planar algebra, as a form of standard invariant, is a method to construct and classify  finite index type $\text{II}_1$ subfactors. 
The module embedding theorem has been known to Vaughan Jones since he first defined the graph planar algebra \cite{Jo00}. 
The proof for finite depth case appears in \cite{JP10,CHPS18,GMPPS18}. 
Many nontrivial examples of subfactors have been constructed inspired by this theorem, including
the Extended Haagerup subfactor and its relatives 
\cite{BPMS12,GMPPS18}.

In our setting, the bipartite graph $\Lambda$ can be infinite depth. We refer the reader to \cite{Bu10} for the definition of the infinite depth bipartite graph planar algebra.

\begin{theorem}
The planar algebra constructed from
$\End^\dag_0(\mathcal{M},F)$ with generator $F$
mentioned in Remark \ref{Remark:C(Lambda,nu)}
is isomorphic to the graph planar algebra of bipartite graph $\Lambda$, where $\mathcal{M}$ is an indecomposable semisimple pivotal $\rm C^*$ $\mathcal{A}-$module category,
$\mathcal{A}$ is a 2-shaded rigid $\rm C^*$ multitensor category with generator $X=1^+\otimes X\otimes 1^-$,
$\Lambda$ is the (possibly infinite) fusion graph for $\mathcal{M}$ with respect to the generator $X$,
where the vertex weighting $\nu$ on $\Lambda$ comes from the trace $\Tr^{\mathcal{M}}$ as in Remark \ref{Rmk:fusion graph}.
\end{theorem}
\begin{proof}
Here we provide the sketch of the proof. 
From the unitary pivotal dagger functor $\mathcal{A}\to \End^\dagger(\mathcal{M})$, we obtain a rigid $\rm C^*$ tensor category $\End^\dagger_0(\mathcal{M},F)$ with pivotal structure with generator $F=-\lhd X$ in the sense of \S\ref{C(K) End(M)}.

According to \S\ref{C(K) End(M)} and \S\ref{2-Hilb and edge weighting},
from $\End^\dag_0(\mathcal{M},F)$, we can construct
the 2-category $\mathcal{C}(\Lambda,\nu)$ discussed in Remark \ref{Remark:C(Lambda,nu)}
with its generating $\mathsf{Hilb}$-enriched graph $\Lambda$,
which is equivalent information.
Similar to \cite[\S3.5.3]{GMPPS18}, the planar algebra of $\mathcal{C}(\Lambda, \nu)$ with generator $\Lambda$ is $*$-isomorphic to the graph planar algebra $\mathcal{G}_\bullet$ (in the sense of Burstein \cite{Bu10}) of the fusion graph $\Lambda$ with vertex weighting $\nu$, which corresponds to $F$ in the sense of Remark \ref{Rmk:fusion graph}.

Note that there is a well-know correspondence between \cite{Gh11,DGG14,Pe18}: 
\[\left\{\, \parbox{3cm}{\rm Subfactor planar algebras $\mathcal{P}_\bullet$} \,\right\}\ \, \cong\ \, \left\{\, \parbox{9cm}{\rm Pairs $(\mathcal{A}, X)$ with $\mathcal{A}$ a 2-shaded rigid ${\rm C}^*$ multitensor category with a generator $X$, i.e., $1_\mathcal{A}=1^+\oplus 1^-$, $1^+,1^-$ are simple and $X=1^+\otimes X\otimes 1^-$ } \,\right\}\]

Finally, the pivotal dagger tensor functor $\mathcal{A}\to \End^\dagger_0(\mathcal{M},F)$ gives a planar algebra embedding from the subfactor planar algebra $\mathcal{A}_\bullet$ to the graph planar algebra $\mathcal{G}_\bullet$ of its principal graph.
\end{proof}

If we choose $\mathcal{M}=\mathcal{A}_+=1^+\otimes \mathcal{A}\otimes 1^+$ to be the $\mathcal{A}-$module category, we obtain the module embedding theorem:

\begin{corollary}
Every subfactor planar algebra $\mathcal{P}_\bullet$ embeds into the graph planar algebra of its principal graph.
\end{corollary}

\subsection{Tracial Markov lattices and pivotal bimodule categories}\label{Tracial Markov lattice}

\begin{definition}
Let $\mathcal{C},\mathcal{D}$ be rigid ${\rm C}^*$ (multi)tensor categories with canonical unitary dual functors respectively. We call $M$ a semisimple pivotal ${\rm C}^*$ $\mathcal{C}-\mathcal{D}$ bimodule category, if there exists a pivotal trace $\tr^{\mathcal{M}}$ compatible with the spherical structures in $\mathcal{C}$ and $\mathcal{D}$, i.e.,
\begin{align*}
    \tr^{\mathcal{M}}_{a\rhd m}(f) &= \tr^{\mathcal{M}}_{m}(( \ev^\dagger_a\rhd \id_m)\circ(\id_{\overline{a}}\rhd f)\circ (\ev_a\rhd \id_m))\\
    \tr^{\mathcal{M}}_{m\lhd b}(f) &= \tr^{\mathcal{M}}_{m}((\id_m\lhd \coev^\dagger_b)\circ(f\lhd \id_{\overline{b}})\circ (\id_m\lhd \coev_b)),
\end{align*}
for $f\in \End(a\rhd m\lhd b)$, where $m\in\mathcal{M},\ a\in \mathcal{C},\  b\in\mathcal{D}$.
\end{definition}

\begin{definition}(Tracial Markov lattice)
We call $M$ a tracial Markov lattice if $M$ is a Markov lattice equipped with a unital trace $\tr$ on $\bigcup_{i,j\ge 0}M_{i,j}$ and the conditional expectation $E^{M,l}_{i,j},E^{M,r}_{i,j}$ are trace-preserving, i.e.,
$$\tr\circ E^{M,l}_{i,j}=\tr,\qquad\qquad \tr\circ E^{M,r}_{i,j}=\tr$$
on $M_{i,j}$, $i,j\ge 0$.
\end{definition}

\begin{definition}
We call $M$ a tracial standard $A-B$ bimodule, where $A,B$ are standard $\lambda$-lattices, if $\tr_M|_A=\tr_A,\ \tr_M|_B=\tr_B$ and $M$ is a standard $A-B$ bimodule, see Definition \ref{def ML std bimod}. 
\end{definition}

Similar to Theorem \ref{Thm:tracial MT and pivotal module category}, we have the following theorem:
\begin{theorem}There is a bijective correspondence between equivalence classes of the following:
\[\left\{\, \parbox{5.3cm}{\rm Tracial Markov lattice $M=(M_{i,j})_{i,j\ge 0}$ with $\dim(M_{0,0})=1$ as a standard $A-B$ bimodule over standard $\lambda$-lattices $A,B$} \,\right\}\ \, \cong\ \, \left\{\, \parbox{7.3cm}{\rm Pairs $(\mathcal{M}, Z)$ with $\mathcal{M}$ an indecomposable semisimple ${\rm C}^*$ pivotal $\mathcal{A}-\mathcal{B}$ bimodule category together with a choice of simple base object $Z= 1^+_\mathcal{A}\rhd Z\lhd 1^+_\mathcal{B}$} \,\right\}\]
Equivalence on the left hand side is the trace-preserving $*$-isomorphism on the tracial Markov lattice as standard $A-B$ bimodule; equivalence on the right hand side is the pivotal unitary $\mathcal{A}-\mathcal{B}$ bimodule equivalence between their Cauchy completions which maps the simple base object to simple base object.
\end{theorem}

Let us look at the balanced $(d_0,d_1)$-fair square-partite graph $(\Lambda,\omega)$ from the tracial Markov lattice $M$. Similar to the tracial Markov tower case, the edge-weighting comes from the vertex-weighting. To be precise,
\begin{align*}
    \text{For }P\in V_{00}\sqcup V_{01},\qquad\qquad \sum_{\{e:P\to Q:Q\in V_{10}\sqcup V_{11}\}}\nu(Q) &= d_0\cdot \nu(P)\\
    \text{For }P\in V_{00}\sqcup V_{01},\qquad\qquad \sum_{\{e:P\to Q:Q\in V_{01}\sqcup V_{11}\}}\nu(Q) &= d_1\cdot \nu(P).
\end{align*}


\begin{remark}
As for the biunitary connection, the computation does not change at all. In fact, the biunitary connection is independent of the pivotal structure, see Proposition \ref{prop of BC}(2) and \S\ref{C(Phi) and End(M,F,G)}. This now agrees with the usual definition of biunitary connection for the tracial/pivotal case discussed in \cite{JS97,EK98,MPPS12,MP14}.
\end{remark}

{\footnotesize{

}}


\begin{thebibliography}{GMPPS18}
\bibitem[AV15]{AV15}
Arano, Yuki; Vaes, Stefaan. \textit{$\rm C^*$-tensor categories and subfactors for totally disconnected groups.} Operator algebras and applications—the Abel Symposium 2015, 1--43, Abel Symp., 12, Springer, 2017. \mathscinet{MR3837590}
\bibitem[Ba97]{Ba97}
Baez, John C. \textit{Higher-dimensional algebra. II. $2$-Hilbert spaces.} Adv. Math. 127 (1997), no. 2, 125--189. \mathscinet{MR1448713}
\bibitem[BDH14]{BDH14}
Bartels, Arthur; Douglas, Christopher L.; Henriques, André. \textit{Dualizability and index of subfactors.} Quantum Topol. 5 (2014), no. 3, 289--345. \mathscinet{MR3342166}
\bibitem[Bi97]{Bi97}
Bisch, Dietmar. \textit{Bimodules, higher relative commutants and the fusion algebra associated to a subfactor.} Operator algebras and their applications (Waterloo, ON, 1994/1995), 13--63, Fields Inst. Commun., 13, Amer. Math. Soc., Providence, RI, 1997. \mathscinet{MR1424954}
\bibitem[BPMS12]{BPMS12}
Bigelow, Stephen; Peters, Emily; Morrison, Scott; Snyder, Noah. \textit{Constructing the extended Haagerup planar algebra}. Acta Math. 209 (2012), no. 1, 29--82. \mathscinet{MR2979509}
\bibitem[Bu10]{Bu10}
Burstein, R. D. \textit{Automorphisms of the bipartite graph planar algebra.} J. Funct. Anal. 259 (2010), no. 9, 2384--2403. \mathscinet{MR2674118}
\bibitem[CHPS18]{CHPS18}
Coles, Desmond, Peter Huston, David Penneys, and Srivatsa Srinivas. \textit{The module embedding theorem via towers of algebras.}, 2018, arXiv preprint \arXiv{1810.07049}
\bibitem[DGG14]{DGG14}
Das, Paramita; Ghosh, Shamindra Kumar; Gupta, Ved Prakash. \textit{Perturbations of planar algebras}. Math. Scand. 114 (2014), no. 1, 38--85. \mathscinet{MR3178106}
\bibitem[DY15]{DY15}
De Commer, Kenny; Yamashita, Makoto. \textit{Tannaka-Kreĭn duality for compact quantum homogeneous spaces II. Classification of quantum homogeneous spaces for quantum $\rm SU(2)$.} J. Reine Angew. Math. 708 (2015), 143--171. \mathscinet{MR3420332}
\bibitem[EGNO15]{EGNO15}
Pavel Etingof, Shlomo Gelaki, Dmitri Nikshych, and Victor Ostrik, \textit{Tensor categories}, Mathematical Surveys and Monographs, vol. 205, American Mathematical Society, Providence, RI, 2015, \mathscinet{MR3242743}
\bibitem[EK98]{EK98}
Evans, David E.; Kawahigashi, Yasuyuki. \textit{Quantum symmetries on operator algebras.} Oxford Mathematical Monographs. Oxford Science Publications. The Clarendon Press, Oxford University Press, New York, 1998. xvi+829 pp. ISBN: 0-19-851175-2 \mathscinet{MR1642584}
\bibitem[FL02]{FL02}
Frank, Michael; Larson, David R. \textit{Frames in Hilbert $C^\ast$-modules and $C^\ast$-algebras.} J. Operator Theory 48 (2002), no. 2, 273--314. \mathscinet{MR1938798}
\bibitem[FP19]{FP19}
Ferrer, Giovanni; Hernández Palomares, Roberto. \textit{Classifying module categories for generalized Temperley-Lieb-Jones $*$-2-categories}. Internat. J. Math. 31 (2020), no. 4, 2050027, 30 pp. \mathscinet{MR4098904}
\bibitem[Gh11]{Gh11}
Ghosh, Shamindra Kumar. \textit{Planar algebras: a category theoretic point of view.} J. Algebra 339 (2011), 27--54. \mathscinet{MR2811311}
\bibitem[GHJ89]{GHJ89}
Goodman, Frederick M.; de la Harpe, Pierre; Jones, Vaughan F. R. \textit{Coxeter graphs and towers of algebras.} Mathematical Sciences Research Institute Publications, 14. Springer-Verlag, New York, 1989. {\rm x}+288 pp. ISBN: 0-387-96979-9 \mathscinet{MR0999799}
\bibitem[GLR85]{GLR85}
Ghez, P.; Lima, R.; Roberts, J. E. \textit{$W^\ast$-categories}. Pacific J. Math. 120 (1985), no. 1, 79--109. \mathscinet{MR0808930}
\bibitem[GMPPS18]{GMPPS18}
Grossman, Pinhas, Scott Morrison, David Penneys, Emily Peters, and Noah Snyder. \textit{The Extended Haagerup fusion categories}, 2018, arXiv preprint \arXiv{1810.06076}.
\bibitem[HP17]{HP17}
Hartglass, Michael; Penneys, David. \textit{${\rm C}^*$-algebras from planar algebras I: Canonical ${\rm C}^*$-algebras associated to a planar algebra.} Trans. Amer. Math. Soc. 369 (2017), no. 6, 3977--4019. \mathscinet{MR3624399}
%
\bibitem[Jo83]{Jo83}
Jones, Vaughan F. R. \textit{Index for subfactors}. Invent. Math. 72 (1983), no. 1, 1--25. \mathscinet{MR0696688}
\bibitem[Jo99]{Jo99}
Jones, Vaughan F. R. \textit{Planar algebras I}, 1999, \arXiv{math.QA/9909027} 
\bibitem[Jo00]{Jo00}
Jones, Vaughan F. R. \textit{The planar algebra of a bipartite graph.} Knots in Hellas '98 (Delphi), 94--117, Ser. Knots Everything, 24, World Sci. Publ., River Edge, NJ, 2000. \mathscinet{MR1865703}
\bibitem[Jo15]{Jo15}
Jones, Vaughan F. R. \textit{von Neumann algebras}, 2015, available at \url{https://math.vanderbilt.edu/jonesvf/VONNEUMANNALGEBRAS2015/VonNeumann2015.pdf}.
\bibitem[JP10]{JP10}
Jones, Vaughan F. R.; Penneys, David. \textit{The embedding theorem for finite depth subfactor planar algebras.} Quantum Topol. 2 (2011), no. 3, 301--337. \mathscinet{MR2812459}
\bibitem[JP17]{JP17}
Jones, Corey; Penneys, David. \textit{Operator algebras in rigid $\rm C^*$-tensor categories.} Comm. Math. Phys. 355 (2017), no. 3, 1121--1188. \mathscinet{MR3687214}
\bibitem[JP19]{JP19}
Jones, Corey; Penneys, David. \textit{Realizations of algebra objects and discrete subfactors.} Adv. Math. 350 (2019), 588--661. \mathscinet{MR3948170}
\bibitem[JS97]{JS97}
Jones, V.; Sunder, V. S. \textit{Introduction to subfactors.} London Mathematical Society Lecture Note Series, 234. Cambridge University Press, Cambridge, 1997. {\rm xii}+162 pp. ISBN: 0-521-58420-5 \mathscinet{MR1473221}
\bibitem[KW00]{KW00}
Kajiwara, Tsuyoshi; Watatani, Yasuo. \textit{Jones index theory by Hilbert $C^*$-bimodules and $K$-theory.} Trans. Amer. Math. Soc. 352 (2000), no. 8, 3429--3472. \mathscinet{MR1624182}
\bibitem[Li14]{Li14}
Liu, Zhengwei. \textit{A universal skein theory for subfactor planar algebras.}
Talk in Vanderbilt University Subfactor Seminar August 22, 2014.
Abstract available at \url{https://math.vanderbilt.edu/peters10/subfactor_seminar_fall_2014.html}
\bibitem[Lo89]{Lo89}
Longo, Roberto. \textit{Index of subfactors and statistics of quantum fields.} I. Comm. Math. Phys. 126 (1989), no. 2, 217--247. \mathscinet{MR1027496}
\bibitem[LR96]{LR96}
Longo, R.; Roberts, J. E. \textit{A theory of dimension.} $K$-Theory 11 (1997), no. 2, 103--159. \mathscinet{MR1444286}
\bibitem[MP14]{MP14}
Morrison, Scott; Peters, Emily. \textit{The little desert? Some subfactors with index in the interval $(5,3+\sqrt{5})$}. Internat. J. Math. 25 (2014), no. 8, 1450080, 51 pp. \mathscinet{MR3254427}
\bibitem[MPPS12]{MPPS12}
Morrison, Scott; Penneys, David; Peters, Emily; Snyder, Noah. \textit{Subfactors of index less than 5, Part 2: Triple points.} Internat. J. Math. 23 (2012), no. 3, 1250016, 33 pp. \mathscinet{MR2902285}
\bibitem[M\"u03]{Mu03}
M\"uger, Michael. \textit{From subfactors to categories and topology. I. Frobenius algebras in and Morita equivalence of tensor categories.} J. Pure Appl. Algebra 180 (2003), no. 1-2, 81--157. \mathscinet{MR1966524}
\bibitem[MW10]{MW10}
Scott Morrison and Kevin Walker, \textit{The graph planar algebra embedding}, 2010, preprint available at \href{http://tqft.net/gpa}{tqft.net/gpa}
\bibitem[Oc88]{Oc88}
Ocneanu, Adrian. \textit{Quantized groups, string algebras and Galois theory for algebras.} Operator algebras and applications, Vol. 2, 119--172, London Math. Soc. Lecture Note Ser., 136, Cambridge Univ. Press, Cambridge, 1988. \mathscinet{MR0996454}
\bibitem[PP86]{PP86}
Pimsner, Mihai; Popa, Sorin. \textit{Entropy and index for subfactors.} Ann. Sci. École Norm. Sup. (4) 19 (1986), no. 1, 57--106. \mathscinet{MR0860811}
\bibitem[Pe18]{Pe18}
Penneys, David. \textit{Unitary dual functors for unitary multitensor categories}, 2018, arXiv preprint \arXiv{1808.00323}
\bibitem[Po94]{Po94}
Popa, Sorin. \textit{Classification of amenable subfactors of type II}. Acta Math. 172 (1994), no. 2, 163--255. \mathscinet{MR1278111}
\bibitem[Po95]{Po95}
Popa, Sorin. \textit{An axiomatization of the lattice of higher relative commutants of a subfactor.} Invent. Math. 120 (1995), no. 3, 427--445. \mathscinet{MR1334479}
\bibitem[RV16]{RV16}
Reutter, David J.; Vicary, Jamie. \textit{Biunitary constructions in quantum information.} High. Struct. 3 (2019), no. 1, 109--154. \mathscinet{MR3939047}
\bibitem[Sc13]{Sc13}
Schaumann, Gregor. \textit{Traces on module categories over fusion categories}. J. Algebra 379 (2013), 382--425. \mathscinet{MR3019263}
\bibitem[Se11]{Se11}
Selinger, P. \textit{A survey of graphical languages for monoidal categories.} New structures for physics, 289--355, Lecture Notes in Phys., 813, Springer, Heidelberg, 2011. \mathscinet{MR2767048}
\bibitem[Wa90]{Wa90}
Watatani, Yasuo. \textit{Index for $\rm C^*$-subalgebras.} Mem. Amer. Math. Soc. 83 (1990), no. 424, {\rm vi}+117 pp. \mathscinet{MR0996807}
\end{thebibliography}
\end{document}